\documentclass[12pt,oneside,a4paper]{article}
\usepackage{amsmath,amsthm, amssymb}
\usepackage{hyperref}

\usepackage[left=3.0cm,right=3.0cm,top=2.0cm,bottom=2.0cm]{geometry}
\usepackage[utf8]{inputenc}
\usepackage{esint}
\usepackage[expansion=false]{microtype}

\newtheoremstyle{mytheorem}
  {3pt}
  {3pt}
  {\itshape}
  {}
  {\bfseries}
  {.}
  {1em}
  {}

\newtheorem{theoremA}{Theorem}

\newtheorem{theoremC}{Theorem}

\newtheorem{theoremD}{Theorem}

\newtheorem{theoremE}{Theorem}

\newtheorem{theoremF}{Theorem}

\newtheorem{CorollaryX}{Corollary}

\newtheorem{definition}{Definition}[section]
\newtheorem{proposition}[definition]{Proposition}

\theoremstyle{mytheorem}
\newtheorem{theorem}[definition]{Theorem}

\newtheorem{lemma}[definition]{Lemma}

\newtheorem{corollary}[definition]{Corollary}

\newtheorem{question}[definition]{Question}

\theoremstyle{remark}
\newtheorem{remark}[definition]{Remark}

\theoremstyle{definition}

\begin{document}
\title{Locally Conformally Flat Manifolds with Positive Scalar Curvature:
Kleinian Groups, Moduli Spaces, and Euclidean Rigidity}
\author{Jialong Deng \thanks{Jialongdeng@gmail.com}}
\date{}

\maketitle
\begin{abstract}
We study locally conformally flat (LCF) Riemannian manifolds with nonnegative or positive scalar curvature (PSC), using the conformal boundary of the developing image. For closed oriented LCF $n$-manifolds with PSC and infinite fundamental group, $n\ge5$, we bound the macroscopic dimension of their Riemannian universal covers by $\lfloor(n-1)/2\rfloor$, establish the existence of a nontrivial homotopy group above the middle dimension, and, under mild additional hypotheses, bound the Hausdorff dimension of the limit sets of their Kleinian groups in the interval $(1,(n-2)/2)$.
 In particular, no closed aspherical manifold admits an LCF metric with PSC. If the scalar curvature is at least $n(n-1)$ and the manifold is not isometric to the round sphere, then every smooth nonzero-degree map to the sphere expands somewhere when its Kleinian group is elementary, while the developing map expands somewhere whenever the fundamental group is infinite. We prove that the space of LCF metrics with PSC and its moduli space are contractible in dimension three for finite fundamental group and for $S^2\times S^1$, and that the moduli space is empty or contractible for smooth manifolds homeomorphic to spherical space forms but not diffeomorphic to $S^n$ in dimensions $n\ge4$. For complete open simply connected LCF manifolds of nonnegative scalar curvature, we obtain Euclidean rigidity under additional topological hypotheses at infinity (and, in dimension three, from vanishing second homology alone). We also show that, in dimensions $n\ge4$, the Euclidean conclusion can fail when neither of the two additional topological hypotheses is assumed: we construct complete contractible examples of PSC that are not homeomorphic to $\mathbb{R}^n$.

\end{abstract}

\tableofcontents
\section{Introduction}
This paper is the second in a project
\cite{2025arXiv251213528D} devoted to the geometry of
locally conformally flat (LCF) Riemannian manifolds with positive scalar
curvature (PSC).  By Kuiper's theorem \cite{zbMATH03061883}, every closed,
oriented, simply connected LCF \(n\)-manifold, \(n\ge 3\), is conformally
diffeomorphic to the round sphere \((S^n,g_{st})\).  The case of
nontrivial fundamental group is considerably more flexible.  Whereas
\cite{2025arXiv251213528D} treats LCF \(4\)-manifolds with negative scalar curvature, the present paper treats
LCF \(n\)-manifolds in all dimensions \(n\ge 3\), under nonnegative or positive
scalar curvature assumptions, and develops their geometric, topological, and
metric consequences.  The results of this paper are driven by a single
mechanism: scalar curvature imposes quantitative control on the conformal
boundary of the developing image, and this control is translated into global
rigidity, dimension bounds, and topological restrictions by means of Kleinian
group theory, shape theory, potential theory, capacity, and topology.  The
control itself appears through Hausdorff dimension bounds for the limit set,
boundary behavior of conformal factors, and Patterson--Sullivan theory.

In dimensions $n\ge 4$, Schoen and Yau~\cite{zbMATH04075988} prove
that if an LCF manifold $(M^n,g)$ has PSC, then its developing map into $(S^n,g_{st})$
is injective, without any assumption on the fundamental group.  With
its image denoted by $\Omega\subset S^n$, the associated holonomy representation
identifies $\pi_1(M)$ with a Kleinian group $\Gamma$, yielding a
realization $ M^n\;\cong\;\Omega/\Gamma.$
With $\Lambda(\Gamma):=S^n\setminus\Omega$, the limit set of
$\Gamma$, the Schoen--Yau theorem further gives
\begin{equation}\label{eq:SY-upper}
        \dim_{\mathcal H}\bigl(\Lambda(\Gamma)\bigr)\;\le\;\frac{n-2}{2}.
\end{equation}

These results naturally lead to a problem posed by Yau in his 1990
problem list~\cite[Problem~2]{MR1216573}, which asks for a description
of the possible algebraic and geometric structures of the fundamental
group $\pi_1(M)$ of a closed LCF manifold.  They are also related to
another problem in Yau's problem list~\cite[Problem~36]{MR1216573},
which asks whether one can characterize the domain $\Omega$.  Yau
remarks that, in the case where the scalar curvature can be made
constant, the limit set $\Lambda(\Gamma)$ coincides with the singular
set of the natural Yamabe equation.

On the other hand, using the Gauss--Bonnet--Chern formula, Gursky
\cite{zbMATH00750638} shows that if \(n=4\) or \(6\)
and \(H_1(M;\mathbb Z)=0\), then
\((M^n,g)\) is conformally equivalent to the round sphere
\((S^n,g_{st})\).  He asks whether this rigidity phenomenon
persists in other dimensions \(n\ge3\).  In higher dimensions, however,
the Gauss--Bonnet--Chern integrand becomes less directly effective for
extracting scalar curvature rigidity.  We therefore take a different
route: Patterson--Sullivan theory for the Kleinian group
\(\Gamma\) converts the size of the limit set into algebraic information
about \(\Gamma\), and this yields a mechanism available in dimensions
\(n\geq5\).

When the fundamental group is finite, we answer Gursky's question affirmatively when \(n\not\equiv3\pmod4\) and negatively when \(n\equiv3\pmod4\); see Theorem~\ref{nLCF}.
In the four-dimensional case, the author \cite[p.~12]{2025arXiv251213528D} gives an alternative proof of Gursky's result using the classification of closed, LCF $4$-manifolds with PSC.

The most intriguing cases of Yau's two problems and Gursky's question arise when the fundamental group is infinite and $n \ge 5$. Our main result in this direction is the following theorem.

\begin{theoremA}\label{thm:intro-A}
Let $(M^n,g)$, $n\geq5$, be a closed, connected, oriented, locally conformally
flat Riemannian manifold with positive scalar curvature and infinite
fundamental group.  Then there exists an integer
$i\ge\lfloor n/2\rfloor+1$ such that $\pi_i(M^n)\ne 0$, and its universal Riemannian cover $(\widetilde M,\widetilde g)$ has macroscopic dimension at most $\left\lfloor\frac{n-1}{2}\right\rfloor$.  
 
Let $M^n\cong\Omega/\Gamma$ be the Kleinian realization given by the
 developing map.  If, in addition,
$H_1(M^n;\mathbb{Z})=0$, then $\Gamma$ is
non-elementary.  Under these additional assumptions, if $\Gamma$ is
torsion-free, then the Hausdorff dimension of its limit set satisfies
\begin{equation}\label{eq:HD-intro}
        1\;<\;\dim_{\mathcal H}\!\bigl(\Lambda(\Gamma)\bigr)\;<\;\frac{n-2}{2}.
\end{equation}
\end{theoremA}

Theorem~\ref{thm:intro-A} gives a partial answer to Yau's Problems~2 and~36 and to
Gursky's question.  The lower bound in~\eqref{eq:HD-intro} is the new
content: combined with~\eqref{eq:SY-upper}, it traps
$\dim_{\mathcal H}(\Lambda(\Gamma))$ in the interval $(1,(n-2)/2)$
under the hypotheses $H_1(M;\mathbb{Z})=0$ and $\Gamma$ torsion-free.
This constrains the geometry of $\Omega$ (addressing Problem~36) and
exhibits nontrivial algebraic structure of $\pi_1(M)$ (addressing
Problem~2).   In particular, since $\dim_{\mathcal H}(\Lambda(\Gamma))>1$
implies $\Lambda(\Gamma)\ne\varnothing$ and hence $M\not\cong S^n$
conformally, Gursky's rigidity fails when $\pi_1(M)$ is infinite. Using  Schoen and Yau's Liouville theorem \cite{MR4836036}, the proof can be generalized to the
case of nonnegative scalar curvature, see Corollary \ref{nonnegative}.  
 
Theorem~\ref{thm:intro-A} also verifies Gromov's macroscopic dimension conjecture~\cite{zbMATH00867495} for LCF manifolds. It asserts that if $(M^n,g)$ is a closed Riemannian manifold with PSC, then the macroscopic dimension of its universal cover is at most $n-2$. The conjecture is established for $3$-manifolds by Bolotov~\cite{zbMATH01985808}, while in higher dimensions only partial results are available; see, for example,~\cite{zbMATH06371564,zbMATH08005926}.

The first conclusion follows by combining Schoen--Yau's low degree homotopy
vanishing and limit set estimate~\cite{zbMATH04075988} with Izeki's
convex cocompactness theorem~\cite{zbMATH01785298} and Kapovich's formula for
cohomological dimension~\cite[Proposition~9.6]{MR2491697}.
Non-elementarity of $\Gamma$ under $H_1(M^n;\mathbb{Z})=0$ is
established by a separate algebraic argument (Lemma~\ref{Abel} in
the body).  For the Hausdorff dimension bound, Sullivan's
theorem~\cite{zbMATH03903608} gives
$\delta(\Gamma)=\dim_{\mathcal H}(\Lambda(\Gamma))$ for non-elementary
convex cocompact groups.  To obtain the strict lower bound, set
$d=\dim_{\mathrm{top}}\Lambda(\Gamma)$.  If $d=0$, Kapovich's
formula~\cite[Proposition~9.6]{MR2491697} gives
$\operatorname{cd}_{\mathbb Z}(\Gamma)=1$, and the Stallings--Swan theorem,
in the form recorded by Kapovich \cite[Theorem~2.5]{MR2491697}, then implies
that $\Gamma$ is free.  If $d=1$ and
$\dim_{\mathcal H}\Lambda(\Gamma)\le1$, Kapovich's rigidity theorem
\cite[Theorem~1.3]{MR2491697} shows that the limit set is a round circle;
torsion-freeness and convex cocompactness then identify $\Gamma$ with a
closed hyperbolic surface group.  Both alternatives contradict
$H_1(M^n;\mathbb Z)=0$.  Nayatani's criterion
\cite[Corollary~3.4]{zbMATH01028179} supplies the strict upper bound.  For
the macroscopic dimension, the proof realizes \(\pi_1(M)\) as a convex cocompact Kleinian group whose limit set, by PSC, has strictly bounded Hausdorff dimension, the floor function arising from integrality of topological
dimension.  The geometric action on the convex hull identifies this limit set with
the Gromov boundary, so a boundary formula computes the asymptotic dimension, which
quasi-isometry invariance transfers to the universal cover and which bounds the
macroscopic dimension.

An immediate topological consequence of Theorem~\ref{thm:intro-A}, which also plays a
role in the proof of Theorem~\ref{Lipschitz} below, is the following.

\begin{CorollaryX}[Corollary \ref{LCF aspherical}]\label{CorollaryB}
A closed, connected, oriented, aspherical manifold of dimension $n\geq5$ does not admit a
locally conformally flat metric with positive scalar curvature.
\end{CorollaryX}

Beyond these topological consequences, conformally flat geometries also arise
as backgrounds for rigid supersymmetry in gauged-supergravity constructions
\cite{zbMATH06267307}. This connection
provides additional motivation for studying their global rigidity. We next
investigate pointwise Lipschitz properties of natural maps associated with LCF
manifolds of PSC.

\begin{theoremC}\label{Lipschitz}
Let $(M^{n}, g)$, $n \ge 4$, be a closed, connected, oriented, locally conformally flat manifold with scalar curvature
$\mathrm{Sc}_{g} \ge n(n-1)$, and assume that $(M^{n}, g)$ is not isometric to $(S^{n}, g_{st})$.
View its fundamental group $\pi_1(M)$ as a Kleinian group.
\begin{enumerate}
\item[(1)] If $\pi_1(M)$ is elementary, then every smooth map of
nonzero degree $f\colon(M^{n},g)\to (S^{n},g_{st})$ has a
point $p\in M$ at which the pointwise Lipschitz constant of $f$ is
strictly greater than~$1$.
\item[(2)] If $\pi_1(M)$ is infinite (whether elementary or non-elementary),
then for the developing map
$\mathrm{dev} \colon (\widetilde{M}, \tilde g) \to (S^{n}, g_{st})$,
there exists a point $x \in \widetilde{M}$ such that the pointwise Lipschitz constant satisfies
$\mathrm{Lip}_{x}(\mathrm{dev}) > 1$.
\end{enumerate}
\end{theoremC}

For assertion~\textup{(1)}, suppose that $\pi_1(M)$ is elementary. If the
pointwise Lipschitz constant were
everywhere bounded above by~$1$, then Proposition~\ref{Llarull} would imply
that $(M^{n}, g)$ is isometric to $(S^{n}, g_{st})$, a contradiction.
For assertion~\textup{(2)}, suppose that $\pi_1(M)$ is infinite. Injectivity
of the developing map identifies $\widetilde M$ conformally with a proper
domain $\Omega\subsetneq S^n$. With
$(\mathrm{dev}^{-1})^*\tilde g=u^{4/(n-2)}g_{st}$ on $\Omega$,
Proposition~\ref{minimum} gives $\min_\Omega u<1$. Since the pointwise
Lipschitz constant of $\mathrm{dev}$ at the point corresponding to
$x\in\Omega$ is $u(x)^{-2/(n-2)}$, the conclusion follows.

The developing-map framework used in
Theorems~\ref{thm:intro-A} and~\ref{Lipschitz} also provides information
about the topology of spaces and moduli spaces of locally conformally flat
metrics with positive scalar curvature. Questions concerning the topology and rigidity of such moduli spaces have recently arisen in dimension three.
Moroianu, Pino Carmona, and
Shahbazi\footnote{I thank Carlos Shahbazi for bringing their beautiful paper to my attention.}
\cite{2026arXiv260303272M} show that every closed three-dimensional
heterotic soliton with vanishing torsion which is LCF and has constant scalar curvature is rigid, in the sense that it is an isolated point of the
corresponding moduli space.  In a different direction, Bamler and Kleiner
\cite{2019arXiv190908710B} prove, using Ricci flow through singularities,
that if \(M^3\) is a closed, connected, oriented smooth three-manifold, then
the space of positive scalar curvature metrics on \(M^3\) is either empty or
contractible.    They ask in Question~1.8 whether the same conclusion holds for the space of conformally flat PSC metrics.

  In dimension $3$, the notion of conformal flatness used
in their paper coincides with the notion of LCF used in our work:
a Riemannian metric is LCF if and only if its Cotton tensor vanishes,
equivalently, if its Schouten tensor satisfies the Cotton--York
condition; in dimensions $n\ge 4$ this is equivalent to the vanishing
of the Weyl tensor.  Ricci flow does not preserve local conformal
flatness in dimension $3$, so the Bamler--Kleiner approach is
unavailable in the LCF setting.  They also note that such a contractibility statement would imply Theorem~1.6 and Corollary~1.7 of their paper.

\begin{theoremD}\label{thm:intro-D}
Let $M^n$ be a connected, closed, oriented smooth $n$-manifold. Denote by $X(M^n)$ the space of locally conformally flat metrics with positive scalar curvature on $M^n$, and by $\mathcal{M}(M^n)$ its moduli space. Then:
\begin{enumerate}
\item[(i)] If $n = 3$ and either $\pi_1(M^3)$ is finite or $M^3$ is diffeomorphic to $S^2 \times S^1$, then both $X(M^3)$ and $\mathcal{M}(M^3)$ are contractible.
\item[(ii)] If $n \ge 4$ and $M^n$ is homeomorphic to a spherical space form but not diffeomorphic to $S^n$, then $\mathcal{M}(M^n)$ is either empty or contractible.
\end{enumerate}
\end{theoremD}

Theorem~\ref{thm:intro-D}(i) gives a positive answer to Bamler and Kleiner's
Question~1.8 for $3$-manifolds with finite fundamental group, and
extends it to the case $M^3\cong S^2\times S^1$. The empty alternative in
Theorem~\ref{thm:intro-D}(ii) is essential. For example, if $M$ is an exotic sphere,
then Kuiper's theorem implies that $X(M)=\varnothing$, since an LCF metric on
$M$ would force $M$ to be diffeomorphic to $S^n$. The
author shows~\cite{zbMATH07368018} that the conformal class of a
metric is dense in the moduli space in the Gromov--Hausdorff topology;
whether $\mathcal{M}(M^n)$ is contractible in that topology remains
open.

For $n=3$ with finite fundamental group, applying the normalization argument
of Bamler and Kleiner~\cite[Lemma~9.2(a)]{2019arXiv190908710B} on the universal
cover and then descending the unique minimizing conformal factor to $M$
produces a canonical $\mathrm{Diff}(M)$-equivariant normalized round
representative in each conformal class. Interpolation of the
corresponding conformal factors gives a strong deformation retraction, and
their Theorem~1.2 makes its target contractible. Equivariance allows the
deformation to descend to the moduli space, whose round locus is a point by
de Rham rigidity \cite{MR43468}. For $M\cong S^2\times S^1$, the same
interpolation retracts onto the locally cylindrical metrics, whose moduli
space is identified explicitly with $(0,\infty)\times[0,\pi]$. For $n\ge4$,
nonemptiness of $X(M)$ and Kuiper's theorem, followed by a Cartan fixed-point
argument, identify $M$ diffeomorphically with a spherical space form.
Marques' result \cite[Corollary~3.2]{MR2950765} supplies a positive constant sectional curvature representative in each
conformal class, whereas Obata's theorem \cite{zbMATH03374588} makes its
curvature one normalization unique. The implicit function theorem and the
strict Lichnerowicz--Obata spectral gap give continuity of this canonical
projection. Interpolation then yields the equivariant retraction, and de
Rham rigidity completes the proof.

Theorems~\ref{thm:intro-A}, \ref{Lipschitz}, and~\ref{thm:intro-D} use the developing map to study closed LCF  manifolds with PSC.  We next turn to complete noncompact manifolds.  In the simply
connected case, the holonomy is trivial, and the developing map realizes
\(M\) itself as a domain $ \Omega\subset S^n .$
The central object is then the conformal boundary $ \Lambda:=S^n\setminus\Omega .$
Since no nontrivial Kleinian group remains, the topology of \(M\) must be read directly from \(\Lambda\).  Shape theory provides the mechanism for doing this without assuming any regularity of \(\Lambda\).

This point of view is related to a classical rigidity problem for open
three-manifolds.  Gromov and Lawson
\cite[Corollary~10.9]{MR720933}, by minimal surface methods, show that a
complete noncompact \(3\)-manifold with uniformly positive scalar curvature and finitely
generated fundamental group is simply connected at infinity.  Since an open
contractible \(3\)-manifold which is simply connected at infinity is
homeomorphic to \(\mathbb R^3\) \cite{275662}, it follows that
any complete open contractible \(3\)-manifold with uniformly positive scalar
curvature is homeomorphic to \(\mathbb R^3\); see also the index-theoretic
proof in \cite[Theorem~1]{MR2721617}.  Motivated by the work of Gromov--Lawson, the following conjecture has circulated in the field for decades, but has only recently begun to receive attention: A complete, connected, open, contractible \(3\)-manifold with nonnegative scalar curvature is diffeomorphic to \(\mathbb{R}^3\).

Under additional geometric or topological assumptions, the conjecture is 
partially verified in recent work; see~\cite[p.~2]{2026arXiv260301887C}.
The LCF case requires new methods: the conformal boundary $\Lambda$ may be a wild compact set with no smoothness or
rectifiability, so the minimal surface and index theoretic methods  do not apply.  We introduce shape-theoretic methods to
extract topological information on $\Lambda$ directly from the global
topology of $M$, without imposing any regularity on $\Lambda$.  To the
best of our knowledge, this is the first use of shape theory to study
developing map limit sets in the scalar curvature setting.  Furthermore, the dictionary relating properties (algebraic, coarse, etc.) of Kleinian groups to the global and local shape of their limit sets, as well as to the homology and end topology of the associated domains of discontinuity, will be developed in a subsequent paper in this  project~\cite{DengKleinianScalarShape}.

\begin{theoremE}\label{thm:intro-E}
Let \((M^n,g)\), \(n \ge 3\), be a complete, open, simply connected, locally
conformally flat \(n\)-manifold with nonnegative scalar curvature. Let \(\Phi:M\to \Phi(M)=:\Omega\subset S^n\) be the developing map, and set
\(\Lambda:=S^n\setminus\Omega\). Assume that one of the following holds:
\begin{enumerate}
\item[(i)] \(n=3\) and \(H_2(M;\mathbb Z)=0\);
\item[(ii)] \(n \ge 4\), \(M\) is \(\left\lfloor \frac{n-2}{2}\right\rfloor\)-connected at infinity, and \(\Lambda\) satisfies the small loops condition in \(S^n\).
\end{enumerate}
Then \(M\) is homeomorphic to \(\mathbb R^n\).
\end{theoremE}

The proof proceeds as follows. By the Liouville theorem of Schoen and Yau,
\(M^n\) embeds as a domain \(\Omega\subset S^n\). The Ma--Qing
estimate~\cite[Theorem~1.3]{zbMATH08016940} yields
$\dim_{\mathcal H}(\Lambda)\leq\frac{n-2}{2}$ and
$\dim(\Lambda)\leq\left\lfloor\frac{n-2}{2}\right\rfloor=:d.$
When \(n=3\), \v{C}ech--Alexander duality and the hypothesis
\(H_2(M;\mathbb Z)=0\) imply that \(\Lambda\) is connected; hence it is a
single point. Suppose that \(n\geq4\). The \(d\)-connectedness at infinity
implies that \(\Omega\) is one-ended. The non-separation theorem then
identifies the ends of \(\Omega\) with the components of \(\Lambda\), so
\(\Lambda\) is connected. The connectivity at infinity, combined with the
small loops condition and the dimension bound, gives the vanishing of the
homotopy pro-groups of \(\Lambda\). Morita's
finite-dimensional Whitehead theorem~\cite[Theorem~1.2, p.~394]{zbMATH03497097}
therefore implies that \(\Lambda\) has the shape of a point and hence has
Property \(UV^\infty\). The same pro-triviality at infinity also yields the
cellularity criterion. Repov\v{s}'s
theorem~\cite[Theorem, p.~564]{zbMATH04011439} for \(n=4\), and McMillan's
cellularity criterion in its cell-like form
\cite[Theorem~3.2.3, p.~107]{zbMATH05624618} for \(n\geq5\), originating in
\cite[Theorem~1, p.~327]{zbMATH03190942}, then show that \(\Lambda\) is
cellular in \(S^n\). Consequently, collapsing \(\Lambda\) to a point produces
a space homeomorphic to \(S^n\), and hence \(\Omega\) is homeomorphic to
\(\mathbb R^n\).

Theorem~\ref{thm:intro-E}(i) settles the conjecture above for three-dimensional LCF
manifolds with nonnegative scalar curvature.  In dimension four,
Theorem~\ref{LCF-4-manifold} gives a differentiable strengthening in
which the hypothesis in Theorem~\ref{thm:intro-E}(ii) is replaced by
\(H_3(M;\mathbb R)=0\), bounded geometry, and uniformly positive scalar
curvature, and \(M\) is shown to be diffeomorphic to the standard
\(\mathbb R^4\). Beyond the
topological hypotheses, Corollary~\ref{Contractible for} identifies five
analytic and asymptotic-geometric conditions, each independently
implying Euclidean rigidity, thereby partially answering Yau's
Problem~36~\cite[Problem~36]{MR1216573}. On the other hand, the
Euclidean conclusion in Theorem~\ref{thm:intro-E}\textup{(ii)} can fail when neither of its additional
topological hypotheses is assumed, as Corollary~\ref{cor:sharpness} makes
precise; the following theorem supplies the underlying examples.

\begin{theoremF}\label{F}
For every $n\geq 4$, there exists a complete open contractible locally conformally flat
$n$-manifold with positive scalar curvature that is not homeomorphic to
$\mathbb{R}^n$. Moreover, for every $n\geq 6$, there exists such a manifold
with uniformly positive scalar curvature.
\end{theoremF}

Karakhanyan's theorem \cite[Theorem~1.1]{zbMATH07873599} motivates the
metric construction.  It characterizes the compact sets $K\subset S^n$
for which $S^n\setminus K$ admits a complete scalar-flat metric conformal
to the round metric: this occurs if and only if the corresponding
spherical Bessel capacity vanishes, in the sense of
\cite[Section~2.2]{zbMATH07873599}.

 The proof has complementary topological and analytic parts.  For $n\geq 4$,
we construct an increasingly thin nested tower whose bonding maps kill
homology while preserving fundamental-group information.  Its limit is a
\v{C}ech-acyclic compactum whose complement is contractible but not simply
connected at infinity, and hence is not Euclidean.  Analytically, the tower
is chosen so that the limit set has zero critical Bessel capacity.
Potential theory then provides a conformal factor whose blow-up along the
limit set yields a complete scalar-flat metric; constant shifts of this
factor produce complete LCF metrics with PSC.  For $n\geq 6$, we use
Newman's construction \cite{zbMATH03054974} (see also
\cite[Example~3.2.2]{zbMATH07206284}) to obtain an acyclic piecewise
linear (PL) $2$-complex whose
complement is contractible but not homeomorphic to $\mathbb{R}^n$,
and whose area measure is Ahlfors $2$-regular.  A conformal potential of
this measure is then chosen with the homogeneity dictated by the
conformal Laplacian: its blow-up gives completeness, while the
corresponding potential estimates yield scalar curvature bounded
between two positive constants.  The restriction $n\ge6$ is precisely
the condition that the $2$-dimensional spine lie at or below the
critical dimension $(n-2)/2$.

\paragraph*{Organization of the Paper.}  In Section~\ref{3}, for manifolds with finite fundamental group, we answer Gursky's question affirmatively when \(n\not\equiv3\pmod4\) and negatively when \(n\equiv3\pmod4\), and we obtain a strict lower bound for the Hausdorff dimension of the limit set of some infinite Kleinian group. In Section~\ref{4}, we show that, for certain maps, the pointwise Lipschitz constant is strictly greater than one at some point. In Section~\ref{5}, we prove contractibility of both the space of metrics and its moduli space in the specified three-dimensional cases, and prove the empty-or-contractible alternative for the moduli spaces of the higher-dimensional spherical-space-form manifolds appearing in Theorem~\ref{thm:intro-D}. In Section~\ref{Euclidean 1}, we prove Euclidean rigidity for complete open LCF manifolds with nonnegative scalar curvature under additional topological, analytic, or bounded geometry hypotheses. In Section~\ref{counterexample psc}, we construct complete LCF metrics with PSC on contractible manifolds.

\paragraph*{Acknowledgments.}          
This research was supported through the program   ``Oberwolfach Leibniz Fellows'' by the Mathematisches Forschungsinstitut Oberwolfach.   This work originates from a broader project initiated during my postdoctoral stay at the Yau Center. I thank Akito Futaki and Shing-Tung Yau for their support during that time, and this work was supported by the YMSC Overseas Shuimu Scholarship and NSFC 12401063, and partially supported by NSFC 12271284. Results were presented in seminars in 2025. I thank Gerhard Huisken, Thomas Schick, and Uwe Semmelmann for helpful discussions in person.

\paragraph*{Disclosure on AI assistance.}
The author developed the proof structure of Theorem~\ref{F}, while ChatGPT-5.6 Sol was used to assist in elaborating the detailed arguments. The author also used ChatGPT-5.6 Sol to polish and copy-edit this manuscript, including refining the presentation of some mathematical arguments, improving mathematical exposition and language, checking grammar and style, and providing typesetting assistance.
 The author assumes full responsibility for the final manuscript, including all mathematical claims and source attributions.

\section{Limit sets of Kleinian groups}\label{3}

The following theorem gives a sharp answer to Gursky's question when the fundamental group is finite: rigidity holds if \(n\not\equiv3\pmod4\), whereas counterexamples exist in every dimension \(n\equiv3\pmod4\).

\begin{theorem}\label{nLCF}
Let \( (M^n, g) \) be a closed, connected, oriented, smooth Riemannian manifold of dimension \( n \geq 3 \), which is locally conformally flat with positive scalar curvature. Assume further that \( H_1(M^n; \mathbb{Z}) = 0 \) and that the fundamental group \( \pi_1(M) \) is finite.

\begin{itemize}
\item[(i)] If \(n\) is even or \(n\equiv1\pmod4\), then \((M^n,g)\) is conformally equivalent to the standard sphere \((S^n,g_{st})\).

\item[(ii)] If \(n\equiv3\pmod4\), then \(\pi_1(M)\) is trivial or isomorphic to the binary icosahedral group \(I_{120}^*\cong\mathrm{SL}_2(\mathbb F_5)\); accordingly, \((M^n,g)\) is conformally equivalent to \((S^n,g_{st})\) or to a spherical space form \(S^n/\Gamma\) with \(\Gamma\cong I_{120}^*\). The congruence restriction is optimal: for every \(n\equiv3\pmod4\), both alternatives occur.
\end{itemize}
\end{theorem}

\emph{Proof outline.}
Kuiper's theorem and the Cartan fixed-point theorem identify \(M\)
conformally with a spherical space form \(S^n/\Gamma\). Since
\(H_1(M;\mathbb Z)=0\), the group \(\Gamma\) is perfect, so
Wolf's classification in Allcock's reformulation gives \(\Gamma=\{e\}\) or
\(\Gamma\cong I_{120}^*\); in the latter case, restriction to
\(Q_8\subset I_{120}^*\) equips \(\mathbb R^{n+1}\) with a quaternionic
vector-space structure and forces \(4\mid n+1\). Conversely, the diagonal
action of \(I_{120}^*\subset\operatorname{Sp}(1)\) on
\(S^{4r-1}\subset\mathbb H^r\) realizes the nontrivial alternative in every
admissible dimension, while the round sphere realizes the trivial one.

\begin{proof}
Let \(p\colon\widetilde M\to M\) be the universal covering, let
\(\widetilde g:=p^*g\), and denote the deck transformation group by \(\Pi\).
Since \(\Pi\cong\pi_1(M)\) is finite, \(p\) is finite-sheeted; hence
\(\widetilde M\) is closed. Moreover, \((\widetilde M,\widetilde g)\) is
simply connected and locally conformally flat. Kuiper's theorem, recalled in
the Introduction, therefore gives a conformal diffeomorphism
\(D\colon(\widetilde M,\widetilde g)\to (S^n,g_{st})\).
Each \(\delta\in\Pi\) is an isometry of \(\widetilde g\), because
\(p\circ\delta=p\). Thus \(\rho(\delta):=D\circ\delta\circ D^{-1}\)
defines an injective homomorphism
\(\rho\colon\Pi\to\operatorname{Conf}(S^n)\). Its image is finite and acts
freely on \(S^n\): if \(\rho(\delta)\) fixes \(x\in S^n\), then \(\delta\)
fixes \(D^{-1}(x)\), and hence \(\delta=e\).

We next conjugate this conformal action to an isometric action. Boundary
extension gives a group isomorphism
\(\operatorname{Conf}(S^n)\cong\operatorname{Isom}(\mathbb H^{n+1})\):
every conformal diffeomorphism of
\(S^n=\partial_\infty\mathbb H^{n+1}\) extends uniquely to a hyperbolic
isometry, and every hyperbolic isometry induces a conformal boundary map.
Recall also the Cartan fixed-point theorem: if a compact group acts
isometrically on a complete, simply connected Riemannian manifold of
nonpositive sectional curvature, then it has a common fixed point. Applying
this theorem to the finite group corresponding to \(\rho(\Pi)\) in
\(\operatorname{Isom}(\mathbb H^{n+1})\), we obtain a common fixed point
\(x_0\in\mathbb H^{n+1}\).

Choose \(A\in\operatorname{Isom}(\mathbb H^{n+1})\) such that \(A(x_0)=0\),
where \(0\) is the origin in the Poincar\'e ball model, and let
\(a\in\operatorname{Conf}(S^n)\) be its boundary map. Set
\[
\rho_0(\delta):=a\rho(\delta)a^{-1},
\qquad
\Gamma:=\rho_0(\Pi),
\qquad
F:=a\circ D.
\]
The hyperbolic extensions of the elements of \(\Gamma\) fix \(0\). The
stabilizer of \(0\) in the isometry group of the Poincar\'e ball is naturally
identified with \(\mathrm O(n+1)\), and its boundary action is precisely
\(\operatorname{Isom}(S^n,g_{st})\). Hence
\(\Gamma\subset\operatorname{Isom}(S^n,g_{st})\).
The map \(\rho_0\colon\Pi\to\Gamma\) is an isomorphism, the action of
\(\Gamma\) remains free, and \(F\circ\delta=\rho_0(\delta)\circ F\)
for every \(\delta\in\Pi\).
Consequently, \(F\) descends to a diffeomorphism
\(\overline F\colon M=\widetilde M/\Pi\to S^n/\Gamma\).

We verify explicitly that \(\overline F\) is conformal. Write
\(F^*g_{st}=e^{2u}\widetilde g\) for some
\(u\in C^\infty(\widetilde M)\). For every \(\delta\in\Pi\),
\[
\begin{aligned}
e^{2u\circ\delta}\widetilde g
 &=\delta^*(F^*g_{st})
  =(F\circ\delta)^*g_{st} \\
 &=(\rho_0(\delta)\circ F)^*g_{st}
  =F^*g_{st}
  =e^{2u}\widetilde g .
\end{aligned}
\]
Thus \(u\circ\delta=u\), so \(u=\overline u\circ p\) for some
\(\overline u\in C^\infty(M)\). Let
\(\pi_\Gamma\colon S^n\to S^n/\Gamma\) be the quotient map and let
\(g_\Gamma\) be the quotient round metric, characterized by
\(\pi_\Gamma^*g_\Gamma=g_{st}\). Since
\(\pi_\Gamma\circ F=\overline F\circ p\), we have
\[
\begin{aligned}
p^*(\overline F^*g_\Gamma)
 &=F^*(\pi_\Gamma^*g_\Gamma)
  =F^*g_{st}
  =e^{2u}\widetilde g
  =p^*(e^{2\overline u}g).
\end{aligned}
\]
Because \(p\) is a surjective local diffeomorphism,
\(\overline F^*g_\Gamma=e^{2\overline u}g\).
Therefore, \((M,g)\) is conformally equivalent to the spherical space form
\((S^n/\Gamma,g_\Gamma)\), and
\(\pi_1(M)\cong\Pi\cong\Gamma\).

The Hurewicz theorem in degree one gives
\(H_1(M;\mathbb Z)\cong\pi_1(M)^{\mathrm{ab}}\cong\Gamma^{\mathrm{ab}}\).
Hence \(H_1(M;\mathbb Z)=0\) implies that \(\Gamma\) is perfect. We now use
Wolf's classification in Allcock's intrinsic reformulation
\cite[Theorem~1.1 and Remark~1.3, pp.~5563--5564]{zbMATH06871580}.
At the beginning of \S~3, Allcock notes that the trivial group is the perfect
base case of type~(I) and that the only other perfect group in
Theorem~1.1 is the binary icosahedral group \(2A_5\)
\cite[p.~5569]{zbMATH06871580}. Hence
\[
\Gamma=\{e\}
\qquad\text{or}\qquad
\Gamma\cong I_{120}^*\cong2A_5
\cong\mathrm{SL}_2(\mathbb F_5).
\]

Suppose that \(\Gamma\cong I_{120}^*\), and let
\[
\sigma\colon\Gamma\longrightarrow\mathrm O(V),
\qquad V:=\mathbb R^{n+1},
\]
be the orthogonal representation defining the action of \(\Gamma\) on
\(S^n\), where \(S^n\) is regarded as the unit sphere in \(V\). Freeness of
the action gives
\(\ker\bigl(\sigma(\gamma)-\operatorname{id}_V\bigr)=\{0\}\)
for every \(\gamma\in\Gamma\setminus\{e\}\).

We first exhibit a quaternion subgroup of \(I_{120}^*\). Under the double covering
\(\operatorname{Sp}(1)\to\mathrm{SO}(3)\), the binary icosahedral group
\(I_{120}^*=2A_5\) is the inverse image of the icosahedral subgroup
\(A_5\subset\mathrm{SO}(3)\)
\cite[p.~5563]{zbMATH06871580}. Choose a tetrahedral subgroup
\(A_4\subset A_5\). Its inverse image \(2A_4\) is a subgroup of \(2A_5\);
in the standard quaternionic realization
\cite[(3.1), p.~5573]{zbMATH06871580},
\[
2A_4=
\left\{
 \pm1,\ \pm i,\ \pm j,\ \pm k,\
 \frac{\pm1\pm i\pm j\pm k}{2}
\right\},
\]
where the four signs in the last term are chosen independently. Hence
\(Q_8:=\{\pm1,\pm i,\pm j,\pm k\}\subset2A_4\subset I_{120}^*\).
After choosing an isomorphism \(\Gamma\cong I_{120}^*\), we regard this
\(Q_8\) as a subgroup of \(\Gamma\).

Let \(z=-1\in Q_8\). Since \(z\ne e\), freeness implies
\(\ker\bigl(\sigma(z)-\operatorname{id}_V\bigr)=\{0\}\).
Moreover, \(z^2=e\), so the minimal polynomial of \(\sigma(z)\) divides
\(X^2-1=(X-1)(X+1)\).
The roots are distinct; hence \(\sigma(z)\) is diagonalizable over
\(\mathbb R\) with eigenvalues in \(\{1,-1\}\). Its \(1\)-eigenspace is zero,
and therefore \(\sigma(z)=-\operatorname{id}_V\).

Set \(\mathsf I:=\sigma(i)\) and \(\mathsf J:=\sigma(j)\). Since
\(i^2=j^2=z\) and \(ji=zij\) in \(Q_8\), we obtain
\(\mathsf I^2=\mathsf J^2=-\operatorname{id}_V\) and
\(\mathsf J\mathsf I=-\mathsf I\mathsf J\).
By the defining relations of the quaternion algebra, there is consequently
a unital \(\mathbb R\)-algebra homomorphism
\[
\Phi\colon\mathbb H\longrightarrow\operatorname{End}_{\mathbb R}(V),
\qquad
\Phi(i)=\mathsf I,\quad
\Phi(j)=\mathsf J.
\]
Its kernel is a proper two-sided ideal of the division algebra \(\mathbb H\);
hence \(\Phi\) is injective. Thus \(V\) is a finite-dimensional left
\(\mathbb H\)-vector space. It therefore has an \(\mathbb H\)-basis, so
\(V\cong\mathbb H^m\) for some \(m\geq1\). Consequently,
\(n+1=\dim_{\mathbb R}V=4m\), and hence \(n\equiv3\pmod4\).

It follows that if \(n\) is even or \(n\equiv1\pmod4\), then
\(\Gamma\not\cong I_{120}^*\), and therefore \(\Gamma=\{e\}\). The conformal
equivalence constructed above then identifies \((M,g)\) with
\((S^n,g_{st})\), proving~(i). If \(n\equiv3\pmod4\), the classification
gives exactly the two alternatives in~(ii).

It remains to establish optimality. Let \(n=4r-1\), where \(r\geq1\), and
realize \(I_{120}^*\subset\operatorname{Sp}(1)\) as a group of unit
quaternions \cite[p.~5563]{zbMATH06871580}. Let it act diagonally on
\(\mathbb H^r\) by
\(\xi\cdot(v_1,\ldots,v_r):=(\xi v_1,\ldots,\xi v_r)\).
The quaternion norm is multiplicative, so this action preserves the Euclidean
norm and restricts to an isometric action on
\(S^{4r-1}\subset\mathbb H^r\). It is free: if
\(\xi\cdot(v_1,\ldots,v_r)=(v_1,\ldots,v_r)\), choose an index \(\ell\) with
\(v_\ell\ne0\). Then \((\xi-1)v_\ell=0\), and multiplication on the right by
\(v_\ell^{-1}\) gives \(\xi=1\).

The action is orientation preserving. Indeed, the determinant of the
diagonal action defines a continuous map
\(\operatorname{Sp}(1)\to\{\pm1\}\).
Since \(\operatorname{Sp}(1)\) is connected and the determinant equals \(1\)
at the identity, this map is identically \(1\). Therefore,
\(N_r:=S^{4r-1}/I_{120}^*\)
is a closed, connected, oriented spherical space form. Its quotient round
metric has constant sectional curvature \(1\); hence it is locally
conformally flat and has scalar curvature \((4r-1)(4r-2)>0\).
Since \(4r-1\geq3\), the sphere \(S^{4r-1}\) is simply connected. Thus
\(\pi_1(N_r)\cong I_{120}^*\ne\{e\}\).
Because \(I_{120}^*\) is perfect,
\(H_1(N_r;\mathbb Z)\cong\pi_1(N_r)^{\mathrm{ab}}\cong(I_{120}^*)^{\mathrm{ab}}=0\).
Hence \(N_r\) realizes the nontrivial alternative in every dimension
\(n\equiv3\pmod4\), while the standard round sphere realizes the
trivial-group alternative. For \(r=1\), \(N_1\) is the Poincar\'e homology
sphere \(P^3=S^3/I_{120}^*\).
\end{proof}

\begin{remark}
The PSC  hypothesis is retained only to formulate the
result as an answer to Gursky's question; it is not used in the classification
argument. Thus the conclusions of Theorem~\ref{nLCF} remain valid without
any assumption on the scalar curvature. If, moreover, \(M\) is an integral
homology sphere, then the binary-icosahedral alternative can occur only when
\(n=3\): for \(n=4r-1\geq7\), a spherical space form with fundamental group
\(I_{120}^*\) has
\(H_3(M;\mathbb Z)\cong H_3(BI_{120}^*;\mathbb Z)\cong\mathbb Z/120\).
Consequently, an LCF integral homology sphere with finite fundamental group
is conformally equivalent to \(S^n\) when \(n\geq4\), while in dimension
three it is conformally equivalent either to \(S^3\) or, up to orientation,
to the Poincar\'e homology sphere.
\end{remark}

\subsection{Properties of Kleinian Groups}

In light of Theorem~\ref{nLCF}, we now focus on closed, LCF Riemannian \(n\)-manifolds with PSC for \(n\geq5\) and infinite fundamental group. Additional hypotheses on \(H_1(M;\mathbb Z)\) will be imposed only where explicitly stated.

Let \( B^{n+1} := \{ x \in \mathbb{R}^{n+1} \mid |x| < 1 \} \) be the Poincar\'e ball model of \(\mathbb H^{n+1}\), endowed with the hyperbolic metric
\[
g_H = 4(1 - |x|^2)^{-2} \sum_{i=1}^{n+1} (dx^i)^2.
\]
Every element of \(\operatorname{Conf}(S^n)\) extends uniquely to a
diffeomorphism of the closed ball
\(\overline{B^{n+1}}:=B^{n+1}\cup S^n\) whose restriction to
\(B^{n+1}\) is an isometry of \(g_H\). Conversely, every isometry of
\((B^{n+1},g_H)\) extends continuously to the boundary \(S^n\), where it
induces a conformal transformation of \((S^n,g_{st})\). Therefore, boundary
restriction induces a natural group isomorphism
\(\operatorname{Isom}(B^{n+1},g_H)\cong\operatorname{Conf}(S^n)\).

Let \( \Gamma \) be a discrete subgroup of the conformal group \( \operatorname{Conf}(S^n) \) of the standard sphere \( (S^n, g_{st}) \); that is, \( \Gamma \) is a Kleinian group. Fix \(o\in B^{n+1}\). Its limit set is
\[
\Lambda(\Gamma):=\overline{\Gamma\cdot o}^{\,B^{n+1}\cup S^n}\cap S^n,
\]
and this definition is independent of the choice of \(o\in B^{n+1}\).

\begin{proposition}\label{virtual}
Let $\Gamma$ be a Kleinian group and let $\Gamma'\leq\Gamma$ be a
finite-index subgroup. Then
$\Lambda(\Gamma')=\Lambda(\Gamma)$.
\end{proposition}

\begin{proof}
Fix $o\in B^{n+1}$. Since $\Gamma'\subseteq\Gamma$, every accumulation
point of the $\Gamma'$-orbit of $o$ is an accumulation point of the
$\Gamma$-orbit of $o$. Thus
$\Lambda(\Gamma')\subseteq\Lambda(\Gamma)$.

Let $\xi\in \Lambda(\Gamma)$. Then there exists a sequence of distinct elements $\{\gamma_k\}_{k\ge 1}\subset \Gamma$ such that $\gamma_k o\to \xi$ in $B^{n+1}\cup S^n$. Because $[\Gamma:\Gamma']<\infty$, there is a disjoint right-coset decomposition
$\Gamma=\bigsqcup_{i=1}^{m}\Gamma' g_i.$
By the pigeonhole principle, some coset $\Gamma' g_j$ contains infinitely many $\gamma_k$. Passing to that subsequence (still denoted $\gamma_k$), write
$\gamma_k=\gamma'_k g_j$, $\gamma'_k\in \Gamma'.$ The $\gamma'_k$ are distinct: if $\gamma'_k=\gamma'_\ell$, then $\gamma_k=\gamma_\ell$, contrary to the choice of distinct $\gamma_k$.

Set $o':=g_j o\in B^{n+1}$. Then
$\gamma_k o=\gamma'_k(g_j o)=\gamma'_k o'\to \xi.$ Thus $\xi$
is an accumulation point of the $\Gamma'$-orbit of $o'$, so
$\xi\in\Lambda(\Gamma')$ by the independence of the base point. Therefore
$\Lambda(\Gamma)\subseteq\Lambda(\Gamma')$, and the two inclusions prove the
claim.
\end{proof}

Let \((M^n,g)\) be a closed, connected, oriented, smooth, LCF Riemannian
manifold with \(n\geq4\) and PSC. Schoen and
Yau~\cite[Theorem~4.5]{zbMATH04075988} prove that the developing map of
\((\widetilde M,\widetilde g)\) is injective and that its holonomy
representation
\(\rho\colon\pi_1(M)\to\operatorname{Conf}(S^n)\)
is faithful and has discrete image. Set \(\Gamma:=\rho(\pi_1(M))\). The image
of the developing map is an open subset \(\Omega\subset S^n\) on which
\(\Gamma\) acts properly discontinuously, and \(M\) is diffeomorphic to the
quotient \(\Omega/\Gamma\).

Moreover, \(\Omega=\Omega(\Gamma):=S^n\setminus\Lambda(\Gamma)\). Indeed,
the maximality of the domain of discontinuity gives
\(\Omega\subseteq\Omega(\Gamma)\)~\cite{zbMATH00195006}, and hence
\(\Lambda(\Gamma)\subseteq S^n\setminus\Omega\). Since
\(M\cong\Omega/\Gamma\) is compact, there exists a compact set
\(K\subset\Omega\) such that \(\Gamma K=\Omega\). Suppose that
\(x\in\partial\Omega\cap\Omega(\Gamma)\). Choose \(y_k\in\Omega\) with
\(y_k\to x\), and write \(y_k=\gamma_kz_k\), where \(z_k\in K\). Choose a
compact neighborhood \(C\) of \(x\) contained in \(\Omega(\Gamma)\). Proper
discontinuity implies that
\(\{\gamma\in\Gamma:\gamma K\cap C\ne\varnothing\}\)
is finite. For all sufficiently large \(k\), \(y_k\in C\); after passing to a
subsequence, \(\gamma_k=\gamma\) is therefore constant. Since \(\gamma K\)
is compact and hence closed, \(x\in\gamma K\subset\Omega\), contradicting
\(x\in\partial\Omega\). Thus \(\Omega\) is closed in \(\Omega(\Gamma)\).
Furthermore, Schoen and Yau's estimate for the complement of the developing
image~\cite[Theorem~4.7]{zbMATH04075988} and the inclusion
\(\Lambda(\Gamma)\subseteq S^n\setminus\Omega\) give
\[
\dim_{\mathrm{top}}\Lambda(\Gamma)
\leq\dim_{\mathcal H}\Lambda(\Gamma)
\leq\frac{n-2}{2}<n-1.
\]
The non-separation theorem for compact subsets of the sphere now implies that
\(S^n\setminus\Lambda(\Gamma)\) is connected
\cite[p.~48, Corollary~1 to Theorem~IV.4]{MR0006493}. Thus
\(\Omega(\Gamma)\) is connected, and the nonempty subset \(\Omega\), being
both open and closed in \(\Omega(\Gamma)\), equals \(\Omega(\Gamma)\).

 The following standard facts may be found in \cite{zbMATH00195006}; for the
 structure of elementary groups, see also
 \cite[\S~5.5, especially Lemma~1 and Theorem~5.5.9]{MR4221225}. A Kleinian
 group $\Gamma$ is called \textit{elementary} if its limit set
 $\Lambda(\Gamma)$ has at most two points; otherwise, it is called
 \textit{non-elementary}. If $\Lambda(\Gamma)$ is empty, then $\Gamma$ is
 finite. If $\Lambda(\Gamma)$ consists of a single point, then $\Gamma$
 contains an abelian subgroup of finite index of rank $k$ with
 $1\leq k \leq n$. If $\Lambda(\Gamma)$ consists of two points, then $\Gamma$ contains an
 infinite cyclic subgroup of finite index. Conversely, every abelian Kleinian
 group is elementary. Together with Proposition~\ref{virtual}, these facts show
 that a Kleinian group is elementary if and only if it is virtually abelian. If
 $\Gamma$ is non-elementary, then $\Lambda(\Gamma)$ is the unique minimal
 nonempty closed \(\Gamma\)-invariant subset of \(S^n\).
 
For a Kleinian group \(\Gamma\), the \emph{critical exponent}
\(\delta(\Gamma)\) is defined by
\[
\delta(\Gamma):=\inf\left\{s>0\,\middle|\,
\sum_{\gamma\in\Gamma}\exp\bigl(-s\,d_H(x,\gamma y)\bigr)<\infty\right\},
\]
where \(x,y\in B^{n+1}\) and \(d_H\) denotes the hyperbolic distance
induced by \(g_H\). The value is independent of the choices of \(x\) and
\(y\). For a finite group, the definition gives \(\delta(\Gamma)=0\). If
\(\Gamma\) is non-elementary, then \(0<\delta(\Gamma)\leq n\). If
\(\Gamma'\leq\Gamma\), then \(\delta(\Gamma')\leq\delta(\Gamma)\).

A Kleinian group \( \Gamma \) is said to be \emph{convex cocompact} if the quotient
$(\Omega(\Gamma) \cup B^{n+1}) / \Gamma$ is compact. If
\(\#\Lambda(\Gamma)\geq2\), this is equivalent to the compactness of
\(\mathrm{CH}(\Gamma)/\Gamma\), where \(\mathrm{CH}(\Gamma)\subset B^{n+1}\)
is the closed convex hull of the union of all complete geodesics with distinct
ideal endpoints in \(\Lambda(\Gamma)\). Equivalently, it is the smallest
closed convex subset whose closure in \(\overline{B^{n+1}}\) contains
\(\Lambda(\Gamma)\)~\cite{MR1218098}.
Indeed, any closed convex subset with this boundary-closure property contains
every such complete geodesic and hence contains the closed convex hull above.
For \(\#\Lambda(\Gamma)\leq1\), we set \(\mathrm{CH}(\Gamma)=\varnothing\)
by convention.

A Kleinian group \( \Gamma \) is said to be \emph{geometrically finite} if there exists a uniform bound on the orders of its finite subgroups and the \( \varepsilon \)-neighborhood of \( \mathrm{CH}(\Gamma)/\Gamma \) in \( B^{n+1}/\Gamma \) has finite volume for some \( \varepsilon > 0 \).

\begin{proposition} \label{geo finite}
Let \( \Gamma \) be a geometrically finite (resp. convex cocompact) Kleinian group, and let \( \Gamma' \leq \Gamma \) be a subgroup of finite index. Then \( \Gamma' \) is also geometrically finite (resp. convex cocompact).
\end{proposition}

\begin{proof}
By Proposition~\ref{virtual}, the limit set \( \Lambda(\Gamma') \) coincides with \( \Lambda(\Gamma) \). Consequently, the hyperbolic convex hull \( \mathrm{CH}(\Gamma') \subset B^{n+1} \) equals \( \mathrm{CH}(\Gamma) \). The quotient \( \mathrm{CH}(\Gamma')/\Gamma' \) is a finite-sheeted orbifold covering of \( \mathrm{CH}(\Gamma)/\Gamma \). Likewise, \((\Omega(\Gamma')\cup B^{n+1})/\Gamma'\) is a finite-sheeted orbifold covering of \((\Omega(\Gamma)\cup B^{n+1})/\Gamma\), which proves convex cocompactness when applicable. Moreover, the \( \varepsilon \)-neighborhood of \( \mathrm{CH}(\Gamma')/\Gamma' \) has finite volume, since finiteness of volume is preserved under finite-sheeted orbifold coverings.

In addition, any finite subgroup of \( \Gamma' \) is also a finite subgroup of \( \Gamma \), so the orders of finite subgroups of \( \Gamma' \) are uniformly bounded. This proves geometric finiteness.
\end{proof}

Convex cocompact Kleinian groups are geometrically finite. Conversely, geometrically finite Kleinian groups without parabolic elements are precisely the convex cocompact ones~\cite{MR1218098}. If \(\Gamma\) is a non-elementary convex cocompact Kleinian group, then
Sullivan's theorem~\cite{zbMATH03903608} states that
$$
\Gamma \text{ non-elementary and convex cocompact}
\quad\Longrightarrow\quad
\delta(\Gamma)=\dim_{\mathcal H}\Lambda(\Gamma),
$$

where \(\dim_{\mathcal H}\) denotes Hausdorff dimension.

\subsection{The Hausdorff Dimension of the Limit Set}\label{Hausdorff}

Let \(q\colon\Omega\to M\cong\Omega/\Gamma\) be the quotient map, and
continue to write \(g\) for the pullback metric \(q^*g\) on \(\Omega\).
Schoen and Yau's estimate~\cite[Theorem~4.7]{zbMATH04075988} gives
\(\dim_{\mathcal H}\Lambda(\Gamma)\leq\frac{n-2}{2}\).
For non-elementary \(\Gamma\), Proposition~\ref{prop:standing} below sharpens
this inequality to a strict one. If \(\Gamma\) is elementary, then, because
\(\Gamma\) is infinite, its limit set has one or two points and hence
\(\dim_{\mathcal H}\Lambda(\Gamma)=0\). Under the additional hypotheses of
Theorem~\ref{Kleinian}, torsion-freeness yields a strict lower bound. We first
record the critical-exponent estimate used repeatedly below.

\begin{proposition}\label{prop:standing}
Let $(M^n,g)$, $n\geq4$, be a closed, connected, oriented, locally
conformally flat manifold with positive scalar curvature and infinite
fundamental group, and write $M=\Omega/\Gamma$ as above. Then $\Gamma$ is
Gromov-hyperbolic and convex
cocompact. If $\Gamma$ is non-elementary, then
\[
0<\dim_{\mathcal H}\Lambda(\Gamma)=\delta(\Gamma)<\frac{n-2}{2}.
\]
\end{proposition}

\begin{proof}
The manifold $M^n$ is not finitely covered by a torus, so
Izeki's results~\cite[Theorems~2 and~3]{zbMATH01785298} show that $\Gamma$ is
Gromov-hyperbolic and convex cocompact. The first assertion follows. For the
remaining assertion, suppose that $\Gamma$ is non-elementary. Because
$\Gamma\cong\pi_1(M)$ is finitely generated and
$\operatorname{Conf}(S^n)\cong\mathrm{PO}(n+1,1)$ has a faithful
finite-dimensional linear representation (for example, its adjoint
representation), Selberg's lemma~\cite{zbMATH03319625} provides a torsion-free subgroup
$\Gamma_0\leq\Gamma$ of finite index. Then $\Gamma_0$ is non-elementary by
Proposition~\ref{virtual} and convex cocompact by
Proposition~\ref{geo finite}. Let $g_0$ be the
pullback of $g$ to $M_0:=\Omega/\Gamma_0$. Then
$\Lambda(\Gamma_0)=\Lambda(\Gamma)$ by Proposition~\ref{virtual}, so
$\Omega=\Omega(\Gamma_0)$. Choose a right-coset decomposition
$\Gamma=\bigsqcup_{i=1}^m\Gamma_0g_i$. For $s>0$,
\[
\sum_{\gamma\in\Gamma}e^{-s d_H(x,\gamma y)}
=\sum_{i=1}^m\sum_{\gamma_0\in\Gamma_0}
e^{-s d_H(x,\gamma_0g_i y)}.
\]
By base-point independence, every inner series has critical exponent
$\delta(\Gamma_0)$. Since the outer sum is finite,
$\delta(\Gamma)=\delta(\Gamma_0)$.

The manifold $M_0$ is closed, and $g_0$ is locally conformally flat with
positive scalar curvature; in particular, $M_0$ is not finitely covered by a
torus. We have already shown that $\Gamma_0$ is torsion-free,
non-elementary, and convex cocompact. Therefore Nayatani's criterion
\cite[Corollary~3.4]{zbMATH01028179} gives
$\delta(\Gamma_0)<\frac{n-2}{2}.$ Finally, Sullivan's
theorem~\cite{zbMATH03903608} gives
$\dim_{\mathcal H}\Lambda(\Gamma)=\delta(\Gamma)$, while the standard
critical-exponent inequality recalled above gives $\delta(\Gamma)>0$.
\end{proof}

\begin{theorem}\label{Kleinian}
Let \( (M^n,g) \), \(n\geq5\), be a closed, connected, oriented, locally conformally flat Riemannian manifold with positive scalar curvature and infinite fundamental group. Then \(M\) can be realized as a quotient \(\Omega/\Gamma\), as described above, with \(\Gamma\cong\pi_1(M)\). There exists an integer
$i \geq \Big\lfloor \frac{n}{2} \Big\rfloor + 1$
such that $\pi_i(M^n) \neq 0$.

If, in addition, \( H_1(M^n;\mathbb{Z}) = 0 \), then \( \Gamma \) is non-elementary.

If, moreover, \(\Gamma\) is torsion-free, then the Hausdorff dimension of the limit set satisfies
\[
1 < \dim_{\mathcal H}\Lambda(\Gamma) < \frac{n-2}{2}.
\]
\end{theorem}

Theorem~\ref{Kleinian} immediately yields a partial result toward the long-standing conjecture that a closed, oriented aspherical manifold does not admit a Riemannian metric of PSC.
For other partial results related to this conjecture, see \cite{zbMATH07375613}.

\begin{corollary}\label{LCF aspherical}
Let \( M^n \) be a closed, connected, oriented, aspherical manifold with \( n \geq 5 \).
Then \( M^n \) does not admit a locally conformally flat Riemannian metric with positive scalar curvature.
\end{corollary}

\begin{proof}
The fundamental group \(\pi_1(M)\) is infinite. Indeed, if it were finite,
then the universal cover would be both a closed oriented manifold and a
contractible space, contradicting its nonzero top-dimensional homology.
If such a metric existed, Theorem~\ref{Kleinian} would imply that there exists
an integer $i\geq\lfloor n/2\rfloor+1$ such that
\(\pi_i(M^n)\neq0\), which contradicts the asphericity of \(M\).
\end{proof}

The proof of Theorem~\ref{Kleinian} uses Izeki's convex cocompactness theorem, Nayatani's critical-exponent criterion, and Kapovich's cohomological-dimension and rigidity theorems. We first record the elementary algebraic lemma needed to exclude elementary holonomy.

\begin{lemma}\label{Abel}
Let $G$ and $F$ be groups fitting into a short exact sequence
\[
1 \longrightarrow \mathbb{Z} \xrightarrow{\;\iota\;} G \xrightarrow{\;\pi\;} F \longrightarrow 1,
\]
where $F$ is finite. Then the abelianization $G^{\mathrm{ab}} := G/[G,G]$ is nontrivial.
\end{lemma}

\begin{proof}

Let $H:=\iota(\mathbb Z)$ and $a:=\iota(1)$. Since $\iota$ is injective, $H\cong\mathbb Z$. We write $H$ multiplicatively inside $G$, so $a^k:=\iota(k)$ for $k\in\mathbb Z$. Exactness gives $\ker\pi=H$, hence $H\trianglelefteq G$.

Let $\pi_0\colon G\to G/H$ be the quotient map. Then there exists an
isomorphism $\bar\pi\colon G/H\xrightarrow{\;\sim\;}F$ such that
$\pi=\bar\pi\circ\pi_0$.
We henceforth identify $F$ with $G/H$ via $\bar\pi$. Thus
$m:=[G:H]=|G/H|=|F|<\infty.$ Whenever we refer to an element of $F$, or
to a coset $gH\in G/H$, this identification is tacitly used.

Because $H$ is normal, for every $g\in G$ and $x\in H$ we have $gxg^{-1}\in H$ (indeed, $\pi(gxg^{-1})=\pi(g)\pi(x)\pi(g)^{-1}=1$, so $gxg^{-1}\in\ker\pi=H$). Define
$\varepsilon\colon G\to \mathrm{Aut}(H)$ by
$\varepsilon(g)(x):=gxg^{-1}$ for $x\in H$. For $g_1,g_2\in G$ and $x\in H$,
\[
\varepsilon(g_1g_2)(x)
=(g_1g_2)x(g_1g_2)^{-1}
=g_1\big(g_2xg_2^{-1}\big)g_1^{-1}
=\varepsilon(g_1)\!\left(\varepsilon(g_2)(x)\right),
\]
so $\varepsilon(g_1g_2)=\varepsilon(g_1)\circ\varepsilon(g_2)$ and $\varepsilon$ is a group homomorphism.

Since $H\cong\mathbb Z$ is infinite cyclic, fix the generator $a=\iota(1)$ of $H$. Any automorphism of $H$ must send $a$ to a generator of $H$, i.e., $a$ or $a^{-1}$; conversely, either choice extends uniquely by $a^k\mapsto a^{\pm k}$ for all $k\in\mathbb Z$. If $a'$ is another generator, then $a'=a^{\pm1}$, hence the sign is independent of the chosen generator. Therefore $\mathrm{Aut}(H)=\{\pm\mathrm{id}_H\}\cong\{\pm1\}.$ Under the identification $\mathrm{Aut}(H)\cong\{\pm1\}$, write $\varepsilon(g)\in\{\pm1\}$ so that $gag^{-1}=a^{\varepsilon(g)}$. Then for each $k\in\mathbb Z$, $ga^kg^{-1}=(gag^{-1})^k=a^{\varepsilon(g)k},$ so the single sign $\varepsilon(g)$ determines the whole automorphism $\varepsilon(g)\in\mathrm{Aut}(H)$.

For $h\in H$, $H$ is abelian, so $hah^{-1}=a$; hence $\varepsilon(h)=1$ and $H\subseteq\ker\varepsilon$.

 Define
 $\bar\varepsilon:G/H\to \{\pm1\},\qquad \bar\varepsilon(gH):=\varepsilon(g).$
This map is well defined: if
$gH=g'H$, then $g^{-1}g'\in H\subseteq\ker\varepsilon$, so
$\varepsilon(g')=\varepsilon(g)\varepsilon(g^{-1}g')=\varepsilon(g)$. Moreover, for $g_1H,g_2H\in G/H$ one has
\[
\bar\varepsilon(g_1H\cdot g_2H)
=\bar\varepsilon(g_1g_2H)
=\varepsilon(g_1g_2)
=\varepsilon(g_1)\varepsilon(g_2)
=\bar\varepsilon(g_1H)\bar\varepsilon(g_2H),
\]
so $\bar\varepsilon$ is a homomorphism. Via our identification $F\simeq G/H$, we view $\bar\varepsilon$ as a homomorphism $F\to\{\pm1\}$ such that
$\varepsilon=\bar\varepsilon\circ\pi.$

\emph{Case 1: $\varepsilon$ (equivalently $\bar\varepsilon$) is nontrivial.}
Choose $g\in G$ with $\varepsilon(g)=-1$. Since $\{\pm1\}$ is abelian,
$\varepsilon([x,y])=1\qquad(x,y\in G),$ so
$[G,G]\subseteq\ker\varepsilon$. 

We recall the universal property of abelianization: let $q:G\to G^{\mathrm{ab}}:=G/[G,G]$ be the quotient map. If $f:G\to A$ is a homomorphism with $A$ abelian and $[G,G]\subseteq\ker f$, then there exists a \emph{unique} homomorphism $\widehat f:G^{\mathrm{ab}}\to A$ with $f=\widehat f\circ q$.

Hence $\varepsilon$ factors through the abelianization: applying this with
$f=\varepsilon$ gives a unique
$\widetilde\varepsilon:G^{\mathrm{ab}}\to \{\pm1\}$ such that $\varepsilon=\widetilde\varepsilon\circ q .$ 
Explicitly,
$\widetilde\varepsilon(g[G,G])=\varepsilon(g)$; this is well defined because
$[G,G]\subseteq\ker\varepsilon$. Since
$\widetilde\varepsilon(g[G,G])=-1\neq1$, the map
$\widetilde\varepsilon$ is nontrivial. The only homomorphism from the trivial
group is the trivial map, so $G^{\mathrm{ab}}$ cannot be trivial. Thus
$G^{\mathrm{ab}}\neq0$.

\emph{Case 2: $\varepsilon$ is trivial.}
Then $\varepsilon(g)=1$ for all $g\in G$, i.e., $gag^{-1}=a$. For $h=a^k\in H$,
$ghg^{-1}=(gag^{-1})^k=a^k=h,$ so $H\subseteq Z(G)$.

Fix a set of right coset representatives
$\mathcal T=\{t_1,\dots,t_m\}\subset G$, so that $G=\bigsqcup_{i=1}^m Ht_i,$
and fix the order $1<2<\cdots<m$. For each $g\in G$ and $1\le i\le m$ there exist unique elements $h_i(g)\in H$ and $\sigma_g(i)\in\{1,\dots,m\}$ such that
\begin{equation}\label{eq:coset}
t_i g = h_i(g)\, t_{\sigma_g(i)}.
\end{equation}

The existence and uniqueness of such elements can be established as follows.   The right cosets $\{Ht_j\}_{j=1}^m$ partition $G$, so $t_ig$ lies in exactly one $Ht_j$; set $\sigma_g(i):=j$ and define $h_i(g)\in H$ by \eqref{eq:coset}.  
If $t_i g = h t_j = h' t_{j'}$ with $h,h'\in H$, then $Ht_j=Ht_{j'}$, so $j=j'$ because the cosets $Ht_j$ are pairwise disjoint. With $j$ fixed, $h t_j=h' t_j$ implies $h=h'$ since right multiplication by $t_j^{-1}$ is injective in a group.

Define the transfer map (also called the Verlagerung)
\[
\mathrm{Ver}:G\to H,\qquad \mathrm{Ver}(g):=\prod_{i=1}^m h_i(g),
\]
with
the factors taken in the fixed order $1,\dots,m$.
 Because $H$ is abelian and the product is finite, any permutation of the factors yields the same element of $H$; fixing an order just makes the definition explicit.

The map $\mathrm{Ver}$ is a group homomorphism.  Fix $g_1,g_2\in G$. For
each $i$ there exist $h_i^{(1)}\in H$ and $\sigma_{g_1}(i)$ with
$t_i g_1 = h_i^{(1)}\, t_{\sigma_{g_1}(i)}.$ Likewise, for each $i$ there
exist $h_{\sigma_{g_1}(i)}^{(2)}\in H$ and
$\sigma_{g_2}(\sigma_{g_1}(i))$ with
$t_{\sigma_{g_1}(i)} g_2 = h_{\sigma_{g_1}(i)}^{(2)}\,
t_{\sigma_{g_2}\circ\sigma_{g_1}(i)}.$ Multiplying gives
\[
t_i(g_1g_2)=h_i^{(1)}h_{\sigma_{g_1}(i)}^{(2)}\, t_{\sigma_{g_2}\circ\sigma_{g_1}(i)}.
\]
By the uniqueness in \eqref{eq:coset}, $h_i(g_1g_2)=h_i^{(1)}h_{\sigma_{g_1}(i)}^{(2)}$. Hence
\begin{equation}\label{eq:Ver_prod}
\mathrm{Ver}(g_1g_2)
=\prod_{i=1}^m h_i^{(1)}h_{\sigma_{g_1}(i)}^{(2)}
=\Big(\prod_{i=1}^m h_i^{(1)}\Big)\Big(\prod_{i=1}^m h_{\sigma_{g_1}(i)}^{(2)}\Big).
\end{equation}

The map $\sigma_{g_1}$ is a permutation.  If $\sigma_{g_1}(i)=\sigma_{g_1}(j)$, then
\[
Ht_i g_1 = H t_{\sigma_{g_1}(i)} = H t_{\sigma_{g_1}(j)} = H t_j g_1.
\]
Right-multiplying by $g_1^{-1}$ gives $Ht_i=Ht_j$, hence $i=j$ by disjointness.  Thus the map $\sigma_{g_1}$ is injective.  Fix $k\in\{1,\dots,m\}$. Since the cosets $Ht_r$ partition $G$, there exists a unique $r$ with $t_k g_1^{-1}\in Ht_r$, say $t_k g_1^{-1}=h' t_r$ with $h'\in H$. Multiplying on the right by $g_1$ yields $t_k=h' t_r g_1$, so $Ht_k=Ht_r g_1$, hence $\sigma_{g_1}(r)=k$.  
Thus $\sigma_{g_1}$ is bijective (a permutation). Therefore the multiset $\{h_{\sigma_{g_1}(i)}^{(2)}:1\le i\le m\}$ is just a reordering of $\{h_i^{(2)}:1\le i\le m\}$, and since $H$ is abelian,
$\prod_{i=1}^m h_{\sigma_{g_1}(i)}^{(2)}=\prod_{i=1}^m h_i^{(2)}.$ Applying this to \eqref{eq:Ver_prod} gives
$\mathrm{Ver}(g_1g_2)=\mathrm{Ver}(g_1)\mathrm{Ver}(g_2)$.

For any $x,y\in G$,
\[
\mathrm{Ver}([x,y])=\mathrm{Ver}(x)\mathrm{Ver}(y)\mathrm{Ver}(x)^{-1}\mathrm{Ver}(y)^{-1}=1,
\]
because $H$ is abelian. Hence $[G,G]\subseteq\ker\mathrm{Ver}$. By the universal property of abelianization, there is a unique homomorphism
$\overline{\mathrm{Ver}}:G^{\mathrm{ab}}\to H$ such that $ \mathrm{Ver}=\overline{\mathrm{Ver}}\circ q,$ where $q:G\to G^{\mathrm{ab}}$ is the quotient map.

Since $H\subseteq Z(G)$,  $\mathrm{Ver}(h)=h^{\,m}$ for all $h\in H$. Indeed, if $h\in H$, then $t_i h = h t_i$ for every $i$. Comparing with \eqref{eq:coset} and invoking the uniqueness of $h_i(\cdot)$ and $\sigma_{\cdot}(\cdot)$ yields $\sigma_h(i)=i$ and $h_i(h)=h$ for all $i$. Consequently,
$\mathrm{Ver}(h)=\prod_{i=1}^m h = h^{\,m}.$

Let $j:H\hookrightarrow G$ denote the inclusion and let
$j_*:H\to G^{\mathrm{ab}},$ $j_*:=q\circ j$ be the
homomorphism induced by inclusion. Since  $\mathrm{Ver}(h)=h^{\,m}$ for all
$h\in H\subseteq Z(G)$, we obtain
$(\overline{\mathrm{Ver}}\circ j_*)(h)=\mathrm{Ver}(h)=h^{\,m}$, $(h\in H).$
Since $H\cong\mathbb Z$, it is torsion-free: if $h^m=1$ with $h=a^p$, then $a^{mp}=1$ in $H$, hence $mp=0$ in $\mathbb Z$, so $p=0$ and $h=1$. Thus $h\mapsto h^m$ is injective, so the composite $\overline{\mathrm{Ver}}\circ j_*$ is nontrivial.

If $G^{\mathrm{ab}}=0$, then every homomorphism from $G^{\mathrm{ab}}$ is trivial. In particular, $\overline{\mathrm{Ver}}$ is trivial, and hence
$\overline{\mathrm{Ver}}\circ j_*$ is trivial, contradicting the previous paragraph. Therefore $G^{\mathrm{ab}}\neq0$. If $F=1$, then $m=1$ and $G=H\cong\mathbb Z$, so $G^{\mathrm{ab}}\cong\mathbb Z\neq0$; this is already encompassed by Case~2.
\end{proof}

\begin{remark}\label{free}
The infinite cyclic group \( \mathbb{Z} \) in Lemma~\ref{Abel} cannot be replaced by a free group of rank at least two. For instance, let \( A_5 \) denote the alternating group on five letters. Since \( A_5 \) is perfect, the free product \( G = A_5 \ast A_5 \) is also perfect, as the abelianization of a free product is the direct sum of the abelianizations of its factors.

Consider the natural folding homomorphism \( \pi: G \to A_5 \), which restricts to the identity on each factor. Hence $\ker\pi$ meets every conjugate of either factor trivially, and the Kurosh subgroup theorem implies that $\ker\pi$ is free. Using the rational Euler characteristic of virtually free groups and its multiplicativity under passage to finite-index subgroups, we obtain
\[
\chi(G)=\frac1{60}+\frac1{60}-1=-\frac{29}{30},
\qquad
\chi(\ker\pi)=60\chi(G)=-58,
\]
so \(\ker\pi\) has rank \(1-(-58)=59\). Thus $\ker\pi\cong F_{59}$, yielding the short exact sequence
\[
1 \longrightarrow F_{59} \longrightarrow G \xrightarrow{\pi} A_5 \longrightarrow 1.
\]
Thus, $G$ admits a free normal subgroup of rank $59$ and finite index while remaining perfect, showing that Lemma~\ref{Abel} fails if $\mathbb{Z}$ is replaced by a free group of rank at least two.
\end{remark}

\begin{lemma}\label{lem:strict-lower}
Let $\Gamma\leq\operatorname{Isom}(\mathbb H^{n+1})$ be a torsion-free,
non-elementary, convex cocompact Kleinian group. If
$\Gamma^{\mathrm{ab}}=0$, then
$\dim_{\mathcal H}\!\bigl(\Lambda(\Gamma)\bigr)>1$.
\end{lemma}

\begin{proof}
Suppose, for contradiction, that
$\dim_{\mathcal H}\!\bigl(\Lambda(\Gamma)\bigr)\leq 1,$ and put
$d:=\dim_{\mathrm{top}}\!\bigl(\Lambda(\Gamma)\bigr).$ The limit set is a
nonempty compact metric space, and topological dimension
does not exceed Hausdorff dimension. Hence $d\in\{0,1\}$.

Because $\Gamma$ is convex cocompact, it is geometrically finite and has no
parabolic elements. Thus the peripheral family $\Pi'$ in
\cite[Proposition~9.6]{MR2491697} is empty, and that proposition gives
$\operatorname{cd}_{\mathbb Z}(\Gamma)=d+1.$ If $d=0$, then
$\operatorname{cd}_{\mathbb Z}(\Gamma)=1$. We use the
following form of the Stallings--Swan theorem, recorded by Kapovich
\cite[Theorem~2.5]{MR2491697}: if $G$ is a torsion-free group satisfying
$\operatorname{cd}_R(G)\leq1$, then $G$ is free. Applying this theorem with
$R=\mathbb Z$ and $G=\Gamma$, we conclude that $\Gamma$ is free. Since
$\Gamma$ is non-elementary, it is a free group of rank at least two and
therefore has nontrivial abelianization, contradicting
$\Gamma^{\mathrm{ab}}=0$.

It remains to consider $d=1$. Our assumption then gives
$\dim_{\mathrm{top}}\!\bigl(\Lambda(\Gamma)\bigr)
 =
\dim_{\mathcal H}\!\bigl(\Lambda(\Gamma)\bigr)=1.$ By
\cite[Theorem~1.3]{MR2491697}, the limit set is a round circle $C$ and
$\Gamma$ preserves the totally geodesic plane
$P\cong\mathbb H^2\subset B^{n+1}$ bounded by $C$. Consider the restriction
homomorphism $\rho\colon\Gamma\to \operatorname{Isom}(P).$ Its kernel
is a discrete subgroup of the compact pointwise stabilizer of
$P$, hence is finite and therefore trivial because $\Gamma$ is torsion-free.

Set $\Gamma_P:=\rho(\Gamma)$. Since $\rho$ is injective, the proper
discontinuity of the action of $\Gamma$ on $\mathbb H^{n+1}$ implies that
$\Gamma_P$ acts properly discontinuously on $P$. Moreover,
$\mathrm{CH}(\Gamma)=P,$ so convex cocompactness implies that
$P/\Gamma_P$ is compact. Since $\Gamma_P$
is torsion-free, $\Sigma:=P/\Gamma_P$ is a closed hyperbolic surface, possibly
nonorientable. In either case $H_1(\Sigma;\mathbb Z)\neq0$. Therefore
$\Gamma^{\mathrm{ab}}\cong\Gamma_P^{\mathrm{ab}}
\cong H_1(\Sigma;\mathbb Z)\neq0$,
again contradicting $\Gamma^{\mathrm{ab}}=0$.

Both possibilities are impossible. Therefore
$\dim_{\mathcal H}\!\bigl(\Lambda(\Gamma)\bigr)>1$.
\end{proof}

\begin{remark}\label{rem:torsion-caveat}
Torsion-freeness is used essentially in Lemma~\ref{lem:strict-lower}: it is
needed both for the cohomological-dimension-one conclusion and to eliminate
the finite kernel of the action on the invariant hyperbolic plane. Passing to
a finite-index torsion-free subgroup by Selberg's lemma does not remove this
issue, because trivial abelianization need not pass to finite-index subgroups.
\end{remark}

\begin{proof}[Proof of Theorem~\ref{Kleinian}]
By Proposition~\ref{prop:standing}, $\Gamma$ is Gromov-hyperbolic and its
Kleinian action is convex cocompact, hence geometrically finite.

Suppose that $\pi_i(M^n)=0$ for all
$i\geq\lfloor n/2\rfloor+1$. We first verify the required low-degree
vanishing. If $\Gamma$ is elementary, then, since $\Gamma$ is infinite,
$\Lambda(\Gamma)$ is nonempty and hence consists of one or two points.
Accordingly, $\Omega(\Gamma)$ is diffeomorphic to either
$\mathbb R^n$ or $S^{n-1}\times\mathbb R$, and hence
\[
\pi_i(M)=\pi_i(\Omega(\Gamma))=0,
\qquad 2\leq i\leq n-2.
\]
If $\Gamma$ is non-elementary, Proposition~\ref{prop:standing} gives
$\dim_{\mathcal H}\Lambda(\Gamma)<\frac{n-2}{2}$.
For every $i\leq\lfloor n/2\rfloor$, this implies
$(i+1)+\dim_{\mathcal H}\Lambda(\Gamma)<n$.
The relative general-position argument in the proof of Schoen and
Yau~\cite[Theorem~4.6(ii)]{zbMATH04075988} now applies: it extends a map
$S^i\to\Omega(\Gamma)$ to $B^{i+1}\to S^n$ and deforms the extension,
relative to its boundary, away from $\Lambda(\Gamma)$ under the strict
inequality above. Therefore,
\[
\pi_i(M)=\pi_i(\Omega(\Gamma))=0,
\qquad 2\leq i\leq\left\lfloor\frac n2\right\rfloor.
\]
Thus, under the supposition above, all higher homotopy groups of $M$ vanish.
The simply connected CW complex $\widetilde M$ is therefore weakly
contractible and hence contractible by Whitehead's theorem. Consequently,
$M$ is aspherical and is a finite $n$-dimensional model for $B\Gamma$.
In particular, $\Gamma$ has finite integral cohomological dimension and is
torsion-free. If $C_*(\widetilde M)$ denotes the cellular chain complex, then
the free cocompact action and the contractibility of $\widetilde M$ give the
cochain identification
\[
\operatorname{Hom}_{\mathbb Z\Gamma}
 \bigl(C_*(\widetilde M),\mathbb Z\Gamma\bigr)
\cong C_c^*(\widetilde M;\mathbb Z).
\]
Consequently,
$H^n(\Gamma;\mathbb Z\Gamma)\cong H_c^n(\widetilde M;\mathbb Z)$,
and compact-support Poincar\'e duality on the oriented manifold
$\widetilde M$ gives
$H_c^n(\widetilde M;\mathbb Z)\cong H_0(\widetilde M;\mathbb Z)\cong\mathbb Z$.
Since $M$ is an $n$-dimensional model for $B\Gamma$, it follows that
$\operatorname{cd}_{\mathbb Z}(\Gamma)=n$.

If $\Gamma$ were elementary, then it would be virtually abelian; being infinite
and Gromov-hyperbolic, it would therefore be virtually cyclic. Since it is
torsion-free, the classification of infinite virtually cyclic groups implies
that $\Gamma\cong\mathbb Z$, which would give
$\operatorname{cd}_{\mathbb Z}(\Gamma)=1<n$. Thus $\Gamma$ is
non-elementary. Since $\Gamma$ is convex cocompact, it is geometrically finite
without parabolic elements, and Kapovich's formula
\cite[Proposition~9.6]{MR2491697}, together with
Proposition~\ref{prop:standing}, yields
\[
n=\operatorname{cd}_{\mathbb Z}(\Gamma)
 =\dim_{\mathrm{top}}\Lambda(\Gamma)+1
 \leq\dim_{\mathcal H}\Lambda(\Gamma)+1
 <\frac n2<n,
\]
a contradiction. Therefore there exists an integer
$i\geq\lfloor n/2\rfloor+1$ such that $\pi_i(M^n)\neq0$.

For the remaining conclusions, assume that
$H_1(M^n;\mathbb Z)=0$.

Suppose, for contradiction, that the infinite Kleinian group $\Gamma$ is
elementary. Then $\Gamma$ is virtually abelian. Since abelian subgroups of
Gromov-hyperbolic groups are finite or virtually cyclic, $\Gamma$ must be
virtually cyclic. Taking the core of a finite-index infinite cyclic subgroup
gives a normal infinite cyclic subgroup of finite index, and hence a short
exact sequence
\[
1 \longrightarrow \mathbb{Z} \longrightarrow \Gamma \longrightarrow F \longrightarrow 1,
\]
where $F$ is finite. By Lemma~\ref{Abel}, such a group has nontrivial
abelianization, contradicting
$H_1(M^n;\mathbb Z)\cong\Gamma^{\mathrm{ab}}=0$. Hence $\Gamma$ is
non-elementary.

Suppose now that $\Gamma$ is torsion-free. Since the holonomy representation
is faithful, $\Gamma^{\mathrm{ab}}\cong H_1(M^n;\mathbb Z)=0.$ Thus
$\Gamma$ is perfect, and Lemma~\ref{lem:strict-lower} gives
$\dim_{\mathcal H}\!\bigl(\Lambda(\Gamma)\bigr)>1.$

For the upper bound, Proposition~\ref{prop:standing} gives
$\dim_{\mathcal H}\!\bigl(\Lambda(\Gamma)\bigr)<(n-2)/2$. Combining this
with the strict lower bound, we obtain
$1<\dim_{\mathcal H}\!\bigl(\Lambda(\Gamma)\bigr)<\frac{n-2}{2}$.
\end{proof}

\begin{remark}
Lemma~\ref{lem:strict-lower} excludes the equality
$\dim_{\mathcal H}\Lambda(\Gamma)=1$ in every dimension under the additional
hypotheses $H_1(M;\mathbb Z)=0$ and $\Gamma$ torsion-free in
Theorem~\ref{Kleinian}. In particular, no congruence condition on $n$ is
needed for the strict lower bound.
\end{remark}

\begin{remark}
More generally, let $\Gamma$ be a non-elementary geometrically finite
Kleinian group. If
$\dim_{\mathrm{top}}\Lambda(\Gamma)=\dim_{\mathcal H}\Lambda(\Gamma)=k$,
then
Kapovich's rigidity theorem \cite[Theorem~1.3]{MR2491697} implies that
$\Lambda(\Gamma)$ is a round $k$-sphere. Under the additional hypotheses
$H_1(M;\mathbb Z)=0$ and $\Gamma$ torsion-free in
Theorem~\ref{Kleinian}, the case $k=1$ is impossible by
Lemma~\ref{lem:strict-lower}.
\end{remark}

\begin{remark}
The proof of Theorem~\ref{Kleinian} also yields further information about
the Hausdorff dimension of \(\Lambda(\Gamma)\). For example, suppose that
\(\Gamma\) is elementary. Since \(\Gamma\) is infinite,
\(\#\Lambda(\Gamma)\in\{1,2\}\). If \(\#\Lambda(\Gamma)=1\), then
\(\Omega\cong\mathbb R^n\), so \(M\) is aspherical, contradicting
Corollary~\ref{LCF aspherical}. Hence \(\#\Lambda(\Gamma)=2\), and therefore
\(\dim_{\mathcal H}\Lambda(\Gamma)=0\). Moreover, \(\Gamma\) contains an
infinite cyclic subgroup of finite index generated by a loxodromic element,
\(\Omega\) is conformally equivalent to the round cylinder
\(S^{n-1}\times\mathbb R\), and \(M\) is finitely covered by
\(S^{n-1}\times S^1\).
\end{remark}

\begin{corollary}\label{nonnegative}
Let $(M^n,g)$ be a closed, connected, oriented, locally conformally flat
Riemannian manifold of dimension $n\geq7$ with nonnegative scalar curvature.
Suppose $(M^n,g)$ is not finitely covered by a torus,
\(H_1(M^n;\mathbb Z)=0\), and $\Gamma:=\pi_1(M)$ is infinite. Then $\Gamma$
is non-elementary. Furthermore, if \(\Gamma\) is torsion-free, then
\[
1<\dim_{\mathcal H}\Lambda(\Gamma)\leq\frac{n-2}{2}.
\]
\end{corollary}

\begin{proof}
Since $n\geq7$, $\mathrm{Sc}_g\geq0$, and $M^n$ is not finitely covered
by a torus, Izeki's theorem \cite[Theorem~2(2)]{zbMATH01785298} shows that the
developing map is injective and that the holonomy image, identified with
$\Gamma=\pi_1(M)$, is a convex cocompact Kleinian group. Izeki's
Theorem~4 identifies the developing image with $\Omega(\Gamma)$, so
$M\cong\Omega(\Gamma)/\Gamma$. Izeki's Theorem~3 shows that $\Gamma$ is
Gromov-hyperbolic.
If $\Gamma$ were elementary, then it would be virtually abelian. Since it is
infinite and Gromov-hyperbolic, it would therefore be virtually cyclic.
Taking the core of a finite-index infinite cyclic subgroup yields a normal
infinite cyclic subgroup of finite index and hence a short exact sequence
\[
1\longrightarrow\mathbb Z\longrightarrow\Gamma\longrightarrow F
\longrightarrow1,
\]
where $F$ is finite. Lemma~\ref{Abel} would then imply
$\Gamma^{\mathrm{ab}}\neq0$, contradicting
$H_1(M^n;\mathbb Z)=0$. Hence $\Gamma$ is non-elementary.

Assume now that $\Gamma$ is torsion-free. Since
$\Gamma^{\mathrm{ab}}=0$, Lemma~\ref{lem:strict-lower} gives
$\dim_{\mathcal H}\!\bigl(\Lambda(\Gamma)\bigr)>1.$ Izeki's result
\cite[Theorem~4]{zbMATH01785298} identifies the
Schoen--Yau invariant with the critical exponent and gives
\(\delta(\Gamma)\leq(n-2)/2\). Since \(\Gamma\) is non-elementary and
convex cocompact, Sullivan's theorem~\cite{zbMATH03903608} gives
$\dim_{\mathcal H}\!\bigl(\Lambda(\Gamma)\bigr)=\delta(\Gamma)\leq\frac{n-2}{2}$,
which proves the assertion.
\end{proof}

\begin{remark}
In Corollary~\ref{nonnegative} the two cases separate. Let
$L_g=-\frac{4(n-1)}{n-2}\Delta_g+\mathrm{Sc}_g$ be the conformal
Laplacian and $\lambda_1(L_g)$ its first eigenvalue, whose sign is a conformal
invariant. If $\mathrm{Sc}_g\not\equiv0$, then
$\lambda_1(L_g)>0$: indeed, the Rayleigh quotient gives
$\lambda_1(L_g)\geq0$, while equality would let a positive first eigenfunction
$u$ satisfy $L_gu=0$. In that case,
\[
0=\int_MuL_gu\,dV_g
=\frac{4(n-1)}{n-2}\int_M|\nabla u|^2\,dV_g
 +\int_M\mathrm{Sc}_g u^2\,dV_g.
\]
Both terms are nonnegative, so $u$ is constant and
$\mathrm{Sc}_g\equiv0$, a contradiction. Thus the conformal
class of $g$ contains a metric of positive scalar curvature with the same
holonomy group. Proposition~\ref{prop:standing} therefore gives
$\dim_{\mathcal H}\Lambda(\Gamma)<\frac{n-2}{2}$.
If $\mathrm{Sc}_g\equiv0$, then $\lambda_1(L_g)=0$. Choose a
torsion-free finite-index subgroup $\Gamma_0\leq\Gamma$ by Selberg's lemma and
pull $g$ back to $M_0:=\Omega(\Gamma_0)/\Gamma_0$.
Here $\Lambda(\Gamma_0)=\Lambda(\Gamma)$ and
$\Omega(\Gamma_0)=\Omega(\Gamma)$ by Proposition~\ref{virtual}; moreover,
$\Gamma_0$ is non-elementary and convex cocompact by
Propositions~\ref{virtual} and~\ref{geo finite}. The pullback metric is
scalar-flat and lies in the canonical conformal class on
$\Omega(\Gamma_0)/\Gamma_0$. The finite-coset argument in the proof of
Proposition~\ref{prop:standing} gives
$\delta(\Gamma_0)=\delta(\Gamma)$. If
$\delta(\Gamma)<(n-2)/2$, Nayatani's criterion
\cite[Corollary~3.4]{zbMATH01028179} would therefore imply that this conformal
class contains a metric of positive scalar curvature, contradicting the
conformal invariance of the sign of $\lambda_1$. Since
$\delta(\Gamma)\leq(n-2)/2$ by \cite[Theorem~4]{zbMATH01785298}, it follows
that
$\delta(\Gamma)=\dim_{\mathcal H}\Lambda(\Gamma)=\frac{n-2}{2}$.
No torsion-freeness assumption on $\Gamma$ is needed for this equality.
\end{remark}

\begin{remark}
Suppose, in the setting of Corollary~\ref{nonnegative} but with the
torus-cover exclusion removed, that $M$ is finitely covered by a torus $T^n$.
Then $M$ is flat. Indeed, scalar-curvature rigidity for the
torus~\cite{MR720933} implies that the pullback of $g$ to $T^n$ is flat, and
hence $g$ is flat. After the developing map is postcomposed with a M\"obius
transformation, the classical Euclidean developing description identifies
$\widetilde M$ with $\mathbb R^n=S^n\setminus\{\infty\}$ and $\Gamma$ with
its Bieberbach deck group. It has a finite-index translation subgroup
$L\cong\mathbb Z^n$. In the isometric upper-half-space model of
$\mathbb H^{n+1}$, with $o=(0,1)$ and $T_v(x,t)=(x+v,t)$, one has
\[
d_H(o,T_vo)=2\operatorname{arsinh}\!\left(\frac{|v|}{2}\right)
=2\log|v|+O(1)
\qquad\text{as }|v|\to\infty.
\]
Consequently, the Poincar\'e series of $L$ has the same convergence behavior
as $\sum_{v\in\mathbb Z^n}(1+|v|)^{-2s}$,
whose critical exponent is $n/2$. Finite-index invariance therefore gives
$\delta(\Gamma)=n/2$. The orbit of the translation lattice $L$ accumulates
only at $\infty$, so $\Lambda(L)=\{\infty\}$. Proposition~\ref{virtual}
then gives
$\Lambda(\Gamma)=\{\infty\}$ and $\dim_{\mathcal H}\Lambda(\Gamma)=0$.
Thus the equality between the critical exponent and the Hausdorff dimension
of the limit set does not extend to elementary groups.
\end{remark}

\begin{remark}
By Bieberbach's theorem, every closed flat manifold is finitely covered by a
torus. Thus the hypothesis that \(M\) is not finitely covered by a torus in
Corollary~\ref{nonnegative} excludes the flat case described above.
\end{remark}

On a closed manifold, positive scalar curvature implies
$\mathrm{Sc}_g\geq c>0$ for some constant $c$. The results of Schoen and
Yau~\cite{zbMATH04075988} used above hold in dimensions at least four
under this assumption. In that paper, they also remark that their result
\cite[Proposition 4.4$^{\prime}$]{zbMATH04075988} would remain valid under the
assumption of nonnegative scalar curvature, provided the positive energy
theorem can be extended to the case of complete manifolds---namely, when
\(M\) has one asymptotically flat end and other ends that are merely complete.
Since every orientable \(3\)-manifold is spin, Witten's proof of the positive
mass theorem applies to the orientable \(3\)-dimensional case as well.
Consequently, Schoen and Yau's result~\cite{zbMATH04075988} also holds for
orientable \(3\)-manifolds; see the discussion in
\cite[Appendix~A, p.~805]{zbMATH05042496} for further details.

Recently, Lesourd, Unger, and Yau~\cite[Theorem~1.2]{MR4836036} use minimal hypersurfaces to prove that there does not exist a complete Riemannian metric with PSC on \( T^3 \# X \), where \( X \) is an arbitrary (possibly noncompact) manifold.  
Furthermore, they show that the nonexistence of a complete smooth metric with PSC on \( T^n \# X \) implies the following Liouville theorem~\cite[Theorem~1.7]{MR4836036}, see also~\cite[Corollary~4]{zbMATH07817078}:

\begin{theorem}[Liouville theorem]\label{Liouville theorem}
Let \((M^n,g)\), \(n\geq3\), be a complete, locally conformally flat
manifold with nonnegative scalar curvature. If
\(\Phi\colon M^n\to S^n\) is a conformal map, then \(\Phi\) is injective and
the boundary \(\partial\Phi(M)\) has zero Newtonian capacity (i.e.,
$2$-capacity).
\end{theorem}

In another paper, Lesourd, Unger, and Yau~\cite[Theorem~1.2]{MR4773185} use $\nu$-buddles to prove the above Schoen–Yau conjecture and obtain a new proof of Liouville theorem, which is completely in the spirit of the approach outlined in~\cite{zbMATH04075988}.

\begin{remark}
Corollary~\ref{nonnegative} also holds for \( n \ge 5 \), because
the Liouville theorem implies that Izeki’s theorem~\cite[Theorem~3]{zbMATH01785298} 
holds under the assumptions \( \dim M \ge 3 \) and nonnegative scalar curvature.  
Namely, if \( (M^n, g') \)  \( (n\geq 3) \) is a compact, connected, locally conformally 
flat manifold that is not covered by a torus, and if there exists a metric 
\( g \in [g'] \) with nonnegative scalar curvature, then \( \pi_1(M) \) is a hyperbolic group.  Indeed, the fact that \( \partial \Phi(\widetilde M) \) has zero Newtonian capacity implies that \( \Phi(\widetilde M) \) is dense in \( S^n \), and that 
$\Omega(\Gamma) = \Phi(\widetilde M),$ where \( \Gamma \) is the image of the holonomy representation.  Thus $(M, g') = \Omega(\Gamma)/\Gamma.$
Since the scalar curvature of $g$ in \( g' \) is nonnegative, one has
$\delta(\Gamma) \le \frac{n-2}{2}.$ Hence \( \Gamma \) is a convex cocompact Kleinian group.  Moreover, every convex cocompact Kleinian group is hyperbolic~\cite[Proposition~7]{zbMATH01785298} \footnote{I thank H.~Izeki for explaining this argument to me.}.  
\end{remark}

\subsection{The Macroscopic Dimension of the Universal Riemannian Cover}

The notion of macroscopic dimension was introduced by Gromov to study manifolds with positive scalar curvature~\cite{zbMATH00867495}. He conjectured that if $(M^n,g)$ is a closed Riemannian manifold with positive scalar curvature, then its universal cover satisfies
$\dim_{mc}(\widetilde M)\le n-2.$
Here, $\dim_{mc}(X)\le k$ means that there exist a $k$-dimensional simplicial complex $K$ and a continuous map $f:X\to K$ such that
$\operatorname{diam}\bigl(f^{-1}(y)\bigr)\le b$
for all $y\in K$ and for some constant $b>0$.  We call such maps \emph{uniformly cobounded}. The macroscopic dimension of $X$, denoted by $\dim_{mc}(X)$, is the least integer $k$ such that there exists a continuous uniformly cobounded map $f\colon X\to K$, where $K$ is a $k$-dimensional simplicial complex.

Gromov's conjecture is proved for $3$-manifolds by Bolotov~\cite{zbMATH01985808}. In higher dimensions, some partial results are known; see, for example,~\cite{zbMATH06371564,zbMATH08005926}. As an application of the results established in Subsection~\ref{Hausdorff}, we verify Gromov's conjecture for LCF manifolds.

\begin{theorem}[Macroscopic dimension]\label{macdim}
Let $(M^n,g)$, $n\geq5$, be a closed, connected, oriented, locally conformally flat Riemannian manifold with positive scalar curvature and infinite fundamental group. Then its universal Riemannian cover $(\widetilde M,\widetilde g)$ has macroscopic dimension at most $\left\lfloor\frac{n-1}{2}\right\rfloor$, that is,
\begin{equation*}
\dim_{mc}(\widetilde M,\widetilde g)
\le
\left\lfloor\frac{n-1}{2}\right\rfloor
\le
n-2.
\end{equation*}
\end{theorem}

\medskip
\noindent\textit{Strategy of the proof.}
Izeki's theorem makes \(\Gamma=\pi_1(M)\) a convex cocompact,
Gromov-hyperbolic Kleinian group.  If \(\Gamma\) is non-elementary, then
Nayatani's criterion~\cite[Corollary~3.4]{zbMATH01028179} and Sullivan's
theorem~\cite{zbMATH03903608} give
\(\dim_{\mathcal H}\Lambda(\Gamma)<\frac{n-2}{2}\); if \(\Gamma\) is
elementary, its limit set has topological dimension zero.  In either case,
the integrality of topological dimension gives the required floor function.
For \(\Gamma\) non-elementary, the Buyalo--Lebedeva theorem \cite{zbMATH05232578} applied to the
convex hull \(C=\mathrm{CH}(\Gamma)\), whose
Gromov boundary is identified with \(\Lambda(\Gamma)\) in
Lemmas~\ref{lem:hull} and~\ref{lem:boundary}, computes
\(\dim_{as}C=\dim_{\mathrm{top}}\Lambda(\Gamma)+1\), which the Milnor--\v{S}varc
lemma transports to \((\widetilde M,\widetilde g)\) and the comparison
\(\dim_{mc}\leq\dim_{as}\) converts into the stated bound; the elementary
case is settled directly in Lemma~\ref{lem:elementary} by the proper
uniformly cobounded map \(x\mapsto\log|x|\) on \(\widetilde M\cong\mathbb R^n\setminus\{0\}\).

\begin{lemma}\label{lem:elementary}
Let $(M^n,g)$ satisfy the hypotheses of Theorem~\ref{macdim}, and identify
$\Gamma=\pi_1(M)$ with its Kleinian holonomy group. If \(\Gamma\) is
elementary, then \(\#\Lambda(\Gamma)=2\) and
$\dim_{mc}(\widetilde M,\widetilde g)=1$.
 \end{lemma}

\begin{proof}
By Proposition~\ref{prop:standing}, $\Gamma$ is Gromov-hyperbolic and convex
cocompact.
Since \(\Gamma\) is infinite and discrete, \(\Lambda(\Gamma)\neq\emptyset\).
Suppose \(\Lambda(\Gamma)=\{\xi\}\). Normalizing \(\xi=\infty\), the
group \(\Gamma\) consists of M\"obius transformations of \(S^n\) fixing
\(\infty\), that is, of Euclidean similarities \(x\mapsto\lambda Ax+b\)
with \(\lambda>0\) and \(A\in\mathrm O(n)\).  If some \(\gamma\in\Gamma\) had \(\lambda\neq1\), then \(I-\lambda A\) would be invertible, so \(\gamma\) would have a unique fixed point in \(\mathbb R^n\). Together with \(\infty\), this would give
the two distinct fixed points of a loxodromic element. Both fixed points
of a loxodromic element belong to its limit set and hence to
\(\Lambda(\Gamma)\), contradicting
\(\Lambda(\Gamma)=\{\infty\}\). Therefore every element has
\(\lambda=1\), and consequently $\Gamma\leq\operatorname{Isom}(\mathbb R^n).$
   The developing map identifies \(\widetilde M\)
with \(\Omega(\Gamma)=\mathbb R^n\), so \(M=\mathbb R^n/\Gamma\) is
closed and $\Gamma$ is crystallographic and hence virtually $\mathbb{Z}^n$ by Bieberbach's theorem. For \(n\geq2\) such a group is not Gromov-hyperbolic. Thus \(\Lambda(\Gamma)\) is nonempty and is not a singleton; elementarity therefore gives \(\#\Lambda(\Gamma)=2\).

Normalize \(\Lambda(\Gamma)=\{0,\infty\}\). The developing map then
identifies \(\widetilde M\) conformally with
$\Omega(\Gamma) =S^n\setminus\{0,\infty\} =\mathbb R^n\setminus\{0\}.$
Under this identification, define
$\tau\colon\widetilde M \to \mathbb R,$
$\tau(x)=\log |x|.$
This map is continuous and surjective. Every \(\gamma\in\Gamma\) preserves the unordered pair
\(\{0,\infty\}\).  If \(\gamma\) fixes both points, then
\[
 \gamma(x)=\lambda Ax,
 \qquad \lambda>0,\quad A\in\mathrm O(n),
\]
whereas if \(\gamma\) interchanges them, then
\[
 \gamma(x)=\frac{\lambda Ax}{|x|^2},
 \qquad \lambda>0,\quad A\in\mathrm O(n).
\]
Define \(\rho(\gamma)\in\operatorname{Isom}(\mathbb R)\) by
\[
 \rho(\gamma)(t)=
 \begin{cases}
  t+\log\lambda,
   &\text{if }\gamma(x)=\lambda Ax,\\[2mm]
  \log\lambda-t,
   &\text{if }\displaystyle
     \gamma(x)=\frac{\lambda Ax}{|x|^2}.
 \end{cases}
\]
Then $\tau\circ\gamma=\rho(\gamma)\circ\tau$. Moreover, for
\(\gamma,\gamma'\in\Gamma\),
$\rho(\gamma\gamma')\circ\tau
 =\tau\circ(\gamma\gamma')
=\rho(\gamma)\circ\rho(\gamma')\circ\tau.$
Since \(\tau\) is surjective, it follows that
$\rho(\gamma\gamma')=\rho(\gamma)\circ\rho(\gamma').$ Thus $\rho\colon\Gamma\to \operatorname{Isom}(\mathbb R)
$ is a homomorphism. Because \(M=\widetilde M/\Gamma\) is compact, there exists a compact
set \(F\subset\widetilde M\) such that
$\Gamma F=\widetilde M.$ Choose a compact interval \(J=[a,b]\subset\mathbb R\) containing
\(\tau(F)\).  We claim that $ \rho(\Gamma)J=\mathbb R.$
Indeed, given \(t\in\mathbb R\), choose \(x\in\widetilde M\) with
\(\tau(x)=t\).  Write \(x=\gamma y\) for some
\(\gamma\in\Gamma\) and \(y\in F\).  Then
$t=\tau(\gamma y)=\rho(\gamma)\tau(y),
 \qquad \tau(y)\in J,$ which proves
the claim.

The map \(\tau\) is proper.  Indeed, if \(L\subset\mathbb R\) is
compact and \(L\subset[r,s]\), then
$\tau^{-1}(L)
 \subset
\{x\in\mathbb R^n:e^r\leq |x|\leq e^s\}.$ The left-hand side is closed,
while the annulus on the right is compact
in \(\mathbb R^n\setminus\{0\}\).  Hence \(\tau^{-1}(L)\) is compact.

In particular,
$\tau^{-1}(J)
 =
\{x\in\mathbb R^n:e^a\leq |x|\leq e^b\}$ is compact. Set
$B:=\operatorname{diam}_{\widetilde g}\tau^{-1}(J)<\infty.$
Given \(t\in\mathbb R\), choose \(\gamma\in\Gamma\) and \(t'\in J\)
such that \(t=\rho(\gamma)t'\).  Equivariance, applied also to
\(\gamma^{-1}\), gives $\tau^{-1}(t)=\gamma\,\tau^{-1}(t').$
Because every deck transformation is a
\(\widetilde g\)-isometry,
\[
 \operatorname{diam}_{\widetilde g}\tau^{-1}(t)
 =
 \operatorname{diam}_{\widetilde g}\tau^{-1}(t')
 \leq
 \operatorname{diam}_{\widetilde g}\tau^{-1}(J)
 =B.
\]
Thus \(\tau\) is uniformly cobounded.  Endowing \(\mathbb R\) with its
standard locally finite simplicial structure with vertex set
\(\mathbb Z\), we obtain a continuous uniformly cobounded map from
\(\widetilde M\) to a \(1\)-dimensional simplicial complex.  Therefore
$ \dim_{mc}(\widetilde M,\widetilde g)\leq1.$
Suppose \(\varphi\colon\widetilde M\to K^{0}\) were a uniformly
cobounded continuous map to a \(0\)-dimensional simplicial complex.
Such a complex is discrete and \(\widetilde M\) is connected, so
\(\varphi\) is constant and its single nonempty fiber is all of
\(\widetilde M\); uniform coboundedness would give
\(\operatorname{diam}_{\widetilde g}\widetilde M<\infty\).  But
\((\widetilde M,\widetilde g)\) is complete, being the Riemannian
universal cover of a closed manifold, so by Hopf--Rinow a bounded
\(\widetilde M\) would be compact. A properly discontinuous deck group acting
on a compact universal cover is finite, contradicting the infiniteness of
\(\pi_1(M)\). Hence \(\dim_{mc}(\widetilde M,\widetilde g)=1\).
\end{proof}

We next establish the asymptotic-dimension identity used in the proof of
Theorem~\ref{macdim}. Let $\Gamma$ be a non-elementary convex cocompact
Kleinian group and put $C:=\mathrm{CH}(\Gamma)$. We identify precisely the
boundary that appears in \cite[Theorem~6.4]{zbMATH05232578}. For a closed
subset \(A\subset\mathbb H^{n+1}\), write
$\partial_\infty^{i}A
 :=
\overline A^{\,\mathbb H^{n+1}\cup S^n}\cap S^n$ for its ideal boundary,
and write \(\partial_\infty^{G}X\) for the
sequence-defined Gromov boundary of a hyperbolic space \(X\), as in
\cite[\S~6.1]{zbMATH05232578}. We shall show that
\[
 \partial_\infty^{G}C\cong\partial_\infty^{i}C=\Lambda(\Gamma).
\]

\begin{lemma}\label{lem:hull}
Let \(\Gamma\) be a non-elementary convex cocompact Kleinian group, and let
$ C:=\mathrm{CH}(\Gamma)$
be the closed convex hull of \(\Lambda(\Gamma)\).  Then \(C\) is a
nonempty closed convex subset of \(\mathbb H^{n+1}\).  Equipped with its
intrinsic metric, \(C\) is a proper geodesic
\(\operatorname{CAT}(-1)\) space, hence a Gromov-hyperbolic space, and
\(\Gamma\) acts on \(C\) properly, cocompactly, and isometrically.
Moreover, \(\partial_\infty^{i}C=\Lambda(\Gamma)\).
\end{lemma}

\begin{proof}
The set \(C\) is closed and convex by construction.  Since \(\Gamma\) is
non-elementary, we may choose distinct \(\xi,\eta\in\Lambda(\Gamma)\);
the complete hyperbolic geodesic with ideal endpoints \(\xi\) and \(\eta\)
lies in \(C\), so
\(C\neq\varnothing\).

Convexity implies that the ambient hyperbolic geodesic joining any two
points of \(C\) lies in \(C\).  Thus \(C\) is geodesic and
$d_C=d_{\mathbb H^{n+1}}\big|_{C\times C}.$
Since \(C\) is closed in the proper space \(\mathbb H^{n+1}\), it is
proper.  Since it is convex in the \(\operatorname{CAT}(-1)\) space
\(\mathbb H^{n+1}\), it is itself \(\operatorname{CAT}(-1)\) and hence
Gromov-hyperbolic.  

The limit set is \(\Gamma\)-invariant, so \(\Gamma\) preserves \(C\) and
acts on it isometrically.  Since \(\Gamma\) is discrete, it acts
properly on \(\mathbb H^{n+1}\), and therefore properly on the closed
invariant subset \(C\).  Its action on \(C\) is cocompact by convex
cocompactness.

Fix \(o\in C\).  Since \(C\) is \(\Gamma\)-invariant, \(\Gamma o\subset C\).
Conversely, cocompactness supplies a compact \(K\subset C\) with
\(\Gamma K=C\); writing \(R:=\max_{k\in K}d_{\mathbb H^{n+1}}(o,k)<\infty\),
we see that every \(x\in C\) is of the form \(x=\gamma k\) with \(\gamma\in\Gamma\),
\(k\in K\), whence
\(d_{\mathbb H^{n+1}}(x,\gamma o)=d_{\mathbb H^{n+1}}(k,o)\leq R\).  Thus, $C$ lies in the $R$-neighborhood of the orbit $\Gamma o$, that is,
\begin{equation}\label{eq:orbit-coarse-dense}
\Gamma o\subset C\subset N_R(\Gamma o).
\end{equation}
We use that hyperbolic \(R\)-balls are Euclidean-small near the sphere. In
the ball model, if \(t=\tanh(R/2)\), then
\(B_{\mathbb H^{n+1}}(x,R)\) is a Euclidean ball of radius
\[
 r_E(x,R)=\frac{t(1-|x|^2)}{1-t^2|x|^2}.
\]
Consequently,
\[
\operatorname{diam}_{\mathrm{eucl}}B_{\mathbb H^{n+1}}(x,R)
\leq\frac{4t}{1-t^2}(1-|x|).
\]
Thus two sequences at hyperbolic distance at most \(R\), one of which tends
to the sphere, have the same ideal limit.

Let \(\xi\in\partial_\infty^{i}C\) and choose \(x_j\in C\) with
\(x_j\to\xi\) in \(\mathbb H^{n+1}\cup S^n\); as \(\xi\in S^n\) we have
\(|x_j|\to1\).  By \eqref{eq:orbit-coarse-dense} pick
\(\gamma_j\in\Gamma\) with
\(d_{\mathbb H^{n+1}}(x_j,\gamma_j o)\leq R\); by the previous
paragraph \(\gamma_j o\to\xi\), so \(\xi\in\Lambda(\Gamma)\).
Conversely, by definition, $\Lambda(\Gamma)$ is the set of ideal accumulation points of the orbit $\Gamma o$, independently of the choice of base point. Since $\Gamma o\subset C$, every such accumulation point belongs to
$\overline C^{\,\mathbb H^{n+1}\cup S^n}\cap S^n=\partial_\infty^{i}C.$
Therefore \(\partial_\infty^{i}C=\Lambda(\Gamma)\).
\end{proof}

Fix a point $o\in C$. We write the Gromov product as
\((x\mid y)_o:=\tfrac12\bigl(d(o,x)+d(o,y)-d(x,y)\bigr)\geq0\).
Every \(\operatorname{CAT}(-1)\) space is
\(\delta\)-hyperbolic for a universal \(\delta\geq0\); by
Lemma~\ref{lem:hull} both \(\mathbb H^{n+1}\) and \(C\) are
\(\operatorname{CAT}(-1)\), so we fix one such \(\delta\) valid for both,
and both then satisfy
\((x\mid y)_o\geq\min\{(x\mid z)_o,(z\mid y)_o\}-\delta\).

For a $\delta$-hyperbolic space \(X\), the product extends to
\(\partial_\infty^{G}X\) by
\[
 (\xi\mid\eta)_o:=\inf\liminf_j(x_j\mid y_j)_o,
\]
where the infimum ranges over representing sequences
\((x_j)\in\xi\) and \((y_j)\in\eta\). For \emph{any} such pair,
\begin{equation}\label{eq:BL-boundary-estimate}
 (\xi\mid\eta)_o
 \;\leq\;\liminf_{j\to\infty}(x_j\mid y_j)_o
 \;\leq\;(\xi\mid\eta)_o+2\delta .
\end{equation}
The topology of \(\partial_\infty^{G}X\) is that in which
\(\xi_k\to\xi\) iff \((\xi_k\mid\xi)_o\to\infty\), equivalently the one
induced by any visual metric \(\rho\), i.e., one satisfying
\(c_1a^{-(\xi\mid\eta)_o}\leq\rho(\xi,\eta)\leq c_2a^{-(\xi\mid\eta)_o}\)
for some \(a>1\), \(c_1,c_2>0\).

We use ball-model coordinates with \(o\) at the center, normalize
\(d_{\mathrm{ch}}(\alpha,\beta)=|\alpha-\beta|\) on
\(S^n\subset\mathbb R^{n+1}\), and let
\(\mathcal I\colon\partial_\infty^{G}\mathbb H^{n+1}\to S^n\) send the
class of a Gromov sequence to its ideal limit. Since
\(\mathbb H^{n+1}\) is a proper \(\operatorname{CAT}(-1)\) space, Bourdon's
boundary theory~\cite[\S\S~1.4 and~2.4--2.5]{MR1341941} shows that, under
this identification, every sequence representing \(\alpha\) converges to
\(\mathcal I\alpha\) in \(\mathbb H^{n+1}\cup S^n\), and that the Gromov
product extends continuously to the complement of the boundary diagonal in
\((\mathbb H^{n+1}\cup S^n)^2\).
Consequently, if
\(\alpha\neq\beta\) and \((x_j)\in\alpha\), \((y_j)\in\beta\) are any
representing sequences, then
\(\lim_{j\to\infty}(x_j\mid y_j)_o=(\alpha\mid\beta)_o\)
exists and is independent of the representatives. To compute this value, let
\(\theta\) be the round angle between \(\mathcal I\alpha\) and
\(\mathcal I\beta\), and use the unit-speed radial representatives
\[
x_t=\tanh(t/2)\,\mathcal I\alpha,
\qquad
y_t=\tanh(t/2)\,\mathcal I\beta.
\]
The hyperbolic law of cosines gives
\[
\cosh d(x_t,y_t)
=\cosh^2t-\sinh^2t\cos\theta
=1+2\sinh^2t\sin^2(\theta/2).
\]
Hence
\[
d(x_t,y_t)=2t+2\log\sin(\theta/2)+o(1),
\qquad t\to\infty,
\]
and therefore
\((\alpha\mid\beta)_o=-\log\sin(\theta/2)\). Since the chordal distance is
\(2\sin(\theta/2)\), we obtain
\begin{equation}\label{eq:chordal-gromov-product}
 d_{\mathrm{ch}}\bigl(\mathcal I\alpha,\mathcal I\beta\bigr)
 =2e^{-(\alpha\mid\beta)^{\mathbb H^{n+1}}_o} .
\end{equation}

\begin{lemma}\label{lem:boundary}
The inclusion \(C\hookrightarrow\mathbb H^{n+1}\) induces
\(\iota_\infty\colon\partial_\infty^{G}C\to
\partial_\infty^{G}\mathbb H^{n+1}\), and \(\mathcal I\circ\iota_\infty\)
is a homeomorphism of \(\partial_\infty^{G}C\) onto \(\Lambda(\Gamma)\)
with its round subspace topology.
\end{lemma}

\begin{proof}
The map \(\iota_\infty\) is well defined and injective.
Since \(d_C\) is the restriction of \(d_{\mathbb H^{n+1}}\), the product
\((x\mid y)_o\) of points of \(C\) depends only on three mutual
distances that agree in either space, so
\begin{equation}\label{eq:products-agree}
 (x\mid y)^C_o=(x\mid y)^{\mathbb H^{n+1}}_o
 \qquad(x,y\in C).
\end{equation}
Hence a sequence in \(C\) is a Gromov sequence for \(C\) iff it is one
for \(\mathbb H^{n+1}\), and two such are equivalent for \(C\) iff they
are equivalent for \(\mathbb H^{n+1}\).  The first gives
well-definedness, the second injectivity.

Let \(\xi\neq\eta\) in \(\partial_\infty^{G}C\), with representing
sequences \((x_j)\in\xi\), \((y_j)\in\eta\). Then the same pair
represents \(\iota_\infty\xi,\iota_\infty\eta\), and the
estimate \eqref{eq:BL-boundary-estimate} holds in both spaces with the
same \(\delta\).  When this estimate is applied in both \(C\) and \(\mathbb H^{n+1}\) to the
same pair, both enclosing intervals contain
\(\liminf_j(x_j\mid y_j)_o\), whence
\begin{equation}\label{eq:boundary-product-comparison}
 \Bigl|
 (\xi\mid\eta)^C_o
 -(\iota_\infty\xi\mid\iota_\infty\eta)^{\mathbb H^{n+1}}_o
 \Bigr|\leq2\delta .
\end{equation}

Substituting \eqref{eq:boundary-product-comparison} into
\eqref{eq:chordal-gromov-product} gives, for \(\xi\neq\eta\),
\begin{equation}\label{eq:chordal-is-visual}
 2e^{-2\delta}e^{-(\xi\mid\eta)^C_o}
 \;\leq\;
 d_{\mathrm{ch}}\bigl(
 \mathcal I(\iota_\infty\xi),\mathcal I(\iota_\infty\eta)\bigr)
 \;\leq\;
 2e^{2\delta}e^{-(\xi\mid\eta)^C_o} .
\end{equation}
Thus the pullback of \(d_{\mathrm{ch}}\) is a visual metric on
\(\partial_\infty^{G}C\) with parameter \(a=e\) and constants
\(c_1=2e^{-2\delta}\), \(c_2=2e^{2\delta}\). It therefore induces
the topology of \(\partial_\infty^{G}C\).  Hence
\(\mathcal I\circ\iota_\infty\) is a homeomorphism onto its image for
the topology induced by \(d_{\mathrm{ch}}\), which is the round subspace
topology, \(d_{\mathrm{ch}}\) and the round metric being bi-Lipschitz
equivalent on \(S^n\).

We now prove \(\operatorname{im}(\mathcal I\circ\iota_\infty)
=\partial_\infty^{i}C\).  Let \(\xi\in\partial_\infty^{G}C\) be represented by
\((x_j)\subset C\). Viewed as a sequence in \(\mathbb H^{n+1}\), it
represents \(\iota_\infty\xi\). By the standard ball-model identification
used to define \(\mathcal I\),
\(x_j\to \mathcal I(\iota_\infty\xi)\in S^n\).
Since \(x_j\in C\) for every
\(j\), that limit lies in
\(\overline C^{\,\mathbb H^{n+1}\cup S^n}\cap S^n=\partial_\infty^{i}C\).

Let \(\zeta\in\partial_\infty^{i}C\) and choose
\(x_j\in C\) with \(x_j\to\zeta\) in \(\mathbb H^{n+1}\cup S^n\).  In the
ball model with \(o\) at the center one has
\(e^{-d(o,x)}=\frac{1-|x|}{1+|x|}\) and
\(\sinh^2\!\bigl(\tfrac{d(x,y)}{2}\bigr)
=\frac{|x-y|^2}{(1-|x|^2)(1-|y|^2)}\), so \(e^{d(x,y)}\leq4\sinh^2(d(x,y)/2)+2\)
gives
\[
 e^{-2(x\mid y)_o}
 =e^{d(x,y)}e^{-d(o,x)}e^{-d(o,y)}
 \leq 4|x-y|^{2}+2\bigl(1-|x|\bigr)\bigl(1-|y|\bigr).
\]
As \(i,j\to\infty\), we have \(|x_i-x_j|\to0\) and
\((1-|x_i|)(1-|x_j|)\to0\). Hence
\((x_i\mid x_j)_o\to\infty\). Thus \((x_j)\) is a Gromov sequence in
\(\mathbb H^{n+1}\), and its class has ideal limit \(\zeta\). By
\eqref{eq:products-agree}, \((x_j)\) is a Gromov sequence in \(C\) as well;
let \(\xi\in\partial_\infty^{G}C\) be its class there.  By construction
\(\iota_\infty\) sends the \(C\)-class of a sequence to the
\(\mathbb H^{n+1}\)-class of the same sequence, so
\(\iota_\infty\xi=[(x_j)]_{\mathbb H^{n+1}}\) and hence
\(\mathcal I(\iota_\infty\xi)=\zeta\).

Hence \(\operatorname{im}(\mathcal I\circ\iota_\infty)
=\partial_\infty^{i}C\), which equals \(\Lambda(\Gamma)\) by
Lemma~\ref{lem:hull}. Therefore \(\mathcal I\circ\iota_\infty\)
is a homeomorphism of \(\partial_\infty^{G}C\) onto \(\Lambda(\Gamma)\)
with its round subspace topology.
\end{proof}

In particular, there is a canonical homeomorphism
$\partial_\infty^{G}C\cong\Lambda(\Gamma)$, where the limit set carries its
round subspace topology. Thus,
$\dim_{\mathrm{top}}(\partial_\infty^{G}C)=\dim_{\mathrm{top}}\Lambda(\Gamma)$.

The Buyalo--Lebedeva theorem \cite[Theorem~6.4]{zbMATH05232578} says that the asymptotic dimension of every cobounded, Gromov-hyperbolic, proper, geodesic space $X$ equals the topological dimension of its boundary at infinity plus 1. Here a metric space $X$ is cobounded if there is a bounded subset $A \subset X$ such that the orbit of $A$ under the isometry group of $X$ covers $X$.

We verify that $C$ satisfies the four hypotheses of the Buyalo--Lebedeva theorem. Recall from Lemma~\ref{lem:hull} that $C$ is a proper, geodesic, Gromov-hyperbolic space. Moreover, the action of $\Gamma$ on $C$ is cocompact. Hence there exists a compact, and therefore bounded, subset $K\subset C$ such that $\Gamma K=C$. Since $\Gamma$ acts by isometries, its image is contained in $\operatorname{Isom}(C)$, and the $\operatorname{Isom}(C)$-orbit of $K$ also covers $C$. Thus the action of $\operatorname{Isom}(C)$ on $C$ is cobounded.
Therefore, the Buyalo--Lebedeva theorem \cite[Theorem~6.4]{zbMATH05232578}, together with Lemmas~\ref{lem:hull} and~\ref{lem:boundary}, yields
\[
\dim_{as}(C)
=
\dim_{\mathrm{top}}(\partial_\infty^{G}C)+1
=
\dim_{\mathrm{top}}\Lambda(\Gamma)+1.
\]

By Lemma~\ref{lem:hull}, \(C\) is proper and geodesic, hence a length
space, and \(\Gamma\) acts on it properly, cocompactly, and
isometrically. The Milnor--\v{S}varc lemma applies: \(\Gamma\)
is finitely generated, and for any \(o\in C\) the orbit map
\(\gamma\mapsto\gamma o\) is a quasi-isometry from \(\Gamma\) to \(C\) for the word
metric associated with a finite generating set supplied by the
Milnor--\v{S}varc lemma. Word metrics of any two finite generating sets are
quasi-isometric, so for any finite generating set \(S\), $(\Gamma,d_S)$ is
quasi-isometric to $C$.

A quasi-isometry is a coarse equivalence, and asymptotic dimension is
invariant under coarse equivalence \cite[Proposition~22]{zbMATH05292205}; thus,
\[
\dim_{as}(\Gamma,d_S)=
\dim_{as}(C)
=
\dim_{\mathrm{top}}(\partial_\infty^{G}C)+1
=
\dim_{\mathrm{top}}\Lambda(\Gamma)+1.
\]

\begin{proof}[Proof of Theorem~\ref{macdim}]
By Proposition~\ref{prop:standing}, $\Gamma=\pi_1(M)$ is Gromov-hyperbolic
and its Kleinian action is convex cocompact, hence geometrically finite. The
elementary case follows from Lemma~\ref{lem:elementary}, since $n\geq5$.

Suppose now that $\Gamma$ is non-elementary. Proposition~\ref{prop:standing}
gives
$\dim_{\mathcal H}\Lambda(\Gamma)=\delta(\Gamma)<\frac{n-2}{2}$.
The limit set is a compact metric space, and hence
$d:=\dim_{\mathrm{top}}\Lambda(\Gamma)
\leq\dim_{\mathcal H}\Lambda(\Gamma)<\frac{n-2}{2}$.
The integer-valuedness of $d$ now gives
$d+1\leq\left\lfloor\frac{n-1}{2}\right\rfloor$.
Indeed, if $n=2m$, then $d<m-1$, so $d\leq m-2$; if $n=2m+1$,
then $d<m-\frac12$, so $d\leq m-1$.

We use asymptotic dimension $\dim_{as}$, in the sense of
\cite{zbMATH05292205}, to complete the non-elementary case, since
$\dim_{mc}(X)\leq \dim_{as}(X)$ for every metric space $X$
\cite[p.~230]{zbMATH06371564}. The identity established above gives
$\dim_{as}(\Gamma,d_S)=\dim_{\mathrm{top}}\Lambda(\Gamma)+1$,
where $d_S$ is a word metric.
The deck transformation action of $\Gamma$ on
$(\widetilde M,\widetilde g)$ is proper, cocompact, and isometric. Since
$M$ is closed, $(\widetilde M,\widetilde g)$ is complete, proper, and
geodesic. By the Milnor--\v{S}varc lemma,
$(\widetilde M,\widetilde g)$ is quasi-isometric to $\Gamma$, equipped with
any word metric. Since asymptotic dimension is invariant under
quasi-isometry \cite[Proposition~22]{zbMATH05292205}, we obtain
$\dim_{as}(\widetilde M,\widetilde g)=\dim_{as}(\Gamma,d_S).$ Therefore,
for $n\geq5$, we have
\[
\dim_{mc}(\widetilde M,\widetilde g)
\leq
\dim_{as}(\widetilde M,\widetilde g)
=
\dim_{as}(\Gamma,d_S)
=
\dim_{\mathrm{top}}\Lambda(\Gamma)+1
\leq
\left\lfloor\frac{n-1}{2}\right\rfloor
\leq
n-2.
\]
\end{proof}

\begin{remark}
The assumption of PSC in Theorem~\ref{macdim} cannot be weakened to
nonnegative scalar curvature, since the flat torus $T^n$ has universal cover
$\mathbb R^n$, and $\dim_{mc}(\mathbb R^n)=n$.  
Furthermore, since $\dim_{mc}(\mathbb H^n)=n$, taking the metric product of a round sphere with a closed hyperbolic manifold shows that the upper bound is sharp. Indeed, every integer in the range
$\{0,1,\ldots,\left\lfloor\frac{n-1}{2}\right\rfloor\}$
is realized as the upper bound for $\dim_{mc}$.
\end{remark}

\begin{remark}
For a manifold $M$ as in Theorem~\ref{macdim}, we have
$
\dim_{mc}(\widetilde M,\widetilde g)
\le \left\lfloor\frac{n-1}{2}\right\rfloor
\le n-2<n.
$
Hence $M$ is md-small in the sense of Dranishnikov~\cite{zbMATH06371564}, and \cite[Theorem~5.4(2)]{zbMATH06371564} implies that the locally finite fundamental class of $M$ maps to zero in the integral coarse homology group $HX_n(M;\mathbb Z)$.
\end{remark}

\begin{proof}[Proof of Theorem~\ref{thm:intro-A}]
The macroscopic dimension estimate follows immediately from Theorem~\ref{macdim}, while the remaining assertions are consequences of Theorem~\ref{Kleinian}.
\end{proof}

\section{Maps between LCF manifolds with PSC}\label{4}

In this section, we study the geometric and topological constraints imposed by
nonzero-degree maps between closed LCF manifolds with PSC. We first show that
the vanishing of the Hausdorff dimension of the holonomy limit set is preserved
under such maps. We then exploit the dichotomy between elementary and
non-elementary holonomy Kleinian groups to establish Llarull-type rigidity
results. In the elementary case, smooth \(1\)-Lipschitz maps of nonzero degree to
the standard sphere are rigid; whenever the fundamental group is infinite---in
particular, in the non-elementary case---the developing map expands some
tangent vector at some point.

\begin{proposition}\label{degree}
Let $(M^n,g_M)$ and $(N^n,g_N)$, $n\geq4$, be closed, connected,
oriented, locally conformally flat Riemannian manifolds with positive
scalar curvature. Let $\Gamma_M$ and $\Gamma_N$ be their holonomy
Kleinian groups. Assume that
\(
\dim_{\mathcal H}\Lambda(\Gamma_M)=0.
\)
If there exists a continuous map $f\colon M^n\to N^n$ of nonzero degree, then
$\Gamma_N$ is elementary. In particular,
$\dim_{\mathcal H}\Lambda(\Gamma_N)=0.$
\end{proposition}

\begin{proof}
For each $X\in\{M,N\}$, let $\widetilde X$ be the universal cover. Then the developing map
$\mathrm{dev}_X\colon\widetilde X\to S^n$
is injective, with image a $\Gamma_X$-invariant domain
$\Omega_X\subset S^n$. Its holonomy representation is an isomorphism,
$\rho_X\colon\pi_1(X)\stackrel{\cong}{\to}\Gamma_X,$
and the developing map gives an equivariant diffeomorphism
$\widetilde X\cong\Omega_X$. Thus
$X\cong\Omega_X/\Gamma_X$, and the action of $\Gamma_X$ on $\Omega_X$
is free.

Choose basepoints and put
\[
\varphi:=\rho_N\circ f_\#\circ\rho_M^{-1}
\colon\Gamma_M\longrightarrow\Gamma_N,
\qquad H:=\varphi(\Gamma_M).
\]
Let
\(q\colon\Omega_N/H\to \Omega_N/\Gamma_N\cong N\)
be the connected covering associated with the subgroup
\(\rho_N^{-1}(H)\leq\pi_1(N)\). Since
\(
\rho_N^{-1}(H)=f_\#\bigl(\pi_1(M)\bigr),
\)
the lifting criterion gives a map
\(\bar f\colon M\Omega_N/H\)
with \(q\circ\bar f=f\).
If $[\Gamma_N:H]=\infty$, then $q$ is infinite-sheeted. Its total space
cannot be compact, since a covering of the compact manifold \(N\) with
compact total space has finite fibers. Thus \(\Omega_N/H\) is a connected
noncompact $n$-manifold, and hence
$H_n(\Omega_N/H;\mathbb Z)=0$. Therefore
\(
f_*[M]=q_*\bar f_*[M]=0,
\)
contradicting $f_*[M]=\deg(f)[N]\neq0$. Therefore
$[\Gamma_N:H]<\infty$.

We claim that $\Gamma_M$ is elementary. If it were non-elementary,
Proposition~\ref{prop:standing} would give
$\dim_{\mathcal H}\Lambda(\Gamma_M)=\delta(\Gamma_M)>0$,
a contradiction. Hence $\Gamma_M$ is elementary and therefore virtually
abelian, by the structure of elementary Kleinian groups recalled above.
Choose an abelian subgroup $A\leq\Gamma_M$ of finite index. Then
$\varphi(A)$ is abelian and
$[H:\varphi(A)]\leq[\Gamma_M:A]<\infty$.
Consequently,
$[\Gamma_N:\varphi(A)]
=[\Gamma_N:H]\,[H:\varphi(A)]
\leq[\Gamma_N:H]\,[\Gamma_M:A]<\infty$.

Since $\varphi(A)$ is an abelian Kleinian group, it is elementary, so its
limit set contains at most two points. Proposition~\ref{virtual}, applied
to the finite-index inclusion $\varphi(A)\leq\Gamma_N$, gives
$\Lambda(\Gamma_N)=\Lambda(\varphi(A))$. Hence $\Lambda(\Gamma_N)$
contains at most two points, so $\Gamma_N$ is elementary, and
consequently $\dim_{\mathcal H}\Lambda(\Gamma_N)=0$.
\end{proof}

\begin{remark}
The Hausdorff dimension of the limit set is not determined by the abstract
isomorphism type of a Kleinian group. Let \(\Sigma_h\) be a closed,
orientable surface of genus \(h\geq2\), and let
$
\rho_0\colon\pi_1(\Sigma_h)\to \mathrm{PSL}(2,\mathbb R)
$
be a cocompact Fuchsian representation. Set
\(\Gamma_0:=\rho_0(\pi_1(\Sigma_h))\). The group \(\Gamma_0\) preserves a
totally geodesic copy of \(\mathbb H^2\subset\mathbb H^3\), and its limit
set is the ideal boundary of this plane, hence a round circle. Therefore
$\dim_{\mathcal H}\Lambda(\Gamma_0)=1.$ One can choose a quasi-Fuchsian representation
$\rho_1\colon\pi_1(\Sigma_h)\to\mathrm{PSL}(2,\mathbb C)$
which is not conjugate into \(\mathrm{PSL}(2,\mathbb R)\), and set
\(\Gamma_1:=\rho_1(\pi_1(\Sigma_h))\).
The limit set \(\Lambda(\Gamma_1)\) is a quasicircle, and hence a Jordan
curve in \(\mathbb{CP}^1\); consequently, $1\leq\dim_{\mathcal H}\Lambda(\Gamma_1)\leq2.$
Bowen's theorem~\cite[Theorem~2]{MR556580}, in the standard formulation
recalled by Farre--Pozzetti--Viaggi
\cite[\S~1, p.~2]{2024arXiv240720071F}, states that the limit set of
a quasi-Fuchsian representation has Hausdorff dimension \(1\) if and only
if the representation is Fuchsian, that is, conjugate into
\(\mathrm{PSL}(2,\mathbb R)\). Since \(\rho_1\) is non-Fuchsian, it follows
that $\dim_{\mathcal H}\Lambda(\Gamma_1)>1.$

Both \(\rho_0\) and \(\rho_1\) are faithful, so
$\rho_1\circ\rho_0^{-1}\colon\Gamma_0\to \Gamma_1$
is an abstract group isomorphism. Finally, for every \(n\geq2\), the standard inclusion
\(\operatorname{Conf}(S^2)\hookrightarrow\operatorname{Conf}(S^n)\)
regards \(\Gamma_0\) and \(\Gamma_1\) as Kleinian groups in
\(\operatorname{Conf}(S^n)\), without changing their limit sets or their
Hausdorff dimensions.
\end{remark}

\begin{proposition}\label{Llarull}
Let \((M^{n},g)\), \(n\geq4\), be a closed, connected, oriented,
locally conformally flat manifold with scalar curvature
\(\mathrm{Sc}_{g}\geq n(n-1)\). Identify
\(\Gamma:=\pi_1(M)\) with its holonomy Kleinian group, and suppose that
\(\Gamma\) is elementary. If there exists a smooth distance-nonincreasing (i.e.,
\(1\)-Lipschitz) map
\(
f\colon(M^n,g)\to (S^{n},g_{st})
\)
of nonzero degree, then \(f\) is a Riemannian isometry.
\end{proposition}

\begin{proof}

 Identifying \(\widetilde{M}\) with
$\Omega:=\mathrm{dev}(\widetilde{M})\subset S^n$ and transporting the
lifted metric to $\Omega$, we write
\(\tilde g=u^{\frac{4}{n-2}}g_{st}\big|_{\Omega}\)
for some positive smooth function \(u\colon\Omega\to\mathbb R_{>0}\).
By the Kleinian realization established in Section~\ref{3},
\(\Omega=\Omega(\Gamma)\), and hence
\(\partial\Omega=\Lambda(\Gamma)\). Since \(\Gamma\) is elementary,
\(\partial\Omega\) contains at most two points.

The boundary $\partial\Omega$ cannot consist of a single point. Indeed,
if \(\partial\Omega=\{\xi\}\), then
\(\widetilde M\cong\Omega=S^n\setminus\{\xi\}\cong\mathbb R^n\),
so \(\widetilde M\) is contractible and \(M\) is aspherical. If $n=4$,
this contradicts the Schoen--Yau obstruction to positive scalar curvature
on closed aspherical $4$-manifolds~\cite[Theorem~6]{zbMATH04075990}. If
$n\geq5$, it contradicts Corollary~\ref{LCF aspherical}.

Assume that \( \partial \Omega \) consists of exactly two points. Then
\(\widetilde M\) is diffeomorphic to \(S^{n-1}\times\mathbb R\) and
has two ends. Choose a finite CW structure on $M$ and lift it to
$\widetilde M$. The resulting complex is a path-connected free
$\pi_1(M)$-CW complex with finite quotient, so
\cite[Corollary~13.5.12]{zbMATH05189637} implies that $\pi_1(M)$ has two
ends. It therefore contains a finite-index subgroup isomorphic to
$\mathbb Z$ by \cite[Theorem~13.5.9]{zbMATH05189637}. Hence $M$ admits a
connected finite cover
\(\widehat M\) such that
\(\pi_1(\widehat M)\cong\mathbb Z\). Since \(M\) is closed and
oriented, so is \(\widehat M\).

We claim that \(\widehat M\) is spin. Its universal cover satisfies
\(\widetilde{\widehat M}\cong S^{n-1}\times\mathbb R\simeq S^{n-1}\).
Since \(n\geq4\), it follows that
\(\pi_2(\widehat M)=\pi_2(S^{n-1})=0\). The Hopf exact sequence
\[
\pi_2(\widehat M)\longrightarrow H_2(\widehat M;\mathbb Z)
\longrightarrow H_2(\pi_1(\widehat M);\mathbb Z)\longrightarrow0
\]
and the equality
\(H_2(\pi_1(\widehat M);\mathbb Z)
=H_2(\mathbb Z;\mathbb Z)=0\) imply that
\(H_2(\widehat M;\mathbb Z)=0\). The universal coefficient theorem
gives the exact sequence
\[
0\longrightarrow
\operatorname{Ext}_{\mathbb Z}
\bigl(H_1(\widehat M;\mathbb Z),\mathbb Z_2\bigr)
\longrightarrow H^2(\widehat M;\mathbb Z_2)
\longrightarrow
\operatorname{Hom}_{\mathbb Z}
\bigl(H_2(\widehat M;\mathbb Z),\mathbb Z_2\bigr)
\longrightarrow0.
\]
Since
\(H_1(\widehat M;\mathbb Z)
\cong\pi_1(\widehat M)^{\mathrm{ab}}
\cong\mathbb Z\),
we have
\(\operatorname{Ext}_{\mathbb Z}(\mathbb Z,\mathbb Z_2)=0\), while
\(H_2(\widehat M;\mathbb Z)=0\) gives
\(\operatorname{Hom}_{\mathbb Z}(0,\mathbb Z_2)=0\). Therefore
\(
H^2(\widehat M;\mathbb Z_2)=0.
\)
Thus \(w_2(T\widehat M)=0\), and since
\(w_1(T\widehat M)=0\), the manifold \(\widehat M\) is spin. In fact,
for \(n\geq5\), one may choose \(\widehat M\) to be diffeomorphic to
\(S^{n-1}\times S^1\); see~\cite{504170}.

Give \(\widehat M\) the pullback orientation, let
\(p\colon\widehat M\to M\) denote the finite covering, and equip
\(\widehat M\) with the lifted metric \(\widehat g\). Set
\(\widehat f:=f\circ p\colon(\widehat M,\widehat g)
\to (S^n,g_{st})\).
Then \(\widehat f\) is smooth and distance-nonincreasing, and
\(
\deg(\widehat f)=\deg(p)\deg(f)\neq0.
\)
Since \(\mathrm{Sc}_{\widehat g}\geq n(n-1)\), Llarull's rigidity
theorem~\cite{MR1600027} implies that \(\widehat f\) is a Riemannian
isometry. This would make
\(\widehat M\) simply connected, contradicting
\(\pi_1(\widehat M)\cong\mathbb Z\). Thus \(\partial\Omega\) cannot
consist of two points.

Consequently, $\partial\Omega=\varnothing$, so $\Omega=S^n$. Thus
$\widetilde M$ is compact and spin, and the universal covering
\(\pi\colon\widetilde M\to M\) is finite-sheeted. Give
\(\widetilde M\) the pullback orientation and set
\(\widetilde f:=f\circ\pi\colon(\widetilde M,\tilde g)
\to (S^n,g_{st})\).
Then \(\widetilde f\) is smooth and distance-nonincreasing, and
\(
\deg(\widetilde f)=\deg(\pi)\deg(f)\neq0.
\)
Llarull's rigidity theorem implies that
\(\widetilde f\) is a Riemannian isometry. Hence
\((\widetilde M,\tilde g)\) is isometric to \((S^n,g_{st})\).
Conjugation by $\widetilde f$ identifies $\Gamma=\pi_1(M)$ with a finite
subgroup of $\mathrm{O}(n+1)$ acting freely on $S^n$. With this
identification,
$(M,g)\cong(S^n/\Gamma,g_{st})$.

Because the hypotheses do not guarantee that the spherical space form \(M\)
is spin, we cannot apply Llarull's rigidity theorem directly to \(M\). We
instead use the degree formula to descend the rigidity conclusion to $M$.
Since $f$ is smooth and $1$-Lipschitz,
\(\|df_p\|\leq1\) and \(|\operatorname{Jac}f(p)|\leq1\)
for every \(p\in M\).
The degree formula gives
\[
\begin{aligned}
|\deg(f)|\,\operatorname{Vol}_{g_{st}}(S^n)
&=
\left|\int_M f^*(d\operatorname{vol}_{g_{st}})\right| \\
&\leq
\int_M|\operatorname{Jac}f|\,d\operatorname{vol}_g \\
&\leq
\operatorname{Vol}_g(M)
=\frac{\operatorname{Vol}_{g_{st}}(S^n)}{|\Gamma|}.
\end{aligned}
\]
Since $|\deg(f)|$ and $|\Gamma|$ are positive integers, this forces
\(
|\deg(f)|=|\Gamma|=1.
\)
Thus \(\Gamma\) is trivial. Hence the universal covering
\(\pi\colon(\widetilde M,\tilde g)\to(M,g)\) is one-sheeted and is
therefore a Riemannian isometry. Since
\(\widetilde f=f\circ\pi\) is a Riemannian isometry, so is
\(f=\widetilde f\circ\pi^{-1}\).
\end{proof}

\begin{remark}
Proposition~\ref{Llarull} also holds in dimension three without assuming either local conformal flatness or that $\pi_1(M)$ is elementary. Indeed, in 2022 Lee and Tam \cite{LeeTam2026} extend  Cecchini, Hanke, and Schick's result \cite{zbMATH08176144} from even dimensions to all dimensions, and every oriented $3$-manifold is spin. These two results do not require local conformal flatness, and the smoothness assumptions on the metric and the map can be relaxed to less than \(C^{2}\).  The weighted Llarull rigidity theorem for weighted Riemannian manifolds is established by the author in \cite[Theorem~1.3]{zbMATH07342230}. These results rely on the spin condition, and the proofs follow Llarull's index-theoretic rigidity argument, adapted to the present geometric setting.

Even without assuming local conformal flatness or the spin condition, in 2021 the author proves a Llarull-type rigidity theorem (see \cite[Theorem~A]{zbMATH07544449}) under the additional hypothesis that the map is harmonic and satisfies Condition~C.
\end{remark}

We next record the conformal minimum estimate used below.

\begin{proposition}\label{minimum}
Let \(\Omega\subsetneq S^n\), \(n\geq3\), be a connected open subset
carrying a complete metric
\(\tilde g=u^{4/(n-2)}g_{st}\big|_\Omega\), where
\(u\colon\Omega\to\mathbb{R}_{>0}\) is smooth. Then
\(\sup_\Omega u=+\infty\). If, moreover,
\(\mathrm{Sc}_{\tilde g}\geq n(n-1)\), then \(u\)
attains its minimum on $\Omega$, and $\min_\Omega u<1$.
\end{proposition}

\begin{proof}
We first show that completeness forces \(\sup_{\Omega}u=+\infty\).
Suppose, to the contrary, that
\(A:=\sup_{\Omega}u<\infty\).
Fix \(y_0\in\Omega\). Since \(\Omega\subsetneq S^n\) is open,
\(K:=S^n\setminus\Omega\) is nonempty and compact. Choose \(p\in K\) such that
\(\ell:=d_{g_{st}}(y_0,p)=d_{g_{st}}(y_0,K)>0\).
Let \(\gamma\colon[0,\ell]\to S^n\) be a unit-speed minimizing
\(g_{st}\)-geodesic from \(y_0\) to \(p\). The minimality of \(\ell\)
implies that \(\gamma([0,\ell))\subset\Omega\). Hence, for
\(0\le s<t<\ell\),
\[
d_{\tilde g}\bigl(\gamma(s),\gamma(t)\bigr)
\leq
\int_s^t u(\gamma(r))^{\frac{2}{n-2}}\,dr
\leq A^{\frac{2}{n-2}}(t-s).
\]
Thus, if \(t_j\uparrow\ell\), then \(\bigl(\gamma(t_j)\bigr)\) is
\(d_{\tilde g}\)-Cauchy. By completeness and the Hopf--Rinow theorem,
it converges to some point of \(\Omega\). Since the topology induced by
\(d_{\tilde g}\) is the manifold topology, this would also be its limit
in \(S^n\). On the other hand,
\(\gamma(t_j)\to p\notin\Omega\) in \(S^n\), a contradiction.
Therefore \(\sup_{\Omega}u=+\infty\).

Now assume that \(\mathrm{Sc}_{\tilde g}\geq n(n-1)\). By hypothesis,
\((\Omega,\tilde g)\) is complete. Write
\[
\tilde g=e^{2\psi}g_{st}\big|_\Omega,
\qquad
\psi:=\frac{2}{n-2}\log u.
\]
We use the following result of Ma and Qing
\cite[Lemma~3.1]{zbMATH07456625}: if \((X^n,h)\) is a compact
Riemannian manifold, \(U\subset X\) is a domain, and
\(\widehat h=e^{2\eta}h\big|_U\) is complete with
\[
\mathrm{Sc}_{\widehat h}^{-}
:=\max\{-\mathrm{Sc}_{\widehat h},0\}
\in L^p(U,\widehat h)
\]
for some \(p>n/2\), then
\(\eta(x)\to+\infty\) as \(x\to\partial U\).

In the present setting,
\(\mathrm{Sc}_{\tilde g}^{-}\equiv0\), and hence
\(\mathrm{Sc}_{\tilde g}^{-}\in L^p(\Omega,\tilde g)\) for every
\(p>n/2\). Applying the result with
\[
(X,h)=(S^n,g_{st}),\qquad
U=\Omega,\qquad
\widehat h=\tilde g,\qquad
\eta=\psi,
\]
we obtain \(\psi(x)\to+\infty\) as \(x\to\partial\Omega\).
Equivalently,
\[
v:=u^{-1}\longrightarrow0
\qquad\text{as }x\to\partial\Omega.
\]

Define \(\bar v\colon\overline\Omega\to[0,\infty)\) by
\(\bar v:=v\) on \(\Omega\) and \(\bar v:=0\) on \(\partial\Omega\).
This extension is continuous.
Indeed, if \(z_j\in\overline\Omega\) and
\(z_j\to p\in\partial\Omega\), then the terms lying in
\(\partial\Omega\) have value zero, while the subsequence of terms
lying in \(\Omega\), if infinite, has values tending to zero by the
preceding boundary limit. Since \(\overline\Omega\) is compact,
\(\bar v\) attains its maximum. This maximum is positive and cannot
occur on \(\partial\Omega\); hence it occurs at some
\(x_{\min}\in\Omega\). Consequently, \(u\) attains its global minimum
at \(x_{\min}\).

It remains to prove that this minimum is strictly less than \(1\).
For the remainder of the proof, we use the Laplacian
\(\Delta:=-\operatorname{div}_{g_{st}}\nabla\).
Thus $\Delta$ has nonnegative spectrum.
Suppose, to the contrary, that \(\min_{\Omega}u\geq1\). The conformal
scalar-curvature equation is
\[
\frac{4(n-1)}{n-2}\Delta u+n(n-1)u
=
\mathrm{Sc}_{\tilde g}\,
u^{\frac{n+2}{n-2}}.
\]
It follows that
\[
\Delta u
\geq
\frac{n(n-2)}4
\left(u^{\frac{n+2}{n-2}}-u\right)
\geq0
\qquad\text{on }\Omega.
\]
Equivalently,
\(\operatorname{div}_{g_{st}}\nabla u\leq0\). Since \(u\) attains its
global minimum at a point of the connected open set \(\Omega\), the
strong minimum principle implies that \(u\) is constant. This
contradicts \(\sup_{\Omega}u=+\infty\). Therefore
\(\min_{\Omega}u<1\).
\end{proof}

For a continuous map \(F\colon(X,g_X)\to(Y,g_Y)\) between Riemannian
manifolds, define its pointwise Lipschitz constant by
\[
\mathrm{Lip}_x(F):=\limsup_{y\to x,\,y\neq x}
\frac{d_{g_Y}(F(y),F(x))}{d_{g_X}(y,x)}.
\]
If $F$ is smooth, then $\mathrm{Lip}_x(F)=\|dF_x\|$. In particular, if
$X$ is connected and $\mathrm{Lip}_x(F)\leq1$ for every $x\in X$, then
integrating $\|dF\|\leq1$ along piecewise smooth curves and taking the
infimum of their lengths shows that $F$ is globally $1$-Lipschitz.
Consequently, Llarull's rigidity theorem implies that, if
$(X^n,g)$ is a closed, connected spin manifold with
$\mathrm{Sc}_g\geq n(n-1)$ and is not isometric to the standard round
sphere, then every smooth nonzero-degree map
\(f\colon(X^n,g)\to(S^n,g_{st})\) satisfies
$\mathrm{Lip}_{x_0}(f)>1$ at some point $x_0\in X$.
The next theorem gives the corresponding conclusion for developing maps.

\begin{theorem}\label{dev}
Let \((M^{n},g)\), \(n\geq3\), be a closed, connected, locally
conformally flat manifold satisfying
\(\mathrm{Sc}_{g}\geq n(n-1)\). Assume that $\pi_1(M)$ is infinite. Let
\(\mathrm{dev}\colon(\widetilde M,\tilde g)\to (S^n,g_{st})\)
denote a developing map, where \(\tilde g\) is the lifted metric. Then
there exists \(x\in\widetilde M\) such that
\(\mathrm{Lip}_{x}(\mathrm{dev})>1\).
\end{theorem}

\begin{proof}
Since \(M\) is closed, the lifted metric \(\tilde g\) is complete, and
$\mathrm{Sc}_{\tilde g}\geq n(n-1)$. By
Theorem~\ref{Liouville theorem}, the developing map is injective. Hence, with
\(\Omega:=\mathrm{dev}(\widetilde M)\), the map
\(\varphi:=\mathrm{dev}\colon\widetilde M\to\Omega\)
is a conformal diffeomorphism. The connected domain \(\Omega\) is proper:
if \(\Omega=S^n\), then \(\widetilde M\cong S^n\) would be compact, so its
properly discontinuous deck-transformation group
\(\pi_1(M)\) would be finite, a contradiction.

Transporting \(\tilde g\) by \(\varphi\), write
\(g_{\Omega}:=(\varphi^{-1})^*\tilde g
=u^{\frac{4}{n-2}}g_{st}\big|_{\Omega}\)
for some smooth \(u\colon\Omega\to(0,\infty)\). Then \(\varphi\) is an
isometry, \((\Omega,g_{\Omega})\) is complete, and
\(\mathrm{Sc}_{g_{\Omega}}\geq n(n-1)\).

The developing map decomposes as $\mathrm{dev} = \iota \circ \varphi,$
where $\iota:\Omega \hookrightarrow S^{n}$ is the inclusion. We compute the pointwise operator norm of 
$d(\mathrm{dev})_{p} : (T_{p}\widetilde{M},\tilde{g}) \to (T_{\mathrm{dev}(p)} S^{n}, g_{st}).$ Since $\varphi$ is an isometry $(\widetilde{M},\tilde{g})\cong(\Omega,g_{\Omega})$, one has $\| d(\mathrm{dev})_{p} \|_{(\tilde{g}\to g_{st})}
=\| d\iota_{\varphi(p)} \|_{(g_{\Omega}\to g_{st})}.$ Thus it suffices to study the inclusion $\iota : (\Omega, g_{\Omega}) \hookrightarrow (S^{n}, g_{st}).$

Fix $x\in \Omega$ and $v\in T_{x}\Omega$. Since $d\iota_{x}$ is the identity on $T_{x}\Omega$, its operator norm is
\[
\| d\iota_{x} \|_{(g_{\Omega}\to g_{st})}
=
\sup_{v\neq 0}
\frac{|v|_{g_{st}}}{|v|_{g_{\Omega}}}
=
u(x)^{-\frac{2}{n-2}}.
\]

Since $g_{\Omega}$ is complete and $\operatorname{Sc}_{g_{\Omega}} \ge n(n-1)$,  Proposition \ref{minimum} implies that there exists a point $x_0\in \Omega$ such that $u(x_0)=\min_{\Omega} u < 1$. Thus, 
$\| d\iota_{x_{0}} \|=u_{\min}^{-\frac{2}{n-2}}> 1.$

Setting $p_{0} := \varphi^{-1}(x_{0}) \in \widetilde{M}$, we conclude that
$\| d(\mathrm{dev})_{p_{0}} \|_{(\tilde{g}\to g_{st})}
=
\| d\iota_{x_{0}} \|_{(g_{\Omega}\to g_{st})}
=
u_{\min}^{-\frac{2}{n-2}}
> 1$.
This shows that the developing map strictly expands lengths at $p_{0}$, completing the proof.

\end{proof}

\begin{proof}[Proof of Theorem~\ref{Lipschitz}]
Suppose first that $\pi_1(M)$ is elementary and that the conclusion
fails for a smooth nonzero-degree map
\(f\colon(M^n,g)\to(S^n,g_{st})\). Then
$\mathrm{Lip}_p(f)\leq1$ for every $p\in M$. Since $f$ is smooth,
\(\|df_p\|=\mathrm{Lip}_p(f)\leq1\)
for every \(p\in M\).

Fix distinct points \(x,y\in M\), set \(\ell:=d_g(x,y)\), and let
\(\gamma\colon[0,\ell]\to M\) be a unit-speed minimizing geodesic from
\(x\) to \(y\). Then
\[
\begin{aligned}
d_{g_{st}}(f(x),f(y))
&\leq L_{g_{st}}(f\circ\gamma)\\
&=\int_0^\ell |df_{\gamma(t)}(\dot\gamma(t))|_{g_{st}}\,dt\\
&\leq\int_0^\ell |\dot\gamma(t)|_g\,dt
=\ell=d_g(x,y).
\end{aligned}
\]
Thus $f$ is globally $1$-Lipschitz. Proposition~\ref{Llarull} therefore
implies that $f$ is a Riemannian isometry, contradicting the assumption
that $(M^n,g)$ is not
isometric to $(S^n,g_{st})$.

If $\pi_1(M)$  is infinite,  the conclusion follows from
Theorem~\ref{dev}.
\end{proof}

\section{Moduli spaces of LCF metrics with PSC}\label{5}
In this section, we answer the question of Bamler and
Kleiner~\cite[Question~1.8]{2019arXiv190908710B} affirmatively for closed,
oriented $3$-manifolds with finite fundamental group and for
$S^2\times S^1$. Retaining their notation, the question is:
\begin{quote}
Is $\mathrm{Met}_{\mathrm{PSC}}(M) \cap \mathrm{Met}_{\mathrm{CF}}(M)$ always empty or contractible?
Equivalently, is the space of conformally flat metrics with positive Yamabe constant
always empty or contractible?
\end{quote}
We then extend the method to smooth $n$-manifolds, $n\ge4$, that are
homeomorphic to spherical space forms but not diffeomorphic to $S^n$.
Throughout, $\mathrm{Met}(M^n)$ and $\mathrm{Diff}(M^n)$ are equipped with
their $C^\infty$-topologies, and all subspaces of $\mathrm{Met}(M^n)$ carry
the subspace topology. We write $g_{st}$ for the round metric of sectional
curvature $1$ on the sphere whose dimension is clear from context. We use the
sign convention
\(\Delta_g=\operatorname{div}_g\nabla\),
so that $-\Delta_g$ has nonnegative spectrum. Let
\[
X(M^n) \ :=\ \Bigl\{\, g \in \mathrm{Met}(M^n) \ \Bigm| \ g \text{ is locally conformally flat and } \mathrm{Sc}_g > 0 \,\Bigr\},
\]
equipped with the natural action of $\mathrm{Diff}(M^n)$ given by
$\psi\cdot g:=(\psi^{-1})^*g$. Thus, in the notation of the quotation above,
$X(M^3)
=
\mathrm{Met}_{\mathrm{PSC}}(M^3)
\cap
\mathrm{Met}_{\mathrm{CF}}(M^3)$.
We define the moduli space
$\mathcal{M}(M^n) := X(M^n) / \mathrm{Diff}(M^n)$, endowed with the quotient topology.

We shall repeatedly use the following fact. Let $(N,h)$ be a closed, connected
Riemannian $3$-manifold locally isometric to
\((S^2\times\mathbb R,g_{\mathrm{cyl}})\),
where \(g_{\mathrm{cyl}}:=g_{st}+dt^2\),
and let $(\widetilde N,\widetilde h)$ be its universal Riemannian covering.
Then:
\begin{enumerate}
\item[(i)] $(\widetilde N,\widetilde h)$ is isometric to
      $(S^2\times\mathbb R,g_{\mathrm{cyl}})$;
\item[(ii)] $(N,h)$ is isometric to a quotient of
      $(S^2\times\mathbb R,g_{\mathrm{cyl}})$ by a group of cylinder
      isometries.
\end{enumerate}
Indeed, since $N$ is closed, $(N,h)$ is complete; hence
$(\widetilde N,\widetilde h)$ is complete and simply connected. The Ricci
endomorphism of $\widetilde h$ is parallel, because this is true in every
cylindrical chart. Its eigendistributions corresponding to the eigenvalues
$1$ and $0$ have ranks $2$ and $1$, respectively, and are parallel. The de
Rham decomposition theorem therefore gives an isometric splitting
\[
(\widetilde N,\widetilde h)
\cong
(\Sigma^2,h_\Sigma)\times(\mathbb R,dt^2),
\]
where $(\Sigma,h_\Sigma)$ is complete, simply connected, and has constant
sectional curvature $1$. By the Killing--Hopf theorem,
$(\Sigma,h_\Sigma)$ is isometric to $(S^2,g_{st})$, proving~(i). Given an
isometry
\(J\colon(\widetilde N,\widetilde h)
\to (S^2\times\mathbb R,g_{\mathrm{cyl}})\),
the group
\(J\,\mathrm{Deck}(\widetilde N/N)\,J^{-1}
\subset
\mathrm{Isom}(S^2\times\mathbb R,g_{\mathrm{cyl}})\)
acts freely and properly discontinuously, and $J$ descends to the quotient
isometry asserted in~(ii).

\begin{proposition}\label{X3-contractible}
Let $M^3$ be a closed, connected, oriented smooth $3$-manifold with finite
fundamental group. Then both $X(M^3)$ and $\mathcal{M}(M^3)$ are contractible.
\end{proposition}

\begin{proof}
By the elliptization theorem of Perelman~\cite{2002math.....11159P,2003math......3109P}, $M$ is diffeomorphic to a spherical space form
$S^3/\Gamma$, where $\Gamma<\mathrm{SO}(4)$ is finite and acts freely on $S^3$.
In particular, $M$ carries a metric of constant sectional curvature $1$.

Let
\[
\mathrm{Met}^{CC}(M):=
\Bigl\{\, h\in \mathrm{Met}(M)\ \Bigm|\ h \text{ is locally isometric to } (S^3,g_{st})\,\Bigr\}.
\]
In particular, the curvature scale in $\mathrm{Met}^{CC}(M)$ is fixed: every metric in
this space has constant sectional curvature exactly $1$.
Moreover,
\(\mathrm{Met}^{CC}(M)\neq\varnothing\)
and
\(\mathrm{Met}^{CC}(M)\subset X(M)\),
because every metric locally isometric to $(S^3,g_{st})$ is locally
conformally flat and has scalar curvature $6$.

We use the superscript $CC$ for the space of normalized locally spherical
metrics above,
whereas $\mathrm{Met}_{CC}(M)$ in
\cite[Theorem~1.2]{2019arXiv190908710B} consists of metrics locally isometric
to either the unit round sphere $(S^3,g_{st})$ or the unit round cylinder
$(S^2\times\mathbb R,g_{st}+dt^2)$. On the present manifold the two spaces
coincide. Indeed, if $M$ carried a locally cylindrical metric, part~(i) of the
local-to-global cylinder observation would identify its universal cover with
$S^2\times\mathbb R$, which is noncompact. Since $\pi_1(M)$ is finite,
however, the universal cover is a finite-sheeted cover of the compact
manifold $M$ and is therefore compact, a contradiction. Thus the cylindrical
alternative cannot occur. Hence
\cite[Theorem~1.2]{2019arXiv190908710B} implies that
$\mathrm{Met}^{CC}(M)$ is contractible.

For each $g\in X(M)$, define
\[
V_g:=\Bigl\{\phi\in C^\infty(M)\ \Bigm|\ e^{2\phi}g\in \mathrm{Met}^{CC}(M)\Bigr\}.
\]
Fix the universal covering
\[
p\colon\widetilde M\longrightarrow M,
\qquad
\mathcal D:=\mathrm{Deck}(p),
\qquad
|\mathcal D|=|\pi_1(M)|,
\]
and write $\widetilde g:=p^*g$. Since $M$ is a spherical space form,
$\widetilde M$ is diffeomorphic to $S^3$. Define
\[
\widetilde V_{\widetilde g}
:=
\Bigl\{\widetilde\phi\in C^\infty(\widetilde M)
\ \Bigm|\
(\widetilde M,e^{2\widetilde\phi}\widetilde g)
\text{ is isometric to }(S^3,g_{st})
\Bigr\}.
\]
Pullback identifies the downstairs admissible set with the deck-invariant
part of the upstairs admissible set:
\[
p^*(V_g)
=
\widetilde V_{\widetilde g}
\cap C^\infty(\widetilde M)^{\mathcal D}.
\]
Indeed, if $\phi\in V_g$, then
$p^*(e^{2\phi}g)=e^{2p^*\phi}\widetilde g$ is a complete metric of constant
sectional curvature $1$ on the simply connected manifold $\widetilde M$; by
the Killing--Hopf theorem it is isometric to $(S^3,g_{st})$. Conversely, every
$\mathcal D$-invariant
$\widetilde\phi\in\widetilde V_{\widetilde g}$ descends to a unique
$\phi\in C^\infty(M)$, and the descended metric $e^{2\phi}g$ is locally
isometric to $(S^3,g_{st})$, so $\phi\in V_g$.

Apply the proof of \cite[Lemma~9.2(a)]{2019arXiv190908710B} to
$(\widetilde M,\widetilde g)$. It gives a unique minimizer
$\widetilde\phi_{\widetilde g}\in\widetilde V_{\widetilde g}$ of
\[
\widetilde{\mathcal F}_{\widetilde g}(\widetilde\phi)
:=
\int_{\widetilde M}e^{-\widetilde\phi}\,d\mu_{\widetilde g}.
\]
For every $\kappa\in\mathcal D$, one has $\kappa^*\widetilde g=\widetilde g$
and hence
\(e^{2(\widetilde\phi_{\widetilde g}\circ\kappa)}\widetilde g
=
\kappa^*\!\left(e^{2\widetilde\phi_{\widetilde g}}\widetilde g\right)\),
so
$\widetilde\phi_{\widetilde g}\circ\kappa
\in\widetilde V_{\widetilde g}$; indeed, one precomposes the associated
isometry to $(S^3,g_{st})$ with $\kappa$. Moreover,
\(\widetilde{\mathcal F}_{\widetilde g}
(\widetilde\phi_{\widetilde g}\circ\kappa)
=
\widetilde{\mathcal F}_{\widetilde g}(\widetilde\phi_{\widetilde g})\).
Thus $\widetilde\phi_{\widetilde g}\circ\kappa$ is another minimizer, and
uniqueness implies that $\widetilde\phi_{\widetilde g}$ is
$\mathcal D$-invariant. It therefore descends to a function
$\phi(g)\in V_g$. Define the downstairs functional by
\(\mathcal F_g(\phi):=\int_M e^{-\phi}\,d\mu_g\).
Because $\widetilde g=p^*g$, the Riemannian volume densities satisfy
$d\mu_{\widetilde g}=p^*(d\mu_g)$. Hence, for every $\phi\in V_g$, the
covering-integration formula gives
\[
\widetilde{\mathcal F}_{\widetilde g}(p^*\phi)
=
\int_{\widetilde M}e^{-p^*\phi}\,d\mu_{\widetilde g}
=
\int_{\widetilde M}p^*\!\left(e^{-\phi}\,d\mu_g\right)
=
|\mathcal D|\int_M e^{-\phi}\,d\mu_g
=
|\mathcal D|\,\mathcal F_g(\phi).
\]
Since the unique upstairs minimizer is deck-invariant, this identity shows
that $\phi(g)$ minimizes $\mathcal F_g$ on $V_g$. If $\phi\in V_g$ were
another minimizer, then $p^*\phi$ would also minimize
$\widetilde{\mathcal F}_{\widetilde g}$; upstairs uniqueness would give
$p^*\phi=\widetilde\phi_{\widetilde g}=p^*\phi(g)$, and hence
$\phi=\phi(g)$.

Finally, the map $g\mapsto p^*g$ is continuous in the
$C^\infty$-topology, and the proof of
\cite[Lemma~9.2(a)]{2019arXiv190908710B} gives the $C^\infty$-continuity of
the upstairs minimizer
$\widetilde g\mapsto\widetilde\phi_{\widetilde g}$. The pullback map
\(p^*\colon C^\infty(M)
\to C^\infty(\widetilde M)^{\mathcal D}\)
is a linear homeomorphism for the Fr\'echet $C^\infty$-topologies; continuity
of its inverse follows, for example, from a finite collection of evenly
covered coordinate charts. Therefore
\(g\longmapsto
\phi(g)
=
(p^*)^{-1}\!\bigl(\widetilde\phi_{p^*g}\bigr)\)
is continuous in the $C^\infty$-topology. Define
\(r(g):=e^{2\phi(g)}g\in \mathrm{Met}^{CC}(M)\).
Then $r\colon X(M)\to \mathrm{Met}^{CC}(M)$ is continuous.

We claim that $r(g)=g$ for every $g\in \mathrm{Met}^{CC}(M)$. Fix such a metric
$g$, and let $\phi\in V_g$. Since both $g$ and $e^{2\phi}g$ belong to
$\mathrm{Met}^{CC}(M)$, both are metrics of constant sectional curvature
$1$ on the same manifold $M$. Their universal covers are therefore isometric to
$(S^3,g_{st})$, and the covering degree is $|\pi_1(M)|$ in both cases.
Hence
\[
\operatorname{Vol}(M,g)=\operatorname{Vol}(M,e^{2\phi}g)
=\frac{\operatorname{Vol}(S^3,g_{st})}{|\pi_1(M)|}.
\]
Since $d\mu_{e^{2\phi}g}=e^{3\phi}\,d\mu_g$, it follows that
\(\int_M e^{3\phi}\,d\mu_g=\operatorname{Vol}(M,e^{2\phi}g)=\operatorname{Vol}(M,g)\).

Now consider the function $f(s):=3e^{-s}+e^{3s}-4$, $s\in \mathbb R$. We have
\(
f''(s)=3e^{-s}+9e^{3s}>0,
\)
so $f$ is strictly convex, and
\(f'(s)=-3e^{-s}+3e^{3s}=0
\Longleftrightarrow s=0\).
Therefore $f(s)\ge 0$ for all $s$, with equality if and only if $s=0$.
Applying this pointwise to $s=\phi(x)$ and integrating, we obtain
\[
\int_M \bigl(3e^{-\phi}+e^{3\phi}\bigr)\,d\mu_g \ge 4\,\operatorname{Vol}(M,g),
\]
with equality if and only if $\phi\equiv 0$. Using the preceding identity for
$\int_M e^{3\phi}\,d\mu_g$, we get
\(3\mathcal F_g(\phi)+\operatorname{Vol}(M,g)\ge 4\,\operatorname{Vol}(M,g)\),
hence
\(\mathcal F_g(\phi)\ge \operatorname{Vol}(M,g)=\mathcal F_g(0)\),
with equality if and only if $\phi\equiv 0$. Thus $0$ is the unique minimizer of
$\mathcal F_g$ on $V_g$. By the uniqueness established above, we conclude that
$\phi(g)\equiv 0$ and $r(g)=g$ for every $g\in \mathrm{Met}^{CC}(M)$.

We now construct a strong deformation retraction of $X(M)$ onto
$\mathrm{Met}^{CC}(M)$. For a metric $g\in X(M)$, set
$u(g):=e^{\phi(g)/2}>0,$ so that $r(g)=u(g)^4\,g.$
Let
\(
L_g:=-8\Delta_g+\mathrm{Sc}_g
\)
be the conformal Laplacian. In dimension three, if $h=u^4g$, then
$\mathrm{Sc}_h=u^{-5}L_g(u)$.
For $t\in[0,1]$, define
\[
u_t(g):=(1-t)+t\,u(g),
\qquad
H(g,t):=u_t(g)^4\,g.
\]
Since $g\mapsto \phi(g)$ is continuous in the $C^\infty$-topology, the same is
true for $g\mapsto u(g)$. Affine interpolation, pointwise powers, and tensor
multiplication are continuous in the $C^\infty$ Fr\'echet topology. Hence the
preceding formula defines a jointly continuous map
\(H\colon X(M)\times[0,1]\to \mathrm{Met}(M)\).
We claim that its image is contained in $X(M)$.
First, local conformal flatness is a conformal invariant, and
$H(g,t)$ is conformal to $g$, so $H(g,t)$ is LCF.
Second, since $g\in X(M)$, $L_g(1)=\mathrm{Sc}_g>0.$
Since $r(g)\in \mathrm{Met}^{CC}(M)$, it has PSC, and hence
\(
L_g(u(g))=u(g)^5\,\mathrm{Sc}_{r(g)}>0
\)
pointwise. By linearity of $L_g$,
\(
L_g(u_t(g))=(1-t)L_g(1)+tL_g(u(g))>0
\)
pointwise for all $t\in[0,1]$. Therefore
\(\mathrm{Sc}_{H(g,t)}=u_t(g)^{-5}L_g(u_t(g))>0\).
Thus $H(g,t)\in X(M)$ for all $(g,t)$.

Consequently,
\(H\colon X(M)\times[0,1]\to X(M)\)
is continuous,  $H(g,0)=g,$ and  $H(g,1)=r(g).$

Finally, if $g\in \mathrm{Met}^{CC}(M)$, then $r(g)=g$, hence $u(g)\equiv 1$,
so $u_t(g)\equiv 1$ and therefore
\(
H(g,t)=g\qquad\text{for all }t\in[0,1].
\)
It follows that $H$ is a strong deformation retraction of $X(M)$ onto
$\mathrm{Met}^{CC}(M)$. Since $\mathrm{Met}^{CC}(M)$ is contractible, so is $X(M)$.

It remains to prove that $\mathcal{M}(M)=X(M)/\mathrm{Diff}(M)$ is contractible.
We first verify that $H$ is $\mathrm{Diff}(M)$-equivariant. Let $\psi\in\mathrm{Diff}(M)$, and write
\(
\psi\cdot g:=(\psi^{-1})^*g.
\)
If $\phi\in V_g$, then
\(e^{2(\phi\circ\psi^{-1})}(\psi\cdot g)
=(\psi^{-1})^*(e^{2\phi}g)\in \mathrm{Met}^{CC}(M)\),
so $\phi\circ\psi^{-1}\in V_{\psi\cdot g}$. Moreover,
\[
\mathcal F_{\psi\cdot g}(\phi\circ\psi^{-1})
=\int_M e^{-(\phi\circ\psi^{-1})}\,d\mu_{(\psi^{-1})^*g}
=\int_M e^{-\phi}\,d\mu_g
=\mathcal F_g(\phi).
\]
Thus the bijection
\(V_g\to V_{\psi\cdot g}\),
\(\phi\longmapsto \phi\circ\psi^{-1}\),
preserves the functional $\mathcal F$. By uniqueness of the minimizer,
$\phi(\psi\cdot g)=\phi(g)\circ\psi^{-1}.$ Consequently,
$r(\psi\cdot g)=\psi\cdot r(g).$ Since $u(\psi\cdot g)=u(g)\circ\psi^{-1}$, it follows that $u_t(\psi\cdot g)=u_t(g)\circ\psi^{-1}$ and hence
\(H(\psi\cdot g,t)=\psi\cdot H(g,t)\)
for all \((g,t)\in X(M)\times[0,1]\).

Let
\(q_M\colon X(M)\to \mathcal{M}(M):=X(M)/\mathrm{Diff}(M)\)
be the quotient map. Since \(H\) is \(\mathrm{Diff}(M)\)-equivariant, the map
$(g,t)\mapsto q_M(H(g,t))$ is constant on the fibers of
$q_M\times \mathrm{id}_{[0,1]}$. Indeed, if
\((q_M\times \mathrm{id}_{[0,1]})(g,t)=(q_M\times \mathrm{id}_{[0,1]})(g',t')\),
then \(t=t'\) and \(g'=\psi\cdot g\) for some \(\psi\in \mathrm{Diff}(M)\). Hence, by
\(\mathrm{Diff}(M)\)-equivariance of \(H\),
\[
q_M(H(g',t'))=q_M(H(\psi\cdot g,t))
=q_M(\psi\cdot H(g,t))=q_M(H(g,t)).
\]

Since $\mathrm{Diff}(M)$ acts on $X(M)$ by homeomorphisms, $q_M$ is open. Indeed,
for every open set $U\subset X(M)$,
\(q_M^{-1}(q_M(U))
=
\bigcup_{\psi\in\mathrm{Diff}(M)}\psi\cdot U\)
is open. Hence $q_M\times\mathrm{id}_{[0,1]}$ is an open continuous
surjection, and therefore a quotient map. Thus $(g,t)\mapsto q_M(H(g,t))$
factors uniquely through a continuous map
\[
\overline H\colon\mathcal{M}(M)\times[0,1]\longrightarrow \mathcal{M}(M),
\qquad
\overline H([g],t)=[H(g,t)].
\]
Because $\mathrm{Met}^{CC}(M)$ is $\mathrm{Diff}(M)$-invariant, the
equivariant retraction
$r\colon X(M)\to \mathrm{Met}^{CC}(M)$ induces a continuous map
on orbit spaces, denoted by $\overline r$, with
$\overline r([g])=[r(g)]$. The inclusion of $\mathrm{Met}^{CC}(M)$ in
$X(M)$ induces a continuous injection on orbit spaces, denoted by $j$, with
$j([h])=[h]$. Moreover,
$\overline r\circ j=\mathrm{id}$. Hence $j$ is a topological embedding, and
we identify $\mathrm{Met}^{CC}(M)/\mathrm{Diff}(M)$ with its image in
$\mathcal{M}(M)$. Since $H$ fixes $\mathrm{Met}^{CC}(M)$ pointwise,
$\overline H$ is a strong deformation retraction of $\mathcal{M}(M)$ onto
$\mathrm{Met}^{CC}(M)/\mathrm{Diff}(M)$.

We now show that $\mathrm{Met}^{CC}(M)/\mathrm{Diff}(M)$ is a point.
Let $h_1,h_2\in \mathrm{Met}^{CC}(M)$. Then $(M,h_1)$ and $(M,h_2)$ are closed
spherical space forms on the same smooth manifold $M$, hence they are
diffeomorphic. By de Rham rigidity for spherical space forms
\cite[pp.~66--67]{MR43468}, diffeomorphic spherical space forms are
isometric, so there
exists an isometry $\Psi\colon(M,h_1)\to (M,h_2)$.
In particular, $\Psi\in \mathrm{Diff}(M)$ and
$h_1=\Psi^*h_2.$ Thus any two elements of $\mathrm{Met}^{CC}(M)$ lie in the same
$\mathrm{Diff}(M)$-orbit. Therefore $\mathrm{Met}^{CC}(M)/\mathrm{Diff}(M)$
is a singleton. Since $\overline H$ is a strong deformation retraction onto a
point, $\mathcal{M}(M)$ is contractible.
\end{proof}

\begin{proposition}\label{prop:X-and-M-S2S1-contractible}
  Both $X(S^2\times S^1)$ and $\mathcal{M}(S^2\times S^1)$ are contractible.
\end{proposition}

The proof follows the strategy of Proposition~\ref{X3-contractible}, except
that the case of the moduli space
$\mathcal{M}(S^2\times S^1)$ requires the independent classification in
Lemma~\ref{lem:cylinder-moduli}, whose proof follows the proposition.

\begin{proof}
Write $X:=X(S^2\times S^1)$. Following
\cite{2019arXiv190908710B}, let
$\mathrm{Met}^{CC}(S^2\times S^1)\subset\mathrm{Met}(S^2\times S^1)$ denote
the subspace of metrics locally isometric to the standard round cylinder
$(S^2\times\mathbb R,g_{\mathrm{cyl}})$. Thus the sectional curvature of the
spherical factor is normalized to be $1$. In
\cite[Theorem~1.2]{2019arXiv190908710B}, the notation $\mathrm{Met}_{CC}$ also
includes locally spherical metrics. On $S^2\times S^1$ the spherical
alternative cannot occur, since a closed locally spherical manifold has finite
fundamental group, whereas $\pi_1(S^2\times S^1)\cong\mathbb Z$. Thus the
space used here coincides with the space in that theorem. Identify
$S^1=\mathbb R/\mathbb Z$, and let $g_{\mathbb R/\mathbb Z}$ denote the
metric on $S^1$ induced by $dt^2$. The product metric
$g_{st}+g_{\mathbb R/\mathbb Z}$ is locally isometric to the round cylinder;
hence
$\mathrm{Met}^{CC}(S^2\times S^1)\neq\varnothing$. Moreover, every such metric is LCF and has
PSC, so $\mathrm{Met}^{CC}(S^2\times S^1)\subset X$.
By definition, $\mathrm{Met}^{CC}(S^2\times S^1)$ is $\mathrm{Diff}(S^2\times S^1)$-invariant.

Since every $g\in X(S^2\times S^1)$ is conformally flat and
$S^2\times S^1$ is diffeomorphic to the quotient
\((S^2\times\mathbb R)/\langle(x,t)\mapsto(x,t+1)\rangle\),
\cite[Lemma~9.2(b)]{2019arXiv190908710B} applies to any
$g\in X(S^2\times S^1)$ and yields a unique function
$\phi(g)\in C^\infty(S^2\times S^1)$ such that
$e^{2\phi(g)} g$ is isometric to a quotient of the standard round cylinder.
Every such quotient metric is locally isometric to
$(S^2\times\mathbb R,g_{\mathrm{cyl}})$ and therefore belongs to
$\mathrm{Met}^{CC}(S^2\times S^1)$. Moreover, the assignment
\(g \mapsto \phi(g)\) is continuous in the \(C^\infty\)-topology.
 
Hence, we obtain a continuous map
\(r\colon X \to \mathrm{Met}^{CC}(S^2\times S^1) \subset X\),
\(r(g) = e^{2\phi(g)}\, g\).
If $g\in \mathrm{Met}^{CC}(S^2\times S^1)$, part~(ii) of the
local-to-global cylinder observation shows that $g$ is itself isometric to a
quotient of the standard round cylinder. Thus $\phi\equiv0$ has the property required in
\cite[Lemma~9.2(b)]{2019arXiv190908710B}; uniqueness gives
$\phi(g)\equiv0$, and hence $r(g)=g$. Thus
$r|_{\mathrm{Met}^{CC}(S^2\times S^1)}=\mathrm{id}$.

Following the argument of Proposition~\ref{X3-contractible}, we construct a strong deformation retraction of $X$ onto $\mathrm{Met}^{CC}(S^2\times S^1)$. Fix $g\in X$ and set $u(g):=e^{\phi(g)/2}>0$, so that $r(g)=e^{2\phi(g)}g=u(g)^4\,g$.
Let $L_g:=-8\Delta_g+\mathrm{Sc}_g$ be the conformal Laplacian in dimension $3$; then the conformal change formula gives $\mathrm{Sc}_{u^4g}=u^{-5}L_g u$.
Since $g\in X$, we have $L_g(1)=\mathrm{Sc}_g>0$, and since $r(g)\in \mathrm{Met}^{CC}(S^2\times S^1)\subset X$, we have
$\mathrm{Sc}_{r(g)}>0$, hence
\(L_g(u(g))=u(g)^5\,\mathrm{Sc}_{r(g)}>0\)
pointwise on \(S^2\times S^1\).
For $t\in[0,1]$, define
\[
u_t(g):=(1-t)\cdot 1+t\cdot u(g)\ >\ 0,
\qquad
H(g,t):=u_t(g)^4\,g.
\]
Linearity of $L_g$ gives
\(
L_g(u_t(g))=(1-t)L_g(1)+tL_g(u(g))>0,
\)
so $\mathrm{Sc}_{H(g,t)}=u_t(g)^{-5}L_g(u_t(g))>0$ for all $t$.
Moreover, local conformal flatness is preserved under conformal changes, so $H(g,t)$ is LCF.
Since $g\mapsto\phi(g)$, and hence $g\mapsto u(g)$, is continuous in the
$C^\infty$-topology, and since affine interpolation, pointwise powers, and
tensor multiplication are continuous in the $C^\infty$ Fr\'echet topology,
the map $(g,t)\mapsto u_t(g)^4g$ is jointly continuous.
Thus $H\colon X\times[0,1]\to X$ is a continuous homotopy with
$H(g,0)=g$ and $ H(g,1)=r(g).$
If $g\in \mathrm{Met}^{CC}(S^2\times S^1)$, then $u(g)\equiv 1$, hence $H(g,t)=g$ for all $t$.
Therefore $H$ is a strong deformation retraction of $X$ onto $\mathrm{Met}^{CC}(S^2\times S^1)$.

By \cite[Theorem~1.2]{2019arXiv190908710B}, $\mathrm{Met}^{CC}(S^2\times S^1)$ is contractible. Since $X$ strongly deformation retracts onto $\mathrm{Met}^{CC}(S^2\times S^1)$, it follows that $X$ is contractible.

We claim that $r$ and $H$ are $\mathrm{Diff}(S^2\times S^1)$-equivariant, hence descend to the quotient $\mathcal{M}(S^2\times S^1)=X/\mathrm{Diff}(S^2\times S^1)$.
Let $\psi\in \mathrm{Diff}(S^2\times S^1)$ and $\widehat g:=\psi\cdot g=(\psi^{-1})^*g$.
Define
\(\widehat\phi:=\phi(g)\circ\psi^{-1}\in C^\infty(S^2\times S^1)\).
We compute
\[
\begin{aligned}
e^{2\widehat\phi}\widehat g
&=e^{2(\phi(g)\circ\psi^{-1})}(\psi^{-1})^*g\\
&=(\psi^{-1})^*\bigl(e^{2\phi(g)}g\bigr)\\
&=\psi\cdot\bigl(e^{2\phi(g)}g\bigr).
\end{aligned}
\]
Since $e^{2\phi(g)}g$ is isometric to a quotient of the standard round
cylinder, so is its pullback $e^{2\widehat\phi}\widehat g$. Thus
$\widehat\phi$ has the property required in
\cite[Lemma~9.2(b)]{2019arXiv190908710B} for the metric
$\widehat g=\psi\cdot g$.

The same lemma asserts that such a function is unique.
Therefore $\widehat\phi=\phi(\widehat g)$, i.e.
\(
\phi(\psi\cdot g)=\phi(g)\circ\psi^{-1}.
\)
Hence,
\[
\begin{aligned}
\psi\cdot r(g)
&=(\psi^{-1})^*(e^{2\phi(g)}g)\\
&=e^{2(\phi(g)\circ\psi^{-1})}(\psi\cdot g)\\
&=e^{2\phi(\psi\cdot g)}(\psi\cdot g)
=r(\psi\cdot g).
\end{aligned}
\]
It also follows that $u(\psi\cdot g)=u(g)\circ\psi^{-1}$. Hence $u_t(\psi\cdot g)=u_t(g)\circ\psi^{-1}$ and therefore
$H(\psi\cdot g,t)=\psi\cdot H(g,t)$.
Let $q_M\colon X\to \mathcal{M}(S^2\times S^1)$ be the quotient
map. Since $\mathrm{Diff}(S^2\times S^1)$ acts on $X$ by homeomorphisms,
$q_M$ is
open: for every open $U\subset X$,
\(q_M^{-1}(q_M(U))
=
\bigcup_{\psi\in\mathrm{Diff}(S^2\times S^1)}\psi\cdot U\)
is open. Hence $q_M\times\mathrm{id}_{[0,1]}$ is an open continuous
surjection, and therefore a quotient map. Thus the equivariant homotopy $H$
descends to a continuous strong deformation retraction
\(\overline H\colon
\mathcal{M}(S^2\times S^1)\times[0,1]
\to \mathcal{M}(S^2\times S^1)\)
of $\mathcal{M}(S^2\times S^1)$ onto the image of $\mathrm{Met}^{CC}(S^2\times S^1)$ in the quotient. As in the proof of Proposition~\ref{X3-contractible}, the equivariant retraction $r$ shows that the natural map
\[
\mathcal{M}^{CC}(S^2\times S^1)
:=
\mathrm{Met}^{CC}(S^2\times S^1)/\mathrm{Diff}(S^2\times S^1)
\longrightarrow
\mathcal{M}(S^2\times S^1)
\]
is a topological embedding. We therefore identify its image with
$\mathcal{M}^{CC}(S^2\times S^1)$.

Lemma~\ref{lem:cylinder-moduli} shows that
$\mathcal{M}^{CC}(S^2\times S^1)$ is homeomorphic to
$(0,\infty)\times[0,\pi]$, which is convex and hence contractible. Since
$\mathcal{M}(S^2\times S^1)$ deformation retracts onto this subspace,
$\mathcal{M}(S^2\times S^1)$ is contractible.
\end{proof}

\begin{lemma}\label{lem:cylinder-moduli}
\(\mathcal{M}^{CC}(S^2 \times S^1)\) is homeomorphic to \((0,\infty) \times [0,\pi]\).
\end{lemma}

\begin{proof}
We construct continuous maps
\(Q\colon(0,\infty)\times[0,\pi]\to  \mathcal{M}^{CC}(S^2\times S^1)\)
and
\(\overline P\colon\mathcal{M}^{CC}(S^2\times S^1)\to (0,\infty)\times[0,\pi]\)
and prove that they are inverse homeomorphisms.

Fix once and for all the universal covering
\[
\varpi\colon S^2\times \mathbb R\longrightarrow S^2\times S^1,
\qquad
\varpi(x,s)=(x,[s]),
\]
with deck generator
\[
\tau\colon S^2\times \mathbb R\longrightarrow S^2\times \mathbb R,
\qquad
\tau(x,s)=(x,s+1).
\]
Thus \(\mathrm{Deck}(\varpi)=\langle \tau\rangle\cong \mathbb Z\).

Let \(h\in \mathrm{Met}^{CC}(S^2\times S^1)\), and write
$\widetilde h:=\varpi^*h$. By part~(i) of the local-to-global cylinder
observation, there is an isometry
\(J_h\colon(S^2\times\mathbb R,\widetilde h)
\to (S^2\times\mathbb R,g_{\mathrm{cyl}})\).
Set
\(\gamma:=J_h\circ\tau\circ J_h^{-1}
\in\mathrm{Isom}(S^2\times\mathbb R,g_{\mathrm{cyl}})\).
Because \(\gamma\) is conjugate to \(\tau\), the action of \(\langle\gamma\rangle\) is free and \(\gamma\) has infinite order.

The Ricci endomorphism of the round cylinder has eigendistributions $TS^2$
and $T\mathbb R$, with the distinct eigenvalues $1$ and $0$. Hence every
isometry preserves both distributions. Writing an isometry as
$\Psi=(\Psi_1,\Psi_2)$, this gives $\partial_t\Psi_1=0$ and
$d_x\Psi_2=0$. Since both factors are connected, it follows that
$\Psi(x,t)=(A(x),b(t))$. Bijectivity of $\Psi$ makes $A$ and $b$
diffeomorphisms, while $\Psi^*g_{\mathrm{cyl}}=g_{\mathrm{cyl}}$ gives
$A^*g_{st}=g_{st}$ and $b^*dt^2=dt^2$. Consequently,
\[
\mathrm{Isom}(S^2\times \mathbb R, g_{\mathrm{cyl}})
\cong \mathrm{Isom}(S^2, g_{st}) \times \mathrm{Isom}(\mathbb{R}, dt^2)
\cong \mathrm{O}(3) \times (\mathbb{R} \rtimes \{\pm 1\}).
\]
Thus
\(
\gamma(x,t)=(Ax,\delta t+L)
\)
for some $A\in\mathrm{O}(3)$, $\delta\in\{\pm1\}$, and $L\in\mathbb R$.

We claim that \(\delta=1\). If \(\delta=-1\), then
$\gamma^2(x,t)=(A^2x,t)$.
Since \(A^2\in \mathrm{SO}(3)\), there exists \(x_0\in S^2\) such that
\(A^2x_0=x_0\). Hence $\gamma^2(x_0,t)=(x_0,t)$ for all $t\in\mathbb R$.
Since \(\gamma\) has infinite order, \(\gamma^2\neq\mathrm{id}\), contradicting
the freeness of the \(\langle\gamma\rangle\)-action. Thus \(\delta=1\).

We next claim that \(A\in \mathrm{SO}(3)\). The manifold
\(S^2\times S^1\) is oriented, and \(\tau\) preserves the lifted orientation
on \(S^2\times\mathbb R\). Since orientation preservation is invariant under
conjugation, the conjugate $\gamma=J_h\tau J_h^{-1}$ also preserves the product orientation.
The translation \(t\mapsto t+L\) preserves the orientation of $\mathbb R$;
hence $A$ preserves the orientation of $S^2$. Therefore
$A\in\mathrm{SO}(3)$.

We have \(L\neq 0\). Indeed, if \(L=0\), then $\gamma(x,t)=(Ax,t).$
Since \(A\in \mathrm{SO}(3)\), there exists \(x_0\in S^2\) such that \(Ax_0=x_0\). Hence
$\gamma(x_0,t)=(x_0,t)$ for all $t\in \mathbb R$, again contradicting freeness. Thus \(L\neq 0\).

Finally, after postcomposing $J_h$ with the reflection
$(x,t)\mapsto(x,-t)$ if necessary, we may assume $L>0$. After this
normalization, the parameters do not depend on the choices. Indeed, if
$J_h'=\Psi\circ J_h$, where
$\Psi(x,t)=(Bx,\varepsilon t+c)$, then
\((\Psi\gamma\Psi^{-1})(x,t)
=(BAB^{-1}x,t+\varepsilon L)\).
After restoring the convention that the translation length is positive,
$L$ is unchanged and $A$ is replaced by an $\mathrm O(3)$-conjugate. If the
deck generator $\tau$ is replaced by $\tau^{-1}$, then the holonomy element becomes
$\gamma^{-1}(x,t)=(A^{-1}x,t-L)$; conjugating by $(x,t)\mapsto(x,-t)$ to
restore the normalization gives $(A^{-1}x,t+L)$. Thus neither operation
changes the positive translation length $L$ or the rotation angle
$\theta\in[0,\pi]$ of $A$. Hence
$h$ determines the well-defined parameters $(L,\theta)$, and its holonomy
admits the normalized representative
\[
\gamma_{L,A}(x,t):=(Ax,t+L),
\qquad
A\in \mathrm{SO}(3),\quad L>0.
\]

We now construct the continuous map
$Q\colon (0,\infty) \times [0,\pi] \to \mathcal{M}^{CC}(S^2 \times S^1)$.
For \((L,\theta)\in (0,\infty)\times[0,\pi]\), let
\(
\gamma_{L,\theta}(x,t):=(R_\theta x,t+L),
\)
where \(R_\theta\in\mathrm{SO}(3)\) denotes the standard rotation by angle \(\theta\) about the \(x_3\)-axis. Let
$(S^2\times \mathbb R)/\langle \gamma_{L,\theta}\rangle$
carry the quotient metric \(\bar g_{L,\theta}\).

To identify this quotient with \(S^2\times S^1\), define
\[
\widetilde F_{L,\theta}\colon S^2\times \mathbb R\longrightarrow S^2\times \mathbb R,
\qquad
\widetilde F_{L,\theta}(x,s)=(R_{s\theta}x,Ls).
\]
This is a diffeomorphism, with inverse
\(\widetilde F_{L,\theta}^{-1}(y,t)=\bigl(R_{-(t/L)\theta}y,t/L\bigr)\).
Moreover,
\(\widetilde F_{L,\theta}\circ \tau=\gamma_{L,\theta}\circ \widetilde F_{L,\theta}\).
Hence \(\widetilde F_{L,\theta}\) descends to a diffeomorphism
\(F_{L,\theta}\colon S^2\times S^1\to
(S^2\times \mathbb R)/\langle \gamma_{L,\theta}\rangle\).
Define
\(h_{L,\theta}:=F_{L,\theta}^*\bar g_{L,\theta}\in \mathrm{Met}^{CC}(S^2\times S^1)\).

On the fixed universal cover one has
\(\varpi^*h_{L,\theta}=\widetilde F_{L,\theta}^*g_{\mathrm{cyl}}\).
Let $Z$ be the Killing field on $S^2$ generating $a\mapsto R_a$. For
$(V,a),(W,b)\in T_xS^2\oplus\mathbb R\partial_s$, direct differentiation
gives
\[
(\varpi^*h_{L,\theta})\bigl((V,a),(W,b)\bigr)
=g_{st}(V+a\theta Z_x,W+b\theta Z_x)+L^2ab.
\]
This tensor is $\tau$-invariant and depends smoothly on $(L,\theta)$,
including at $\theta=0$ and $\theta=\pi$. Hence the map
$(L,\theta)\longmapsto h_{L,\theta}$
is continuous in the \(C^\infty\)-topology. Thus there is a continuous map
\[
Q\colon(0,\infty)\times[0,\pi]\longrightarrow \mathcal{M}^{CC}(S^2\times S^1),
\qquad
Q(L,\theta):=[h_{L,\theta}].
\]

We now construct the map \(\overline P\).
Let \(h\in \mathrm{Met}^{CC}(S^2\times S^1)\), and let
$\widetilde h:=\varpi^*h$.
Define
\[
f_h\colon S^2\times \mathbb R\longrightarrow \mathbb R,
\qquad
f_h(p):=d_{\widetilde h}(p,\tau p).
\]
Since \(\tau\) is an isometry of \((S^2\times \mathbb R,\widetilde h)\), one has
$f_h(\tau p)=f_h(p),$
so \(f_h\) descends to a continuous function on the compact quotient \(S^2\times S^1\), still denoted \(f_h\).

Set
\[
L(h):=\min_{S^2\times S^1} f_h,
\qquad
D(h):=\max_{S^2\times S^1} f_h,
\qquad
\Theta(h):=\sqrt{D(h)^2-L(h)^2}.
\]
Since $\tau$ acts freely, $f_h(p)>0$ for every $p$. Since $f_h$ is continuous
on the compact quotient, it attains a positive minimum. Hence $L(h)>0$.

We now compute these invariants from the normalized holonomy. Let $J_h$ be the
isometry above satisfying $J_h\tau J_h^{-1}=\gamma_{L,A}$, where
$A\in\mathrm{SO}(3)$ has rotation angle $\theta\in[0,\pi]$ and $L>0$.
Because isometries preserve the distance function, the displacement $f_h(p)$
transforms as
\[
f_h(p)
=d_{g_{\mathrm{cyl}}}(J_h(p),J_h(\tau p))
=d_{g_{\mathrm{cyl}}}((x,t),\gamma_{L,A}(x,t)),
\]
where $(x,t):=J_h(p)$. Thus we obtain
\[
f_h(p)^2 = d_{g_{st}}(x, Ax)^2 + d_{\mathbb{R}}(t, t+L)^2 = d_{g_{st}}(x, Ax)^2 + L^2.
\]
Let $v$ be a unit vector on the rotation axis of $A$, chosen arbitrarily if
$\theta=0$, and put $a:=\langle x,v\rangle$. Then
\(\langle x,Ax\rangle=a^2+(1-a^2)\cos\theta\),
so $d_{g_{st}}(x,Ax)$ ranges over $[0,\theta]$. If $\theta>0$, its minimum
is attained on the rotation axis and its maximum on the orthogonal equator;
if $\theta=0$, it vanishes everywhere. It follows that
\(L(h)^2 = L^2\) and \(D(h)^2 = \theta^2 + L^2\).
Therefore
\(L(h)=L\), \(D(h)=\sqrt{L^2+\theta^2}\), and \(\Theta(h)=\theta\).

Thus
\[
P\colon\mathrm{Met}^{CC}(S^2\times S^1)\longrightarrow (0,\infty)\times[0,\pi],
\qquad
P(h):=\bigl(L(h),\Theta(h)\bigr)
\]
is well defined.

We next prove that \(P\) is continuous. Because the $C^\infty$-topology on
$\mathrm{Met}^{CC}(S^2\times S^1)$ is metrizable, it suffices to prove
sequential continuity. Let \(h_i\to h\) in the \(C^\infty\)-topology. In
particular, \(h_i\to h\) in \(C^0\). Since \(h\) is positive definite on the compact manifold \(S^2\times S^1\), after discarding finitely many terms and renumbering, we may choose
$\varepsilon_i\in(0,1),$ $\varepsilon_i\to 0,$ such that
\[
(1-\varepsilon_i)h\le h_i\le (1+\varepsilon_i)h
\]
as bilinear forms on \(T(S^2\times S^1)\). Pulling back by \(\varpi\), we obtain
\[
(1-\varepsilon_i)\widetilde h\le \widetilde h_i\le (1+\varepsilon_i)\widetilde h
\]
on \(T(S^2\times \mathbb R)\), where \(\widetilde h_i:=\varpi^*h_i\).

For every piecewise \(C^1\) curve \(c\) in \(S^2\times \mathbb R\),
\[
\sqrt{1-\varepsilon_i}\,\ell_{\widetilde h}(c)
\le
\ell_{\widetilde h_i}(c)
\le
\sqrt{1+\varepsilon_i}\,\ell_{\widetilde h}(c).
\]
Taking infima over curves joining fixed endpoints yields
\[
\sqrt{1-\varepsilon_i}\,d_{\widetilde h}
\le
d_{\widetilde h_i}
\le
\sqrt{1+\varepsilon_i}\,d_{\widetilde h}.
\]
Hence
\[
\sqrt{1-\varepsilon_i}\,f_h\le f_{h_i}\le \sqrt{1+\varepsilon_i}\,f_h
\]
pointwise on \(S^2\times S^1\). If
\(\eta_i:=\max\!\left\{\sqrt{1+\varepsilon_i}-1,\ 1-\sqrt{1-\varepsilon_i}\right\}\),
then \(\eta_i\to 0\) and
\(\|f_{h_i}-f_h\|_{C^0(S^2\times S^1)}\le \eta_i\,D(h)\to 0\).
Therefore
\(L(h_i)\to L(h)\), \(D(h_i)\to D(h)\), and \(\Theta(h_i)\to \Theta(h)\).
Thus \(P\) is continuous.

We next establish that $P$ is invariant under pullback by
$\mathrm{Diff}(S^2\times S^1)$. Let $\psi\in\mathrm{Diff}(S^2\times S^1)$,
and let $\widetilde\psi\colon S^2\times\mathbb R\to S^2\times\mathbb R$ be a
lift. Conjugation by $\widetilde\psi$ induces an automorphism of
$\mathrm{Deck}(\varpi)\cong\mathbb Z$. Hence there exists
$\varsigma\in\{\pm1\}$ such that
\(\widetilde\psi\circ\tau\circ\widetilde\psi^{-1}=\tau^\varsigma\).
By naturality of pullback,
\(\varpi^*(\psi^*h)=\widetilde\psi^*(\varpi^*h)
=\widetilde\psi^*\widetilde h\).
Therefore, for every $p\in S^2\times\mathbb R$,
\[
\begin{aligned}
f_{\psi^*h}(p)
&=d_{\widetilde\psi^*\widetilde h}(p,\tau p)\\
&=d_{\widetilde h}(\widetilde\psi p,\widetilde\psi\tau p)\\
&=d_{\widetilde h}(\widetilde\psi p,\tau^\varsigma\widetilde\psi p).
\end{aligned}
\]
If $\varsigma=1$, the last expression equals $f_h(\widetilde\psi p)$. If
$\varsigma=-1$, the symmetry of the distance function and the fact that
$\tau$ is an isometry of $\widetilde h$ give
\(d_{\widetilde h}(\widetilde\psi p,\tau^{-1}\widetilde\psi p)
=d_{\widetilde h}(\tau\widetilde\psi p,\widetilde\psi p)
=f_h(\widetilde\psi p)\).
Thus $f_{\psi^*h}=f_h\circ\widetilde\psi$ in both cases. The two functions
have the same range, and hence
\(L(\psi^*h)=L(h)\), \(D(\psi^*h)=D(h)\), and \(\Theta(\psi^*h)=\Theta(h)\).
It follows that $P$ is $\mathrm{Diff}(S^2\times S^1)$-invariant. By the
universal property of the quotient, $P$ descends to a unique continuous map
\(\overline P\colon\mathcal{M}^{CC}(S^2\times S^1)
\to (0,\infty)\times[0,\pi]\).

We conclude by showing that $Q$ and $\overline P$ are inverse homeomorphisms.
For any $(L,\theta)\in(0,\infty)\times[0,\pi]$, the construction of the
model metric $h_{L,\theta}$ provides an isometry
\(\widetilde F_{L,\theta}\colon
(S^2\times\mathbb R,\varpi^*h_{L,\theta})
\to (S^2\times\mathbb R,g_{\mathrm{cyl}})\)
that conjugates $\tau$ to
$\gamma_{L,\theta}(x,t)=(R_\theta x,t+L)$. For
$p\in S^2\times\mathbb R$, write
\(
(x,t):=\widetilde F_{L,\theta}(p).
\)
Then
\(f_{h_{L,\theta}}(p)
=\sqrt{d_{g_{st}}(x,R_\theta x)^2+L^2}\).
By the displacement computation above,
$d_{g_{st}}(x,R_\theta x)$ ranges over $[0,\theta]$; in particular, this
statement also covers $\theta=0$. It follows that
$L(h_{L,\theta})=L$ and $\Theta(h_{L,\theta})=\theta$. Thus,
\(\overline P(Q(L, \theta)) = \overline P([h_{L, \theta}]) = (L, \theta)\).
Conversely, let $h\in\mathrm{Met}^{CC}(S^2\times S^1)$ have invariants
$P(h)=(L,\theta)$. By the holonomy normalization established above, there is
an isometry
\(J_h\colon(S^2\times\mathbb R,\varpi^*h)
\to (S^2\times\mathbb R,g_{\mathrm{cyl}})\)
such that $J_h\tau J_h^{-1}=\gamma_{L,A}$, where $A\in\mathrm{SO}(3)$
has rotation angle $\theta$. Choose $C\in\mathrm{SO}(3)$ such that
$CAC^{-1}=R_\theta$, and set $K_C(x,t):=(Cx,t)$. Then
\((K_C\circ J_h)\tau(K_C\circ J_h)^{-1}
=K_C\gamma_{L,A}K_C^{-1}=\gamma_{L,\theta}\).
Hence $K_C\circ J_h$ descends to an isometry
\[
\overline K\colon
(S^2\times S^1,h)
\longrightarrow
\bigl((S^2\times\mathbb R)/\langle\gamma_{L,\theta}\rangle,
\bar g_{L,\theta}\bigr).
\]
Composing with $F_{L,\theta}^{-1}$ gives an isometry
\(F_{L,\theta}^{-1}\circ\overline K\colon
(S^2\times S^1,h)
\to (S^2\times S^1,h_{L,\theta})\).
Consequently,
\(Q(\overline P([h]))=Q(L,\theta)=[h_{L,\theta}]=[h]\).

The identities $\overline P\circ Q=\mathrm{id}$ and
$Q\circ\overline P=\mathrm{id}$, together with the continuity of $\overline P$
and $Q$,
show that $\overline P$ and $Q$ are inverse homeomorphisms. Thus
\(\mathcal{M}^{CC}(S^2\times S^1)\cong(0,\infty)\times[0,\pi]\).
\end{proof}

We now consider a closed, oriented smooth manifold $M^n$, $n\geq3$, that is
homeomorphic to a spherical space form but not diffeomorphic to $S^n$. We do
not assume that the given smooth structure is the quotient smooth structure;
in particular, $M$ need not carry a round metric. If $M$ admits an LCF metric,
then Kuiper's theorem~\cite{zbMATH03061883}, followed by the Cartan
fixed-point argument used below, shows that $M$ is diffeomorphic to a quotient
$S^n/\Gamma$, where $\Gamma\subset\mathrm{O}(n+1)$ is finite and acts freely
on $S^n$. This implication is used only under the standing assumption
$X(M)\neq\varnothing$.

\begin{proposition}\label{spherical space}
Assume $n\geq3$. Let $M^n$ be a closed, oriented smooth manifold homeomorphic
to a spherical space form but not diffeomorphic to $S^n$. Define
\[
\mathrm{Met}^{CC}(M)
:=
\left\{
h\in\mathrm{Met}(M)
\ \middle|\ 
h\text{ is locally isometric to }(S^n,g_{st})
\right\}.
\]
If $X(M)$ is nonempty, then $\mathrm{Met}^{CC}(M)$ is nonempty, there is a
$\mathrm{Diff}(M)$-equivariant strong deformation retraction of $X(M)$ onto
$\mathrm{Met}^{CC}(M)$, and $\mathrm{Met}^{CC}(M)$ is a single
$\mathrm{Diff}(M)$-orbit. Consequently, the moduli space
\(
\mathcal{M}(M):=X(M)/\mathrm{Diff}(M)
\)
is either empty or contractible.
\end{proposition}

\paragraph*{Proof strategy.}
Assuming $X(M)\neq\varnothing$, the strategy of the proof is to construct a
canonical, $\mathrm{Diff}(M)$-equivariant strong deformation retraction from
$X(M)$ onto the subspace of round metrics $\mathrm{Met}^{CC}(M)$. Marques'
result \cite[Corollary~3.2]{MR2950765} gives a round representative in every
conformal class in $X(M)$, and, because $M$ is not diffeomorphic to $S^n$,
Obata's theorem \cite{zbMATH03374588} makes this representative unique after
normalizing its scalar curvature to $n(n-1)$. This defines a projection
$r\colon X(M)\to\mathrm{Met}^{CC}(M)$. We establish the continuity of $r$ by
applying the Banach-space implicit function theorem to the Yamabe-type equation
for the conformal factor; the crucial invertibility of the linearized conformal
Laplacian follows directly from the Lichnerowicz--Obata spectral gap
($\lambda_1>n$). A simple convex interpolation of these continuous conformal
factors yields the required retraction. Finally, de Rham rigidity implies that
$\mathrm{Met}^{CC}(M)$ is a single $\mathrm{Diff}(M)$-orbit, so the equivariant
homotopy descends to the quotient and contracts $\mathcal{M}(M)$ to a point.

\begin{proof}
If $X(M)=\varnothing$, there is nothing to prove. Assume
$X(M)\neq\varnothing$, and choose $g_*\in X(M)$. Since $\pi_1(M)$ is finite,
the universal cover is a finite-sheeted cover of $M$ and hence compact. The
lift of $g_*$ makes it a closed, simply connected LCF manifold. Kuiper's
theorem~\cite{zbMATH03061883} identifies the universal cover conformally and
diffeomorphically with $S^n$. Under this identification, the deck group acts
freely by conformal transformations of $S^n$. This group is finite and, after
a conformal conjugation, is contained in $\mathrm{O}(n+1)$; for example, this
follows by applying the Cartan fixed-point theorem to its isometric action on
$\mathbb H^{n+1}$. Thus $M$ is diffeomorphic to $S^n/\Gamma$ for a finite
subgroup $\Gamma\subset\mathrm{O}(n+1)$ acting freely on $S^n$. Since $M$ is
not diffeomorphic to $S^n$, the group $\Gamma$ is nontrivial.

Moreover, $n$ is odd. Indeed, if $n$ were even, then
\(2=\chi(S^n)=|\Gamma|\,\chi(S^n/\Gamma)\)
would force $|\Gamma|=2$. The only free orthogonal involution of $S^n$ is the
antipodal map, whose quotient $\mathbb{RP}^n$ is nonorientable when $n$ is
even, contradicting the orientability of $M$.

For any $g\in X(M)$, Marques' result
\cite[Corollary~3.2]{MR2950765}, applied through the above diffeomorphism,
shows that the conformal class $[g]$ contains a metric of constant sectional
curvature. This curvature is positive because the path in Marques' corollary
consists of metrics with positive scalar curvature. After multiplying this metric by a positive constant, we
obtain a metric $g_0\in[g]$ of constant sectional curvature $+1$. Since $M$ is
not diffeomorphic to $S^n$, $(M,g_0)$ is not conformally equivalent to the
round sphere. Obata's theorem \cite{zbMATH03374588} therefore implies that any metric of
constant scalar curvature in $[g]$ is a constant multiple of $g_0$.
Normalizing the scalar curvature to $n(n-1)$ therefore gives a unique round
representative
\(
h\in[g]\cap\mathrm{Met}^{CC}(M).
\)
We define
\(r\colon X(M)\to \mathrm{Met}^{CC}(M)\),
\(r(g):=h\).
This uniqueness immediately gives
\(r(\psi\cdot g)=\psi\cdot r(g)\)
for \(\psi\in\mathrm{Diff}(M)\).

For each $\eta\in X(M)$, let $v_\eta>0$ be the unique smooth function such
that \(r(\eta)=v_\eta^{\frac{4}{n-2}}\eta\).
We show that $r$ is continuous. Fix \(g\in X(M)\), and write
$h:=r(g)=u^{\frac{4}{n-2}}g$, where $u:=v_g$. Set
\[
a_n:=\frac{4(n-1)}{n-2},
\qquad
q:=\frac{n+2}{n-2},
\qquad
L_{g'}:=-a_n\Delta_{g'}+\mathrm{Sc}_{g'}.
\]
By the conformal covariance of the conformal Laplacian,
\[
L_{g'}(vw)=v^qL_{\widehat g'}w,
\qquad
\widehat g'=v^{\frac{4}{n-2}}g';
\]
in particular,
\(L_{g'}(v)=\mathrm{Sc}_{\widehat g'}\,v^q\).
Since \(h\) is round, \(\mathrm{Sc}_h=n(n-1)\), and therefore
\(
L_g(u)=n(n-1)u^q.
\)

Fix \(k\ge 0\) and \(0<\alpha<1\). Let
\[
\mathcal G^{k+2,\alpha}
:=
\bigl\{g'\in C^{k+2,\alpha}(\operatorname{Sym}^2(T^*M))
\mid g' \text{ is positive definite}\bigr\},
\]
and
\[
\mathcal P^{k+2,\alpha}
:=
\bigl\{v\in C^{k+2,\alpha}(M)\mid v>0\bigr\}.
\]
Since pointwise positive definiteness and pointwise positivity are open conditions,
\(\mathcal G^{k+2,\alpha}\) and \(\mathcal P^{k+2,\alpha}\) are open subsets of Banach spaces.
Define
\(\mathcal Y\colon\mathcal G^{k+2,\alpha}\times
\mathcal P^{k+2,\alpha}\to  C^{k,\alpha}(M)\)
by
\(
\mathcal Y(g',v):=L_{g'}v-n(n-1)v^q.
\)
The inversion map $g'\mapsto(g')^{-1}$ is $C^\infty$ on
$\mathcal G^{k+2,\alpha}$, differentiation is bounded between the relevant
H\"older spaces, and multiplication and tensor contraction are continuous
multilinear, hence smooth, operations on these spaces. In local coordinates,
$L_{g'}v$ is a finite sum of contractions
of $(g')^{-1}$, derivatives of $g'$ of order at most two, and derivatives of
$v$ of order at most two. It follows, by localization, that
\((g',v)\longmapsto L_{g'}v\)
is a $C^\infty$ map from
$\mathcal G^{k+2,\alpha}\times C^{k+2,\alpha}(M)$ to $C^{k,\alpha}(M)$.
Moreover, for every $v_0\in\mathcal P^{k+2,\alpha}$, compactness of $M$ gives
a neighborhood of $v_0$ on which all functions are bounded above and bounded
away from zero. For $j\geq1$, the derivatives of the power map on this
neighborhood are
\[
D^j(v\mapsto v^q)_v(\varphi_1,\ldots,\varphi_j)
=q(q-1)\cdots(q-j+1)v^{q-j}\varphi_1\cdots\varphi_j.
\]
The H\"older algebra estimates therefore show that $v\mapsto v^q$ is
$C^\infty$ from $\mathcal P^{k+2,\alpha}$ to $C^{k+2,\alpha}(M)$, and hence
also to $C^{k,\alpha}(M)$. Thus $\mathcal Y$ is $C^\infty$ (in particular,
$C^1$), and $\mathcal Y(g,u)=0$.

The derivative of \(\mathcal Y\) in the \(v\)-variable at \((g,u)\) is
\[
D_v\mathcal Y_{(g,u)}(\varphi)
= \frac{d}{dt}\bigg|_{t=0} \mathcal Y(g,u+t\varphi)
=
L_g\varphi-q\,n(n-1)u^{q-1}\varphi.
\]
Since \(u>0\), every \(\varphi\in C^{k+2,\alpha}(M)\) can be written uniquely as
\(\varphi=uw\)
with \(w=u^{-1}\varphi\in C^{k+2,\alpha}(M)\).

Since $L_g(uw)=u^qL_h(w)$ and $L_h=-a_n\Delta_h+n(n-1)$, we obtain
\[
\begin{aligned}
D_v\mathcal Y_{(g,u)}(uw)
&=L_g(uw)-qn(n-1)u^qw\\
&=u^q\bigl(L_h(w)-qn(n-1)w\bigr)\\
&=u^q\bigl(-a_n\Delta_h w+(1-q)n(n-1)w\bigr).
\end{aligned}
\]
Since
\[
q=\frac{n+2}{n-2},
\qquad
a_n=\frac{4(n-1)}{n-2},
\qquad
(1-q)n(n-1)=-na_n,
\]
we have
\(
D_v\mathcal Y_{(g,u)}(uw)=-a_n\,u^q(\Delta_h+n)w.
\)
Equivalently,
\[
D_v\mathcal Y_{(g,u)}=M_{u^q}\circ A_h\circ M_{u^{-1}},
\qquad
A_h:=-a_n(\Delta_h+n),
\]
where \(M_f\) denotes multiplication by \(f\). We claim that
\(D_v\mathcal Y_{(g,u)}:C^{k+2,\alpha}(M)\to  C^{k,\alpha}(M)\)
is an isomorphism. Let $\Gamma_h$ denote the deck group of the round universal
covering
\(
(S^n,g_{st})\to (M,h).
\)
This group is nontrivial, because otherwise $M$ would be diffeomorphic to
$S^n$. Since $h$ has constant sectional curvature $1$, one has
$\mathrm{Ric}_h=(n-1)h$. The Lichnerowicz estimate~\cite{zbMATH03156874}
therefore gives
\(
\lambda_1(-\Delta_h)\ge n.
\)
If equality held, a first eigenfunction would
lift to a $\Gamma_h$-invariant first spherical harmonic on $S^n$, hence to a
nonzero linear function $x\mapsto\langle v,x\rangle$ with
$\zeta v=v$ for every $\zeta\in\Gamma_h$. Thus every element of $\Gamma_h$
would fix $v/|v|\in S^n$. Since $\Gamma_h\neq\{1\}$, choosing
$\zeta\in\Gamma_h\setminus\{1\}$ contradicts the freeness of the
$\Gamma_h$-action. Therefore
\(
\lambda_1(-\Delta_h)>n.
\)
Since $\lambda_0(-\Delta_h)=0$ and every nonzero eigenvalue is at least
$\lambda_1(-\Delta_h)>n$, the number $n$ is not in the spectrum of
$-\Delta_h$. Hence $\ker A_h=\{0\}$.

Since $n\notin\operatorname{spec}(-\Delta_h)$, the spectral theorem and the
identification of the domain of $-\Delta_h$ with $W^{2,2}(M)$ show that
\(A_h=a_n(-\Delta_h-n)\colon W^{2,2}(M)\to L^2(M)\)
is an isomorphism. Hence, for every $\rho\in C^{k,\alpha}(M)$, there is a
unique $w\in W^{2,2}(M)$ satisfying $A_h w=\rho$. Equivalently,
\(
\Delta_h w=-nw-a_n^{-1}\rho.
\)
Starting from $w\in W^{2,2}(M)$, repeated application of global $L^p$
regularity~\cite[Theorem~2.5(a)]{zbMATH04030435}, together with the Sobolev
embedding theorem, gives
\(w\in W^{2,p}(M)\)
for every \(p<\infty\).
Indeed, one iterates the Sobolev embedding and the equation until
$w\in W^{2,p_0}(M)$ for some $p_0>n/2$; the Sobolev exponent strictly
increases at each subcritical step. At a critical step $p=n/2$, one uses
$W^{2,n/2}(M)\hookrightarrow L^r(M)$ for every finite $r$ and chooses
$r>n/2$; the equation and global $L^r$ regularity then give
$w\in W^{2,r}(M)$. Thus in all cases $w$ is bounded once this threshold is
crossed. The right-hand side then belongs to $L^p(M)$ for every $p<\infty$,
and another
application of global $L^p$ regularity gives the displayed conclusion. Choose
$p$ sufficiently large and then choose $\beta$ with
\(\alpha<\beta<1-\frac np\).
The Sobolev embedding gives $w\in C^{1,\beta}(M)$, so the right-hand side belongs to
$C^{0,\alpha}(M)$. Global H\"older regularity
\cite[Theorem~2.5(b)]{zbMATH04030435} then gives $w\in C^{2,\alpha}(M)$.
Iterating the same regularity theorem in the equation yields
\(
w\in C^{k+2,\alpha}(M).
\)

Since $A_h$ is a second-order linear differential operator with smooth
coefficients on the closed manifold $M$, the map
\(A_h\colon C^{k+2,\alpha}(M)\to  C^{k,\alpha}(M)\)
is bounded. It is injective because its kernel is already trivial in
$W^{2,2}(M)$,
and it is surjective by the preceding argument. The bounded inverse theorem
therefore implies that
\(A_h^{-1}\colon C^{k,\alpha}(M)\to C^{k+2,\alpha}(M)\)
is bounded.
Since multiplication by each of the positive smooth functions \(u\) and
\(u^q\) is an isomorphism on H\"older spaces,
\(D_v\mathcal Y_{(g,u)}\) is an isomorphism as well.

By the Banach-space implicit function theorem, there exist an open
neighborhood
$\mathcal{U}_g^{k+2,\alpha}\subset\mathcal{G}^{k+2,\alpha}$ of $g$, an open
neighborhood $\mathcal{W}_u^{k+2,\alpha}\subset\mathcal{P}^{k+2,\alpha}$
of $u$, and a $C^1$ map
\(\sigma_{k,\alpha}\colon
\mathcal{U}_g^{k+2,\alpha}\to\mathcal{W}_u^{k+2,\alpha}\)
such that $\sigma_{k,\alpha}(g)=u$ and the zero set of $\mathcal Y$ in
$\mathcal{U}_g^{k+2,\alpha}\times\mathcal{W}_u^{k+2,\alpha}$ is precisely
the graph of $\sigma_{k,\alpha}$. In particular,
\(\mathcal Y\bigl(g',\sigma_{k,\alpha}(g')\bigr)=0\)
for \(g'\in\mathcal{U}_g^{k+2,\alpha}\).

Now let \(g'\in \mathcal{U}_g^{k+2,\alpha}\cap X(M)\), and set
$u_{g'}:=\sigma_{k,\alpha}(g')$. Then
\(u_{g'}\in C^{k+2,\alpha}(M)\) satisfies
\(
L_{g'}(u_{g'})=n(n-1)u_{g'}^q.
\)
Rewriting this equation as
\[
-a_n\Delta_{g'}u_{g'}
=n(n-1)u_{g'}^q-\mathrm{Sc}_{g'}u_{g'}
\]
and using that $g'$ is smooth and $u_{g'}>0$, we see that the
right-hand side lies in $C^{k+2,\alpha}(M)$. Hence the global Schauder
estimate \cite[Theorem~2.5(b)]{zbMATH04030435}, applied with $k$ replaced by
$k+2$, gives
$u_{g'}\in C^{k+4,\alpha}(M)$. Iterating this estimate gives
$u_{g'}\in C^\infty(M)$.

Define
\(\widehat h(g'):=u_{g'}^{\frac{4}{n-2}}g'\).
Then \(\widehat h(g')\) is smooth and has scalar curvature \(n(n-1)\). By the
existence and uniqueness results established earlier, \(r(g')\in[g']\cap
\mathrm{Met}^{CC}(M)\) is the unique round representative of \([g']\). In particular,
\(r(g')\) is Einstein and has scalar curvature \(n(n-1)\). Since
\(\widehat h(g')\in[g']=[r(g')]\) also has constant scalar curvature,
Obata's theorem implies that $\widehat h(g')=c\,r(g')$
for some \(c>0\). Since both metrics have scalar curvature \(n(n-1)\), we have \(c=1\). Hence
\(\widehat h(g')=r(g')\)
for \(g'\in \mathcal{U}_g^{k+2,\alpha}\cap X(M)\).
Since both conformal factors are positive, this identity also gives
$u_{g'}=v_{g'}$ on $\mathcal{U}_g^{k+2,\alpha}\cap X(M)$.

We now prove continuity in the Fr\'echet $C^\infty$-topology. Let
$g_i\in X(M)$ and suppose that $g_i\to g$ in
$C^\infty(\operatorname{Sym}^2(T^*M))$. Fix $m\geq0$ and
$\alpha\in(0,1)$. For all sufficiently large $i$, one has
$g_i\in\mathcal{U}_g^{m+2,\alpha}$. The identity just proved shows that the
implicit-function branch agrees on
$\mathcal{U}_g^{m+2,\alpha}\cap X(M)$ with the canonical conformal factor.
Hence
\(v_{g_i}=\sigma_{m,\alpha}(g_i)
\to \sigma_{m,\alpha}(g)=v_g\)
in \(C^{m+2,\alpha}(M)\).
Since $m$ is arbitrary, we have
\(v_{g_i}\to v_g\)
in \(C^\infty(M)\).
Consequently,
\(r(g_i)=v_{g_i}^{\frac{4}{n-2}}g_i
\to
v_g^{\frac{4}{n-2}}g=r(g)\)
in \(C^\infty(\operatorname{Sym}^2(T^*M))\).
Thus both $g\mapsto v_g$ and $r$ are sequentially continuous. Since the
relevant $C^\infty$ topologies are metrizable, both maps are continuous. In
particular,
\(r\colon X(M)\to \mathrm{Met}^{CC}(M)\)
is continuous.

For $g\in X(M)$, set
\[
H(g,t):=\bigl((1-t)+t\,v_g\bigr)^{\frac{4}{n-2}}g,
\qquad t\in[0,1].
\]
Since affine interpolation, pointwise positive powers, and tensor
multiplication are continuous operations in the $C^\infty$ Fr\'echet
topology, the map $(g,t)\mapsto H(g,t)$ is jointly continuous.
Since $L_g(1)=\mathrm{Sc}_g>0$ and
\(L_g(v_g)=v_g^{\frac{n+2}{n-2}}\mathrm{Sc}_{r(g)}>0\),
linearity of $L_g$ gives
$L_g\bigl((1-t)+t\,v_g\bigr)>0$. Hence
$\mathrm{Sc}_{H(g,t)}>0$. Moreover, $H(g,t)$ is conformal to $g$, so it is
locally conformally flat. Therefore $H(g,t)\in X(M)$ for every $t$.
By construction, $H(g,0) = g$ and $H(g,1) = r(g) \in \mathrm{Met}^{CC}(M)$. Moreover, if $g \in \mathrm{Met}^{CC}(M)$, the uniqueness of the round representative yields $r(g) = g$, whence the conformal factor $v_g$ is identically $1$. It follows that $H(g,t)=g,$ $t\in[0,1]).$
Because $H$ continuously deforms $X(M)$ into $\mathrm{Met}^{CC}(M)$ while keeping $\mathrm{Met}^{CC}(M)$ pointwise fixed at all times, $H$ is a strong deformation retraction of $X(M)$ onto $\mathrm{Met}^{CC}(M)$.

Since $r(\psi\cdot g)=\psi\cdot r(g)$, it follows that
$v_{\psi\cdot g}=v_g\circ\psi^{-1}$, and hence
\(H(\psi\cdot g,t)=\psi\cdot H(g,t)\)
for \(\psi\in\mathrm{Diff}(M)\).
Let $q_M\colon X(M)\to\mathcal{M}(M)$ be the quotient map. Since
$\mathrm{Diff}(M)$ acts on $X(M)$ by homeomorphisms, $q_M$ is open:
for every open $U\subset X(M)$,
\(q_M^{-1}(q_M(U))
=
\bigcup_{\psi\in\mathrm{Diff}(M)}\psi\cdot U\)
is open. Hence $q_M\times\mathrm{id}_{[0,1]}$ is an open continuous
surjection, and therefore a quotient map. Consequently, $H$ descends to a
well-defined continuous homotopy
\(\overline H\colon\mathcal{M}(M)\times[0,1]\to \mathcal{M}(M)\)
from the identity map to the map induced by $r$.

It remains to observe that the time-one map $\overline H(\cdot,1)$ is
constant. Any two metrics $h_1,h_2\in\mathrm{Met}^{CC}(M)$ are round metrics on the same
smooth manifold $M$, which has already been shown to be diffeomorphic to a
spherical space form. Since $n$ is odd, de Rham's rigidity theorem
\cite[pp.~66--67]{MR43468} implies that diffeomorphic spherical space forms
are isometric. Hence
$\mathrm{Met}^{CC}(M)$ is a single $\mathrm{Diff}(M)$-orbit, and
\(
\mathrm{Met}^{CC}(M)/\mathrm{Diff}(M)
\)
is a singleton. The homotopy $\overline H$ therefore contracts
$\mathcal{M}(M)$ to a point.
\end{proof}

\begin{remark}
The difficulty in extending Proposition~\ref{spherical space} to the case
\(M=S^n\) is not the existence of a path to a round metric within each conformal
class: this already follows from \cite[Corollary~3.2]{MR2950765}.
Rather, the Obata normalization used above does not provide a canonical choice of a round representative.

Indeed, when the deck group of a round representative is nontrivial, the proof
above uses the non-spherical case of Obata's rigidity theorem together with the
strict Lichnerowicz--Obata spectral gap
to single out a unique normalized round metric in each conformal class. This makes
the retraction map \(r\colon X(M) \to \mathrm{Met}^{CC}(M)\) canonical and allows one to prove
its continuity. In the sphere case \(M=S^n\), this mechanism breaks down: the round
conformal class is exceptional, since it contains the entire
\(\mathrm{Conf}(S^n,g_{st})\)-orbit of \(g_{st}\). Thus, even after fixing the normalization \(\mathrm{Sc}_h=n(n-1)\), there is no unique round representative in a given conformal class.

Consequently, the above argument does not directly extend to \(S^n\).
\end{remark}

\begin{proof}[Proof of Theorem~\ref{thm:intro-D}]
Part~(i) follows from Propositions~\ref{X3-contractible} and
\ref{prop:X-and-M-S2S1-contractible}.

For part~(ii), recall that $X(M)$ is the space of LCF metrics with positive
scalar curvature on $M$. If $X(M)=\varnothing$, then
$\mathcal{M}(M)=\varnothing$. If $X(M)\neq\varnothing$, Proposition~
\ref{spherical space} gives a $\mathrm{Diff}(M)$-equivariant strong
deformation retraction of $X(M)$ onto $\mathrm{Met}^{CC}(M)$, which is a
single $\mathrm{Diff}(M)$-orbit. The induced homotopy therefore contracts
$\mathcal{M}(M)$. Hence $\mathcal{M}(M)$ is either empty or contractible.
\end{proof}

\section{Euclidean Rigidity for LCF Manifolds}\label{Euclidean 1}

This section establishes Euclidean rigidity for complete, open, simply connected LCF manifolds with nonnegative scalar curvature. In dimension three, any such contractible manifold is conformally diffeomorphic to Euclidean space: the Ma--Qing dimension bound and \v{C}ech--Alexander duality force the limit set to be a single point. In higher dimensions, \(\lfloor(n-2)/2\rfloor\)-connectivity at infinity together with the small loops condition on the limit set implies, through shape-theoretic arguments, that the manifold is homeomorphic to \(\mathbb{R}^n\). We also obtain conformal Euclidean rigidity under several analytic hypotheses and a smooth rigidity theorem in dimension four under bounded geometry.

\subsection{Contractible LCF 3-manifolds}
Since Schoen and Yau proved that every complete, noncompact \(3\)-manifold
with positive Ricci curvature is diffeomorphic to
\(\mathbb{R}^3\)~\cite[Theorem~3, p.~217]{MR0645740}, it is natural to ask which complete noncompact Riemannian \(3\)-manifolds with PSC are diffeomorphic to \(\mathbb{R}^3\). Consider the smooth metric
\[
g_0:=g_E+\Bigl(\sum_{i=1}^3 x_i\,dx_i\Bigr)^2
\]
on \(\mathbb R^3\). In spherical coordinates,
\[
g_0=(1+r^2)\,dr^2+r^2\,d\theta^2+r^2\sin^2\theta\,d\phi^2.
\]
Since \(g_0\geq g_E\), every \(d_{g_0}\)-Cauchy sequence is
\(d_{g_E}\)-Cauchy and hence converges in the Euclidean topology. The local
uniform equivalence of the two smooth metrics then gives convergence in
\(d_{g_0}\). Thus \(g_0\) is complete. A
warped-product computation after the change of variables
\(dt=\sqrt{1+r^2}\,dr\) gives
\(\mathrm{Sc}_{g_0}=2(r^2+3)/(r^2+1)^2>0\). In addition,
\((S^2\times\mathbb{R},g_{st}\oplus dt^2)\) also has PSC. The round
cylinder is LCF, since, with \(r=e^\tau\),
\[
g_E=dr^2+r^2g_{st}=r^2\bigl(d\tau^2+g_{st}\bigr).
\]
The metric \(g_0\) is also LCF. Indeed, with \(t\) as above, on \(r>0\),
\[
g_0=dt^2+r(t)^2g_{st}
   =r(t)^2\bigl(ds^2+g_{st}\bigr),
\qquad ds=\frac{dt}{r(t)}.
\]
Thus \(g_0\) is conformal to the round cylinder away from the origin. Since local conformal flatness is
conformally invariant, \(g_0\) is LCF on \(\mathbb R^3\setminus\{0\}\). In
dimension three, local conformal flatness is equivalent to the vanishing of
the Cotton tensor. Hence the Cotton tensor of \(g_0\) vanishes on
\(\mathbb R^3\setminus\{0\}\), and, by continuity, also at the origin.
Therefore \(g_0\) is LCF on all of \(\mathbb R^3\).

These two examples show that additional topological or geometric conditions
are required to settle the question completely. Gromov and Lawson~\cite[Corollary~10.9]{MR720933} proved that every complete noncompact \(3\)-manifold with uniformly positive scalar curvature and finitely generated fundamental group is simply connected at infinity. Hence, if such a manifold is contractible, it is homeomorphic to \(\mathbb{R}^3\).
Chang, Weinberger, and Yu later gave an index-theoretic proof~\cite[Theorem~1]{MR2721617}.

This leads to the following conjecture, which has circulated in the field for decades: every complete, connected, open, contractible \(3\)-manifold with nonnegative scalar curvature is diffeomorphic to \(\mathbb{R}^3\). The conjecture has been verified using inverse mean curvature flow under the assumption of bounded geometry (i.e., \(|\mathrm{Rm}|\leq C\) and the injectivity radius is bounded below by \(C^{-1}\)); see~\cite[p.~2]{2026arXiv260301887C}.
Here we replace the condition of bounded geometry with that of local conformal flatness and establish the following result.

\begin{theorem}\label{contractible}
Let \( M^3 \) be a connected, open, contractible \(3\)-manifold admitting a complete, locally conformally flat metric \(g\) with nonnegative scalar curvature.
Then \( (M^3,g) \) is conformally diffeomorphic to \( (\mathbb{R}^3, g_E) \).
\end{theorem}

The proof is as follows. By Theorem~\ref{Liouville theorem}, \(M\) embeds conformally as a contractible domain \(\Omega\subset S^3\), and \(\partial\Omega\) has zero Newtonian \(2\)-capacity. \v{C}ech--Alexander duality first shows that \(\Omega\) is dense in \(S^3\), so that \(S^3\setminus\Omega=\partial\Omega\), and a second application shows that \(\partial\Omega\) is connected. The Ma--Qing estimate gives \(\dim_{\mathcal H}(\partial\Omega)\leq\frac12\). Since every nonsingleton connected metric space has Hausdorff dimension at least \(1\), the boundary consists of a single point. Hence \(\Omega\cong\mathbb R^3\), and the developing map followed by stereographic projection is the required conformal diffeomorphism.

Improving Schoen--Yau's estimate~\cite[Theorem~2.7]{zbMATH04075988}, Ma and Qing proved the following special case of their theorem~\cite[Theorem~1.3]{zbMATH08016940}.
\begin{theorem}[Ma--Qing]\label{Ma}
Let \((M^n,\bar g)\), \(n\geq3\), be a complete Riemannian manifold, and let \(S\subset M\) be compact. Let \(D\subset M\) be a bounded open neighborhood of \(S\). Suppose that
\(g=u^{\frac{4}{n-2}}\bar g\) is a conformal metric on \(D\setminus S\) that is geodesically complete near \(S\). If the scalar curvature of \(g\) is nonnegative, then
\[
\dim_{\mathcal{H}}^{\bar g}(S) \le \frac{n-2}{2}.
\]
\end{theorem}

Here the superscript in \(\dim_{\mathcal H}^{\bar g}\) records the background
metric; when the background is \(g_{st}\), we suppress it.

In the applications below, we verify the following sufficient condition for
geodesic completeness near \(S\): every locally rectifiable curve in
\(D\setminus S\) that converges in the background topology to a point of \(S\)
has infinite length with respect to \(g\).

We shall also use the following capacity--dimension implication. Let
\(E\subsetneq S^n\) be compact with zero Newtonian \(2\)-capacity, and choose a
stereographic projection
\(\sigma\colon S^n\setminus\{p\}\to\mathbb R^n\) with \(p\notin E\). On a
relatively compact neighborhood of \(E\) in this chart, the round and
Euclidean metrics and volume forms are uniformly comparable. Consequently,
the corresponding local Sobolev \(2\)-capacities have the same null sets, so
\(\sigma(E)\) has zero Euclidean Newtonian \(2\)-capacity. The latter agrees,
up to normalization, with the Riesz \((1,2)\)-capacity. Since \(\sigma(E)\)
is bounded, \cite[Proposition~5.1.4(b)]{zbMATH00824833} gives
\(C_{1,2}(\sigma(E))=0\). For every \(0<\varepsilon<2\),
\cite[Theorem~5.1.13]{zbMATH00824833}, with
\((\alpha,p)=(1,2)\) and \(h(r)=r^{n-2+\varepsilon}\), applies because
\[
\int_0^1\frac{h(r)}{r^{n-2}}\,\frac{dr}{r}
=\int_0^1r^{\varepsilon-1}\,dr<\infty.
\]
It gives \(\Lambda_h(\sigma(E))=0\), where \(\Lambda_h\) denotes the
associated ball-gauge Hausdorff measure. For this power gauge, vanishing of
\(\Lambda_h\) is equivalent to
\(\mathcal H^{n-2+\varepsilon}(\sigma(E))=0\). It follows that
\(\dim_{\mathcal H}\sigma(E)\leq n-2\). Finally, \(\sigma\) is bi-Lipschitz
on a neighborhood of \(E\), and hence \(\dim_{\mathcal H}E\leq n-2\).

\begin{proof}[Proof of Theorem~\ref{contractible}]
Since \( (M^3, g) \) is a simply connected, complete, LCF manifold with nonnegative scalar curvature, its developing map is defined on \( M \) itself, and
$\Phi\colon M\to S^3$ is conformal. By Theorem~\ref{Liouville theorem}, \(\Phi\) is injective and the boundary \(\partial\Phi(M)\) has zero Newtonian \(2\)-capacity.
Because \(\Phi\) is an injective local diffeomorphism, it is a conformal
diffeomorphism from \(M\) onto the open domain
\(\Omega:=\Phi(M)\subset S^3\). In particular, \(\Omega\) is contractible.

We now show that \( \partial \Omega = S^3 \setminus \Omega \).
Set \(K:=S^3\setminus\Omega\). Since \(M\) is noncompact and \(\Phi\) is a diffeomorphism onto \(\Omega\), we have \(\Omega\neq S^3\), and hence \(K\neq\varnothing\). Since \(\partial\Omega\) has zero Newtonian \(2\)-capacity, the capacity--dimension implication above gives
$\dim_{\mathcal H}(\partial\Omega)\le 1.$
By the Szpilrajn inequality~\cite[Chapter~VII]{MR0006493}, the
covering dimension of \(\partial\Omega\) is at most \(1\), so
$\check H^{\,2}(\partial\Omega;\mathbb Z)=0.$
By \v{C}ech--Alexander duality,
\(\widetilde H_0(S^3\setminus \partial\Omega;\mathbb Z)
\cong
\widetilde{\check H}^{\,2}(\partial\Omega;\mathbb Z)=0\).
Therefore \(S^3\setminus \partial\Omega\) is connected. On the other hand,
$S^3\setminus \partial\Omega=\Omega\sqcup \operatorname{int}(K),$
a disjoint union of open sets. Since \(\Omega\neq\varnothing\), connectedness forces
$\operatorname{int}(K)=\varnothing.$
Thus \(\Omega\) is dense in \(S^3\), and hence
\(\partial\Omega=\overline{\Omega}\setminus\Omega=S^3\setminus\Omega=K.\)

Since \(\partial\Omega=S^3\setminus\Omega\) and \(\Omega\) is contractible,
\v{C}ech--Alexander duality yields
\(\widetilde{\check H}^{\,0}(\partial\Omega;\mathbb Z)
\cong
\widetilde H_{\,2}(\Omega;\mathbb Z)=0\).
Because \(\partial\Omega\) is nonempty and compact, this implies that
\(\partial\Omega\) is connected.

Set \(S:=S^3\setminus\Omega=\partial\Omega\). Since \(S\) is closed in the
compact manifold \(S^3\), it is compact. Choose \(x_0\in\Omega\) and
\(\varepsilon>0\) such that
\(\overline{B_{g_{st}}(x_0,\varepsilon)}\subset\Omega,\) and set
\(D:=S^3\setminus\overline{B_{g_{st}}(x_0,\varepsilon)}.\)
Then \(D\) is a bounded open neighborhood of \(S\) in \(S^3\).

Let \(\tilde g:=(\Phi^{-1})^*g\) be the metric transported from \(M\) to
\(\Omega\). Since \(\Phi\colon M\to\Omega\) is a conformal diffeomorphism,
the restriction of \(\tilde g\) to \(D\setminus S\) has the
form \(\tilde g=u^4g_{st}\) for some smooth positive function \(u\), and
\(\mathrm{Sc}_{\tilde g}=\mathrm{Sc}_g\circ\Phi^{-1}\geq0.\)

It remains to verify that \(\tilde g\) is geodesically complete near \(S\).
Since \(\Phi\colon(M,g)\to(\Omega,\tilde g)\) is an isometry and \((M,g)\) is complete,
the metric space \((\Omega,d_{\tilde g})\) is complete.

Suppose, for contradiction, that there exists a locally rectifiable curve
\(\gamma\colon[0,1)\to D\setminus S\subset\Omega\)
such that \(\gamma(t)\to p\in S\) as \(t\uparrow 1\) and
$L_{\tilde g}(\gamma)<\infty.$ Then for \(0\le s<t<1\),
\(d_{\tilde g}(\gamma(s),\gamma(t))
\le L_{\tilde g}\bigl(\gamma|_{[s,t]}\bigr)\).
Since the total length of \(\gamma\) is finite, the right-hand side tends to \(0\)
as \(s,t\uparrow 1\). Thus \(\gamma(t)\) is a Cauchy curve in
\((\Omega,d_{\tilde g})\). By completeness, there exists \(q\in \Omega\) such that
\(\gamma(t)\to q\) with respect to \(d_{\tilde g}\).

But the distance topology of a smooth Riemannian metric agrees with the manifold
topology. Thus \(\gamma(t)\to q\) in \(\Omega\), and continuity of the inclusion
\(\Omega\hookrightarrow S^3\) gives \(\gamma(t)\to q\) in \(S^3\). Since \(S^3\)
is Hausdorff and \(\gamma(t)\to p\in S=S^3\setminus\Omega\) in the background
topology, uniqueness of limits would give \(p=q\), a contradiction.
Therefore every locally rectifiable curve in \(D\setminus S\) converging in the
background topology to a point of \(S\) has infinite \(\tilde g\)-length.
Hence \(\tilde g\) is geodesically complete near \(S\).

We may therefore apply Theorem~\ref{Ma} with background manifold
\((S^3,g_{st})\), singular set \(S\), and the bounded open
neighborhood \(D\). It follows that
\(\dim_{\mathcal H}(S)\le \frac{3-2}{2}=\frac12.\) Since
\(S=\partial\Omega\), it follows that
$\dim_{\mathcal H}(\partial\Omega)\le \frac12.$

Hence the topological dimension of \(\partial\Omega\) is zero; in particular, \(\partial\Omega\) is totally disconnected. On the other hand, \(\partial\Omega\) is connected, which forces \(\partial\Omega\) to consist of a single point. Consequently, $\Omega$ is diffeomorphic to $\mathbb{R}^3$. Because stereographic projection provides a conformal diffeomorphism from the punctured sphere to Euclidean space, we conclude that $(M,g)$ is conformally diffeomorphic to $(\mathbb R^3, g_E)$.
\end{proof}

\begin{remark}
The nonnegativity assumption on the scalar curvature in
Theorem~\ref{contractible} cannot be removed. Indeed, let
$\Omega\subset S^3$ be a round ball. By the Loewner--Nirenberg
theorem~\cite{MR358078}, $\Omega$ admits a complete metric $g$, conformal
to $g_{st}$, with $\mathrm{Sc}_g=-1$. Although $\Omega$ is contractible,
it is not conformally diffeomorphic to
$S^3\setminus\{q\}\cong\mathbb R^3$. Otherwise, by Liouville theorem \ref{Liouville theorem},
such a conformal diffeomorphism would extend to a M\"obius transformation
of $S^3$, which would map the round sphere $\partial\Omega$ to the single
point $q$, a contradiction. More generally, let \(F^d\subset S^n\) be a closed smooth submanifold. Then
there exists a smooth positive function \(u\) on \(S^n\setminus F\) such that
$g=u^{\frac{4}{n-2}}g_{st}$
is complete and satisfies \(\mathrm{Sc}_g=-1\) if and only if
$d>\frac{n-2}{2}$; see~\cite{zbMATH04052477}.
\end{remark}

This argument does not extend directly to higher dimensions $n \ge 4$, since Theorem~\ref{Ma} only yields the bound $\dim_{\mathcal{H}}(\partial \Omega) \le \frac{n-2}{2}$, which is no longer sufficient to force $\partial \Omega$ to be totally disconnected. However, under additional assumptions, the same strategy remains valid in higher dimensions, with Theorem~\ref{Ma} replaced by results implying that $\dim_{\mathcal{H}}(\partial \Omega) < 1$.

\begin{corollary}\label{Contractible for}
Let \(M^n\) (\(n\ge 4\)) be a connected, open, simply connected \(n\)-manifold
such that $\widetilde H_{n-1}(M;\mathbb Z)=0.$
Assume that \(M\) admits a complete locally conformally flat metric \(g\) with
$\mathrm{Sc}_g \ge 0.$
Assume in addition that at least one of the following holds:
\begin{enumerate}
\item [(i)] \(\mathrm{Ric}_g \ge 0\) outside a compact subset of \(M\);
\item [(ii)]
\[
\mathrm{Ric}_g^- \in L^1(M,g) \cap L^\infty(M,g),\qquad
\mathrm{Sc}_g \in L^\infty(M,g),\qquad
|\nabla_g \mathrm{Sc}_g| \in L^\infty(M,g),
\]
where
\[
\mathrm{Ric}_g^-(x):=
\max\bigl\{0,-\lambda_{\min}(g^{-1}\mathrm{Ric}_g)(x)\bigr\};
\]
\item [(iii)] The Ricci tensor satisfies the critical integrability condition $|\mathrm{Ric}_g| \in L^{n/2}(M,g)$;

\item [(iv)] There exists a point $p \in M$ such that
\[
\lim_{r \to \infty} \frac{\mathrm{Vol}_g\bigl(B_p(r)\bigr)}{r^{n} (\log r)^{\,n-1}} = 0;
\]
\item [(v)] There exists \( p_0 \in (n-2,n) \) such that
\[
A_g^{(p_0)} := \frac{1}{n-2} \left( (p_0-2)\,\mathrm{Ric}_g \;+\; \frac{n-p_0}{2(n-1)}\,\mathrm{Sc}_g\, g \right) \ge 0.
\]

\end{enumerate}
Then \((M,g)\) is conformally diffeomorphic to \((\mathbb R^n,g_E)\).
\end{corollary}

\begin{proof}
As in the proof of Theorem~\ref{contractible}, since $(M^n,g)$ is simply connected and LCF with $\mathrm{Sc}_g \ge 0$, Theorem~\ref{Liouville theorem} yields an injective conformal map $\Phi\colon M\to S^n$. Moreover, the boundary $\partial \Phi(M)$ has vanishing Newtonian $2$-capacity.
Thus \(\Phi\) is a conformal diffeomorphism from \(M\) onto an open
domain $\Omega:=\Phi(M)\subset S^n.$ Since \(\partial\Omega\) has zero Newtonian \(2\)-capacity, the capacity--dimension implication above gives
$\dim_{\mathcal H}(\partial\Omega)\le n-2.$ In particular, the covering dimension of \(\partial\Omega\) is at most \(n-2\), so
$\widetilde{\check H}^{\,n-1}(\partial\Omega;\mathbb Z)=0.$

By \v{C}ech--Alexander duality,
\(\widetilde H_0(S^n\setminus \partial\Omega;\mathbb Z)
\cong
\widetilde{\check H}^{\,n-1}(\partial\Omega;\mathbb Z)=0\).
Therefore \(S^n\setminus \partial\Omega\) is connected.
Set $K:=S^n\setminus \Omega.$
Since \(M\) is noncompact whereas \(S^n\) is compact, we have \(K\neq\varnothing\). Thus,
$S^n\setminus \partial\Omega=\Omega\sqcup \operatorname{int}(K),$
a disjoint union of open sets. Since \(\Omega\neq\varnothing\), connectedness forces
$\operatorname{int}(K)=\varnothing.$
Thus \(\Omega\) is dense in \(S^n\), and hence
\(\partial\Omega=\overline{\Omega}\setminus\Omega=S^n\setminus\Omega=K.\)

Now \(\Omega\cong M\), so our topological hypothesis gives
$\widetilde H_{n-1}(\Omega;\mathbb Z)=0.$ By \v{C}ech--Alexander duality,
\(\widetilde{\check H}^{\,0}(\partial\Omega;\mathbb Z)
\cong
\widetilde H_{n-1}(\Omega;\mathbb Z)=0\).
Hence \(\partial\Omega\) is connected.

Next define the pushforward metric
$\widetilde g:=(\Phi^{-1})^*g$
on \(\Omega\). Then \(\widetilde g\) is conformal to the standard round metric
on \(S^n\), and \(\Phi\colon(M,g)\to(\Omega,\widetilde g)\)
is an isometry. Therefore \((\Omega,\widetilde g)\) is complete. All five
additional hypotheses are preserved by this isometry; under condition~(iv),
the distinguished point becomes \(\widetilde p:=\Phi(p)\).

Hence, if any one of conditions~(i), (ii), or (iii) holds, the finite-point conformal compactification theorem applies: for conditions~(i) and~(ii), this is due to Ma--Qing~\cite[Corollary~1.1]{zbMATH07456625}, while for condition~(iii), it follows from Chen--Li~\cite[Theorem~1.4]{zbMATH07576884}. In all three cases, one concludes that $\partial \Omega = S^n \setminus \Omega$ is a finite set. Since \(\partial\Omega=S^n\setminus\Omega\) is nonempty and connected, it follows that \(\partial\Omega\) consists of a single point, say \(\partial\Omega=\{q\}\).

If condition~(iv) holds, then there exists a point $\widetilde p \in \Omega$ such that
\[
\lim_{r \to \infty} \frac{\mathrm{Vol}_{\widetilde g}\bigl(B_{\widetilde p}(r)\bigr)}{r^{n} (\log r)^{\,n-1}} = 0.
\]
Carron and Herzlich~\cite[Proposition~2.2]{zbMATH01756816} then imply that $\dim_{\mathcal{H}}(\partial \Omega)=0$.

If condition~(v) is satisfied, we apply Liu--Ma--Qing--Zhong's theorem~\cite[Theorem~5.1.1]{zbMATH08029235}. We verify its hypotheses as follows. The set $\partial\Omega$ is closed in $S^n$, and the metric $\widetilde g$ on $S^n \setminus \partial\Omega = \Omega$ is conformal to $g_{st}$. Moreover, $\widetilde g$ is complete, hence geodesically complete near $\partial\Omega$ as in the proof of Theorem~\ref{contractible}. Condition~(v) gives $p_0 \in (n-2,n)$ such that $A_{\widetilde g}^{(p_0)} \ge 0$. Since $n \ge 4$, we have $(n-2,n) \subset [2,n)$. Therefore, Liu--Ma--Qing--Zhong's theorem yields
$\dim_{\mathcal{H}}(\partial\Omega) \le (n-p_0)/2 < 1.$

Thus conditions (iv) and (v) each imply that $\partial\Omega$ consists of a single point,
since a nonempty connected metric space of Hausdorff dimension $<1$ is a singleton.
Consequently, under any one of the five additional conditions, there exists
\(q\in S^n\) such that \(\partial\Omega=\{q\}\), and hence
\(\Omega=S^n\setminus\{q\}\).
By stereographic projection from \(q\), \(\Omega\) is conformally diffeomorphic
to \((\mathbb R^n,g_E)\). Composing with \(\Phi\), we conclude that
\((M,g)\) is conformally diffeomorphic to \((\mathbb R^n,g_E)\).
\end{proof}

\subsection{LCF 4-manifolds}

In dimension four, the author proved in~\cite[Theorem~5.3]{2025arXiv251213528D} that
if \((M^4,g)\) is closed, LCF, and scalar-flat, with
\(\pi_2(M)\neq 0\), then its Riemannian universal cover
\((\widetilde M,\widetilde g)\) is, up to homothety, isometric to
\(\bigl(\mathbb H^2\times S^2,g_{\mathbb H}\oplus g_{st}\bigr)\), where
\(g_{\mathbb H}\) denotes the hyperbolic metric of curvature \(-1\). In the present subsection, we use
bounded geometry to prove a Euclidean rigidity statement for complete
LCF \(4\)-manifolds. The argument uses Huang's Ricci-flow classification
theorem and his analysis of infinite connected sums
\cite[Theorem~1.1 and \S\S~2--3]{zbMATH07110928}.

\begin{theorem}[Euclidean rigidity for LCF \(4\)-manifolds]\label{LCF-4-manifold}
Let \((M^4,g)\) be a complete, smooth, connected, open, locally
conformally flat Riemannian \(4\)-manifold. Assume that \((M,g)\) has
bounded geometry, in the sense that \(|\mathrm{Rm}_g|\leq C\) and
\(\mathrm{inj}_g(x)\geq C^{-1}\) for all \(x\in M\) and some constant
\(C\geq1\), and has uniformly positive scalar curvature, in the sense that
\(\mathrm{Sc}_g\geq \sigma>0\) for some constant \(\sigma\). If
\(\pi_1(M)=0\) and \(H_3(M;\mathbb R)=0\), then \(M\) is diffeomorphic
to the standard \(\mathbb R^4\).
\end{theorem}

\begin{proof}
Since \(\pi_1(M)=0\), the manifold \(M\) is orientable. The curvature
bound gives, after changing constants, a two-sided bound
\(|\mathrm{sec}_g|\leq C_0\), while
\(\operatorname{inj}_g\geq C^{-1}>0\). Thus \((M,g)\) has bounded geometry
in the sense required by Huang's theorem.

Since \((M^4,g)\) is LCF, its Weyl tensor vanishes. In an orthonormal frame,
with \(R_{ijij}\) denoting the sectional curvature of the plane spanned by
\(e_i\) and \(e_j\), the curvature tensor satisfies
\[
R_{ijkl}
=
\frac{1}{2}
\bigl(
\mathrm{Ric}_{ik}\delta_{jl}
+\mathrm{Ric}_{jl}\delta_{ik}
-\mathrm{Ric}_{il}\delta_{jk}
-\mathrm{Ric}_{jk}\delta_{il}
\bigr)
-
\frac{\mathrm{Sc}_g}{6}
\bigl(
\delta_{ik}\delta_{jl}
-\delta_{il}\delta_{jk}
\bigr).
\]
In particular, \(R_{1234}=0\). Consequently, for every point \(x\in M\)
and every orthonormal \(4\)-frame \(\{e_1,e_2,e_3,e_4\}\subset T_xM\),
\[
\begin{aligned}
&R_{1313}+R_{1414}+R_{2323}+R_{2424}-2R_{1234}\\
&\qquad
=\sum_{i=1}^{4}\mathrm{Ric}_{ii}-\frac{2}{3}\mathrm{Sc}_g
=\mathrm{Sc}_g-\frac{2}{3}\mathrm{Sc}_g
=\frac{1}{3}\mathrm{Sc}_g
\geq\frac{\sigma}{3}>0.
\end{aligned}
\]
Hence \((M,g)\) has uniformly positive isotropic curvature.

Moreover, \(M\) contains no essential incompressible space form. Indeed, every
such space form is, in particular, an embedded \(N^3\cong S^3/\Gamma\subset M\) with \(\Gamma\) finite
and nontrivial whose fundamental group injects into that of the ambient manifold,
so that \(\Gamma\cong\pi_1(N)\hookrightarrow\pi_1(M)=0\), which is impossible.

Huang's classification theorem~\cite[Theorem~1.1 and \S~1]{zbMATH07110928}
applies. Thus \(M\) is diffeomorphic to an infinite connected sum of
\(S^4\),
\(\mathbb{RP}^4\), \(S^3\times S^1\), and
\(S^3\widetilde{\times}S^1\) along a countably infinite, connected,
locally finite graph \(G\), with loops
and multiple edges allowed. Here the four vertex manifolds carry their standard
smooth structures. Huang states
in the proof of his Theorem~1.2 that the triviality of \(\pi_1(M)\) forces \(G\)
to be a tree and every vertex piece to be the standard \(S^4\)
\cite[\S~3]{zbMATH07110928}. We give the details. Working in the equivalent
boundary-gluing model obtained from Huang's quotient model by inserting collars,
for each vertex \(v\), let \(Y_v\) be
the corresponding vertex manifold \(X_v\) with the interior of one \(4\)-ball
removed for each incident half-edge (so a loop contributes two balls). Since
local finiteness makes this a finite family and the boundary spheres are simply
connected, Seifert--van Kampen gives \(\pi_1(Y_v)\cong\pi_1(X_v)\).
Choose a nested exhaustion of \(G\) by finite connected subgraphs. Local
finiteness ensures that the corresponding partial sums are compact and that
their interiors exhaust \(M\). Applying Seifert--van Kampen to these partial
sums and using that fundamental groups commute with this directed union gives
\[
        \pi_1(M)\cong
        \Bigl(*_{v\in V(G)}\pi_1(Y_v)\Bigr)*\pi_1(G)
        \cong
        \Bigl(*_{v\in V(G)}\pi_1(X_v)\Bigr)*\pi_1(G).
\]
Since a free product is trivial only when every factor is trivial, every
vertex group is trivial, and \(\pi_1(G)\) is trivial.
Among Huang's four pieces, only \(S^4\) has trivial fundamental group.
Thus every vertex piece is the standard \(S^4\), and \(G\) is a tree.

In the proof of Huang's Theorem~1.2, the one-endedness of \(G\) follows from
the assumption that the manifold is homeomorphic to \(\mathbb R^4\). That
assumption is unavailable here; instead, we use \(H_3(M;\mathbb R)=0\).

Since \(G\) is an infinite locally finite tree, K\"onig's lemma implies that
\(G\) contains a ray and hence has at least one end. We claim that \(G\) has
exactly one end. Suppose otherwise. Then \(G\) contains a bi-infinite
geodesic. Let \(e\) be an edge on this geodesic. Then \(G\setminus e\) has two infinite
components. In Huang's quotient model for the infinite connected sum, the
edge \(e\) is represented by a smooth neck; its central sphere
\(S_e\cong S^3\) separates \(M\) into two connected noncompact open sets
\(M_-\) and \(M_+\). Choose a smooth function \(\chi\colon M\to[0,1]\) which is
identically \(0\) on \(M_-\) outside this neck, identically \(1\) on \(M_+\)
outside the same neck, and depends only on the neck coordinate inside the
neck. Although \(\chi\) is not compactly supported, its differential is
supported in a compact subcylinder of the neck. Hence \(\alpha=d\chi\) is a
closed compactly supported \(1\)-form. Choose rays in the two infinite
components of \(G\setminus e\), beginning at the endpoints of \(e\), and
concatenate the corresponding rays in \(M_-\) and \(M_+\) across the
neck. The local finiteness of \(G\) and the compactness of the vertex blocks
imply that the resulting block decomposition of \(M\) is locally finite.
Consequently, the inverse image of a compact subset of \(M\) meets only a
bounded subarc of the concatenated double ray. After smoothing inside the
necks, we therefore obtain a proper smooth curve
\(\gamma\colon\mathbb R\to M\) crossing the neck once and joining ends
represented by rays in \(M_-\) and \(M_+\), respectively. Therefore
\(\int_\gamma\alpha=1\).
Thus \([\alpha]\neq0\) in \(H^1_c(M;\mathbb R)\): if
\(\alpha=d\varphi\) with \(\varphi\in C^\infty_c(M)\), then
\(\varphi(\gamma(t))\to0\) as \(t\to\pm\infty\), and hence
\(\int_\gamma d\varphi=0\), a contradiction. By the compactly supported de
Rham theorem and Poincar\'e duality for the oriented manifold \(M\),
\(0\neq [\alpha] \in H^1_c(M;\mathbb R)\cong H_3(M;\mathbb R),\) contradicting
\(H_3(M;\mathbb R)=0\). Hence
\(G\) is a one-ended infinite tree.

It remains to identify the resulting infinite connected sum. We have shown
that \(G\) is a locally finite one-ended tree and that every vertex piece is
the standard \(S^4\). Choose a proper ray in \(G\). Every branch off this ray
is finite; otherwise, K\"onig's lemma would produce a second end of \(G\).
The final reduction in Huang's proof of Theorem~1.2 then applies: after
absorbing these finite branches, the
connected sum becomes an infinite connected sum of standard \(S^4\)'s along
the ray \([0,\infty)\)~\cite[\S~3]{zbMATH07110928}. The standard
\(\mathbb R^4\) is such a ray sum, obtained from the decomposition of
\(\mathbb R^4\) into the unit \(4\)-ball and the closed annuli between
successive concentric \(3\)-spheres. Huang's well-definedness theorem for
infinite connected sums, together with the fact that the standard \(S^4\)
admits an orientation-reversing diffeomorphism, shows that the diffeomorphism
type is independent of the gluing embeddings and orientation choices
\cite[Theorem~2.2, Remark~1, and \S~3]{zbMATH07110928}. Hence \(M\) is
diffeomorphic to the standard \(\mathbb R^4\).
\end{proof}

\begin{remark}
For \(n\geq7\), a closed Riemannian \(n\)-manifold equipped with an LCF
metric of PSC does not necessarily have positive isotropic curvature. Indeed, for a  Riemannian \(n\)-manifold, positive isotropic
curvature is equivalent to positive \((n-4)\)-curvature
\cite[Proposition, p.~1470]{LabbiPIC2000}.  A closed LCF $n$-manifold with positive isotropic curvature implies \(b_p=0\) for all \(2\leq p\leq n-2\)
\cite[Corollary~I, p.~1471]{LabbiPIC2000}. Thus, if the manifold is connected
and oriented and also satisfies \(H_1(M;\mathbb Z)=0\), Poincar\'e duality
implies that it is a rational homology sphere. On the other hand, let $(N^3,g_{\mathrm{hyp}})$ be a closed, connected, hyperbolic integer homology $3$-spheres. Then the product manifold
$(N^3\times S^m,g_{\mathrm{hyp}}\oplus g_{st})$ is LCF, and
$\mathrm{Sc}_{g_{\mathrm{hyp}}\oplus g_{st}}=-6+m(m-1)$.
Thus its scalar curvature is negative when $m=2$, zero when $m=3$, and
positive when $m\geq4$. Furthermore,
$H_1(N^3\times S^m;\mathbb Z)=0$ and
$\pi_1(N^3\times S^m)$ is infinite. However, $N^3\times S^m$ is not a
rational homology sphere.
This example demonstrates that an LCF metric with PSC need not itself have positive isotropic curvature.
\end{remark}

\subsection{Applying shape theory to limit sets}
In this subsection, we use topological conditions on the conformal boundary
\(\Lambda\) to characterize \(\Omega\), thereby addressing Yau's
Problem~36~\cite[Problem~36]{MR1216573}.
Throughout this subsection, \(\dim\) denotes covering dimension.

Since the limit set $\Lambda$ of a developing map may be a fractal, we use shape theory instead of classical homotopy theory to study it. An important shape invariant, introduced by Borsuk, is \emph{movability}. A compact metrizable space $X$, embedded in the Hilbert cube $Q$, is said to be \emph{movable} if for every neighborhood $U$ of $X$ there exists a neighborhood $U' \subset U$ of $X$ such that, for any neighborhood $U'' \subset U$ of $X$, there exists a homotopy $H \colon U' \times I \to U$ satisfying $H(x,0)=x$ and $H(x,1) \in U''$ for all $x \in U'$~\cite[Definition~4]{zbMATH07436589}. This definition is independent of the embedding of $X$ into $Q$.

Inspired by the turning (diameter) condition in the definition of
homotopically \((1,c)\)-uniform domains due to Alestalo and
V\"ais\"al\"a~\cite{MR1404095}, we introduce the following weaker spherical
loop-filling condition.
\begin{definition}
Let \(\Omega\subset (S^n,g_{st})\) be a domain, and let \(c\ge 1\).
We say that \(\Omega\) has the \emph{diameter-controlled loop-filling property
with constant \(c\)} if every continuous map $f\colon S^1\to \Omega$
extends to a continuous map
$F\colon B^2\to \Omega$
such that
\[
\operatorname{diam}_{g_{st}}\bigl(F(B^2)\bigr)
\le c\,\operatorname{diam}_{g_{st}}\bigl(f(S^1)\bigr).
\]
\end{definition}

Recall that a \emph{topological \(n\)-cell} is a space homeomorphic to the standard closed unit cube \(I^n\). A compact subset \(X\) of an \(n\)-manifold \(N\) is \emph{cellular in \(N\)} if there exist \(n\)-cells \(B_1,B_2,\dots\) in \(N\) such that
\[
B_{i+1}\subset\operatorname{Int}_N(B_i),
\qquad
X=\bigcap_{i=1}^{\infty}B_i.
\]
A compact metric space \(X\) is called \emph{cell-like} if there exist a manifold \(N\) and an embedding \(\varphi\colon X\to N\) such that \(\varphi(X)\) is cellular in \(N\). For any nonempty finite-dimensional compact metric space \(X\), Lacher~\cite[Theorem~1.1]{zbMATH03290989}
shows that being cell-like is
equivalent to having the shape of a point, and also equivalent to the
statement that every embedding of \(X\) into any absolute neighborhood retract (ANR) has \emph{Property \(UV^{\infty}\)}. Here a compact subset \(X\) of an ANR space \(Y\) is said to have Property \(UV^{\infty}\) if, for every neighborhood \(U\subset Y\) of \(X\), there exists a neighborhood \(V\subset U\) of \(X\) such that the inclusion \(V\hookrightarrow U\) is null-homotopic in \(U\).

Following Venema~\cite{zbMATH03519593}, we make the following definitions. Let \(X\) be a compact subset of an \(n\)-manifold \(N^n\) equipped with a compatible metric. We say that \(X\) satisfies the
\emph{cellularity criterion} if, for every neighborhood \(U\) of \(X\), there
exists a neighborhood \(V\subset U\) of \(X\) such that every loop in
\(V\setminus X\) is null-homotopic in \(U\setminus X\). We say that \(X\)
satisfies the \emph{small loops condition} (SLC) if, for every neighborhood
\(U\) of \(X\), there exist a neighborhood \(V\subset U\) of \(X\) and a number
\(\varepsilon>0\) such that every loop in \(V\setminus X\) of diameter \(<\varepsilon\)
is null-homotopic in \(U\setminus X\). Finally, \(X\) satisfies the
\emph{inessential loops condition} (ILC) if, for every neighborhood \(U\) of
\(X\), there exists a neighborhood \(V\subset U\) of \(X\) such that every loop
in \(V\setminus X\) that is null-homotopic in \(V\) is also null-homotopic in
\(U\setminus X\).

Although SLC is stated using a compatible metric, it is independent of this
choice. Indeed, let \(d_1\) and \(d_2\) be compatible metrics and suppose that
SLC holds for \(d_1\). Given a neighborhood \(U\) of \(X\), choose, as in the
definition, a neighborhood \(V_0\subset U\) and \(\varepsilon_1>0\) for
\(d_1\). Since \(X\) is compact and the ambient manifold is locally compact,
there is a neighborhood \(V\) of \(X\) such that
\(\overline V\) is compact and \(\overline V\subset V_0\). Uniform continuity
of the identity map
\((\overline V,d_2)\to(\overline V,d_1)\) gives \(\varepsilon_2>0\) such that
every subset of \(\overline V\) of \(d_2\)-diameter less than
\(\varepsilon_2\) has \(d_1\)-diameter less than \(\varepsilon_1\). Hence every
loop in \(V\setminus X\) of \(d_2\)-diameter less than \(\varepsilon_2\) is
null-homotopic in \(U\setminus X\). Thus SLC for \(d_1\) implies SLC for
\(d_2\); interchanging the metrics proves equivalence.

Consequently, SLC is invariant under homeomorphism. Indeed, after pulling back a
compatible metric by the homeomorphism, small loops and null-homotopies in the
complement correspond exactly. The preceding metric independence then gives the
claim for arbitrary compatible metrics.

These loop conditions are closely related. For a proper compact subset
\(X\subset S^n\), choose \(p\in S^n\setminus X\) and use stereographic
projection to regard \(X\) as a compact subset of \(\mathbb R^n\); all three
loop conditions, as well as covering dimension, are preserved. Venema's
results~\cite[p.~444]{zbMATH03519593} then show that, if
\(\dim X\leq n-2\), ILC is equivalent to the small loops condition. Moreover,
if \(X\) has the shape of a point, then ILC is equivalent to the cellularity
criterion.

\begin{proposition}\label{lower dim}
Let \(n=4\) or \(5\), and let \((M^n,g)\) be a connected, simply connected,
complete, open, locally conformally flat manifold with nonnegative scalar
curvature.
Assume
\[
\widetilde H_{n-1}(M;\mathbb Z)
=
\widetilde H_{n-2}(M;\mathbb Z)
=
0.
\]
Let \(\Phi\colon M\to S^n\) be the developing map, set
\(\Omega:=\Phi(M)\), and let
\(\Lambda:=S^n\setminus\Omega\). If \(\Lambda\) is movable and \(\Omega\)
has the diameter-controlled loop-filling property with some constant
\(c\geq1\), then \(\Lambda\) is cellular in \(S^n\). Consequently,
\(M\) is homeomorphic to \(\mathbb R^n\).
\end{proposition}

The strategy of the proof is to combine the geometric and homological assumptions with movability to show that $\Lambda$ has the shape of a point. Using the diameter-controlled loop-filling property, we then verify the cellularity criterion. The conclusion follows from the $4$- and $5$-dimensional cellularity theorems.

\begin{proof}
By Theorem~\ref{Liouville theorem}, the developing map is injective and
therefore identifies \(M\) conformally and diffeomorphically with the domain
\(\Omega\subset S^n\), and \(\partial\Omega\) has zero Newtonian
\(2\)-capacity. Since \(M\) is noncompact, \(\Omega\neq S^n\), and
\(\Lambda:=S^n\setminus\Omega\) is a nonempty compact subset of \(S^n\).

We first show that \(\Lambda=\partial\Omega\). The capacity--dimension
implication above gives
\(\dim_{\mathcal H}(\partial\Omega)\leq n-2\). Hence the covering dimension
of \(\partial\Omega\) is at most \(n-2\), and therefore
\(\widetilde{\check H}^{\,n-1}(\partial\Omega;\mathbb Z)=0.\) By
\v{C}ech--Alexander duality,
\(\widetilde H_0(S^n\setminus\partial\Omega;\mathbb Z)=0.\) Thus
\(S^n\setminus\partial\Omega\) is connected. But
\(S^n\setminus\partial\Omega =\Omega\sqcup\operatorname{int}(\Lambda).\)
Since \(\Omega\neq\varnothing\), it follows that
\(\operatorname{int}(\Lambda)=\varnothing\), and hence
\(\Lambda=\partial\Omega\).

Choose \(x_0\in\Omega\) and \(\varepsilon>0\) such that
\(\overline{B_{g_{st}}(x_0,\varepsilon)}\subset\Omega,\) and set
\(D:=S^n\setminus\overline{B_{g_{st}}(x_0,\varepsilon)}.\)
Then \(D\) is a bounded open neighborhood of \(\Lambda\).

Transfer \(g\) to \(\Omega\) by setting \(h:=(\Phi^{-1})^*g.\) Then \(h\) is
complete and, on \(D\setminus\Lambda\),
\(h=u^{\frac{4}{n-2}}g_{st}\) for some smooth positive function \(u\). Moreover,
\(\mathrm{Sc}_h=\mathrm{Sc}_g\circ\Phi^{-1}\geq0.\)
As in the proof of Theorem~\ref{contractible}, completeness implies that \(h\)
is geodesically complete near \(\Lambda\). Applying Theorem~\ref{Ma} with
background manifold \((S^n,g_{st})\), singular set \(\Lambda\), and the
bounded open neighborhood \(D\), we obtain
\(\dim_{\mathcal H}(\Lambda)\leq\frac{n-2}{2}\).
Since covering dimension does not exceed Hausdorff dimension and is
integer-valued, \(n=4\) or \(5\) gives
\(\dim\Lambda\leq1\).

For \(j=0,1\), \v{C}ech--Alexander duality and \(\Omega\cong M\) give
\[
\widetilde{\check H}^{\,j}(\Lambda;\mathbb Z)
\cong
\widetilde H_{n-j-1}(\Omega;\mathbb Z)
\cong
\widetilde H_{n-j-1}(M;\mathbb Z)
=0.
\]
In particular, \(\Lambda\) is connected. First consider the case
\(\dim\Lambda=0\). A connected compact
\(0\)-dimensional metric space is a point. In particular, \(\Lambda\) is
cellular in \(S^n\), and
\(\Omega=S^n\setminus\Lambda \cong S^n\setminus\{\mathrm{pt}\}
\cong \mathbb R^n.\)

We may therefore assume that \(\dim\Lambda=1\). Since
\(\widetilde{\check H}^{\,1}(\Lambda;\mathbb Z)=0\),
Lemma~\ref{Cech} gives
\(\check H_1(\Lambda;\mathbb Z)=0\).

Since \(\Lambda\) is a connected one-dimensional compact metric space, it is
a curve. In the terminology used by Borsuk--Dydak, the first Betti number is
the rank of the first integral homology group in the sense of Vietoris or
\v{C}ech~\cite[p.~163]{zbMATH03707364}. Thus Lemma~\ref{Cech} shows that the
first Betti number of \(\Lambda\) is zero. The consequence of Trybulec's
theorem recorded in~\cite[Theorem~8.2 and the ensuing paragraph,
p.~180]{zbMATH03707364} states that the shape of a movable curve is determined
by this number. The interval \([0,1]\) is a
movable curve with first Betti number \(0\), and it has the shape of a
point. Therefore
\(\operatorname{Sh}(\Lambda)
=
\operatorname{Sh}([0,1])
=
\operatorname{Sh}(*)\).
Thus \(\Lambda\) has the shape of a point. Since \(\Lambda\) is a nonempty
finite-dimensional compact metric space, Lacher's theorem
\cite[Theorem~1.1]{zbMATH03290989} implies that \(\Lambda\) is cell-like and
that its embedding in \(S^n\) has Property \(UV^\infty\).

We next prove that \(\Lambda\) satisfies the small loops condition
\(\mathrm{SLC}\). Let \(U\) be an arbitrary open neighborhood of
\(\Lambda\) in \(S^n\). Since \(\Lambda\subsetneq S^n\), choose proper
open neighborhoods \(U_0\) and \(V\) of \(\Lambda\) such that
\(\overline V\subset U_0\) and \(\overline{U_0}\subset U\). Set
\(\eta:=d_{g_{st}}\bigl(\overline V,S^n\setminus U_0\bigr)>0\)
and \(\varepsilon:=\frac{\eta}{c}\).
Let \(\gamma\colon S^1\to V\setminus\Lambda\) be any loop satisfying
\(\operatorname{diam}_{g_{st}}\bigl(\gamma(S^1)\bigr)<\varepsilon\). Since
\(V\setminus\Lambda\subset\Omega\), the loop \(\gamma\) lies in \(\Omega\).
By the diameter-controlled loop-filling property, \(\gamma\) extends to a map
\(G\colon B^2\to\Omega\) such that
\[
    \operatorname{diam}_{g_{st}}\bigl(G(B^2)\bigr)
    \le c\,\operatorname{diam}_{g_{st}}\bigl(\gamma(S^1)\bigr)
    <\eta.
\]
Because \(G(B^2)\) contains \(\gamma(S^1)\subset\overline V\), every point
of \(G(B^2)\) lies within distance \(<\eta\) of \(\overline V\). By the
definition of \(\eta\), this implies
\(G(B^2)\cap(S^n\setminus U_0)=\varnothing\). Hence
\(G(B^2)\subset U_0\setminus\Lambda\subset U\setminus\Lambda\).
Thus every loop of sufficiently small diameter in \(V\setminus\Lambda\)
bounds a disk in \(U\setminus\Lambda\). This is precisely the small loops
condition.

Because \(\dim\Lambda\le 1\) and \(n\in\{4,5\}\), we have
\(\dim\Lambda\le n-2\). By Venema~\cite[p.~444]{zbMATH03519593}, for a compact metric space \(X\subset N^n\) with \(\dim X\le n-2\), the small loops condition is equivalent to the inessential loops condition. Therefore \(\Lambda\)
satisfies \(\mathrm{ILC}\).

Because \(\Lambda\) has the shape of a point and satisfies \(\mathrm{ILC}\),
Venema's equivalence of loop conditions implies that \(\Lambda\) satisfies
the cellularity criterion \cite[p.~444]{zbMATH03519593}.

We now apply the appropriate cellularity theorem by dimension.

If \(n=4\), then Repov\v{s}'s \(4\)-dimensional cellularity criterion~\cite[Theorem, p.~564]{zbMATH04011439} applies:
a compact metric space with Property \(UV^\infty\) in the interior of a topological \(4\)-manifold
is cellular if and only if it satisfies the cellularity criterion. Since \(\Lambda\) has Property \(UV^\infty\) and satisfies the cellularity criterion, it follows that \(\Lambda\) is cellular in \(S^4\).

If \(n=5\), then \(S^5\) is a PL \(5\)-manifold. The cell-like form of
McMillan's cellularity criterion
\cite[Theorem~3.2.3, p.~107]{zbMATH05624618}, originating in
\cite[Theorem~1, p.~327]{zbMATH03190942}, states that a compact cell-like subset in
the interior of a PL \(n\)-manifold, \(n\ge 5\), is cellular if and only if it
satisfies the cellularity criterion. Therefore \(\Lambda\) is cellular in
\(S^5\).

Thus, in both cases, \(\Lambda\subset S^n\) is cellular. We now apply Morton
Brown's cellular quotient theorem, in the form proved by
Uspenskij~\cite[Theorem~1.3]{zbMATH05002651}. If \(K\subset S^n\) is
cellular, then the quotient space \(S^n/K\) is homeomorphic to \(S^n\).
Applying this to \(K=\Lambda\), we obtain
\(S^n/\Lambda\cong S^n\).
The quotient map
\(\mathfrak q\colon S^n\to S^n/\Lambda\)
collapses exactly \(\Lambda\) to a single point. Since
\(S^n\setminus\Lambda\) is saturated and open, the quotient-map criterion
shows that its restriction is a homeomorphism
\(S^n\setminus\Lambda \cong
(S^n/\Lambda)\setminus\{\mathfrak q(\Lambda)\}.\) Since
\(S^n/\Lambda\cong S^n\), this gives
\(S^n\setminus\Lambda \cong S^n\setminus\{\mathrm{pt}\}
\cong \mathbb R^n.\)
Since \(\Omega=S^n\setminus\Lambda\), we have
\(M\cong\Omega\cong\mathbb R^n\).
\end{proof}

We record the following elementary relation between \v{C}ech cohomology and
\v{C}ech--Vietoris homology in covering dimension at most one.
\begin{lemma}\label{Cech}
Let \(X\) be a compact metrizable space with \(\dim X\le 1\). Suppose that $\check H^{1}(X;\mathbb{Z}) = 0$. Then
$\check H_{1}(X;\mathbb{Z}) = 0,$ where $\check H_{*}$ denotes \v{C}ech--Vietoris homology.
\end{lemma}

\begin{proof}
Let \(\mathcal C\) be the set of all finite open covers of \(X\) of multiplicity
at most \(2\), ordered by refinement. Since \(\dim X\le 1\), every finite open
cover of \(X\) has a finite open refinement in \(\mathcal C\), so \(\mathcal C\)
is cofinal among all finite open covers. It is also directed: if
\(\mathcal U,\mathcal V\in\mathcal C\), then the finite open cover
\(\{U\cap V: U\in\mathcal U,\ V\in\mathcal V,\ U\cap V\neq\varnothing\}\)
has a refinement in \(\mathcal C\).
Hence, by cofinality,
\[
\check H_1(X;\mathbb Z)\cong
\varprojlim_{\mathcal U\in\mathcal C} H_1(N(\mathcal U);\mathbb Z),
\qquad
\check H^1(X;\mathbb Z)\cong
\varinjlim_{\mathcal U\in\mathcal C} H^1(N(\mathcal U);\mathbb Z).
\]

If \(\mathcal V\succeq\mathcal U\), choose a refinement projection
\(p_{\mathcal U\mathcal V}\colon N(\mathcal V)\to N(\mathcal U)\). Different
choices are contiguous, hence induce the same maps on homology and cohomology;
similarly, if \(\mathcal W\succeq\mathcal V\succeq\mathcal U\), then
\(p_{\mathcal U\mathcal V}\circ p_{\mathcal V\mathcal W}\) and
\(p_{\mathcal U\mathcal W}\) are contiguous. Thus the induced maps form the
inverse and direct systems above.

For each \(\mathcal U\in\mathcal C\), since \(\mathcal U\) has multiplicity at most \(2\), its nerve \(N(\mathcal U)\) is a finite simplicial complex of dimension at most \(1\), hence a finite graph. In particular, \(H_1(N(\mathcal U);\mathbb Z)\) is finitely generated free abelian, and \(H_0(N(\mathcal U);\mathbb Z)\) is free abelian. Therefore the universal coefficient theorem gives a natural
isomorphism
\(H^1(N(\mathcal U);\mathbb Z)\cong
\operatorname{Hom}(H_1(N(\mathcal U);\mathbb Z),\mathbb Z)\),
under which
\(p_{\mathcal U\mathcal V}^*\) corresponds to precomposition with
\((p_{\mathcal U\mathcal V})_*\). Equivalently,
\(\langle p_{\mathcal U\mathcal V}^*\alpha,y\rangle
=
\langle \alpha,(p_{\mathcal U\mathcal V})_*y\rangle\).

We prove the contrapositive. Assume \(\check H_1(X;\mathbb Z)\neq 0\). Then
there exists a nonzero compatible family
\(x=(x_{\mathcal U})_{\mathcal U\in\mathcal C}
\in
\varprojlim_{\mathcal U\in\mathcal C} H_1(N(\mathcal U);\mathbb Z)\).
Choose \(\mathcal U_0\in\mathcal C\) with \(x_{\mathcal U_0}\neq 0\). Since
\(H_1(N(\mathcal U_0);\mathbb Z)\) is free abelian, there exists
\(\phi\in H^1(N(\mathcal U_0);\mathbb Z)\) such that
\(\langle \phi,x_{\mathcal U_0}\rangle\neq0.\)

We claim that the image of \(\phi\) in
\(\varinjlim_{\mathcal U\in\mathcal C} H^1(N(\mathcal U);\mathbb Z)
\cong
\check H^1(X;\mathbb Z)\)
is nonzero. Otherwise, since the index category is directed, an element of this
filtered colimit is zero if and only if it becomes zero after passing to some refinement;
thus there exists \(\mathcal V\succeq\mathcal U_0\) such that
$p_{\mathcal U_0\mathcal V}^*(\phi)=0.$
Pairing with \(x_{\mathcal V}\) and using compatibility of \(x\), we get
\[
0
=
\langle p_{\mathcal U_0\mathcal V}^*(\phi),x_{\mathcal V}\rangle
=
\langle \phi,(p_{\mathcal U_0\mathcal V})_*(x_{\mathcal V})\rangle
=
\langle \phi,x_{\mathcal U_0}\rangle,
\]
a contradiction. Thus \(\check H_1(X;\mathbb Z)\neq0\) implies
\(\check H^1(X;\mathbb Z)\neq0\). The contrapositive proves the lemma.
\end{proof}

To replace the additional analytic or geometric hypotheses in Corollary~\ref{Contractible for}
by topological conditions, and to extend Proposition~\ref{lower dim} to higher
dimensions, we recall several notions from shape theory, beginning with the
following definition; see~\cite[Section~7]{zbMATH07436589}.

\begin{definition}[Homotopy pro-groups]
Let \((X,x_0)\) be a pointed compact metric space embedded in an absolute
neighborhood retract \(P\). Let
\(V_1\supset V_2\supset V_3\supset\cdots\supset X\)
be a nested cofinal sequence of open neighborhoods of \(X\) in \(P\). For
\(q\ge 1\), the inclusions \(V_j\hookrightarrow V_i\), \(j\ge i\), induce
homomorphisms
\(p_{ij}\colon\pi_q(V_j,x_0)\to \pi_q(V_i,x_0)\).
The \(q\)-th homotopy pro-group of \((X,x_0)\), denoted
$\operatorname{pro}\text{-}\pi_q(X,x_0),$
is the inverse system
$\bigl\{\pi_q(V_i,x_0),p_{ij}\bigr\}_{j\ge i},$
viewed as an object of the pro-category \(\operatorname{pro}\text{-}\mathbf{Grp}\).
\end{definition}

Since open subsets of an ANR are ANRs, the \(V_i\) form an associated
ANR-neighborhood system for \(X\). The resulting pro-isomorphism class is
independent of the chosen cofinal neighborhood system. Indeed, if
\(\{U_i\}\) and \(\{U'_a\}\) are two nested cofinal systems in the same ambient
ANR, then cofinality gives, after passing to subsequences, alternating
inclusions
\[
U_{i_1}\supset U'_{a_1}\supset U_{i_2}\supset U'_{a_2}\supset\cdots\supset X,
\]
which induce the usual pro-isomorphism between the two inverse systems.

Independence of the ambient ANR is part of the standard well-definedness of
shape invariants~\cite[Sections~2.2 and~7]{zbMATH07436589}. Any two associated ANR-neighborhood systems for \(X\), even
arising from different ANR embeddings of \(X\), are pro-homotopy equivalent.
Consequently, their homotopy pro-groups are pro-isomorphic. Thus
\(\operatorname{pro}\text{-}\pi_q(X,x_0)\) is an invariant of the pointed shape
of \((X,x_0)\).

We write $\operatorname{pro}\text{-}\pi_q(X,x_0)=0$
if this inverse system is pro-isomorphic to the trivial system. For a nested
sequential representative, this is equivalent to saying that for every \(i\)
there exists \(j\ge i\) such that the bonding homomorphism $\pi_q(V_j,x_0)\to \pi_q(V_i,x_0)$ is the zero homomorphism.

For a contractible open $n$-manifold $(n\geq 3)$, simple connectivity at infinity is the
classical end condition appearing in the recognition of Euclidean space.
In the present setting, the open manifold need not be contractible.
Thus we require a higher-dimensional vanishing condition on the homotopy
groups at infinity.

\begin{definition}[\(d\)-connected at infinity]
Let \(d\geq1\) be an integer, and let \(M\) be a noncompact, path-connected
manifold. We say that \(M\) is
\emph{\(d\)-connected at infinity} if \(M\) is one-ended and, for every compact
set \(C\subset M\), there exists a compact set \(K\subset M\), with
\(C\subset K\), such that every map
$S^q\to M\setminus K$
extends to a map $B^{q+1}\to M\setminus C$
for every \(1\le q\le d\).
\end{definition}

With these definitions, we obtain the following Euclidean rigidity theorem
when \(n\geq4\).

\begin{theorem}\label{d-connect}
Let \((M^n,g)\), \(n\geq4\), be a complete, open, simply connected, locally
conformally flat \(n\)-manifold with nonnegative scalar curvature. Let
\(\Phi\colon M\to S^n\) be the developing map, set
\(\Omega:=\Phi(M)\), and let \(\Lambda:=S^n\setminus\Omega\).
Assume that \(M\) is \(\left\lfloor \frac{n-2}{2}\right\rfloor\)-connected at infinity and that \(\Lambda\) satisfies the small loops condition in \(S^n\). Then
\(M\) is homeomorphic to \(\mathbb{R}^n\).
\end{theorem}

The strategy of the proof closely follows that of Proposition~\ref{lower dim}. First, it utilizes \(\left\lfloor \frac{n-2}{2}\right\rfloor\)-connectedness at infinity and the dimension bounds to establish that the neighborhood system around the compact set $\Lambda$ exhibits vanishing homotopy pro-groups. This vanishing permits the use of the shape-theoretic Whitehead theorem to deduce that $\Lambda$ is cell-like. Finally, by verifying the cellularity criterion, we guarantee that collapsing $\Lambda$ to a point yields a space homeomorphic to the ambient sphere $S^n$; consequently, its complement $\Omega$ is forced to be homeomorphic to $\mathbb R^n$.

\begin{proof}
Theorem~\ref{Liouville theorem} implies that the developing map identifies
\(M\) conformally and diffeomorphically with the domain \(\Omega\subset S^n\),
and \(\partial\Omega\) has zero Newtonian \(2\)-capacity.
Since \(M\) is noncompact, \(\Omega\neq S^n\), so
\(\Lambda:=S^n\setminus\Omega\) is a nonempty compact subset of \(S^n\).

We first show that \(\Lambda=\partial\Omega\). The capacity--dimension
implication above gives
\(\dim_{\mathcal H}(\partial\Omega)\leq n-2\), so the covering dimension of
\(\partial\Omega\) is at most \(n-2\), and hence
\(\widetilde{\check H}^{\,n-1}(\partial\Omega;\mathbb Z)=0.\) By
\v{C}ech--Alexander duality,
\(\widetilde H_0(S^n\setminus\partial\Omega;\mathbb Z)=0.\) Thus
\(S^n\setminus\partial\Omega\) is connected. Since
\(S^n\setminus\partial\Omega =\Omega\sqcup\operatorname{int}(\Lambda)\)
and \(\Omega\neq\varnothing\), we obtain
\(\operatorname{int}(\Lambda)=\varnothing\). Therefore
\(\Lambda=\partial\Omega\).

Choose \(x_0\in\Omega\) and \(\varepsilon>0\) such that
\(\overline{B_{g_{st}}(x_0,\varepsilon)}\subset\Omega,\) and set
\(D:=S^n\setminus\overline{B_{g_{st}}(x_0,\varepsilon)}.\)
Then \(D\) is a bounded open neighborhood of \(\Lambda\).

The metric \(h:=(\Phi^{-1})^*g\) on \(\Omega\) is complete and, on
\(D\setminus\Lambda\), has the form \(h=u^{\frac{4}{n-2}}g_{st}\) for some
smooth positive function \(u\). Moreover,
\(\mathrm{Sc}_h=\mathrm{Sc}_g\circ\Phi^{-1}\geq0.\)
As in the proof of Theorem~\ref{contractible}, \(h\) is geodesically complete
near \(\Lambda\). Applying Theorem~\ref{Ma} with background manifold
\((S^n,g_{st})\), singular set \(\Lambda\), and the bounded open
neighborhood \(D\) gives
\(\dim_{\mathcal H}(\Lambda)\leq\frac{n-2}{2}\).
Since topological dimension is bounded above by Hausdorff dimension, we obtain
\(\dim(\Lambda)\leq\left\lfloor\frac{n-2}{2}\right\rfloor=:d\).

Since \(\Phi\colon M\to\Omega\) is a diffeomorphism, the assumed
\(d\)-connectedness at infinity of \(M\) transfers to \(\Omega\). In
particular, \(\Omega\) is noncompact, path-connected, and one-ended. We claim
that \(\Lambda\) is connected. For a space \(A\), let \(\mathcal C(A)\) denote
its set of connected components. Since
\(\dim\Lambda\leq d=\lfloor(n-2)/2\rfloor\) and \(n\geq4\),
we have \(d\leq n-3\). Let \(W\subset S^n\) be connected and open. If
\(W=S^n\), then \(W\setminus\Lambda=\Omega\) is connected. Suppose that
\(W\neq S^n\). Choose \(a\in S^n\setminus W\) and identify
\(S^n\setminus\{a\}\) with \(\mathbb R^n\). Then \(W\) is a region in
\(\mathbb R^n\), and
\(\dim(\Lambda\cap W)\leq\dim\Lambda\leq d\leq n-3\).
Corollary~1 to Hurewicz--Wallman's non-separation theorem
\cite[p.~48, Corollary~1 to Theorem~IV.4]{MR0006493} therefore gives that
\(W\setminus\Lambda=W\setminus(\Lambda\cap W)\)
is connected.

Let \(U\) be an open neighborhood of \(\Lambda\). Because \(U\) is locally
connected, every component \(W\) of \(U\) is open. Moreover,
\(W\setminus\Lambda\neq\varnothing\): otherwise the nonempty open subset
\(W\) of \(S^n\) would be contained in \(\Lambda\), contrary to
\(\dim\Lambda<n\). The preceding non-separation argument shows that
\(W\setminus\Lambda\) is connected. It follows that the inclusion
\(U\setminus\Lambda\hookrightarrow U\) induces a natural bijection
\[
\mathcal C(U\setminus\Lambda)\longrightarrow\mathcal C(U).
\tag{\(*\)}
\]
These bijections commute with the bonding maps arising from inclusions of
neighborhoods.

Let \(\mathcal N_0\) be the set, ordered by reverse inclusion, of open
neighborhoods \(U\) of \(\Lambda\) for which every component of \(U\)
meets \(\Lambda\). This family is cofinal: for any open neighborhood \(U\)
of \(\Lambda\), the union \(U'\) of the components of \(U\) that meet
\(\Lambda\) belongs to \(\mathcal N_0\) and satisfies \(U'\subset U\).
Applying the same construction to \(U_1\cap U_2\), where
\(U_1,U_2\in\mathcal N_0\), shows that \(\mathcal N_0\) is directed. If
\(U\in\mathcal N_0\) and \(W\) is a component of \(U\), then the closure in
\(S^n\) of \(W\setminus\Lambda\) meets \(W\cap\Lambda\), because
\(\operatorname{int}_{S^n}\Lambda=\varnothing\) and hence
\(W\setminus\Lambda\) is dense in \(W\). Thus
\(W\setminus\Lambda\) is not relatively compact in \(\Omega\); in the
terminology of ends, it is an unbounded component of the corresponding
neighborhood of infinity. Thus every component of \(U\setminus\Lambda\), for
\(U\in\mathcal N_0\), is unbounded in \(\Omega\). Guilbault's description of the
Freudenthal end set \cite[Section~3.3.1]{zbMATH07206284} uses an efficient
exhaustion \(\{K_i\}\) of \(\Omega\) by compacta and identifies an end with a
coherent choice of a component of \(\Omega\setminus K_i\). Since
\(\{K_i\}\) is cofinal among the compact subsets of \(\Omega\) and
restriction to a cofinal subsystem does not change the inverse limit, that
description, together with \((*)\), gives
natural bijections
\[
\begin{aligned}
\mathcal E(\Omega)
&\cong
\varprojlim_{K\subset\Omega\,\mathrm{compact}}
   \mathcal C(\Omega\setminus K)\\
&\cong
\varprojlim_{U\in\mathcal N_0}
   \mathcal C(U\setminus\Lambda)
\cong
\varprojlim_{U\in\mathcal N_0}\mathcal C(U).
\end{aligned}
\tag{\(**\)}
\]
Here the second bijection uses the correspondence
\(U=S^n\setminus K\) between neighborhoods of \(\Lambda\) and compact
subsets \(K\) of \(\Omega\), followed by cofinality of \(\mathcal N_0\).

For completeness, we identify the last inverse limit in \((**)\). For
\(x\in\Lambda\), let \(C_U(x)\) denote the component of \(U\) containing
\(x\). The coherent family \((C_U(x))_{U\in\mathcal N_0}\) depends only on
the connected component of \(x\) in \(\Lambda\). If \(x,y\in\Lambda\) lie
in distinct components, then, because components and quasi-components agree
in a compact Hausdorff space, a clopen partition of \(\Lambda\) separates
\(x\) from \(y\). Disjoint open neighborhoods in \(S^n\) of the two compact
parts of that partition, followed by passage to \(\mathcal N_0\), give a
neighborhood \(U\) for which \(C_U(x)\neq C_U(y)\). Thus the induced map
\(\mathcal C(\Lambda)\to
\varprojlim_{U\in\mathcal N_0}\mathcal C(U)\)
is injective.

To prove surjectivity, let \((C_U)_{U\in\mathcal N_0}\) be a coherent
element of the inverse limit. Each
\(A_U:=C_U\cap\Lambda\) is a nonempty clopen, hence compact, subset of
\(\Lambda\). The family \(\{A_U\}_{U\in\mathcal N_0}\) has the finite
intersection property: given \(U_1,\dots,U_r\), choose
\(V\in\mathcal N_0\) with \(V\subset\bigcap_{a=1}^r U_a\); coherence gives
\(A_V\subset\bigcap_{a=1}^r A_{U_a}\). Compactness of \(\Lambda\) therefore
gives a point \(x\in\bigcap_U A_U\), and then \(C_U=C_U(x)\) for every
\(U\). Hence the map is surjective. Combining this identification with
\((**)\), the ends of \(\Omega\) correspond bijectively to the connected
components of \(\Lambda\). Since \(\Omega\) is one-ended, \(\Lambda\) is
connected.

Let \(m:=\dim\Lambda\).
By the dimension bound above, \(m\leq d=\left\lfloor\frac{n-2}{2}\right\rfloor\).
If \(m=0\), then \(\Lambda\) is a connected compact zero-dimensional space,
hence a point. Therefore $M\cong\Omega=S^n\setminus\{\mathrm{pt}\}\cong \mathbb R^n.$

Thus, for the rest of the proof, assume \(m\geq1\). We use the preceding
observation that SLC is independent of the compatible metric and invariant
under homeomorphism. Choose \(p\in S^n\setminus\Lambda\) and identify $S^n\setminus\{p\}\cong \mathbb R^n$
by stereographic projection. Let \(X\subset\mathbb R^n\) be the image of
\(\Lambda\). Since \(p\notin\Lambda\), the set \(X\) is compact and
homeomorphic to \(\Lambda\). Therefore $\dim X=\dim\Lambda,$
and \(X\) satisfies SLC in \(\mathbb R^n\).

By Hollingsworth--Rushing~\cite[Lemma~1, p.~105]{zbMATH03513016}, if \(X\subset\mathbb R^n\) is compact, satisfies SLC, and $\dim X=k\le n-3,$
then \(X\) has arbitrarily small neighborhoods \(W\) such that
\[
\pi_i(W,W\setminus X,x)=0
\qquad
(1\le i\le n-k-1,\ x\in W\setminus X).
\]
Applying this with \(k=m\) and using the dimensional hypothesis together with $n\geq 4$, we have
$m\le d=\left\lfloor\frac{n-2}{2}\right\rfloor\le n-3,$
and hence we obtain arbitrarily small neighborhoods \(W\) of \(X\) satisfying
\[
\pi_i(W,W\setminus X,x)=0
\qquad
(1\le i\le n-m-1,\ x\in W\setminus X).
\]

Pull these neighborhoods back to \(S^n\setminus\{p\}\). We now choose a nested
cofinal sequence of open neighborhoods of \(\Lambda\),
\(V_1\supset V_2\supset V_3\supset\cdots\supset \Lambda\),
as follows. Let \(V_0=S^n\setminus\{p\}\). Inductively choose a
Hollingsworth--Rushing neighborhood \(W_i\) of \(\Lambda\), small enough that
$W_i\subset V_{i-1}$
and small enough to lie in a prescribed neighborhood basis of \(\Lambda\). Let
\(V_i\) be the connected component of \(W_i\) containing \(\Lambda\). This
component is well defined because \(\Lambda\) is connected, and it is open
because \(W_i\) is an open subset of a manifold.

Set $N_i:=V_i\setminus\Lambda.$ The relative vanishing passes to the component pair. Fix
\(x\in N_i=V_i\setminus \Lambda\) and \(1\le j\le d+1\). Let
$[f]\in \pi_j(V_i,N_i,x)$ be represented by a based relative map
\(f\colon(B^j,S^{j-1},*)\to (V_i,N_i,x).\)
Composing with the inclusion of pairs $(V_i,N_i)\hookrightarrow (W_i,W_i\setminus \Lambda),$ we regard \(f\) as a class in
$\pi_j(W_i,W_i\setminus\Lambda,x).$
By the Hollingsworth--Rushing vanishing for \(W_i\), this class is trivial.
Hence there is a based relative null-homotopy
\(H\colon B^j\times I\to W_i\) such that
\[
\begin{aligned}
H(\,\cdot\,,0)&=f,
& H(*,t)&=x &&(t\in I),\\
H(S^{j-1}\times I)&\subset W_i\setminus\Lambda,
& H(B^j\times\{1\})&\subset W_i\setminus\Lambda.
\end{aligned}
\]
Since \(B^j \times I\) is connected and \(H(B^j \times I)\) meets \(V_i\), for instance at the basepoint \(x\), the image \(H(B^j \times I)\) is a connected subset of \(W_i\) meeting the component \(V_i\). Therefore $H(B^j\times I)\subset V_i.$
Consequently
\(H(S^{j-1}\times I)\subset V_i\setminus\Lambda=N_i\)
and
\(H(B^j\times\{1\})\subset V_i\setminus\Lambda=N_i\).
Thus the original class \([f]\) is already trivial in
\(\pi_j(V_i,N_i,x)\). Hence this relative homotopy object is trivial for
\(1\leq j\leq d+1\); for \(j=1\), this means that the relative homotopy
pointed set is a singleton. With this convention, we write
\[
\pi_j(V_i,N_i,x)=0
\qquad
(1\le j\le d+1,\ x\in N_i).
\]
Here we use $d+1\le n-m-1,$ which follows from \(m\le d=\lfloor(n-2)/2\rfloor\).

Each \(V_i\) is a connected open subset of \(S^n\). Since
\(\dim(\Lambda\cap V_i)\leq d\leq n-3\), the same non-separation theorem
used above shows that
\(N_i=V_i\setminus\Lambda\) is connected. It is an open subset of a
manifold and hence locally path connected; consequently, \(N_i\) is path
connected.

The sequence \((N_i)\) is cofinal among neighborhoods of infinity in \(\Omega\).
Indeed, if \(C\subset \Omega\) is compact, then $S^n\setminus C$
is a neighborhood of \(\Lambda\) in \(S^n\). Hence, for all sufficiently large
\(i\), $V_i\subset S^n\setminus C,$
and therefore $N_i=V_i\setminus\Lambda\subset \Omega\setminus C.$

We now show directly that the compact-set definition of \(d\)-connectedness at
infinity gives the based pro-triviality needed below. Choose a proper ray
\(r\colon[0,\infty)\to\Omega\); such rays exist in every connected open
manifold~\cite[Section~3.3.3]{zbMATH07206284}.
Since each \(N_i\) is a neighborhood of infinity, choose
\(0<t_1<t_2<\cdots\), with \(t_i\to\infty\), such that
\(r([t_i,\infty))\subset N_i\), and set \(x_i:=r(t_i)\in N_i\).
The path segment \(r|_{[t_i,t_j]}\) from \(x_i\) to \(x_j\) lies in \(N_i\)
whenever \(j\ge i\),
and gives the usual change-of-basepoint path for the based end tower.

Fix \(i\). Since \(N_i\) is a neighborhood of infinity,
$C_i:=\Omega \setminus N_i=S^n\setminus V_i$
is compact. By \(d\)-connectedness at infinity, there exists a compact set
$K_i\subset \Omega,$ $C_i\subset K_i,$
such that every map $S^q\to \Omega\setminus K_i$
extends to a map
$B^{q+1}\to \Omega\setminus C_i=N_i$ for every \(1\le q\le d\).

Since the sequence \((N_i)\) is cofinal among neighborhoods of infinity, choose
\(j\ge i\) such that $N_j\subset \Omega\setminus K_i.$
We claim that, for every \(1\le q\le d\), the bonding homomorphism
$\pi_q(N_j,x_j)\to \pi_q(N_i,x_i)$ is the zero homomorphism.

Indeed, let $[f]\in \pi_q(N_j,x_j)$
be represented by a based map
$f\colon S^q\to N_j.$ Since $N_j\subset \Omega\setminus K_i,$
the compact-set hypothesis gives an extension \(F\colon B^{q+1}\to N_i\)
of \(f\). Hence the image of \([f]\) in \(\pi_q(N_i,x_j)\) is trivial. The
bonding map to \(\pi_q(N_i,x_i)\) is obtained by inclusion followed by the
standard change-of-basepoint isomorphism along the ray inside \(N_i\), and
change of basepoint sends the zero element to the zero element. Therefore
$\pi_q(N_j,x_j)\to \pi_q(N_i,x_i)$ is zero.

Thus, for every \(i\), there exists \(j\ge i\) such that the bonding map to the
\(i\)-th term is zero. Consequently, for every \(1\le q\le d\), the tower
\(\{\pi_q(N_i,x_i)\}_{i\ge1}\) is pro-trivial.

Next consider the based long exact homotopy sequence of the pair $(V_i,N_i,x_i).$
Since
\[
\pi_j(V_i,N_i,x_i)=0
\qquad
(1\le j\le d+1),
\]
exactness gives, for each \(1\le q\le d\), an isomorphism induced by inclusion,
\[
\pi_q(N_i,x_i)\xrightarrow{\cong}\pi_q(V_i,x_i).
\]

These isomorphisms are compatible with the based bonding maps. Indeed, if
\(j\ge i\), then the inclusion of pairs $(V_j,N_j)\hookrightarrow (V_i,N_i)$
induces a morphism of the corresponding long exact homotopy sequences. The
bonding maps are obtained from these inclusions together with change of
basepoint along the ray segment from \(x_i\) to \(x_j\). Since this segment lies
in \(N_i\subset V_i\), the change-of-basepoint maps for the \(N\)-tower and the
\(V\)-tower commute with the maps induced by \(N_\ell\hookrightarrow V_\ell\).
Thus the levelwise isomorphisms above define an isomorphism of pro-systems
\[
\{\pi_q(N_i,x_i)\}\cong \{\pi_q(V_i,x_i)\}
\qquad
(1\le q\le d).
\]
Since the first pro-system is pro-trivial, it follows that, for every
\(1\le q\le d\), the tower \(\{\pi_q(V_i,x_i)\}_{i\ge1}\) is pro-trivial.

Choose a point \(*\in\Lambda\). Since each \(V_i\) is path connected, choose
a path \(\beta_i\) in \(V_i\) from \(*\) to \(x_i\). For \(j\geq i\), let
\(\alpha_{ij}=r|_{[t_i,t_j]}\), regarded as a path in \(N_i\subset V_i\)
from \(x_i\) to \(x_j\). After changing basepoints by \(\beta_i\) and
\(\beta_j\), the fixed-basepoint bonding homomorphism differs from the
ray-based bonding homomorphism by the automorphism determined by the loop
\(\beta_i*\alpha_{ij}*\overline{\beta_j}\) in \(V_i\): this is an inner
automorphism for \(q=1\) and the standard \(\pi_1(V_i,*)\)-action for
\(q\geq2\). In either case the identity element is fixed. Hence whenever the
ray-based bonding homomorphism is zero, so is the corresponding
fixed-basepoint bonding homomorphism. Therefore, for every \(1\le q\le d\),
the tower
\(\{\pi_q(V_i,*)\}_{i\ge1}\) is pro-trivial.

The nested ANR-neighborhood system \((V_i,*)\) is an associated pointed
ANR-neighborhood system for \((\Lambda,*)\). Hence the inverse system
$\left\{\pi_q(V_i,*),(V_j\hookrightarrow V_i)_\#\right\}_{j\ge i}$
represents
$\operatorname{pro}\text{-}\pi_q(\Lambda,*).$
Since this representative system is pro-trivial, we obtain
\[
\operatorname{pro}\text{-}\pi_q(\Lambda,*)=0
\qquad
(1\le q\le d).
\]
Since
\(m=\dim\Lambda\le d\), we have
\[
\operatorname{pro}\text{-}\pi_q(\Lambda,*)=0
\qquad
(1\le q\le m).
\]

We now apply Morita's finite-dimensional Whitehead theorem
\cite[Theorem~1.2, p.~394]{zbMATH03497097}. In the form needed here, it says that if
\(f\colon(X,x_0)\to(Y,y_0)\) is a shape morphism of pointed connected finite-dimensional topological spaces, and if the induced morphism of homotopy pro-groups
\[
\operatorname{pro}\text{-}\pi_k(f)\colon
\operatorname{pro}\text{-}\pi_k(X,x_0)
\longrightarrow
\operatorname{pro}\text{-}\pi_k(Y,y_0)
\]
is an isomorphism for every \(1\leq k\leq r\), where
\(r:=\max\{\dim X,\dim Y\}\),
then \(f\) is a shape equivalence.

The ordinary constant map
\(c\colon(\Lambda,*)\to (\{*\},*)\)
induces a pointed shape morphism. Indeed, every continuous map between pointed compact metric spaces induces a morphism in the pointed shape category. Concretely, if
\(\{V_i\}\) is an associated pointed ANR-neighborhood system for
\((\Lambda,*)\), then the constant maps $V_i\to \{*\}$
are compatible with the bonding maps and represent the induced shape morphism.

We apply Morita's theorem to this morphism. The spaces \((\Lambda,*)\) and
\((\{*\},*)\) are pointed, connected, and finite-dimensional. Moreover,
$\max\{\dim\Lambda,\dim\{*\}\}=m.$
For every \(1\le k\le m\), we have already shown that $\operatorname{pro}\text{-}\pi_k(\Lambda,*)=0,$
and clearly $\operatorname{pro}\text{-}\pi_k(\{*\},*)=0.$
Therefore
\[
\operatorname{pro}\text{-}\pi_k(c)\colon
\operatorname{pro}\text{-}\pi_k(\Lambda,*)
\longrightarrow
\operatorname{pro}\text{-}\pi_k(\{*\},*)
\]
is an isomorphism for every \(1\le k\le m\). By Morita's theorem, \(c\) is a
shape equivalence. Hence \(\operatorname{Sh}(\Lambda)=\operatorname{Sh}(*)\).

The set \(\Lambda\) is a nonempty finite-dimensional compact metric space.
Lacher's theorem~\cite[Theorem~1.1]{zbMATH03290989} now implies both that
\(\Lambda\) is cell-like and that its embedding
\(\Lambda\hookrightarrow S^n\) has Property \(UV^\infty\). Explicitly, for
every neighborhood \(U\) of \(\Lambda\) in \(S^n\), there is a neighborhood
\(V\subset U\) of \(\Lambda\) such that the inclusion
\(V\hookrightarrow U\) is null-homotopic.

It remains to prove cellularity. Since \(n\ge4\), we have
$d=\left\lfloor\frac{n-2}{2}\right\rfloor\ge1.$
Thus the preceding pro-triviality includes the pro-fundamental group at
infinity.

Let \(U\) be any neighborhood of \(\Lambda\) in \(S^n\). Choose \(i\) such that
$V_i\subset U.$
From the pro-triviality of the tower \(\{\pi_1(N_i,x_i)\}\), choose \(j\ge i\)
such that the bonding homomorphism
$\pi_1(N_j,x_j)\to \pi_1(N_i,x_i)$ is trivial. Put $V:=V_j.$
Let $\gamma\colon S^1\to V\setminus\Lambda=N_j$
be any loop. Since \(N_j\) is path connected, join the basepoint of \(\gamma\)
to \(x_j\) by a path in \(N_j\). This conjugates \(\gamma\) to a loop based at
\(x_j\). Its image in \(\pi_1(N_i,x_i)\) is trivial, so the conjugate loop, and
hence \(\gamma\), is null-homotopic in $N_i\subset U\setminus\Lambda.$

Therefore, for every neighborhood \(U\) of \(\Lambda\), there exists a smaller
neighborhood \(V\subset U\) such that every loop in
$V\setminus\Lambda$ is null-homotopic in $U\setminus\Lambda.$
This is the cellularity criterion.

The remainder is exactly as in the proof of Proposition~\ref{lower dim}. Namely,
the cellularity criterion obtained above, together with Property
\(UV^\infty\) and the cell-like conclusion, implies that \(\Lambda\) is
cellular in \(S^n\): for \(n=4\) this
follows from Repov\v{s}'s theorem~\cite[Theorem, p.~564]{zbMATH04011439}, and for \(n\ge 5\)
from the cell-like form of McMillan's cellularity criterion
\cite[Theorem~3.2.3, p.~107]{zbMATH05624618}, originating in
\cite[Theorem~1, p.~327]{zbMATH03190942}.

Finally, collapse \(\Lambda\) to a point:
\(\mathfrak q\colon S^n\to S^n/\Lambda\).
Since \(\Lambda\) is cellular in \(S^n\), Uspenskij's form of Brown's
cellular quotient theorem~\cite[Theorem~1.3]{zbMATH05002651} shows that the
quotient space \(S^n/\Lambda\) is homeomorphic to \(S^n\). The complement
\(S^n\setminus\Lambda\) is saturated and open for \(\mathfrak q\), so the
quotient-map criterion shows that \(\mathfrak q\) restricts to a homeomorphism
\(S^n\setminus\Lambda \cong
(S^n/\Lambda)\setminus\{\mathfrak q(\Lambda)\}.\)
Therefore
\[
M\cong \Omega=S^n\setminus\Lambda
\cong
(S^n/\Lambda)\setminus\{\mathfrak q(\Lambda)\}
\cong
S^n\setminus\{\mathrm{pt}\}
\cong
\mathbb R^n .
\]
\end{proof}

\begin{proof}[Proof of Theorem~\ref{thm:intro-E}]
For \(n=3\), a connected open \(3\)-manifold \(M^3\) with
\(\pi_1(M)=0\) and \(H_2(M;\mathbb Z)=0\) is contractible. Indeed,
\(\pi_1(M)=0\) implies that \(M\) is orientable and that
\(H_1(M;\mathbb Z)=0\). Poincar\'e duality for the connected noncompact
oriented \(3\)-manifold \(M\) gives
\(H_3(M;\mathbb Z)\cong H_c^0(M;\mathbb Z)=0\).
Moreover, a smooth \(3\)-manifold has the homotopy type of a
\(3\)-dimensional CW complex, so \(H_i(M;\mathbb Z)=0\) for all \(i>3\).
Together with the assumption
\(H_2(M;\mathbb Z)=0\), this gives
\(\widetilde H_i(M;\mathbb Z)=0\) for all \(i\geq0\). If some
\(\pi_k(M)\) were nonzero, choose the least such \(k\geq2\). The Hurewicz
theorem would give \(\pi_k(M)\cong H_k(M;\mathbb Z)=0\), a contradiction.
Thus \(\pi_k(M)=0\) for all \(k\geq1\). Since \(M\) has the homotopy type of
a CW complex, Whitehead's theorem implies that
\(M\) is contractible.

Theorem~\ref{contractible} for \(n=3\) and Theorem~\ref{d-connect} for
\(n\geq4\) complete the proof.
\end{proof}

\section[Non-Euclidean contractible LCF manifolds with PSC]{\texorpdfstring{Non-Euclidean contractible LCF $n$-manifolds\\ with PSC for $n\geq 4$}{Non-Euclidean contractible LCF n-manifolds with PSC for n >= 4}}\label{counterexample psc}

The aim of this section is to show that, for $n\geq4$, the conclusion of
Theorem~\ref{thm:intro-E}(ii) need not hold if both of its additional topological
hypotheses are omitted, and that the disjunction of conditions
\textup{(i)}--\textup{(v)}
in Corollary~\ref{Contractible for} cannot be omitted, even at the level of
homeomorphism type.  These two sharpness conclusions are recorded explicitly
in Corollary~\ref{cor:sharpness} below.

The construction below does not separate the two hypotheses in
Theorem~\ref{thm:intro-E}(ii).  Indeed, since $n\geq4$,
$\lfloor(n-2)/2\rfloor\geq1$, so
$\lfloor(n-2)/2\rfloor$-connectedness at infinity implies simple
connectivity at infinity, whereas the manifold constructed below is not
simply connected at infinity.  Thus it fails the hypothesis in
Theorem~\ref{thm:intro-E}(ii) of being $\lfloor(n-2)/2\rfloor$-connected at infinity.
We do not determine whether its limit continuum $K$ satisfies the small
loops condition.  Consequently, the examples show only that the two
hypotheses in Theorem~\ref{thm:intro-E}(ii) cannot be omitted simultaneously; they do
not establish that either hypothesis can be omitted while the other is
retained.

\begin{theorem}\label{thm:counterexample-psc}
For every integer $n\geq4$, there exists a nondegenerate continuum
$K\subset S^n$ with $\dim_{\mathcal{H}}K=1$ such that, with
$M:=S^n\setminus K$, the following hold.
\begin{enumerate}
\item $M$ is a contractible smooth open manifold that is not simply connected at infinity.

\item There exists a smooth function $\phi:M\to(0,\infty)$ with
$m_0:=\inf_M\phi>0$ such that, for every $t\in(-m_0,\infty)$,
\[
g^{(t)}:=(\phi+t)^{4/(n-2)}g_{st}|_M
\]
is smooth, complete, and locally conformally flat, with
\[
\mathrm{Sc}_{g^{(t)}}=n(n-1)t(\phi+t)^{-(n+2)/(n-2)}.
\]
Hence $\mathrm{Sc}_{g^{(t)}}$ is negative, zero, or positive according as
$t<0$, $t=0$, or $t>0$, and
\[
\lim_{t\to0}\int_M|\mathrm{Sc}_{g^{(t)}}|^{n/2}\,dV_{g^{(t)}}=0.
\]
For $t>0$, the rescaled metric
\[
\bar g^{(t)}
:=t^{-4/(n-2)}g^{(t)}
=\left(1+\frac{\phi}{t}\right)^{4/(n-2)}g_{st}|_M
\]
satisfies
\[
0<\mathrm{Sc}_{\bar g^{(t)}}<n(n-1),
\]
and its scalar curvature decays uniformly to zero at $K$ in the sense
that
\[
 \lim_{\delta\downarrow0}
 \sup_{\substack{x\in M\\ d_{g_{st}}(x,K)<\delta}}
 \mathrm{Sc}_{\bar g^{(t)}}(x)=0.
\]
\end{enumerate}
\end{theorem}

The construction of the metric is inspired by Karakhanyan's existence
theorem \cite[Theorem~1.1]{zbMATH07873599}, which states that
$S^n\setminus K$ admits a complete scalar-flat metric conformal to the
round metric if and only if the corresponding spherical Bessel capacity
vanishes; see \cite[Section~2.2]{zbMATH07873599}. We first fix the notation
for this section and then outline the proof. A \emph{continuum} is a
nonempty compact connected metric space; it is
\emph{nondegenerate} if it contains more than one point.  All manifolds
and metrics are smooth unless stated otherwise.
For \(\beta>0\), \(1<p<\infty\), and \(E\subset\mathbb R^d\), the Bessel
capacity of \(E\) is
\[
\mathcal{C}_{\beta,p}^{\mathbb R^d}(E)
:=
\inf\left\{
\|f\|_{L^p(\mathbb R^d)}^p:
f\in L^p(\mathbb R^d),\ f\geq0,\
G_\beta*f\geq1\ \text{everywhere on }E
\right\},
\]
where
\(G_\beta:=\mathcal{F}^{-1}((1+|\xi|^2)^{-\beta/2})\) is the Bessel
kernel; see
\cite[\S1.2.4]{zbMATH00824833}.  The displayed infimum is the
specialization of \cite[Definition~2.3.3]{zbMATH00824833} to the
kernel \(G_\beta\).
Fix \(n\geq4\) and set
\begin{equation}\label{eq:parameters}
\alpha:=1+\frac2n,
\qquad
q:=\frac n2,
\qquad
s:=n-\alpha q=\frac{n-2}{2},
\qquad
\theta:=\frac1{q-1}=\frac2{n-2}.
\end{equation}
Then
$q\geq2,\qquad\alpha q=\frac{n+2}{2}<n.$
The estimates below repeatedly use
\begin{equation}\label{eq:driving-identities}
n-2-s=s,
\qquad
s\theta=(n-2-s)\theta=1,
\qquad
(n-2)\theta=2.
\end{equation}
Throughout this section we use the unnormalized Hausdorff convention
\[
 \mathcal{H}^t_\rho(A)
 :=\inf\left\{\sum_i(\operatorname{diam}E_i)^t:
 A\subset\bigcup_iE_i,\ \operatorname{diam}E_i\leq\rho\right\},
 \qquad
 \mathcal{H}^t(A):=\lim_{\rho\downarrow0}\mathcal{H}^t_\rho(A),
\]
where the infimum is over countable covers.  Thus no dimensional
normalizing constant is included in \(\mathcal{H}^t\).
Fix $p_\sigma\in S^n$ and a stereographic chart
$\sigma: S^n\setminus\{p_\sigma\}\to \mathbb R^n$ with pole
\(p_\sigma\), and choose a closed round ball
\(B\Subset S^n\setminus\{p_\sigma\}\).  All compact sets in the tower
construction lie in \(\operatorname{int}B\).  For \(E\subset B\), write
$\operatorname{cap}(E):=
\mathcal{C}_{\alpha,q}^{\mathbb R^n}\bigl(\sigma(E)\bigr)$, the Bessel capacity
of \(\sigma(E)\) computed in this fixed chart.
For compact $E\subset B$, the numerical value of $\operatorname{cap}(E)$
may depend on the chart,
but its vanishing does not.  Indeed, if $\widetilde\sigma$ is any
other stereographic chart defined on a neighborhood of $E$, then the
transition diffeomorphism $T:=\widetilde\sigma\circ\sigma^{-1}$ is
smooth on a neighborhood of the compact set $\sigma(E)$.  If
$E=\varnothing$, the assertion below is immediate, so assume that
$E\neq\varnothing$.  Choose $\delta_0>0$ so that the closed
$\delta_0$-neighborhood of $\sigma(E)$ is contained in the domain of
$T$, and set
\(M_0:=\sup\bigl\{\lVert DT(z)\rVert:
\operatorname{dist}(z,\sigma(E))\leq\delta_0\bigr\}<\infty\).
If $x,y\in\sigma(E)$ and $|x-y|<\delta_0$, then the segment joining
$x$ to $y$ remains in this neighborhood, so the mean-value estimate
applies.  If $|x-y|\geq\delta_0$, then
$|T(x)-T(y)|\leq\operatorname{diam}\widetilde\sigma(E)
\leq\delta_0^{-1}\operatorname{diam}\widetilde\sigma(E)\,|x-y|$.
Consequently,
\[
 |T(x)-T(y)|
 \leq
 \max\!\left\{M_0,
 \delta_0^{-1}\operatorname{diam}\widetilde\sigma(E)\right\}|x-y|,
\]
and hence $T|_{\sigma(E)}$ is Lipschitz.  The same argument applied to
$T^{-1}$ shows that $T^{-1}|_{\widetilde\sigma(E)}$ is Lipschitz.
By the Lipschitz-image theorem for Bessel capacity
\cite[Theorem~5.2.1]{zbMATH00824833}, if
$A\subset\mathbb R^n$, $\Phi:A\to\mathbb R^n$ is Lipschitz,
$\beta>0$, and $1<p<n/\beta$ (equivalently, $0<\beta p<n$), then
\[
\mathcal{C}_{\beta,p}^{\mathbb R^n}\bigl(\Phi(A)\bigr)
\leq C\mathcal{C}_{\beta,p}^{\mathbb R^n}(A)
\]
for some $C=C(n,\beta,p,\operatorname{Lip}(\Phi))<\infty$.  Since
$\alpha>0$ and
\(1<q<\frac n\alpha\)
(equivalently, \(0<\alpha q=\tfrac{n+2}{2}<n\)),
the theorem applies with $(\beta,p)=(\alpha,q)$.
The theorem is formulated for a Lipschitz map whose domain is the
arbitrary subset $A$; no extension to all of $\mathbb R^n$ is required.
Applying the estimate first to $T|_{\sigma(E)}$ and then to
$T^{-1}|_{\widetilde\sigma(E)}$ proves
\[
\mathcal{C}_{\alpha,q}^{\mathbb R^n}\bigl(\sigma(E)\bigr)=0
\quad\Longleftrightarrow\quad
\mathcal{C}_{\alpha,q}^{\mathbb R^n}\bigl(\widetilde\sigma(E)\bigr)=0.
\]
This is precisely the spherical capacity-zero condition in
\cite[Section~2.2 and the proof of Theorem~1.1]{zbMATH07873599}: after
choosing a stereographic pole in $S^n\setminus E$, that condition is
expressed by the vanishing of the Euclidean Bessel capacity of the
stereographic image.  The displayed equivalence shows that it is
independent of the allowed pole.  Thus $\operatorname{cap}(E)=0$ is
precisely the fixed-chart expression of the spherical capacity-zero
condition used above.
For noncompact subsets of \(B\), the notation \(\operatorname{cap}\)
always refers to this fixed chart; below we use only its monotonicity and
make no chart-independence assertion for such sets.

The proof is divided into five steps. Steps~1--2 are topological and
construct $K$ and the complement $M:=S^n\setminus K$, while
Steps~3--5 are analytic and construct the metrics.  Once $K$ is known
to be nonempty and compact, the analytic part uses only the
connectedness of $M$, the inclusion $K\subset\operatorname{int}B$, and
the assumption $\operatorname{cap}(K)=0$.

\noindent\emph{Step 1 (seed and tower).}
Let \(F_2=\langle a,b\rangle\).  The endomorphism
\(\omega:F_2\to F_2\) determined by
\(\omega(a)=[a,b]\) and \(\omega(b)=[a,b^{-1}]\) is
injective and induces the zero map on \(H_1(F_2;\mathbb Z)\)
(Lemma~\ref{lem:endo}).  We build nested
closed regular neighborhoods
\(V_{j+1}\Subset\operatorname{int}V_j\subset\operatorname{int}B\) of embedded rank-two
roses with a common wedge point \(x\), marked by isomorphisms
\(\mathfrak{m}_j:F_2\to\pi_1(V_j,x)\), such that
\[
(i_j)_*\circ\mathfrak{m}_{j+1}=\mathfrak{m}_j\circ\omega,
\qquad
\operatorname{cap}(V_j)<2^{-j},
\]
where \(i_j:V_{j+1}\hookrightarrow V_j\); we also choose ball covers of
\(V_j\) of small \((1+\frac1j)\)-content
(Lemma~\ref{successor} and Proposition~\ref{induction}).  The marking
constrains only the homotopy class of each successor rose, whereas the
capacity bound constrains only the open set into which its regular
neighborhood is squeezed.  Thus the two conditions can be imposed
independently.  The continuum
\(K:=\bigcap_jV_j\) satisfies \(\operatorname{cap}(K)=0\) and
\(\dim_{\mathcal{H}}K=1\) (Proposition~\ref{prop:limitset}).

\noindent\emph{Step 2 (topology of the complement).}
Since \(H_1(\omega)=0\), the inclusion
\(i_j:V_{j+1}\hookrightarrow V_j\) induces the zero homomorphism
\((i_j)^*:H^1(V_j;\mathbb Z)\to  H^1(V_{j+1};\mathbb Z)\)
for every \(j\).  Thus \v{C}ech
continuity makes \(K\) \v{C}ech-acyclic
(Proposition~\ref{prop:cech}), while natural Alexander duality for the
finite stages, followed by passage to the direct limit, makes \(M\)
acyclic (Proposition~\ref{prop:acyclic}).
General position for \(2\)-disks against the one-dimensional spines
(\(2+1-n<0\)) makes the stages
\(C_j:= S^n\setminus\operatorname{int}V_j\) simply connected, so
\(M\) is contractible (Proposition~\ref{prop:contract}).  Injectivity of
\(\omega\), on the other hand, yields loops on \(\partial V_k\),
arbitrarily far out in \(M\), that stay essential in \(V_1\); so \(M\) is
not simply connected at infinity and \(M\not\cong\mathbb R^n\)
(Proposition~\ref{prop:sci}).  The nondegeneracy of \(K\) is part of
Proposition~\ref{prop:limitset}.

\noindent\emph{Step 3 (the measure).}
Put \(F:=\sigma(K)\).  By the Hedberg--Wolff inequality,
\(\operatorname{cap}(K)=0\) forces every probability measure on \(F\) to
have infinite Wolff self-energy.  This averaged statement does not pass
pointwise to weak-\(*\) limits.  Applying Sion's minimax theorem to
smoothed truncated potentials, whose dependence on the measure is
concave because \(\theta\leq1\), and then taking a convex series, we
obtain a probability measure \(\mu\) with
\(\operatorname{supp}\mu=F\) and \(\mathcal{W}^\mu\equiv+\infty\) on \(F\)
(Theorem~\ref{thm:evans-wolff}).

\noindent\emph{Step 4 (the scalar-flat metric).}
The Newtonian potential \(u(x):=\int_F|x-y|^{2-n}\,d\mu(y)\) is harmonic
off \(F\) and, because \(s\theta=1\), identically \(+\infty\) on \(F\).
With \(U_{\mathrm{rd}}\) the round conformal factor in the chart,
\(\phi:=u/U_{\mathrm{rd}}\) extends
smoothly and positively across the pole \(p_\sigma\) by Kelvin inversion.
The probability normalization makes the inverted potential bounded there,
and \(\phi(p_\sigma)=2^{-(n-2)/2}\).  Hence
\(h:=\phi^{4/(n-2)}g_{st}|_M\) is smooth and scalar-flat, and a
dyadic annulus estimate, via \((n-2)\theta=2\), converts
\(\mathcal{W}^\mu\equiv+\infty\) into infinite length of every divergent
curve: \(h\) is complete (Lemma~\ref{lem:onechart}).

\noindent\emph{Step 5 (the shifts).}
Since \(\mu\) is a probability measure with bounded support,
\(m_0:=\inf_M\phi>0\).  For \(t\in(-m_0,\infty)\) the metric
\(g^{(t)}:=(\phi+t)^{4/(n-2)}g_{st}|_M\) is uniformly comparable
to \(h\), hence complete, and linearity of the conformal Laplacian gives
\[
\mathrm{Sc}_{g^{(t)}}=n(n-1)\,t\,(\phi+t)^{-(n+2)/(n-2)},
\]
so the sign of the scalar curvature is the sign of \(t\), and
\(\int_M|\mathrm{Sc}_{g^{(t)}}|^{n/2}\,dV_{g^{(t)}}\to0\) as
\(t\to0\); for \(t>0\) the rescaling \(t^{-4/(n-2)}g^{(t)}\) has
scalar curvature in \((0,n(n-1))\) tending to \(0\) at \(K\)
(Lemma~\ref{lem:shifts}).  One harmonic function thus yields complete
locally conformally flat metrics of all three signs.

The restriction \(n\geq4\) enters through three thresholds, each
equivalent to it: polarity of finite graphs, \(s=\frac{n-2}2\geq1\)
(Lemma~\ref{lem:thin}); general position, \(2+1-n<0\)
(Lemma~\ref{lem:sc}); and Sion concavity, \(\theta=\frac2{n-2}\leq1\)
(Theorem~\ref{thm:evans-wolff}).

\subsection{The thin marked tower}

This subsection constructs the nested marked tower and its limit
continuum \(K\), which provide the topological and capacity-theoretic
input for the metric construction carried out later in the section.
The construction is organized as follows:
Lemma~\ref{lem:endo} supplies the group-theoretic input,
Lemma~\ref{lem:thin} provides the analytic input for the tower construction,
Lemma~\ref{lem:sublevel} supplies the level-uniform regular-neighborhood
construction,
Lemma~\ref{successor} combines these inputs,
Proposition~\ref{induction} iterates the construction, and
Proposition~\ref{prop:limitset} records the properties of the limiting
continuum needed below.

Let $\Delta^{n+1}\subset\mathbb R^{n+1}$ be a regular Euclidean simplex
centered at the origin.  Radial projection
$x\mapsto x/|x|$ transports the boundary complex
$\partial\Delta^{n+1}$ to a smooth triangulation of $S^n$; fix the
compatible PL structure induced by this triangulation.  Whenever a
finite PL graph is considered, we pass to a common
adapted subdivision containing it as a subcomplex.  A \emph{rank-two rose} is
a CW graph $R=R_a\vee R_b$, where $R_a$ and $R_b$ are oriented petal
circles with a specified common wedge vertex and labels $a,b$.  Whenever
$R$ is treated as a PL graph, each petal is triangulated by at least three
simplicial edges; regular neighborhoods are taken in the sense of
\cite[\S3]{zbMATH01714503}.  Write $F_2=\langle a,b\rangle$ and
$[x,y]=xyx^{-1}y^{-1}$.

\begin{lemma}[Algebraic seed]\label{lem:endo}
The endomorphism $\omega:F_2\to F_2$,
$\omega(a)=[a,b]$, $\omega(b)=[a,b^{-1}]$, is injective and
induces the zero map on $H_1(F_2;\mathbb{Z})$.
\end{lemma}

\begin{proof}
Put $w_a:=aba^{-1}b^{-1}$ and $w_b:=ab^{-1}a^{-1}b$.  The products
$w_aw_b$ and $w_bw_a$ are freely reduced with second letters $b$ and
$b^{-1}$, respectively, so $w_aw_b\ne w_bw_a$.  Thus $w_a$ and $w_b$ do
not commute.  Since every cyclic group is abelian,
$G:=\langle w_a,w_b\rangle$ is noncyclic.
By the Nielsen--Schreier theorem, $G$ is free.  Since it is noncyclic and
generated by two elements, it has rank two.  Let
$\omega_G:F_2\twoheadrightarrow G$ be the corestriction of $\omega$,
so that $\omega_G(a)=w_a$ and $\omega_G(b)=w_b$,
and choose an isomorphism $\psi:G\xrightarrow{\sim}F_2$.  The composite
$\psi\circ\omega_G$ is a surjective endomorphism of the Hopfian group
$F_2$ and is therefore injective.  Hence $\omega_G$ is injective.  If
$\iota_G:G\hookrightarrow F_2$ denotes the inclusion, then
$\omega=\iota_G\circ\omega_G$, so $\omega$ is injective.  Finally,
$w_a=[a,b]$ and $w_b=[a,b^{-1}]$ lie in $[F_2,F_2]$, so the map
induced by $\omega$ on $H_1(F_2;\mathbb Z)$ is zero.
\end{proof}

The next lemma is the entire analytic input of the construction.
Part (a) controls Hausdorff content and gives
$\dim_{\mathcal{H}}K\le1$ in the limit; part (b) says that both
capacity conclusions hold: the graph has zero capacity, and its open
neighborhoods can be chosen with arbitrarily small capacity.  This is the
outer regularity that lets smallness be imposed on the codimension-zero
stages of the tower, not merely on their spines.

\begin{lemma}[Thin neighborhoods of finite graphs]\label{lem:thin}
Let $\Gamma_0$ be a finite embedded PL graph and let
$W$ be an open set such that
$\Gamma_0\subset W\subset\operatorname{int}B$.  Let
$\delta,\eta>0$ and $\tau>1$.  Then:
\begin{enumerate}
\item[(a)] there is a finite family $\{U_\ell\}$ of open
spherical balls with
\[
 \Gamma_0\subset U:=\bigcup_\ell U_\ell\Subset W,
 \qquad
 \operatorname{diam}U_\ell<\delta,
 \qquad
 \sum_\ell(\operatorname{diam}U_\ell)^\tau<\eta;
\]
\item[(b)] if $n\geq4$, then $\operatorname{cap}(\Gamma_0)=0$.  Moreover,
for every family satisfying part~\textup{(a)}, with
$U=\bigcup_\ell U_\ell$, and every $\varepsilon>0$, there is an open set
$O$ such that
$\Gamma_0\subset O\Subset U$ and
$\operatorname{cap}(O)<\varepsilon$.
\end{enumerate}
\end{lemma}

\begin{proof}
\emph{(a)} Let $\Gamma_0$ consist of closed edges
$e_1,\dots,e_m$ and isolated vertices $v_1,\dots,v_{m'}$. Since
$\Gamma_0$ is compact and $W$ is open, choose $\rho_W>0$ such that the
closed $\rho_W$-neighborhood of $\Gamma_0$ is contained in $W$. Let
$L_i$ be the length of $e_i$ in the round metric. For $r>0$,
parametrize $e_i$ by arclength and divide it into at most
$\lfloor L_i/r\rfloor+1$ subarcs of length less than $r$. A spherical
ball of radius $r$ centered at an endpoint of such a subarc contains
that subarc, because round distance is bounded above by arclength.
Cover each edge in this way and each isolated vertex by one open ball
of radius $r$. Denote the resulting family by
$\{U_\ell\}$ and put $U:=\bigcup_\ell U_\ell$.

Each ball has diameter at most $2r$, and
\[
 \sum_\ell (\operatorname{diam} U_\ell)^\tau
 \le \sum_{i=1}^m \left( \frac{L_i}{r} + 1 \right) (2r)^\tau
      + m'(2r)^\tau
 = 2^\tau \Bigl( \sum_{i=1}^m L_i \Bigr) r^{\tau-1}
      + 2^\tau (m+m')r^\tau.
\]
Since $\tau>1$, both exponents $\tau-1$ and $\tau$ are positive, so the
right-hand side tends to zero as $r\to0$. Choose
$r<\rho_W$ sufficiently small that $2r<\delta$ and the displayed sum is
less than $\eta$. Every closed ball in the family is then contained in
$W$; since the family is finite, $\overline U\subset W$. Thus
$U\Subset W$.

\emph{(b)} The argument proceeds in four steps. The hypothesis $n\ge4$
is used at the junction of Steps~1 and~2, where the critical dimension
$s=\frac{n-2}{2}$ from \eqref{eq:parameters} must satisfy $s\ge1$.
All spherical neighborhoods constructed below are contained in
$W\subset\operatorname{int}B$.

\smallskip\noindent\emph{Step 1 (finite critical measure for the spine).}
The restriction $\sigma|_B$ is smooth on the compact set
$B\subset S^n\setminus\{p_\sigma\}$ and is therefore Lipschitz. Since
$\Gamma_0$ is a finite one-dimensional PL graph,
$\sigma(\Gamma_0)$ is a finite union of rectifiable curves, and hence
$\mathcal{H}^1(\sigma(\Gamma_0))<\infty$. We also claim that
$\mathcal{H}^t(\sigma(\Gamma_0))=0$ for every $t>1$. For $0<\rho\le1$,
choose a countable cover $\{E_i\}$ of $\sigma(\Gamma_0)$ such that
$\operatorname{diam}E_i\le\rho$ and
$\sum_i\operatorname{diam}E_i\le
\mathcal{H}^1_\rho(\sigma(\Gamma_0))+1$.
The same cover gives
\[
\begin{aligned}
 \mathcal{H}^t_\rho(\sigma(\Gamma_0))
 &\le \sum_i(\operatorname{diam}E_i)^t \\
 &\le \rho^{t-1}\sum_i\operatorname{diam}E_i \\
 &\le \rho^{t-1}
 \bigl(\mathcal{H}^1(\sigma(\Gamma_0))+1\bigr).
\end{aligned}
\]
Letting $\rho\downarrow0$ proves the claim.  Since $s=(n-2)/2$, we obtain
$\mathcal{H}^s(\sigma(\Gamma_0))<\infty$: when $n=4$, one has $s=1$, and
this follows from the finite $\mathcal{H}^1$-measure; when $n\ge5$, one
has $s>1$, and the preceding claim gives the stronger identity
$\mathcal{H}^s(\sigma(\Gamma_0))=0$.  This is precisely
where the dimensional hypothesis $n\ge4$, which ensures $s\ge1$, is
used.

\smallskip\noindent\emph{Step 2 (polarity of the spine: checking the hypotheses of Adams--Hedberg).}
We apply the finite-critical-measure theorem \cite[Theorem~5.1.9]{zbMATH00824833}:
if $p>1$ and $0<\beta p<n$, and if $\Lambda_h(E)<\infty$ for the ball-gauge
Hausdorff measure with gauge $h(r)=r^{\,n-\beta p}$, then
$\mathcal{C}_{\beta,p}^{\mathbb R^n}(E)=0$.
Apply the theorem with
$(\beta,p)=(\alpha,q)=(1+2/n,n/2)$.  Then $p=q=n/2\ge2>1$, while
$\beta p=\alpha q=((n+2)/n)(n/2)=(n+2)/2<n$ because $n>2$.  Thus the parameters
lie in the admissible range, and the corresponding critical gauge
exponent is $n-\beta p=s$.

The cited theorem uses the ball-gauge measure $\Lambda_h$ for
$h(r)=r^s$, whereas Step~1 bounds the unnormalized Hausdorff measure
$\mathcal{H}^s$ fixed above.  For $A\subset\mathbb R^n$ and $\rho>0$, let
$\Lambda_h^{(\rho)}(A)$ denote the
infimum in the definition of $\Lambda_h(A)$ over ball covers of $A$
whose radii are at most $\rho$.  Let $\{E_i\}_{i\geq1}$ be a countable
cover of $A$ by nonempty sets with
$d_i:=\operatorname{diam}E_i\leq\rho/2$ and
$\sum_i d_i^s<\infty$, and choose $x_i\in E_i$.  Given $\zeta>0$, if
$d_i>0$, choose $d_i<r_i\leq\rho$ so close to $d_i$ that
$r_i^s<d_i^s+2^{-i}\zeta$; if $d_i=0$, choose instead
$0<r_i\leq\rho$ with $r_i^s<2^{-i}\zeta$.  Then the balls
$\{B(x_i,r_i)\}_{i\geq1}$ cover $A$ and
$\sum_{i\geq1}h(r_i)<\sum_{i\geq1}d_i^s+\zeta$, because
$\sum_{i\geq1}2^{-i}=1$.  Covers with infinite $s$-sum are
irrelevant to the infimum.  Taking infima over the covers
$\{E_i\}_{i\geq1}$, and then letting $\zeta\downarrow0$, gives
$\Lambda_h^{(\rho)}(A)\leq\mathcal{H}^s_{\rho/2}(A)$.  Letting
$\rho\downarrow0$ yields
$\Lambda_h(A)\leq\mathcal{H}^s(A)$.  Thus
Step~1 yields $\Lambda_h(\sigma(\Gamma_0))
 \le \mathcal{H}^s(\sigma(\Gamma_0))<\infty.$  The
finite-critical-measure theorem now gives
$\mathcal{C}_{\alpha,q}^{\mathbb R^n}(\sigma(\Gamma_0))=0,
 \qquad\text{or equivalently}\qquad
 \operatorname{cap}(\Gamma_0)=0.$

\smallskip\noindent\emph{Step 3 (admissible potentials of small norm).}
By the defining formula \cite[Definition~2.3.3]{zbMATH00824833}, the capacity of a
set $E$ is the infimum of $\|f\|_{L^p}^p$ over all nonnegative
$f\in L^p(\mathbb R^n)$ such that $G_\beta*f\ge1$ everywhere on $E$.
The convolution is understood as an extended nonnegative Lebesgue
integral; it is pointwise well defined and independent of the chosen
nonnegative measurable representative of $f$. Since
$\operatorname{cap}(\Gamma_0)=0$ by Step~2, there is a nonnegative
function $f_{\alpha,q}\in L^q(\mathbb R^n)$ such that
\[
G_\alpha*f_{\alpha,q}\ge1\quad\text{on }\sigma(\Gamma_0),
\qquad
\|f_{\alpha,q}\|_{L^q}^q<2^{-q}\varepsilon.
\]
The scale factor $2^{-q}$ is reserved to offset the doubling operation
in Step~4.

\smallskip\noindent\emph{Step 4 (outer regularity: from the spine to an open set).}
The Bessel kernel $G_\alpha$ is nonnegative and lower semicontinuous.
Thus Fatou's lemma, applied along every convergent sequence of base
points, shows that the potential $G_\alpha*f_{\alpha,q}$ is lower
semicontinuous on $\mathbb R^n$; cf.\
\cite[Proposition~2.3.2]{zbMATH00824833}.
Hence its strict superlevel set
$\{G_\alpha*f_{\alpha,q}>\tfrac12\}$ is open and contains
$\sigma(\Gamma_0)$, where the potential is at least $1$.  Since
$\sigma(\Gamma_0)$ is compact and $\sigma(U)$ is an open neighborhood
of it, choose an open set $U'$ such that
$\sigma(\Gamma_0)\subset U'\Subset\sigma(U)$.  The intersection
$O'\ :=\ U'\cap\bigl\{x:\ (G_\alpha*f_{\alpha,q})(x)>\tfrac12\bigr\}$ is an open
neighborhood of $\sigma(\Gamma_0)$. On $O'$, the doubled function
$2f_{\alpha,q}$ has potential $>1$ and is therefore admissible.  Hence,
by the definition of capacity,
\(\mathcal{C}_{\alpha,q}^{\mathbb R^n}(O')
\le 2^{q}\|f_{\alpha,q}\|_{L^q}^{q}<\varepsilon\).
Finally, $O:=\sigma^{-1}(O')$ is open with $\Gamma_0\subset O\Subset U$, because
$O'\Subset\sigma(U)$ and $\sigma$ is a homeomorphism onto its image; and since
$O\subset U\subset B$, its capacity is computed in the fixed chart:
$\operatorname{cap}(O)=\mathcal{C}_{\alpha,q}^{\mathbb R^n}(O')<\varepsilon$.
\end{proof}

\begin{lemma}[Sublevel regular neighborhoods]\label{lem:sublevel}
Let $T$ be a finite triangulation of the fixed PL sphere $S^n$, let $L$
be a nonempty, proper, full subcomplex of $T$, and let
$f_L:|T|\to[0,1]$ be the simplicial map that takes value $0$ at the
vertices of $L$ and value $1$ at all other vertices.  Then
$f_L^{-1}(0)=|L|$.  Define the simplicial complement of $L$ in $T$ by
$C(L,T):=\{\Delta\in T:|\Delta|\cap|L|=\varnothing\}$.
Then $C(L,T)$ is a nonempty, proper, full subcomplex of $T$, and its
$0$--$1$-level function is $f_{C(L,T)}=1-f_L$.  Moreover, for every
$c\in(0,1)$, the set $N_c:=f_L^{-1}([0,c])$
is a closed regular neighborhood of $|L|$ in $S^n$ and a compact PL
$n$-manifold, with $\partial N_c=f_L^{-1}(c)$.  Its topological
interior in $S^n$ is $\operatorname{int}N_c=f_L^{-1}([0,c))$.
There is also a homeomorphism of pairs
\begin{equation}\label{eq:punctured-sublevel-product}
 \Phi_c:
 \bigl(N_c\setminus|L|,\partial N_c\bigr)
 \xrightarrow{\ \cong\ }
 \bigl(\partial N_c\times(0,1],
       \partial N_c\times\{1\}\bigr)
\end{equation}
whose restriction to $\partial N_c$ is $z\mapsto(z,1)$.  Consequently,
$\partial N_c$ is a strong deformation retract of $N_c\setminus|L|$.
\end{lemma}

\begin{proof}
Let $y\in|T|$, and let $\Delta_y$ be its carrier.  If $f_L(y)=0$, then
every vertex of $\Delta_y$ has $f_L$-value $0$; fullness of $L$ therefore
gives $\Delta_y\in L$, and hence $y\in|L|$.  The reverse inclusion is
immediate, so $f_L^{-1}(0)=|L|$.

The collection $C(L,T)$ is a subcomplex: if $\Delta\in C(L,T)$ and
$\tau$ is a face of $\Delta$, then $|\tau|\subset|\Delta|$, and hence
$|\tau|\cap|L|=\varnothing$.  Its vertex set is
\begin{equation}\label{eq:complementary-vertices}
 C(L,T)^{(0)}=T^{(0)}\setminus L^{(0)}.
\end{equation}
Indeed, a vertex of $L$ does not belong to $C(L,T)$; conversely, if
$v\notin L^{(0)}$, then $f_L(v)=1$, so
$v\notin f_L^{-1}(0)=|L|$, and the $0$-simplex $\{v\}$ belongs to
$C(L,T)$.  Because $L$ is nonempty, it has a vertex, so
$C(L,T)\neq T$.  Because $L$ is proper and full, some vertex of $T$
does not belong to $L$: otherwise fullness would imply that every
simplex of $T$ belongs to $L$, contrary to $L\neq T$.  Thus $C(L,T)$
is nonempty and proper.

To prove fullness, let $\Delta\in T$ have all its vertices in
$C(L,T)^{(0)}$.  Equation~\eqref{eq:complementary-vertices} gives
$f_L=1$ at every vertex of $\Delta$, and affineness gives
$f_L\equiv1$ on $|\Delta|$.  Since $|L|=f_L^{-1}(0)$, the simplex
$\Delta$ is disjoint from $|L|$ and therefore belongs to $C(L,T)$.
Finally, $f_{C(L,T)}$ and $1-f_L$ agree at every vertex by
\eqref{eq:complementary-vertices}; both are affine on each simplex of
$T$, so uniqueness of the affine extension gives
$f_{C(L,T)}=1-f_L$ on $|T|$.

Fix $c\in(0,1)$.  Call a simplex $\Delta\in T$ mixed if
its vertices occur at both $f_L$-levels $0$ and $1$.  A mixed simplex
belongs to neither $L$ nor $C(L,T)$.  Conversely, if
$\Delta\notin L\cup C(L,T)$, fullness of $L$ gives a vertex outside
$L^{(0)}$, while fullness of $C(L,T)$ gives a vertex outside
$C(L,T)^{(0)}$, hence in $L^{(0)}$ by
\eqref{eq:complementary-vertices}.  Thus the
mixed simplices are precisely the simplices not belonging to
$L\cup C(L,T)$.  For every mixed simplex $\Delta$, choose a derived vertex
$\widehat\Delta_c\in\mathring\Delta\cap f_L^{-1}(c)$, which
is possible by choosing positive barycentric coordinates whose
total weight on the level-$1$ vertices is $c$.  Let $\widehat T_c$ be
the resulting derived subdivision
of $T$ modulo $L\cup C(L,T)$.  Set
\[
N(L,\widehat T_c)
:=\bigcup_{v\in L^{(0)}}\operatorname{St}(v,\widehat T_c),
\qquad
\dot N(L,\widehat T_c)
:=\bigl\{\Xi\in N(L,\widehat T_c):
|\Xi|\cap|L|=\varnothing\bigr\},
\]
where $\operatorname{St}$ denotes the closed simplicial star.

Every simplex of $\widehat T_c$ that is not a simplex of
$L\cup C(L,T)$ has the form
\[
\Xi=\upsilon*\widehat\Delta_{0,c}*\cdots*
\widehat\Delta_{r,c},
\]
where $\Delta_0\subsetneq\cdots\subsetneq\Delta_r$ are mixed simplices
and $\upsilon$ is either empty or a simplex of $L\cup C(L,T)$ contained
in $\Delta_0$.  Consequently, the vertex levels of such a simplex are
contained in either $\{0,c\}$ or $\{c,1\}$.  A simplex of $L$ has all
its vertices at level $0$.  A simplex of $C(L,T)$ has all its vertices
at level $1$, since a vertex belonging to $L^{(0)}$ would give a point
of its intersection with $|L|$.  Thus the vertex levels of every
simplex of $\widehat T_c$ are contained in either $\{0,c\}$ or
$\{c,1\}$; in particular, no simplex has both a level-$0$ and a
level-$1$ vertex.

The same chain description shows that $L$ is full in $\widehat T_c$.
Indeed, suppose that every vertex of a simplex
$\Xi\in\widehat T_c$ belongs to $L$.  Every such vertex has level
$0$.  A simplex of $C(L,T)$ has only level-$1$ vertices, whereas every
simplex of
$\widehat T_c\setminus\bigl(L\cup C(L,T)\bigr)$ contains at least one
derived vertex $\widehat\Delta_{i,c}$ of level $c$.  Hence $\Xi$ is a
simplex of $L$, which proves fullness in the subdivided complex.
Thus $L$ is the full subcomplex of $\widehat T_c$ induced by the
level-$0$ vertices, and the complex $N(L,\widehat T_c)$ defined above
is its simplicial neighborhood.  Since $L$ is full in $T$ and
$\widehat T_c$ is, by construction, a derived subdivision of $T$
modulo $L\cup C(L,T)$,
\cite[\S3, pp.~227--228]{zbMATH01714503} shows that
$N(L,\widehat T_c)$ is a derived neighborhood of $L$ in $T$.
Because every derived vertex $\widehat\Delta_c$ was chosen on
$f_L^{-1}(c)$, this is precisely Bryant's $c$-neighborhood.

We claim that a simplex $\Xi$ of $\widehat T_c$ belongs to
$N(L,\widehat T_c)$ if and only if it has no level-$1$ vertex.  Suppose
first that $\Xi$ has no level-$1$ vertex.  If it has a level-$0$ vertex
$v$, then $v\in L^{(0)}$ and
$\Xi\in\operatorname{St}(v,\widehat T_c)$.  If all the vertices of
$\Xi$ have level $c$, then $\Xi$ is represented by a chain of mixed
simplices as above with $\upsilon=\varnothing$.  Choose
$v\in L^{(0)}\cap\Delta_0$.  The chain description of the derived
subdivision shows that $v*\Xi$ is a simplex of $\widehat T_c$.
Therefore $\Xi\in\operatorname{St}(v,\widehat T_c)$.

Conversely, suppose that $\Xi\in N(L,\widehat T_c)$.  By the definition
of the closed star, there exist $v\in L^{(0)}$ and a simplex
$\Xi'\in\widehat T_c$ such that $\Xi$ is a face of $\Xi'$ and $v$ is a
vertex of $\Xi'$.  The simplex $\Xi'$ therefore has a level-$0$ vertex.
By the preceding level dichotomy it has no level-$1$ vertex, and hence
neither does its face $\Xi$.  This proves the claim.

Every simplex of $\widehat T_c$ is contained in a simplex of $T$; hence
the restriction of $f_L$ to every simplex $\Xi$ of $\widehat T_c$ is
affine.  If $\Xi$ has no level-$1$ vertex, then $f_L\leq c$ on $|\Xi|$
and $\Xi\in N(L,\widehat T_c)$.  If $\Xi$ has a level-$1$ vertex, then
the level dichotomy implies that it has no level-$0$ vertex.  In that
case, $f_L^{-1}([0,c])\cap|\Xi|$ is the face spanned by the level-$c$
vertices of $\Xi$, with the empty face allowed.  The same is true of
$|N(L,\widehat T_c)|\cap|\Xi|$, because the faces of $\Xi$ belonging to
$N(L,\widehat T_c)$ are, by the claim, precisely those having no
level-$1$ vertex.  It follows simplexwise that
$f_L^{-1}([0,c])=|N(L,\widehat T_c)|$.
Together with the preceding identification, this equality identifies
$N_c$ with the realization of a derived neighborhood of $L$ in $T$;
hence $N_c$ is a regular neighborhood of $|L|$ in $S^n$.  Since
$\widehat T_c$ is finite, $N(L,\widehat T_c)$ is a finite subcomplex,
so $N_c$ is compact and therefore closed in $S^n$.  Moreover, since
$S^n$ is a PL manifold without boundary,
\cite[Theorem~3.6]{zbMATH01714503} shows that $N_c$ is a PL
$n$-manifold and that $\partial N_c=|\dot N(L,\widehat T_c)|$.

By the preceding characterization of $N(L,\widehat T_c)$, its
simplices disjoint from $|L|$ are precisely the simplices all of whose
vertices have level $c$.  Indeed, a simplex of the neighborhood that
has a level-$0$ vertex meets $|L|$, whereas, if all the vertices of a
simplex $\Xi$ have level $c$, then $f_L\equiv c>0$ on $|\Xi|$ and hence
$|\Xi|\cap|L|=\varnothing$, because $|L|=f_L^{-1}(0)$.  On every simplex $\Xi$ of
$\widehat T_c$, the intersection $f_L^{-1}(c)\cap|\Xi|$ is exactly the
face spanned by its level-$c$ vertices, again with the empty face
allowed.  Consequently,
$\partial N_c=|\dot N(L,\widehat T_c)|=f_L^{-1}(c)$.

Since $[0,c)$ is open in $[0,1]$ with its relative topology,
$f_L^{-1}([0,c))$ is open in $S^n$ and is contained in $N_c$.
Conversely, let $z\in f_L^{-1}(c)$.  Its carrier $\Delta$ in $T$ is
mixed.  Normalizing separately the barycentric coordinates on the
level-$0$ and level-$1$ vertices gives unique points $z_i$ in the
corresponding faces of $\Delta$ such that $z=(1-c)z_0+cz_1$.  Choose
$c_m\in(c,1)$ with $c_m\to c$ and put
$z_m:=(1-c_m)z_0+c_mz_1$.  Then $z_m$ has the same carrier as $z$,
$z_m\to z$, and $f_L(z_m)=c_m>c$.  Thus no point of
$f_L^{-1}(c)$ is interior to $N_c$, and consequently
$\operatorname{int}N_c=f_L^{-1}([0,c))$.

It remains to prove the product assertion.  Let
$z\in N_c\setminus|L|$ and put $r:=f_L(z)\in(0,c]$.  The carrier
$\Delta$ of $z$ in $T$ is mixed.  Let $\Delta_0$ and $\Delta_1$ be the
faces of $\Delta$ spanned by its level-$0$ and level-$1$ vertices,
respectively.  Normalizing separately the two groups of barycentric
coordinates gives unique points $z_i\in|\Delta_i|$ such that
$z=(1-r)z_0+rz_1$.
Define
\[
 q_c(z):=(1-c)z_0+cz_1\in f_L^{-1}(c)=\partial N_c,
 \qquad
 \Phi_c(z):=\bigl(q_c(z),r/c\bigr).
\]
Conversely, if $w\in\partial N_c$ and $\vartheta\in(0,1]$, write, by the same
normalized-barycentric-coordinate construction,
$w=(1-c)w_0+cw_1$, and set
$\Psi_c(w,\vartheta):=(1-\vartheta c)w_0+\vartheta c w_1$.
These definitions are independent of the choice of a mixed simplex containing the
point: both are determined by the barycentric coordinates in its unique
carrier, and the formulas restrict compatibly to every common face.
On each mixed simplex the formulas are continuous.  The sets
\[
 |\Delta|\cap f_L^{-1}((0,c])
 \quad\text{and}\quad
 \bigl(|\Delta|\cap f_L^{-1}(c)\bigr)\times(0,1],
\]
as $\Delta$ ranges over the finitely many mixed simplices of $T$, are
finite covers of the domain and codomain by relatively closed subsets,
respectively.  The
pasting lemma therefore shows that $\Phi_c$ and $\Psi_c$ are continuous.
The displayed formulas show directly that
$\Psi_c\circ\Phi_c=\operatorname{id}$ and
$\Phi_c\circ\Psi_c=\operatorname{id}$, so $\Phi_c$ is the claimed
homeomorphism.  If $z\in\partial N_c$, then $r=c$ and
$\Phi_c(z)=(z,1)$.  Finally,
\[
 R_u(z):=\Psi_c\bigl(q_c(z),(1-u)r/c+u\bigr),
 \qquad 0\leq u\leq1,
\]
defines a strong deformation retraction
$(R_u)_{0\leq u\leq1}$ of $N_c\setminus|L|$ onto
$\partial N_c$.  The construction uses only positive levels; no product
structure across the generally degenerate spine $f_L^{-1}(0)=|L|$ is
asserted.

Because $c\in(0,1)$ was arbitrary, the conclusion holds for every $c$
while $T$, $L$, and $f_L$ remain fixed.
\end{proof}

The next lemma constructs one stage of the tower from the preceding stage.
Its topological and analytic requirements are imposed successively.  First,
an embedded rose is chosen whose petals represent $\mathfrak{m}(\omega(a))$ and
$\mathfrak{m}(\omega(b))$.  A sufficiently small regular neighborhood of that
rose is then chosen inside an open set of controlled capacity supplied by
Lemma~\ref{lem:thin}.

\begin{lemma}[Thin marked successor]\label{successor}
Assume \(n\geq4\).  Let $H\subset\operatorname{int}B$ be a closed regular
neighborhood of an embedded rank-two rose with wedge vertex \(x\), assume
$x\in\operatorname{int}H$, and let
$\mathfrak{m}:F_2\xrightarrow{\cong}\pi_1(H,x)$ be a marking.  For every
$\varepsilon>0$, $\tau>1$, $\delta>0$, $\eta>0$, there are an embedded
rank-two PL rose $\Gamma'\subset\operatorname{int}H$ with wedge vertex
$x$, a connected closed regular neighborhood
$H'\Subset\operatorname{int}H$ of $\Gamma'$ satisfying
$\Gamma'\subset\operatorname{int}H'$
(in particular $x\in\operatorname{int}H'$), a marking
$\mathfrak{m}':F_2\xrightarrow{\cong}\pi_1(H',x)$, and a
finite family $\{U'_\ell\}$ of open balls with
\[
 H'\subset\bigcup_\ell U'_\ell\Subset\operatorname{int}H,
 \qquad
 \operatorname{diam}U'_\ell<\delta,
 \qquad
 \sum_\ell(\operatorname{diam}U'_\ell)^\tau<\eta,
\]
and, for the inclusion $i:H'\hookrightarrow H$,
\[
 i_*\circ\mathfrak{m}'=\mathfrak{m}\circ\omega,
 \qquad
 \operatorname{cap}(H')<\varepsilon.
\]
The petals of $\Gamma'$ may be oriented and labeled $a$ and $b$ so
that $\mathfrak{m}'(a)$ and $\mathfrak{m}'(b)$ are represented by the
corresponding oriented petal loops based at $x$.  Moreover, the data may
be chosen together with a triangulation $\mathcal{T}'$ of $S^n$ in which
$\Gamma'$ is the polyhedron of a nonempty, proper, full subcomplex, the simplicial map
$f':|\mathcal{T}'|\to[0,1]$ that takes value $0$ on the vertices of this
subcomplex and value $1$ on all other vertices, and a number
$t'\in(0,1)$ such that
\[
(f')^{-1}(0)=\Gamma',\qquad H'=(f')^{-1}([0,t']),
\]
and, for every $c\in(0,1)$, $(f')^{-1}([0,c])$ is a closed regular
neighborhood of $\Gamma'$ with
\[
\partial\bigl((f')^{-1}([0,c])\bigr)=(f')^{-1}(c).
\]
In particular, $i_*$ is injective.
\end{lemma}

The proof is divided into three steps.  Step~1 uses relative embedding approximation
to realize the prescribed classes by an embedded one-complex.  Step~2
applies Lemma~\ref{lem:thin} to obtain a
cover of arbitrarily small $\tau$-content and a regular neighborhood of
arbitrarily small capacity.  Since this neighborhood deformation retracts
onto its spine, it retains the topological data fixed in Step~1.  Step~3
transfers the marking through this deformation retraction and uses the
injectivity of $\omega$ from Lemma~\ref{lem:endo}.

\begin{proof}
We work throughout in a fixed compatible PL structure on $S^n$, refining
subdivisions as necessary so that the finite PL polyhedra under
consideration are subcomplexes and the base point $x$ is a vertex.  Let
$R=R_a\vee R_b$ denote the standard rank-two rose with wedge point $*$ and
oriented petal circles labeled $a$ and $b$, triangulated so that $*$ is a
vertex and each petal contains at least three simplicial edges.

\smallskip\noindent\emph{Step 1 (construction of an embedded PL rose realizing the algebraic data).}
Since $H$ is a closed regular neighborhood of a compact polyhedron in the
interior of a PL manifold, it is a compact PL manifold with boundary
\cite[Theorem~3.6]{zbMATH01714503}.  Consequently, $\partial H$ has a
collar in $H$ by the PL collaring theorem
\cite[Corollary~2.5]{zbMATH01714503}.  Choose a collar
$c_H \colon \partial H \times [0,1] \hookrightarrow H$
whose image is disjoint from the interior point \(x\); this is possible after
restricting and reparametrizing an initial collar to a sufficiently short
subinterval.  Put
$H_0:=H\setminus c_H\bigl(\partial H\times[0,1)\bigr)
   \subset\operatorname{int}H$.
Define a homotopy $(r_u)_{u\in[0,1]}$ on $H$ by
\[
r_u\bigl(c_H(y,t)\bigr):=c_H\bigl(y,(1-u)t+u\bigr)
\quad (y\in\partial H,\ 0\leq t\leq1),
\qquad
r_u(z):=z\quad(z\in H_0).
\]
The two formulas agree on $c_H(\partial H\times\{1\})$, so
$(r_u)_{u\in[0,1]}$ is a strong deformation retraction of $H$ onto
$H_0$.  The homotopy fixes $x$, and its restriction to
$\operatorname{int}H$ is a strong deformation retraction onto the same
set $H_0$, because
$r_u(c_H(y,t))\in c_H(\partial H\times(0,1])$ whenever $t>0$.  Consequently,
the inclusion $\operatorname{int}H\hookrightarrow H$ is a based homotopy
equivalence and induces an isomorphism
$\pi_1(\operatorname{int}H,x)\xrightarrow{\cong}\pi_1(H,x)$.
In particular, the classes $\mathfrak{m}(\omega(a))$ and
$\mathfrak{m}(\omega(b))$ admit continuous based representatives
\(\gamma_a,\gamma_b:([0,1],\{0,1\})\to
(\operatorname{int}H,x)\)
satisfying $[\gamma_a]=\mathfrak{m}(\omega(a))$ and
$[\gamma_b]=\mathfrak{m}(\omega(b))$ in $\pi_1(H,x)$.  Let
$q_0:[0,1]\to[0,1]/(0\sim1)$ be the quotient map, and, for
$c\in\{a,b\}$, fix an orientation-preserving based homeomorphism
$\vartheta_c:\bigl([0,1]/(0\sim1),q_0(0)\bigr)\to(R_c,*)$.  Since
$\gamma_c(0)=\gamma_c(1)=x$, there is a unique continuous based loop
$\overline\gamma_c:(R_c,*)\to(\operatorname{int}H,x)$ satisfying
$\overline\gamma_c\circ\vartheta_c\circ q_0=\gamma_c$.  By the universal
property of the wedge sum, there is a unique continuous based map
$f_0:(R,*)\to(\operatorname{int}H,x)$ whose restrictions to $R_a$ and
$R_b$ are $\overline\gamma_a$ and $\overline\gamma_b$, respectively.  Thus
\[
[f_0|_{R_a}]=[\overline\gamma_a]=[\gamma_a]=\mathfrak{m}(\omega(a)),
\qquad
[f_0|_{R_b}]=[\overline\gamma_b]=[\gamma_b]=\mathfrak{m}(\omega(b))
\]
in $\pi_1(H,x)$.  The map $f_0$ may have self-intersections away from the
wedge point $x$.

We record the precise relative approximation statement being used.
Bryant's Corollary~4.4
\cite[p.~235]{zbMATH01714503} states that, if $X_0\subset X$ is a
subpolyhedron of a $p$-dimensional polyhedron, $M_0$ is a PL
$n$-manifold with $2p+1\leq n$, and $f:X\to M_0$ is continuous with
$f|_{X_0}$ a PL embedding, then, for every continuous control function
$\varepsilon:X\to(0,\infty)$, the map $f$ is
$\varepsilon$-homotopic relative to $X_0$ to a PL embedding.
Apply
this statement with
\[
 X=R,\qquad X_0=\{*\},\qquad p=1,\qquad
 M_0=\operatorname{int}H,
 \qquad \varepsilon\equiv\varepsilon_0>0.
\]
Here $\operatorname{int}H$ is an open PL $n$-manifold,
$f_0|_{\{*\}}$ is a PL embedding, and $3=2p+1\leq n$.  Consequently,
there is a PL embedding
$e:R\hookrightarrow\operatorname{int}H$ that is
$\varepsilon_0$-homotopic to $f_0$ relative to $\{*\}$.  Hence
$e(*)=x$ and $e\simeq f_0$ relative to $*$ in
$\operatorname{int}H$.  No condition on
$f_0^{-1}(f_0(*))$ occurs among the hypotheses of this corollary.  The
absence of such a hypothesis is not an oversight.  Corollary~4.4 rests
on \cite[Theorem~4.3]{zbMATH01714503}.  The existence and PL character
of $e$ are supplied directly by Corollary~4.4; we use Theorem~4.3 only
to explain why no additional hypothesis
$f^{-1}(f(X_0))=X_0$ is required.  Put
$p_{\mathrm{rel}}:=\dim(X\setminus X_0),$
which is the dimension denoted by $p$ in Theorem~4.3.  In the present
application,
$p_{\mathrm{rel}}=\dim(R\setminus\{*\})=1=p,$
where $p=\dim R$ is the parameter in Corollary~4.4.  Thus
$p_{\mathrm{rel}}\leq n$, and $f_0|_{\{*\}}$ is PL and
nondegenerate on the one-vertex triangulation of $\{*\}$.  Theorem~4.3
therefore gives an approximating map $\widetilde f$, relative to
$X_0=\{*\}$, satisfying
\[
 \dim\bigl(S(\widetilde f)\setminus X_0\bigr)
 \leq2p_{\mathrm{rel}}-n=2-n<0,
 \qquad
 S(\widetilde f):=\overline{
 \{z\in X:\widetilde f^{-1}(\widetilde f(z))\neq\{z\}\}}.
\]
Here the closure is taken in $X$.
Bryant denotes this approximating map by $f'$.\footnote{In the two
dimension estimates following the introduction of $f'$, Bryant prints
$S_f$; in context the intended singular set is $S_{f'}$.}  Because the
right-hand side of the dimension estimate in the preceding display is negative,
Bryant's dimension estimate asserts that
$S(\widetilde f)\setminus X_0=\varnothing$.  If
$\widetilde f(z_1)=\widetilde f(z_2)$ for two distinct points, then both
$z_1$ and $z_2$ belong to $S(\widetilde f)$ and hence to $X_0$,
contradicting the injectivity of
$\widetilde f|_{X_0}=f_0|_{X_0}$.  Thus no point of $X\setminus X_0$
can be identified either with another such point or with a point of
$X_0$; this is precisely why no hypothesis
$f^{-1}(f(X_0))=X_0$ is needed.  In particular, possible interior
returns of the original petal loops to $x$ do not affect the
application.  No metric bound on $e$ is needed; Step~2 constrains its
neighborhood.

Consequently, $\Gamma':=e(R)$ is an embedded PL rank-two rose with wedge
vertex $x\in\operatorname{int}H$, and its oriented petal circles, regarded as
based loops at $x$, satisfy
\begin{equation*}
[e|_{R_a}]=[f_0|_{R_a}]=\mathfrak{m}(\omega(a)),
\qquad
[e|_{R_b}]=[f_0|_{R_b}]=\mathfrak{m}(\omega(b))
\qquad\text{in }\pi_1(H,x),
\end{equation*}
because a based homotopy in $\operatorname{int}H$ restricts to each petal
circle as a based homotopy of loops, and the inclusion induces an isomorphism
$\pi_1(\operatorname{int}H,x)\cong\pi_1(H,x)$.

\smallskip\noindent\emph{Step 2 (construction of a thin regular neighborhood).}
Since $\Gamma'$ is compact and lies in the open set
$\operatorname{int}H$, the distance
$d_1:=\operatorname{dist}_{g_{st}}
 \bigl(\Gamma',S^n\setminus\operatorname{int}H\bigr)$ is positive.  The open set
$W:=\{z\in S^n:d_{g_{st}}(z,\Gamma')<d_1/2\}$ satisfies
$\Gamma'\subset W\Subset\operatorname{int}H$.

Apply Lemma~\ref{lem:thin} to the finite graph
$\Gamma'\subset W\subset\operatorname{int}B$.  For
the given parameters $\tau>1$, $\delta>0$, and $\eta>0$, it produces a
finite family of open balls $\{U'_\ell\}$ such that
\(\Gamma' \subset U' := \bigcup_\ell U'_\ell \Subset W\),
\(\operatorname{diam}U'_\ell<\delta\),
and
\(\sum_\ell(\operatorname{diam}U'_\ell)^\tau<\eta\).
The lemma also gives an open set $O$ satisfying
\(\Gamma'\subset O\Subset U'\)
and
\(\operatorname{cap}(O)<\varepsilon\).

We construct a closed regular neighborhood of $\Gamma'$ inside $O$.
Choose a triangulation $T$ of $S^n$, in the fixed PL structure, such
that $\Gamma'$ is the underlying polyhedron of a full subcomplex $L$;
fullness can be achieved by a derived subdivision.  Since $\Gamma'$ is
a rank-two rose, $L$ is nonempty; and since
$\dim|L|=\dim\Gamma'=1<n=\dim S^n$, one has $L\ne T$.  Thus $L$ is a
nonempty, proper, full subcomplex, so Lemma~\ref{lem:sublevel} applies.  Let
$f_L:|T|=S^n\to[0,1]$ be the simplicial map that takes value $0$ at the
vertices of $L$ and value $1$ at all other vertices.  By that lemma,
$f_L^{-1}(0)=\Gamma'$, and, for every $c\in(0,1)$, the set
$f_L^{-1}([0,c])$ is a closed regular neighborhood of $\Gamma'$ and a
compact PL $n$-manifold, with
$\partial\bigl(f_L^{-1}([0,c])\bigr)=f_L^{-1}(c).$
For $\varepsilon'\in(0,1)$ put
$N_{\varepsilon'}:=f_L^{-1}([0,\varepsilon'])$.

The sets $N_{\varepsilon'}$ are
nested compacta with
$\bigcap_{0<\varepsilon'<1}N_{\varepsilon'}=f_L^{-1}(0)=\Gamma'$.
Since
$O\Subset U'\Subset W\Subset\operatorname{int}H
 \subset\operatorname{int}B\subsetneq S^n$,
the set $S^n\setminus O$ is nonempty and compact.  It is disjoint from
$f_L^{-1}(0)=\Gamma'$, and hence
$m_O:=\min_{S^n\setminus O}f_L>0$.  Choose
$0<\varepsilon'<\min\{1,m_O\}$ and set $H':=N_{\varepsilon'}$.  A
point with $f_L\le\varepsilon'<m_O$ cannot lie in $S^n\setminus O$, so
\(\Gamma'\subset\operatorname{int}H'
=f_L^{-1}([0,\varepsilon'))\subset H'\subset O\),
where the equality follows from Lemma~\ref{lem:sublevel}.  In particular,
$\Gamma'\subset\operatorname{int}H'$ and $x\in\operatorname{int}H'$, as
required.

Retain the auxiliary data
$\mathcal{T}':=T,\qquad f':=f_L,\qquad t':=\varepsilon'.$
Then $(f')^{-1}(0)=\Gamma'$ and
$H'=(f')^{-1}([0,t'])$.

Since $H'\subset O\Subset U'$, one has
$H'\subset\bigcup_\ell U'_\ell$.  The diameter and $\tau$-content bounds
are properties of the fixed family $\{U'_\ell\}$.  Monotonicity of
capacity with respect to inclusion gives
$\operatorname{cap}(H')\leq\operatorname{cap}(O)<\varepsilon$.

\smallskip\noindent\emph{Step 3 (the marking identity and injectivity of the inclusion).}
Since a regular neighborhood collapses onto its spine
\cite[Theorem~3.16]{zbMATH01714503}, the resulting deformation retraction of $H'$
onto $\Gamma'$ fixes $\Gamma'$ (and therefore the base point $x$) pointwise;
in particular $H'$ is connected, retracting onto the connected rose
$\Gamma'$.  Consequently, the inclusion
$\iota:(\Gamma',x)\hookrightarrow(H',x)$ is a based homotopy equivalence
and induces an isomorphism
$\iota_*:\pi_1(\Gamma',x)\xrightarrow{\cong}\pi_1(H',x)$.
The embedding $e:(R,*)\to(\Gamma',x)$ is a based PL homeomorphism, so
$e_*$ is an isomorphism.  By the Seifert--van Kampen theorem, the
fundamental group of the standard rose is free of rank two on its two
oriented petal circles:
$\pi_1(R,*) \cong \langle a,b \rangle \cong F_2.$
Thus $\pi_1(\Gamma',x)$ is likewise free of rank two, generated by the
images of the two oriented petal circles of $\Gamma'$.

Define
$\mathfrak{m}':=\iota_*\circ e_*:F_2\to\pi_1(H',x)$.
Both factors are isomorphisms, so $\mathfrak{m}'$ is a marking.  Explicitly, in
$\pi_1(H',x)$,
$\mathfrak{m}'(a) = [\iota \circ e|_{R_a}], \qquad
\mathfrak{m}'(b) = [\iota \circ e|_{R_b}].$

Let $i:H'\hookrightarrow H$ be the inclusion.  For the generator $a$,
\(i_* \mathfrak{m}'(a) = (i \circ \iota \circ e)_*(a) = [e|_{R_a}]\)
in $\pi_1(H,x)$, because $i\circ\iota\circ e$ regards $e|_{R_a}$ as a
loop in $H$.  By Step~1,
$[e|_{R_a}]=\mathfrak{m}(\omega(a))$.  The same identity holds for $b$.  Thus the
homomorphisms $i_*\circ\mathfrak{m}'$ and $\mathfrak{m}\circ\omega$ agree on the free
generators of $F_2$ and are therefore equal:
\(i_* \circ \mathfrak{m}' = \mathfrak{m} \circ \omega\).

Rearranging this identity yields the factorization
\(i_* = \mathfrak{m} \circ \omega \circ (\mathfrak{m}')^{-1}\).
Here $\mathfrak{m}$ and $(\mathfrak{m}')^{-1}$ are isomorphisms, while $\omega$ is
injective by Lemma~\ref{lem:endo}.  Hence $i_*$ is injective.
\end{proof}

For the iteration, we use the rates
$\varepsilon=\delta=\eta=2^{-(j+1)}$ and
$\tau=1+\frac1{j+1}$; these yield the estimates in
Proposition~\ref{induction}\textup{(2)}--\textup{(3)} that are used in
Proposition~\ref{prop:limitset}.

\begin{proposition}[The tower]\label{induction}
There are closed regular neighborhoods $V_j\subset\operatorname{int}B$ of
embedded rank-two roses $\Gamma_j$ whose petals are oriented and labeled
$a$ and $b$, a common wedge point
$x\in\bigcap_j\operatorname{int}V_j$, finite ball families
$\mathcal{U}_j=\{U_{j,\ell}\}$, and markings
$\mathfrak{m}_j:F_2\xrightarrow{\cong}\pi_1(V_j,x)$ such that for all
$j\ge1$:
\begin{enumerate}
\item $V_{j+1}\Subset\operatorname{int}V_j$;
\item $\operatorname{cap}(V_j)<2^{-j}$;
\item $V_j\subset\bigcup_\ell U_{j,\ell}$ with
$\operatorname{diam}U_{j,\ell}<2^{-j}$ and
$\sum_\ell(\operatorname{diam}U_{j,\ell})^{1+1/j}<2^{-j}$;
\item $(i_j)_*\circ\mathfrak{m}_{j+1}=\mathfrak{m}_j\circ\omega$ for
$i_j:V_{j+1}\hookrightarrow V_j$;
\item $\mathfrak{m}_j(a)$ and $\mathfrak{m}_j(b)$ are represented by the
corresponding oriented petal loops of $\Gamma_j$ based at $x$;
\item there are a triangulation $\mathcal{T}_j$ of $S^n$ in which
$\Gamma_j$ is the polyhedron of a nonempty, proper, full subcomplex, the simplicial map
$f_j:|\mathcal{T}_j|\to[0,1]$ taking value $0$ on the vertices of this
subcomplex and value $1$ on all other vertices, and a number
$t_j\in(0,1)$ such that
\[
f_j^{-1}(0)=\Gamma_j,\qquad V_j=f_j^{-1}([0,t_j]),
\]
and, for every $c\in(0,1)$, $f_j^{-1}([0,c])$ is a closed regular
neighborhood of $\Gamma_j$ with
\[
\partial\bigl(f_j^{-1}([0,c])\bigr)=f_j^{-1}(c).
\]
\end{enumerate}
Consequently, for $k>j$ the inclusion $V_k\hookrightarrow V_j$
induces $\mathfrak{m}_j\circ\omega^{\,k-j}\circ\mathfrak{m}_k^{-1}$ on $\pi_1$; in
particular it is injective.
\end{proposition}

The proof is inductive.  At each stage, Lemma~\ref{successor} constructs
in the preceding neighborhood a marked rose whose oriented petals
represent $\mathfrak{m}_j(\omega(a))$ and $\mathfrak{m}_j(\omega(b))$.
Equivalently, its marking satisfies
$(i_j)_*\circ\mathfrak{m}_{j+1}=\mathfrak{m}_j\circ\omega$.  The lemma also
provides capacity, diameter, and Hausdorff-content bounds for a sufficiently
thin regular neighborhood.
The exponentially decaying parameters yield the estimates used for the
limit in Proposition~\ref{prop:limitset}.

\begin{proof}
Fix a compatible PL structure on $S^n$ and pass to suitable subdivisions
as necessary so that each finite PL polyhedron under consideration is a
subcomplex and $x$ is a vertex.  Let $(R,*)$ denote the standard rank-two
rose $R=R_a\vee R_b$, with oriented petal circles labeled $a,b$ and with
$*$ as their common vertex.  Triangulate each petal by at least three simplicial
edges.  We identify $\pi_1(R,*)$, freely generated by the two oriented
petal loops, with $F_2=\langle a,b\rangle$.

\smallskip\noindent\emph{Step 1 (the base stage).}
We prove the assertion for $j=1$.
Fix an open PL $n$-ball $W_1$ with $\overline{W_1}\subset\operatorname{int}B$
and a PL embedding $e_1:(R,*)\hookrightarrow W_1$.  For example, under
a PL homeomorphism $W_1\cong\mathbb R^n$, realize $R$ as the union of
the boundaries of two $2$-simplices in
$\mathbb R^2\subset\mathbb R^n$ meeting at a single common vertex, and
transport this compact $1$-polyhedron back to $W_1$.  Set
$\Gamma_1:=e_1(R)$ and $x:=e_1(*)$.
Lemma~\ref{lem:thin}, applied to $(\Gamma_1,W_1)$ with
$(\tau,\delta,\eta)=(2,2^{-1},2^{-1})$, and with
$\varepsilon=2^{-1}$ in part~\textup{(b)}, produces a finite family
$\mathcal{U}_1=\{U_{1,\ell}\}$ of open balls and an open set $O_1$ with
\[
\Gamma_1\subset O_1\Subset U^{(1)}:=\bigcup_\ell U_{1,\ell}\Subset W_1,
\qquad
\operatorname{diam}U_{1,\ell}<2^{-1},
\qquad
\sum_\ell(\operatorname{diam}U_{1,\ell})^{2}<2^{-1},
\]
together with $\operatorname{cap}(O_1)<2^{-1}$.  Choose a triangulation
of $S^n$ in which $\Gamma_1$ is the polyhedron of a nonempty proper
subcomplex, and pass to a derived subdivision.  Denote the resulting
triangulation by $\mathcal{T}_1$; the corresponding subcomplex with
polyhedron $\Gamma_1$ is nonempty, proper, and full.
Let
$f_1:|\mathcal{T}_1|\to[0,1]$ take value $0$ on the vertices
of that subcomplex and value $1$ on all other vertices.  By
Lemma~\ref{lem:sublevel}, $f_1^{-1}(0)=\Gamma_1$ and, for
every $c\in(0,1)$, $f_1^{-1}([0,c])$ is a closed regular neighborhood of
$\Gamma_1$ with
\(\partial\bigl(f_1^{-1}([0,c])\bigr)=f_1^{-1}(c)\).
Since $m_{O_1}:=\min_{S^n\setminus O_1}f_1>0$, choose
$0<t_1<\min\{1,m_{O_1}\}$ and set $V_1:=f_1^{-1}([0,t_1])$.  Then
$\Gamma_1\subset \operatorname{int}V_1\subset V_1\subset O_1$.
Thus $\mathcal{T}_1$, $f_1$, and $t_1$ are the data required in
property~\textup{(6)} at $j=1$.
Moreover, $V_1$ collapses onto $\Gamma_1$ by
\cite[Theorem~3.16]{zbMATH01714503}; hence the inclusion
$(\Gamma_1,x)\hookrightarrow(V_1,x)$ is a based homotopy equivalence.
In particular, $V_1\subset W_1\subset\operatorname{int}B$, $x\in\operatorname{int}V_1$,
and $V_1$ is connected.  Property~(2) at
$j=1$ follows from $V_1\subset O_1$ by monotonicity of the capacity;
property~(3) at $j=1$ holds because $V_1\subset O_1\subset U^{(1)}$, while
the diameter and content bounds are properties of the fixed family
$\mathcal{U}_1$. Define
\[
\mathfrak{m}_1:=(\iota_1)_*\circ(e_1)_*:F_2\longrightarrow\pi_1(V_1,x),
\qquad
\iota_1:(\Gamma_1,x)\hookrightarrow(V_1,x);
\]
here $(e_1)_*$ is an isomorphism because $e_1$ is a based PL homeomorphism
onto $\Gamma_1$, and $(\iota_1)_*$ is an isomorphism, so $\mathfrak{m}_1$ is
a marking.  By its definition, $\mathfrak{m}_1(a)$ and
$\mathfrak{m}_1(b)$ are represented by the correspondingly oriented
$a$- and $b$-petals of $\Gamma_1$, which proves property~\textup{(5)} at
$j=1$.  The triple $(V_1,x,\mathfrak{m}_1)$ therefore satisfies every hypothesis of
Lemma~\ref{successor}: $V_1\subset\operatorname{int}B$ is a closed regular neighborhood of the
embedded rank-two rose $\Gamma_1$, the point $x$ lies in
$\operatorname{int}V_1$, and $\mathfrak{m}_1:F_2\xrightarrow{\cong}\pi_1(V_1,x)$
is a marking.

\smallskip\noindent\emph{Step 2 (the inductive stage).}
Let $j\ge1$ and suppose stage $j$ has been constructed: an embedded PL
rank-two rose $\Gamma_j$ with wedge vertex $x$, a closed regular neighborhood
$V_j\subset\operatorname{int}B$ of $\Gamma_j$ with $x\in\operatorname{int}V_j$, and a marking
$\mathfrak{m}_j:F_2\xrightarrow{\cong}\pi_1(V_j,x)$.  Properties~(2) and~(3) at
index $j$ are verified when the corresponding stage is constructed: in
Step~1 for $j=1$ and in the present step for $j\geq2$.  These data are
exactly the hypotheses
of Lemma~\ref{successor} for $(H,x,\mathfrak{m})=(V_j,x,\mathfrak{m}_j)$; apply it with the
parameters
\(\varepsilon=\delta=\eta=2^{-(j+1)}\)
and \(\tau=1+\tfrac1{j+1}\ (>1)\),
and name its output
\[
\begin{aligned}
\Gamma_{j+1}&:=\Gamma',
&\qquad V_{j+1}&:=H',
&\qquad \mathfrak{m}_{j+1}&:=\mathfrak{m}',\\
\mathcal{U}_{j+1}&:=\{U_{j+1,\ell}\}:=\{U'_\ell\},
&\qquad \mathcal{T}_{j+1}&:=\mathcal{T}',\\
f_{j+1}&:=f',
&\qquad t_{j+1}&:=t'.
\end{aligned}
\]
The conclusions of the lemma then read: $\Gamma_{j+1}\subset
\operatorname{int}V_j$ is an embedded rank-two PL rose with wedge vertex $x$;
$V_{j+1}$ is a closed regular neighborhood of $\Gamma_{j+1}$ with
$V_{j+1}\Subset\operatorname{int}V_j$, which is property~(1);
\[
\begin{aligned}
V_{j+1}&\subset\bigcup_\ell U_{j+1,\ell}
          \Subset\operatorname{int}V_j,\\
\operatorname{diam}U_{j+1,\ell}&<2^{-(j+1)},\\
\sum_\ell(\operatorname{diam}U_{j+1,\ell})^{1+\frac1{j+1}}
  &<2^{-(j+1)},
\end{aligned}
\]
which is property~(3) at index $j+1$, with the required content exponent
$1+\frac1{j+1}$.  Moreover,
$\operatorname{cap}(V_{j+1})<2^{-(j+1)}$, which is property~(2) at index
$j+1$.  Finally, for the inclusion
$i_j:V_{j+1}\hookrightarrow V_j$,
\((i_j)_*\circ\mathfrak{m}_{j+1}=\mathfrak{m}_j\circ\omega\),
which is property~(4) at index $j$ (the lemma's final assertion, injectivity
of $(i_j)_*$, is subsumed in Step~3 of the present proof).  The petal assertion of
Lemma~\ref{successor} gives property~\textup{(5)} at index $j+1$, and
its final auxiliary-data assertion gives property~\textup{(6)} at index
$j+1$.  To continue the recursion, we verify
that stage $j+1$ again satisfies the hypotheses of Lemma~\ref{successor}:
$V_{j+1}$ is a closed regular neighborhood of the embedded rose
$\Gamma_{j+1}$, and
$V_{j+1}\subset V_j\subset\operatorname{int}B$; the marking
$\mathfrak{m}_{j+1}$ is supplied by the lemma; and the lemma also gives
$x\in\Gamma_{j+1}\subset\operatorname{int}V_{j+1}$.  Hence
$x\in\bigcap_{j\geq1}\operatorname{int}V_j$.  Each $V_j$ is connected:
this was shown for $j=1$ in Step~1, and for $j\geq2$ it follows because
$V_j$ collapses onto the connected graph $\Gamma_j$
\cite[Theorem~3.16]{zbMATH01714503}.  This completes the recursion.

\smallskip\noindent\emph{Step 3 (composites and injectivity).}
Fix $j\ge1$ and write $\iota_j^k:V_k\hookrightarrow V_j$ for the inclusion whenever
$k>j$. We prove $(\iota_j^k)_*=\mathfrak{m}_j\circ\omega^{\,k-j}\circ\mathfrak{m}_k^{-1}$ by
induction on $k-j$. For $k=j+1$ this is property~(4), rearranged using the
invertibility of $\mathfrak{m}_{j+1}$: $(i_j)_*=\mathfrak{m}_j\circ\omega\circ\mathfrak{m}_{j+1}^{-1}$.
Assuming the identity for $k$, functoriality of $\pi_1$ applied to
$\iota_j^{k+1}=\iota_j^{k}\circ i_k$ gives
\[
(\iota_j^{k+1})_*
=(\iota_j^{k})_*\circ(i_k)_*
=\bigl(\mathfrak{m}_j\circ\omega^{\,k-j}\circ\mathfrak{m}_k^{-1}\bigr)\circ
 \bigl(\mathfrak{m}_k\circ\omega\circ\mathfrak{m}_{k+1}^{-1}\bigr)
=\mathfrak{m}_j\circ\omega^{\,k-j+1}\circ\mathfrak{m}_{k+1}^{-1}.
\]
Since $\omega$ is injective (Lemma~\ref{lem:endo}), so is every power
$\omega^{\,k-j}$, and pre- and post-composition with the isomorphisms
$\mathfrak{m}_k^{-1}$ and $\mathfrak{m}_j$ preserves injectivity; hence each $(\iota_j^k)_*$ is
injective.
\end{proof}

For the remainder of this section, fix such a tower and set
\[
 K:=\bigcap_{j\geq1}V_j,
 \qquad
 M:=S^n\setminus K.
\]

The analytic arguments below use only the properties of $K$ stated in the
next proposition.  The topological arguments additionally use the nesting
and the markings.

\begin{proposition}[The limit continuum]\label{prop:limitset}
$K$ is a nondegenerate continuum with
$K\subset V_{j+1}\subset\operatorname{int}V_j$ for all $j$, and
$\operatorname{cap}(K)=0$, $\mathcal{H}^1(K)>0$, and
$\dim_{\mathcal{H}}K=1$.
\end{proposition}

\begin{proof}
The sets $V_j$ are nonempty compact connected sets nested by inclusion,
so their intersection $K$ is nonempty and compact.  To prove
connectedness, suppose that $K=A\sqcup A'$ is a separation into nonempty
closed subsets.  The compact sets $A$ and $A'$ have disjoint open
neighborhoods $U$ and $U'$ in $S^n$.  The compact sets
$V_j\setminus(U\cup U')$ decrease with $j$ and have empty intersection, because
$K\subset U\cup U'$.  Hence the finite-intersection property gives
$V_j\setminus(U\cup U')=\varnothing$ for some $j$, so
$V_j\subset U\cup U'$.  Since $V_j$ contains $K$, it meets
both $U$ and $U'$, contrary to the connectedness of $V_j$.  Thus $K$ is
connected.

We next prove that $K$ is nondegenerate.  By
Proposition~\ref{induction}, $x\in\bigcap_j\operatorname{int}V_j\subset K$.
Suppose, to the contrary, that $K=\{x\}$.  Choose a contractible
coordinate ball $Q$ such that
$x\in Q\Subset\operatorname{int}V_1$.  The compact sets
$V_j\setminus Q$ decrease to $K\setminus Q=\varnothing$.  By the finite
intersection property, $V_k\setminus Q=\varnothing$ for some $k$.
Replacing $k$ by $\max\{k,2\}$ and using the nesting of the $V_j$, we
may assume that $k>1$.  Thus $V_k\subset Q$.  The based inclusion
$(V_k,x)\hookrightarrow(V_1,x)$ therefore factors through $(Q,x)$ and
hence induces the zero homomorphism on fundamental groups.  This
contradicts Proposition~\ref{induction}, because the induced
homomorphism is injective and $\pi_1(V_k,x)\cong F_2\neq1$.  Thus $K$
is nondegenerate.  Choose $z_0,z_1\in K$ realizing its positive
spherical diameter
$D_K:=\operatorname{diam}_{g_{st}}K$.  The image of the connected set
$K$ under the $1$-Lipschitz function
$z\mapsto d_{g_{st}}(z_0,z)$ is a connected subset of $[0,D_K]$
containing both endpoints, and hence is $[0,D_K]$.  Since Hausdorff
measure is nonincreasing under $1$-Lipschitz maps,
\(0<D_K=\mathcal{H}^1([0,D_K])\leq\mathcal{H}^1(K)\),
so $\dim_{\mathcal{H}}K\geq1$.

By monotonicity and Proposition~\ref{induction}(2),
\(0\leq\operatorname{cap}(K)\leq\operatorname{cap}(V_j)<2^{-j}\)
for every $j$, and therefore $\operatorname{cap}(K)=0$.  Fix
$\tau>1$ and take $j$ sufficiently large that
$p_j:=1+\frac1j<\tau$ and $2^{-j}<1$.  Put
$d_{j,\ell}:=\operatorname{diam}U_{j,\ell}$.  Since
$d_{j,\ell}<2^{-j}<1$, one has
$d_{j,\ell}^{\tau}\leq d_{j,\ell}^{p_j}$.  Proposition~\ref{induction}(3)
therefore gives
\[
\mathcal{H}^\tau_{2^{-j}}(K)
 \leq\sum_\ell d_{j,\ell}^{\tau}
 \leq\sum_\ell d_{j,\ell}^{p_j}
 <2^{-j}.
\]
Thus, for all sufficiently large $j$,
$\mathcal{H}^\tau_{2^{-j}}(K)<2^{-j}$.  Since $2^{-j}\downarrow0$, the
definition of Hausdorff measure gives $\mathcal{H}^\tau(K)=0$.  This holds for
every $\tau>1$, and hence $\dim_{\mathcal{H}}K\leq1$.  Together with
the preceding lower bound, this gives $\dim_{\mathcal{H}}K=1$.
\end{proof}

For $n\geq5$, the capacity conclusion also follows from the dimension
conclusion.  Indeed, then $s=(n-2)/2>1$, and the Lipschitz map
$\sigma|_B$ sends $K$ to a set satisfying
$\mathcal{H}^s(\sigma(K))=0$.  The comparison
$\Lambda_h\leq\mathcal{H}^s$ and the finite-critical-measure theorem used
in Step~2 of Lemma~\ref{lem:thin} therefore give
$\operatorname{cap}(K)=0$.  At the endpoint $n=4$, where $s=1$, the
Hausdorff-content estimate above gives no upper bound for
$\mathcal{H}^1(K)$ and, in particular, does not show that this measure is
finite.  Thus the identity $\dim_{\mathcal{H}}K=1$ alone gives no
capacity conclusion at the critical exponent.  The capacity-thinning
condition in Proposition~\ref{induction}(2) is therefore needed there.

\subsection{Topology of the pair \texorpdfstring{$(K,M)$}{(K,M)}}
This subsection proves that $M$ is contractible but not simply connected
at infinity, as required in Theorem~\ref{thm:counterexample-psc}, and
additionally records that $K$ is \v{C}ech-acyclic.  Two properties of the tower are
used.  Its nesting implies both the cofinality of the $V_j$ among
neighborhoods of $K$ and the exhaustion of $M$ by the $C_j$, as shown in
Lemma~\ref{cofinal}.  The two properties of $\omega$ established in
Lemma~\ref{lem:endo} are used, through the marking identity in
Proposition~\ref{induction}(4), for distinct purposes: the equality
$H_1(\omega)=0$ yields the vanishing results, whereas the injectivity of
$\omega$ yields the obstruction to simple connectivity at infinity.

\begin{lemma}[Cofinality and exhaustion]\label{cofinal}
Put $E_j:=S^n\setminus V_j$ and
$C_j:=S^n\setminus\operatorname{int}V_j$.  Then:
\textup{(a)} every open $U\supset K$ contains some $V_j$; the $V_j$
are cofinal among compact neighborhoods of $K$;
\textup{(b)} the $E_j$ are open and increase to $M$, and every compact
$D\subset M$ satisfies $D\subset E_j\subset C_j$ for some $j$;
\textup{(c)} the $C_j$ are compact and nested, exhaust $M$, and
$C_j\subset E_{j+1}\subset\operatorname{int}_M C_{j+1}$.
\end{lemma}

\begin{proof}
(a) Let $U\supset K$ be open.  Each $V_j\setminus U=V_j\cap(S^n
\setminus U)$ is a closed subset of the compact set $V_j$, hence compact,
and these sets decrease in $j$ because the $V_j$ do; their intersection is
$\bigl(\bigcap_jV_j\bigr)\setminus U=K\setminus U=\varnothing$.  By Cantor's
intersection theorem, a decreasing sequence of nonempty compacta has
nonempty intersection, so some $V_{j_0}\setminus U=\varnothing$, that is,
$V_{j_0}\subset U$.  Since $K\subset V_{j+1}\subset\operatorname{int}V_j$, each
$V_j$ is a compact neighborhood of $K$, and every neighborhood of $K$
contains an open set $U\supset K$ and hence some $V_{j_0}$.  Thus the
$V_j$ are cofinal.

(b) Each $E_j$ is open with $E_j\subset E_{j+1}$, and
$\bigcup_jE_j=S^n\setminus\bigcap_jV_j=M$.  Given compact
$D\subset M$, apply (a) to the open set
$U=S^n\setminus D\supset K$: some $V_j\subset S^n\setminus
D$, that is, $D\subset E_j$; and $E_j\subset C_j$ because
$\operatorname{int}V_j\subset V_j$.

(c) Each $C_j$ is compact, $C_j\subset C_{j+1}$ because
$\operatorname{int}V_{j+1}\subset\operatorname{int}V_j$, and $C_j\subset M$ because
$K\subset\operatorname{int}V_j$; by part~(b), the $C_j$ exhaust $M$.  Finally,
$V_{j+1}\subset\operatorname{int}V_j$ gives $C_j\cap V_{j+1}=\varnothing$, that
is, $C_j\subset E_{j+1}$; and $E_{j+1}$ is open in $S^n$, contained
in $M$ (as $K\subset V_{j+1}$) and in $C_{j+1}$, whence
$E_{j+1}\subset\operatorname{int}_M C_{j+1}$.
\end{proof}

The next lemma is where $H_1(\omega)=0$ enters the topology of
actual spaces: the universal-coefficient isomorphism is natural, so
the abstract algebra becomes a statement about restriction maps.

\begin{lemma}[Cohomology of the stages]\label{transfer}
$V_j\simeq S^1\vee S^1$, so
$\widetilde H^{\,r}(V_j;\mathbb{Z})=0$ for $r\ne1$ and
$H^1(V_j;\mathbb{Z})\cong\mathbb{Z}^2$; and every restriction
$(i_j)^*:H^1(V_j;\mathbb{Z})\to H^1(V_{j+1};\mathbb{Z})$ is zero.
\end{lemma}

\begin{proof}
Each $V_j$ is a closed regular neighborhood of the embedded rank-two rose
$\Gamma_j$, and a regular neighborhood collapses onto its spine
\cite[Theorem~3.16]{zbMATH01714503}.  The deformation retraction associated
with the collapse fixes $\Gamma_j$ pointwise, so the inclusion
$\Gamma_j\hookrightarrow V_j$ is a homotopy equivalence.  Since $\Gamma_j$
is PL homeomorphic to the standard rose, $V_j\simeq S^1\vee S^1$.  Thus
$H^1(V_j;\mathbb{Z})\cong\mathbb{Z}^2$,
$\widetilde H^{\,0}(V_j;\mathbb{Z})=0$ by connectedness, and
$H^r(V_j;\mathbb{Z})=0$ for $r\ge2$.

For a path-connected space, the Hurewicz homomorphism
$h:\pi_1(X,x)\to  H_1(X;\mathbb Z)$ is the abelianization and is
natural in based maps.  Write $h_j$ for
this map on $V_j$.  Since $H_1(V_j;\mathbb Z)$ is abelian, the composite
$h_j\circ\mathfrak{m}_j$ kills $[F_2,F_2]$ and descends to the isomorphism
\(\overline{\mathfrak{m}}_j:\mathbb Z^2=F_2^{\mathrm{ab}}
\xrightarrow{\ \cong\ }H_1(V_j;\mathbb Z)\)
induced on abelianizations by the group isomorphism $\mathfrak{m}_j$.  Apply
$h_j$ to the marking identity
$(i_j)_*\circ\mathfrak{m}_{j+1}=\mathfrak{m}_j\circ\omega$ of
Proposition~\ref{induction}(4): by naturality of $h$, the left side becomes
$(i_j)_*\circ h_{j+1}\circ\mathfrak{m}_{j+1}$, and both sides factor through the
surjection $F_2\to F_2^{\mathrm{ab}}$, giving
\((i_j)_*\circ\overline{\mathfrak{m}}_{j+1}=\overline{\mathfrak{m}}_j\circ H_1(\omega)\),
where $H_1(\omega):\mathbb{Z}^2\to\mathbb{Z}^2$ is the map induced by
$\omega$ on abelianizations.  By Lemma~\ref{lem:endo}, the images
$\omega(a)$ and $\omega(b)$ are commutators, so $H_1(\omega)=0$;
since $\overline{\mathfrak{m}}_{j+1}$ is surjective, $(i_j)_*=0$ on
$H_1(V_{j+1};\mathbb{Z})$.

The universal-coefficient sequence
\[
 0\to\operatorname{Ext}\bigl(H_0(X;\mathbb{Z}),\mathbb{Z}\bigr)
 \to H^1(X;\mathbb{Z})
 \xrightarrow{\ \mathrm{ev}\ }
 \operatorname{Hom}\bigl(H_1(X;\mathbb{Z}),\mathbb{Z}\bigr)\to0
\]
is natural in $X$; only the naturality of the
evaluation map $\mathrm{ev}$ is used, not the (non-natural) splitting.
Since $H_0(V_j;\mathbb{Z})$ and $H_0(V_{j+1};\mathbb{Z})$ are free, both
$\operatorname{Ext}$ terms vanish and $\mathrm{ev}$ is an isomorphism for
both spaces; under these identifications $(i_j)^*$ is precomposition with
$(i_j)_*$,
\[
 (i_j)^*=\operatorname{Hom}\bigl((i_j)_*,\mathbb{Z}\bigr):
 \operatorname{Hom}\bigl(H_1(V_j;\mathbb{Z}),\mathbb{Z}\bigr)
 \longrightarrow
 \operatorname{Hom}\bigl(H_1(V_{j+1};\mathbb{Z}),\mathbb{Z}\bigr).
\]
Since $(i_j)_*=0$, it follows that $(i_j)^*=0$.
\end{proof}

We next compute the cohomology groups of $K$.  For a compact subset
$A\subset S^n$, we use the standard neighborhood realization of \v{C}ech
cohomology
\[
 \check H^{\,r}(A;\mathbb Z)
 \cong
 \varinjlim_{U\supset A}H^r(U;\mathbb Z),
\]
where $U$ ranges over the open neighborhoods of $A$ in $S^n$, ordered by
reverse inclusion, and the transition homomorphisms are restrictions; see
\cite[p.~257]{zbMATH02103273}.  Thus the relevant construction is a
direct limit, and no inverse-limit derived term occurs.

\begin{proposition}[\v{C}ech-acyclicity of the limit continuum]\label{prop:cech}
$\widetilde{\check H}^{\,r}(K;\mathbb{Z})=0$ for all $r\ge0$.
\end{proposition}

\begin{proof}
Let $i_j:V_{j+1}\hookrightarrow V_j$ denote the bonding inclusion and
let $u_j:\operatorname{int}V_{j+1}\hookrightarrow
\operatorname{int}V_j$ be its restriction.  By construction,
$K\subset\operatorname{int}V_j$ for every $j$, and
Lemma~\ref{cofinal}\textup{(a)} shows that the decreasing sequence
$(\operatorname{int}V_j)_{j\geq1}$ is cofinal among the open neighborhoods
of $K$: if $U\supset K$ is open, then some $V_j\subset U$, and hence
$\operatorname{int}V_j\subset U$.  Since $V_j$ is a compact PL manifold
with boundary, the PL collaring theorem
\cite[Corollary~2.5]{zbMATH01714503} allows its boundary to be pushed
inward along a collar.  It follows that the inclusion
$\lambda_j:\operatorname{int}V_j\hookrightarrow V_j$ is a homotopy
equivalence.  Consequently the restriction
\(\rho_j^r:=\lambda_j^*:
H^r(V_j;\mathbb Z)\xrightarrow{\ \cong\ }
H^r(\operatorname{int}V_j;\mathbb Z)\)
is an isomorphism.  Functoriality of restriction gives
\(u_j^*\circ\rho_j^r=\rho_{j+1}^r\circ i_j^*\).
Thus the direct systems defined by the $V_j$ and by their interiors are
naturally isomorphic.  The neighborhood realization above and cofinality
therefore give
\[
\check H^{\,r}(K;\mathbb Z)
\cong
\varinjlim_j\Bigl(
H^{\,r}(V_1;\mathbb Z)\xrightarrow{\,i_1^*\,}
H^{\,r}(V_2;\mathbb Z)\xrightarrow{\,i_2^*\,}\cdots\Bigr).
\]
By Lemma~\ref{transfer}, every $V_j$ is connected, $H^{\,r}(V_j;\mathbb Z)=0$ for $r\geq2$, and $i_j^*:H^1(V_j;\mathbb Z)\to H^1(V_{j+1};\mathbb Z)$ is the zero
homomorphism. For $r\geq2$ the direct system consists of zero groups, so
$\check H^{\,r}(K;\mathbb Z)=0$.  For $r=1$ the system is
$\mathbb Z^2\xrightarrow{0}\mathbb Z^2\xrightarrow{0}\cdots$, and its
colimit is zero: the class in the colimit of an element at the $j$th stage
equals the class of its image at the $(j+1)$st stage, which is zero.  Thus
$\check H^{\,1}(K;\mathbb Z)=0$.

For $r=0$, Proposition~\ref{prop:limitset} shows that $K$ is nonempty and
connected.  Since $\check H^{\,0}(K;\mathbb Z)$ is the group of locally
constant integer-valued functions on $K$, and on a connected space these
are the constants, we have $\check H^{\,0}(K;\mathbb Z)\cong\mathbb Z$
and $\widetilde{\check H}^{\,0}(K;\mathbb Z)=0$.  Combining the three
cases proves $\widetilde{\check H}^{\,r}(K;\mathbb Z)=0$ for every
$r\geq0$.
\end{proof}

Proposition~\ref{prop:cech} records an intrinsic property of the limit
continuum $K$, but it is not needed for Theorem~\ref{thm:counterexample-psc}: the proof of
Proposition~\ref{prop:acyclic} below instead uses natural Alexander
duality at the finite stages and then passes to the direct limit.
Alternatively, compactum Alexander duality
\cite[pp.~255--257]{zbMATH02103273}, applied directly to
Proposition~\ref{prop:cech}, would also imply the acyclicity of $M$.  We
retain the finite-stage argument because it makes the transition maps
explicit.

In the next lemma the content is the naturality, not the isomorphism:
under the Alexander-duality identifications, the complement map
corresponds, up to the conventional sign, to the zero restriction map.
Thus the complements inherit vanishing transitions from
Lemma~\ref{transfer}.

\begin{lemma}[Natural Alexander duality]\label{duality}
Let
\(i_j:V_{j+1}\hookrightarrow V_j\)
and
\(\jmath_j:E_j\hookrightarrow E_{j+1}\)
denote the inclusions.  Fix an orientation class
$\mathfrak o\in H_n(S^n;\mathbb Z)$.  Alexander duality gives isomorphisms
\[
D_{j,i}:
\widetilde H_i(E_j;\mathbb Z)
\longrightarrow
\widetilde H^{\,n-i-1}(V_j;\mathbb Z),
\qquad i\geq0,
\]
such that, for signs $\varepsilon_i\in\{\pm1\}$ depending only on the
degree and the cap-product convention,
\[
D_{j+1,i}\circ(\jmath_j)_*
=
\varepsilon_i\,i_j^*\circ D_{j,i}.
\]
Consequently,
\begin{equation}\label{Ej-homology}
\widetilde H_i(E_j;\mathbb Z)\cong
\begin{cases}
\mathbb Z^2,&i=n-2,\\
0,&i\neq n-2,
\end{cases}
\end{equation}
and every transition
$(\jmath_j)_*:
\widetilde H_i(E_j;\mathbb Z)\to\widetilde H_i(E_{j+1};\mathbb Z)$
vanishes for $i\geq1$.
\end{lemma}

\begin{proof}
Each $V_j$ is a nonempty, locally contractible compact polyhedron and a
proper subset of $S^n$.  Thus the reduced-complement form of Alexander
duality \cite[Corollary~3.45]{zbMATH02103273} applies and, together with
Lemma~\ref{transfer}, gives
\eqref{Ej-homology}, where reduced cohomology in negative degrees is
understood to vanish.  In particular, taking $i=0$, each $E_j$ is
path-connected.  No naturality is needed for this step, as the groups are
computed from the isomorphism statement alone.

Fix $0\leq i\leq n-2$.  Since $n-i-1\geq1$, ordinary and reduced
cohomology agree canonically in the degree under consideration; we use
this identification below.  Since
$E_j=S^n\setminus V_j\subset E_{j+1}$ and $E_j$ is nonempty, the reduced
long exact sequence of the pair $(S^n,E_j)$ contains
\[
\widetilde H_{i+1}(S^n)\longrightarrow
H_{i+1}(S^n,E_j;\mathbb Z)
\xrightarrow{\ \partial_{j,i}\ }
\widetilde H_i(E_j;\mathbb Z)\longrightarrow
\widetilde H_i(S^n).
\]
Since both outer groups vanish, $\partial_{j,i}$ is an isomorphism.  The
cap-product construction in the proof of the compact-set pair-duality
theorem, together with the cap-product naturality formula
\cite[proof of Theorem~3.44 and pp.~248--249]{zbMATH02103273},
supplies an isomorphism
\[
\bar\gamma_{\mathfrak o,j,i+1}:
H_{i+1}(S^n,E_j;\mathbb Z)
\longrightarrow
\bar H^{\,n-i-1}(V_j;\mathbb Z),
\]
determined by the fixed class $\mathfrak o$ and natural with respect to
inclusions of compact pairs.  Here
\[
 \bar H^{\,r}(V_j;\mathbb Z)
 :=\varinjlim_{U\supset V_j}H^r(U;\mathbb Z),
\]
where the open neighborhoods of $V_j$ are ordered by reverse
inclusion and, for $U'\subset U$, the transition homomorphism is the
restriction $H^r(U;\mathbb Z)\to H^r(U';\mathbb Z)$.

Choose a finite triangulation $\mathcal S_j^0$ of $S^n$ in which $V_j$ is the
polyhedron of a subcomplex $\mathcal{L}_j^0$.  Put
\[
 \mathcal S_j:=\operatorname{sd}\mathcal S_j^0,
 \qquad
 \mathcal{L}_j:=\operatorname{sd}\mathcal{L}_j^0.
\]
Then $|\mathcal{L}_j|=V_j$, and $\mathcal{L}_j$ is full in $\mathcal S_j$.
Indeed, a simplex of $\operatorname{sd}\mathcal S_j^0$ is represented by a
chain of simplices of $\mathcal S_j^0$; if all its vertices belong to
$\operatorname{sd}\mathcal{L}_j^0$, every simplex in the chain belongs
to $\mathcal{L}_j^0$.  The subcomplex $\mathcal{L}_j$ is nonempty
because $\Gamma_j\subset V_j$, and it is proper because
$V_j\subset\operatorname{int}B\subsetneq S^n$.

Let $f_{\mathcal{L}_j}:|\mathcal S_j|=S^n\to [0,1]$ be the level
function of Lemma~\ref{lem:sublevel}; thus
$f_{\mathcal{L}_j}^{-1}(0)=V_j$.  For $k\geq1$, set
$c_k:=2^{-k},\qquad N_k:=f_{\mathcal{L}_j}^{-1}([0,c_k]).$
Lemma~\ref{lem:sublevel} shows that every $N_k$ is a closed regular
neighborhood of $V_j$, with
$\operatorname{int}N_k=f_{\mathcal{L}_j}^{-1}([0,c_k))$.
Since $c_{k+1}<c_k$,
\(V_j\subset\operatorname{int}N_{k+1}
\subset N_{k+1}
\subset\operatorname{int}N_k
\subset N_k\).
Moreover,
\(\bigcap_{k\geq1}N_k
=f_{\mathcal{L}_j}^{-1}\!\left(\bigcap_{k\geq1}[0,c_k]\right)
=f_{\mathcal{L}_j}^{-1}(0)=V_j\).

The open sets $\operatorname{int}N_k$ are cofinal among the open
neighborhoods of $V_j$.  To see this, let $U$ be such a neighborhood.
If $U\neq S^n$, then $S^n\setminus U$ is a nonempty compact set
disjoint from $f_{\mathcal{L}_j}^{-1}(0)$, and hence
$m_U:=\min_{S^n\setminus U}f_{\mathcal{L}_j}>0$.
For every $k$ with $c_k<m_U$, one has
$N_k\subset U$, and therefore $\operatorname{int}N_k\subset U$; the
case $U=S^n$ is immediate.

Each $N_k$ collapses onto $V_j$ by
\cite[Theorem~3.16]{zbMATH01714503}, so
$V_j\hookrightarrow N_k$ is a homotopy equivalence.  Since $N_k$ is a
compact PL manifold with boundary, a PL collar of $\partial N_k$
\cite[Corollary~2.5]{zbMATH01714503} shows that
$\operatorname{int}N_k\hookrightarrow N_k$ is also a homotopy
equivalence.  The two-out-of-three property applied to
\(V_j\hookrightarrow\operatorname{int}N_k\hookrightarrow N_k\)
therefore shows that, for every $r\geq0$, restriction induces an
isomorphism
\(\rho_k^r:
H^r(\operatorname{int}N_k;\mathbb Z)
\xrightarrow{\ \cong\ }
H^r(V_j;\mathbb Z)\).
For $k\geq1$, denote the transition restriction by
\(\operatorname{res}_{k,k+1}^r:
H^r(\operatorname{int}N_k;\mathbb Z)
\to H^r(\operatorname{int}N_{k+1};\mathbb Z)\).
Functoriality gives
\(\rho_{k+1}^r\circ\operatorname{res}_{k,k+1}^r=\rho_k^r\).
Since $\rho_k^r$ and $\rho_{k+1}^r$ are isomorphisms,
$\operatorname{res}_{k,k+1}^r$ is an isomorphism.  Cofinality
identifies $\bar H^{\,r}(V_j;\mathbb Z)$ with the colimit of the
sequential system
\(\bigl(H^r(\operatorname{int}N_k;\mathbb Z),
\operatorname{res}_{k,k+1}^r\bigr)_{k\geq1}\).
Under this
identification, the canonical restriction map
\(\tau_j^r:
\bar H^{\,r}(V_j;\mathbb Z)
\to  H^r(V_j;\mathbb Z)\)
is induced by the compatible family $(\rho_k^r)_k$ and is therefore
an isomorphism.  Taking $r=n-i-1$, write this map as
\(\tau_j:
\bar H^{\,n-i-1}(V_j;\mathbb Z)
\xrightarrow{\ \cong\ }
H^{\,n-i-1}(V_j;\mathbb Z)\).
Define
\[
D_{j,i}:=\tau_j\circ\bar\gamma_{\mathfrak o,j,i+1}\circ\partial_{j,i}^{-1},
\qquad 0\leq i\leq n-2.
\]
For $i\geq n-1$ both groups in the statement vanish, and we set
$D_{j,i}:=0$, the unique homomorphism, which is an isomorphism of zero
groups.

Let $\kappa_j:(S^n,E_j)\to(S^n,E_{j+1})$ denote the map
of pairs that is the identity on $S^n$, and let
$\bar\imath_j^{\;*}:\bar H^{\,r}(V_j;\mathbb Z)\to
\bar H^{\,r}(V_{j+1};\mathbb Z)$ be the canonical restriction: every
neighborhood of $V_j$ is a neighborhood of $V_{j+1}$.  For
$0\leq i\leq n-2$,
\begin{align*}
D_{j+1,i}\circ(\jmath_j)_*
&=\tau_{j+1}\circ\bar\gamma_{\mathfrak o,j+1,i+1}\circ\partial_{j+1,i}^{-1}
  \circ(\jmath_j)_*\\
&=\tau_{j+1}\circ\bar\gamma_{\mathfrak o,j+1,i+1}\circ(\kappa_j)_*
  \circ\partial_{j,i}^{-1}
  &&\text{(naturality of $\partial$)}\\
&=\varepsilon_i\,\tau_{j+1}\circ\bar\imath_j^{\;*}\circ\bar\gamma_{\mathfrak o,j,i+1}
  \circ\partial_{j,i}^{-1}
  &&\text{(cap-product naturality, same $\mathfrak o$)}\\
&=\varepsilon_i\,i_j^*\circ\tau_j\circ\bar\gamma_{\mathfrak o,j,i+1}\circ\partial_{j,i}^{-1}
  &&\text{(naturality of $\tau$)}\\
&=\varepsilon_i\,i_j^*\circ D_{j,i}.
\end{align*}
Here the first nontrivial equality is
$\partial_{j+1,i}\circ(\kappa_j)_*=(\jmath_j)_*\circ\partial_{j,i}$,
the naturality of connecting homomorphisms for the map of pairs
$\kappa_j$.  For $i\geq n-1$ every square consists of zero groups and
commutes trivially.  This proves the signed naturality identity for all
$i\geq0$.  Replacing $\mathfrak o$ by $-\mathfrak o$ multiplies
$\bar\gamma_{\mathfrak o}$, and hence both
vertical maps in each square, by $-1$, and therefore affects neither the
signed identity nor the conclusions below.

For $i\geq1$ with $i\neq n-2$ the source of $(\jmath_j)_*$ vanishes by
\eqref{Ej-homology}.  In the remaining degree the signed naturality
identity gives
\((\jmath_j)_*
=\varepsilon_{n-2}\,D_{j+1,n-2}^{-1}\circ i_j^*\circ D_{j,n-2}\),
and $i_j^*:H^1(V_j;\mathbb Z)\to H^1(V_{j+1};\mathbb Z)$ is zero by
Lemma~\ref{transfer}.  Hence $(\jmath_j)_*=0$ in every degree
$i\geq1$.
\end{proof}

\begin{proposition}[Acyclicity of the complement]\label{prop:acyclic}
For every $i\geq0$, $\widetilde H_i(M;\mathbb Z)=0$.
\end{proposition}

\begin{proof}
By Lemma~\ref{cofinal}(b), the sets $E_j$ form an
increasing sequence such that $M=\bigcup_{j\geq1}E_j$,
and every compact subset of $M$ is contained in some $E_j$.  Hence the
inclusions $E_j\hookrightarrow M$ and
$\jmath_j:E_j\hookrightarrow E_{j+1}$ induce a canonical isomorphism
\begin{equation*}
\Theta_i:
\varinjlim_j
\bigl(H_i(E_j;\mathbb Z),(\jmath_j)_*\bigr)
\longrightarrow
H_i(M;\mathbb Z)
\end{equation*}
for every $i\geq0$, by the direct-limit theorem for singular homology
\cite[Proposition~3.33]{zbMATH02103273}.

Now fix $i\geq1$.  In positive degrees reduced and unreduced homology
agree, and Lemma~\ref{duality} states that every transition
$(\jmath_j)_*: H_i(E_j;\mathbb Z)\to H_i(E_{j+1};\mathbb Z)$
is zero.  Consequently, the direct limit vanishes.  Hence
$H_i(M;\mathbb Z)=0$ for every $i\geq1$.

It remains to treat degree zero.  By \eqref{Ej-homology},
$\widetilde H_0(E_j;\mathbb Z)=0$, so every $E_j$ is path-connected.
Choose $p\in E_1$.  The inclusions are nested, so $p\in E_j$ for every
$j$, and $H_0(E_j;\mathbb Z)=\mathbb Z[p]$.
Since $\jmath_j$ fixes $p$, the induced map $(\jmath_j)_*$ sends $[p]$
to $[p]$.  The degree-zero direct system is therefore literally
\[
\mathbb Z[p]\xrightarrow{\mathrm{id}}\mathbb Z[p]
\xrightarrow{\mathrm{id}}\mathbb Z[p]
\xrightarrow{\mathrm{id}}\cdots,
\]
and its colimit is $\mathbb Z[p]$.  The canonical isomorphism $\Theta_0$
is induced by inclusions fixing $p$, so it identifies
$H_0(M;\mathbb Z)$ with $\mathbb Z[p]$.  The
augmentation $\varepsilon:H_0(M;\mathbb Z)\to\mathbb Z$
sends $[p]$ to $1$ and is therefore an isomorphism.  Hence its kernel,
which by definition is $\widetilde H_0(M;\mathbb Z)$, vanishes.  In
particular, $M$ is path connected.

Combining degree zero with the positive-degree calculation proves
$\widetilde H_i(M;\mathbb Z)=0$
for every $i\geq0$.

\end{proof}

The assumption $n\geq4$ enters sharply in the following lemma: pushing
a singular $2$-disk off a $1$-dimensional spine by
general position requires $2+1-n<0$.  We use the lemma to prove that
$M$ is contractible.

\begin{lemma}[Simply connected finite stages]\label{lem:sc}
Assume $n\geq4$.  For every $j$, the spaces
$S^n\setminus\Gamma_j$ and $C_j$ are path-connected and
simply connected.
\end{lemma}

\begin{proof}
Fix $j$ and set $X_j:=S^n\setminus\Gamma_j$.  All distances are
measured in the round metric on $S^n$.

We first record a compactly controlled retraction construction.  By
Proposition~\ref{induction}(6), there are a triangulation $\mathcal{T}_j$
of $S^n$ in which $\Gamma_j$ is the polyhedron of a nonempty, proper, full
subcomplex, the simplicial map
$f_j:|\mathcal{T}_j|=S^n\to[0,1]$ taking value $0$ on the vertices of
this subcomplex and value $1$ on all other vertices, and a number
$t_j\in(0,1)$ such that
$f_j^{-1}(0)=\Gamma_j,\qquad V_j=f_j^{-1}([0,t_j]),$
and, for every $c\in(0,1)$, the sublevel set
$N_c:=f_j^{-1}([0,c])$ is a closed regular neighborhood of $\Gamma_j$
with boundary $f_j^{-1}(c)$.  Lemma~\ref{lem:sublevel} also gives
$\operatorname{int}N_c=f_j^{-1}([0,c))
 \qquad(0<c<1).$
In particular,
$\operatorname{int}V_j=f_j^{-1}([0,t_j)),
 \qquad
 C_j=f_j^{-1}([t_j,1]).$
Fix $0<\varepsilon<t_j$, so that
$N_\varepsilon\subset f_j^{-1}([0,t_j))=\operatorname{int}V_j$, with
$\partial N_\varepsilon=f_j^{-1}(\varepsilon)$ and
$\partial V_j=f_j^{-1}(t_j)$.  Put
$W_\varepsilon:=f_j^{-1}\bigl([\varepsilon,t_j]\bigr),$ the compact slab bounded
by the levels $\varepsilon$ and $t_j$.  Apply
the homeomorphism $\Phi_{t_j}$ of
\eqref{eq:punctured-sublevel-product} to $f_j$.  Since its second
coordinate at $z$ is $f_j(z)/t_j$, its restriction gives a
homeomorphism of triples
\[
\bigl(W_\varepsilon;\,\partial N_\varepsilon,\,\partial V_j\bigr)
\cong
\bigl(\partial V_j\times[\varepsilon/t_j,1];\,
\partial V_j\times\{\varepsilon/t_j\},\,
\partial V_j\times\{1\}\bigr).
\]
Transporting the linear contraction of the second coordinate onto $1$
through this homeomorphism gives a strong deformation retraction of
$W_\varepsilon$ onto $\partial V_j$, fixing $\partial V_j$ pointwise.
Since
\[
S^n\setminus\operatorname{int}N_\varepsilon
=f_j^{-1}\bigl([\varepsilon,1]\bigr)
=W_\varepsilon\cup C_j,
\qquad
W_\varepsilon\cap C_j=f_j^{-1}(t_j)=\partial V_j,
\]
we may glue this homotopy to the identity homotopy on $C_j$: the two
agree on the overlap $(W_\varepsilon\cap C_j)\times I=\partial V_j\times I$,
because the first fixes $\partial V_j$ pointwise at every time; and
$W_\varepsilon\times I$ and $C_j\times I$ are closed in
$(S^n\setminus\operatorname{int}N_\varepsilon)\times I$, being
products of preimages of closed intervals under $f_j$ with $I$, so the
pasting lemma yields a continuous homotopy
$(R_{\varepsilon,u})_{u\in[0,1]}$ on
$S^n\setminus\operatorname{int}N_\varepsilon$.  It satisfies
\[
 R_{\varepsilon,0}=\mathrm{id},
 \qquad
 R_{\varepsilon,u}|_{C_j}=\mathrm{id}_{C_j}\quad(0\leq u\leq1),
 \qquad
 R_{\varepsilon,1}(W_\varepsilon)=\partial V_j\subset C_j.
\]
Hence $(R_{\varepsilon,u})_{u\in[0,1]}$ is a strong deformation
retraction of $S^n\setminus\operatorname{int}N_\varepsilon$ onto $C_j$; denote
its terminal retraction by $r_\varepsilon$, which fixes $C_j$
pointwise.  Finally, if $D\subset X_j$ is compact, then $f_j>0$ on $D$,
because $D\cap f_j^{-1}(0)=D\cap\Gamma_j=\varnothing$; compactness
gives $m_D:=\min_{x\in D}f_j(x)>0$, and any
$0<\varepsilon<\min\{t_j,m_D\}$ yields
$D\cap N_\varepsilon=\varnothing$, hence
$D\subset S^n\setminus\operatorname{int}N_\varepsilon$.

We first prove path connectivity of $X_j$.  Let $p_0,p_1\in X_j$; if
$p_0=p_1$, there is nothing to prove.  Otherwise choose a PL arc
$\zeta:I\to S^n$ with $\zeta(0)=p_0$ and $\zeta(1)=p_1$, and put
$A:=\zeta(\partial I)=\{p_0,p_1\}$.  Since $A\subset X_j$, one has
$A\cap\Gamma_j=\varnothing$.  Pass to a common subdivision in which
$A\subset\zeta(I)$ and $\Gamma_j$ are subpolyhedra, and apply the
general-position theorem \cite[Theorem~4.2]{zbMATH01714503} in the
ambient PL manifold $S^n$, with
$X=\zeta(I),\qquad X_0=A,\qquad Y=\Gamma_j,$
and with the constant control function $\eta_0\equiv1$.  Since
$\dim(\zeta(I)\setminus A)=1$ and $\dim\Gamma_j=1$, it gives an ambient
PL isotopy $\Psi_t:S^n\to S^n$, fixed on $A$, such that
\[
\begin{aligned}
 \dim\bigl(\Psi_1(\zeta(I)\setminus A)\cap\Gamma_j\bigr)
 &\leq\dim(\zeta(I)\setminus A)+\dim\Gamma_j-n\\
 &\leq1+1-n<0.
\end{aligned}
\]
Hence $\Psi_1(\zeta(I)\setminus A)\cap\Gamma_j=\varnothing$.  Because
$\Psi_1|_A=\mathrm{id}_A$ and $A\cap\Gamma_j=\varnothing$, we obtain
$\Psi_1(\zeta(I))\cap\Gamma_j=\varnothing$.  Thus
$\widetilde\zeta:=\Psi_1\circ\zeta$ is an arc in $X_j$ from $p_0$ to $p_1$.
This part uses only $1+1-n<0$, hence only $n\geq3$.

We next prove simple connectivity of $X_j$.  Let
$\gamma_0:S^1\to X_j$ be a loop and set
$d:=\operatorname{dist}\bigl(\gamma_0(S^1),\Gamma_j\bigr)>0$, by compactness
and disjointness.  Choose triangulations of \(S^1\) and \(S^n\), with
\(\Gamma_j\) a subcomplex of the latter.  After subdivision, the
controlled simplicial-approximation theorem
\cite[p.~222]{zbMATH01714503},
applied with control $d/2$, gives a free homotopy
$\mathcal A:S^1\times[0,1]\to S^n$ from $\gamma_0$ to a simplicial
(hence PL) loop $\gamma$.  Since every point track of $\mathcal A$ has
diameter less than $d/2$, the homotopy lies in $X_j$, and
$\operatorname{dist}\bigl(\gamma(S^1),\Gamma_j\bigr)\geq d/2>0$.

Since $n\geq2$, $\pi_1(S^n)=0$, so $\gamma$ extends to a
continuous map $F_0:D^2\to S^n$ with
$F_0|_{\partial D^2}=\gamma$.  With respect to the triangulations produced by the preceding step,
$\gamma$ is already simplicial. After extending the triangulation of $S^1$
over $D^2$, Zeeman's relative simplicial approximation theorem
\cite{zbMATH03195022} deforms
$F_0$, relative to $\partial D^2$, to a PL map
$F:D^2\to S^n$ such that $F|_{\partial D^2}=\gamma$.  No control on this
homotopy is needed, because only the terminal map is used.  Set
$Q:=F(D^2),\qquad Q_0:=\gamma(S^1).$
Then $Q\supseteq Q_0$, with both sets compact polyhedra,
$\dim(Q\setminus Q_0)\leq2$, and
$Q_0\cap\Gamma_j=\varnothing$.  Pass to a common subdivision of the
ambient PL structure in which $Q_0\subset Q$ and $\Gamma_j$ are
subpolyhedra.  If $Q=Q_0$, then
$Q\cap\Gamma_j=\varnothing$ already and we take
$\Psi_t:=\mathrm{id}$ for all $t\in[0,1]$;
otherwise apply \cite[Theorem~4.2]{zbMATH01714503} in the ambient PL
manifold $S^n$, with $X=Q$, $X_0=Q_0$, $Y=\Gamma_j$, and with the
constant control function $\eta_0\equiv1$.  It gives an ambient PL isotopy $\Psi_t$ of
$S^n$ which, being a push rel $Q_0$, fixes $Q_0$ pointwise
\cite[\S4]{zbMATH01714503}, and satisfies
\[
\begin{aligned}
 \dim\bigl(\Psi_1(Q\setminus Q_0)\cap\Gamma_j\bigr)
 &\leq\dim(Q\setminus Q_0)+\dim\Gamma_j-n\\
 &\leq2+1-n<0.
\end{aligned}
\]
Hence $\Psi_1(Q\setminus Q_0)\cap\Gamma_j=\varnothing$.
This is precisely where the hypothesis $n\geq4$ is used.  Since
$\Psi_1(Q)=\Psi_1(Q\setminus Q_0)\cup Q_0$ and
$Q_0\cap\Gamma_j=\varnothing$, we
conclude $\Psi_1(Q)\cap\Gamma_j=\varnothing$.  Consequently,
$G:=\Psi_1\circ F:D^2\to X_j$ is well defined and satisfies
$G|_{\partial D^2}=\Psi_1\circ\gamma=\gamma$.

The controlled homotopy $\mathcal A$ constructed above takes values in
$X_j$.  Identify $\partial D^2$ with $S^1\times\{1\}$ via the
identification $S^1=\partial D^2$ used in
$F|_{\partial D^2}=\gamma$.  The maps $\mathcal A$ and $G$ then agree on
their common domain.
They therefore paste to a continuous map
\[
 \bigl(S^1\times[0,1]\bigr)
 \cup_{\,S^1\times\{1\}=\partial D^2}D^2
 \longrightarrow X_j.
\]
The adjunction space on the left is homeomorphic to $D^2$, with
boundary corresponding to $S^1\times\{0\}$, and the restriction of the
pasted map to that boundary is $\gamma_0$.  Hence $\gamma_0$ is
null-homotopic in $X_j$.  Since $\gamma_0$ was arbitrary,
$\pi_1(X_j)$ is trivial; together with path connectivity, this proves
that $X_j$ is simply connected.

Finally, let $i:C_j\hookrightarrow X_j$ be the inclusion.  The
simplicial function $f_j$ attains the value $1$: it equals $1$ at
every vertex outside the full subcomplex triangulating $\Gamma_j$, and
such vertices exist, since otherwise fullness would force
$\Gamma_j=S^n$, which is impossible because
$\dim\Gamma_j=1<n$.  Hence
$\varnothing\neq f_j^{-1}(1)\subset f_j^{-1}([t_j,1])=C_j;$
fix a basepoint $x_0\in C_j$.

Every compact subset of $X_j$ lies in
$S^n\setminus\operatorname{int}N_\varepsilon$ for some
$0<\varepsilon<t_j$.  By the retraction construction established at the
beginning of this proof, $(R_{\varepsilon,u})_{u\in[0,1]}$ is a strong
deformation retraction of
$S^n\setminus\operatorname{int}N_\varepsilon$ onto $C_j$, with terminal
retraction $r_\varepsilon$, and
$R_{\varepsilon,u}|_{C_j}=\mathrm{id}_{C_j}$ for every $u\in[0,1]$.
In particular, $R_{\varepsilon,u}(x_0)=x_0$ for every $u\in[0,1]$.

Applying these retractions to representatives and homotopies, whose
images are compact, yields the following.  If
$g:(S^k,*)\to(X_j,x_0)$ is a based map with
$g(S^k)\subset S^n\setminus\operatorname{int}N_\varepsilon$, then
$(z,u)\mapsto R_{\varepsilon,u}(g(z))$ is a based homotopy in $X_j$
from $g$ to $r_\varepsilon\circ g$, which maps into $C_j$; hence
$i_*:\pi_k(C_j,x_0)\to\pi_k(X_j,x_0)$ is surjective for every
$k\geq1$.  If a based map $h:(S^k,*)\to(C_j,x_0)$ is null-homotopic in
$X_j$, composing the null-homotopy with $r_\varepsilon$, for
$\varepsilon$ so small that its compact image lies in
$S^n\setminus\operatorname{int}N_\varepsilon$, gives a based
null-homotopy of $r_\varepsilon\circ h=h$ in $C_j$; hence $i_*$ is
injective.  The same two arguments, applied to points and paths, show
that the inclusion induces a bijection
$\pi_0(C_j)\to\pi_0(X_j)$ of sets of path components.  Since $X_j$ is
path-connected and $C_j$ is nonempty, this bijection implies that
$C_j$ is path-connected.  The already established isomorphism
$\pi_1(C_j,x_0)\to\pi_1(X_j,x_0)$ then implies that $C_j$ is simply
connected, because $X_j$ is simply connected.
\end{proof}

\begin{proposition}[Contractibility]\label{prop:contract}
The manifold $M$ is contractible.
\end{proposition}

\begin{proof}
Since $M=\bigcup_{j\geq1}C_j$
and every $C_j$ is path-connected by Lemma~\ref{lem:sc}, the space
$M$ is path-connected.  The image of every loop in $M$ is compact and
therefore lies in some $C_j$ by Lemma~\ref{cofinal}(b).  Since $C_j$ is
simply connected, the loop contracts in $C_j\subset M$.  Hence
$\pi_1(M)=0$.  Proposition~\ref{prop:acyclic} gives
$\widetilde H_k(M;\mathbb Z)=0$ for $k\geq0$.  If some higher
homotopy group of $M$ were nonzero, let $k\geq2$ be the least index
with $\pi_k(M)\ne0$.  Then $M$ would be $(k-1)$-connected, and the
Hurewicz theorem \cite[Theorem~4.32]{zbMATH02103273} would give an isomorphism
$\pi_k(M)\cong H_k(M;\mathbb Z)=0$, a contradiction.  Thus $M$ is
weakly contractible.  Since a smooth manifold has the homotopy type
of a CW complex, choose a CW complex \(Z_M\) and a homotopy equivalence
$Z_M\to M$.  Then the constant map $Z_M\to *$ induces isomorphisms on every
homotopy group, including \(\pi_0\).  Whitehead's theorem
\cite[Theorem~4.5]{zbMATH02103273} makes this map a homotopy equivalence.
Thus \(Z_M\), and hence $M$, is contractible.
\end{proof}

We now distinguish $M$ from $\mathbb R^n$ by examining its topology
at infinity.

\begin{proposition}[Not simply connected at infinity]\label{prop:sci}
The manifold $M$ is not simply connected at infinity.  Consequently,
$M\not\cong\mathbb R^n$.
\end{proposition}

The proof constructs, beyond each compact subset of $M$, a loop on the
boundary of a sufficiently deep neighborhood $V_k$.  The injectivity of
the bonding endomorphism $\omega$ implies that this loop remains
essential in the fixed outer neighborhood $V_1$.  Thus these loops cannot
be contracted outside a fixed compact subset of $M$.

\begin{proof}
Our convention includes one-endedness in simple connectivity at infinity.
Only the following necessary loop condition is used:
for every compact set \(C\subset X\), there is a compact set
\(D\subset X\), with \(C\subset D\), such that every loop in
\(X\setminus D\) is null-homotopic in \(X\setminus C\).  To prove that
\(M\) is not simply connected at infinity, it is enough to exhibit one
compact set \(C\subset M\) for which no such \(D\) exists.

Set $C:=C_1=S^n\setminus\operatorname{int}V_1$.  Since
$K\subset V_2\subset\operatorname{int}V_1$, the sets $C$ and $K$ are
disjoint.  Hence $C$ is a compact subset of $M=S^n\setminus K$.
Moreover, $M\setminus C=\operatorname{int}V_1\setminus K$.
Let \(D\subset M\) be an arbitrary compact set containing \(C\).
By Lemma~\ref{cofinal}(b), there is an index \(j\) such that
$D\subset S^n\setminus V_j$.
Because the sets $S^n\setminus V_j$ are increasing, we may
increase \(j\), if necessary, and choose \(k\geq2\) such that
$D\subset S^n\setminus V_k$.  Equivalently, $D\cap V_k=\varnothing$.
In particular, $\partial V_k\cap D=\varnothing$; below we construct the
required loop on $\partial V_k$.

Let $\alpha_k:(S^1,*)\to(\Gamma_k,x)\subset(V_k,x)$ parametrize the
oriented $a$-petal of the rank-two rose $\Gamma_k$.  By
Proposition~\ref{induction}(5),
$[\alpha_k]=\mathfrak{m}_k(a)\in\pi_1(V_k,x).$
The point $x$ belongs to $K$, because $x\in V_j$ for every $j$.
Consequently, $\alpha_k(S^1)\cap K\neq\varnothing$, so $\alpha_k$ is not
a loop in $M$.

We therefore first use
PL general position to move \(\alpha_k\) away from the spine
\(\Gamma_k\), and then use the punctured-sublevel product structure of
the regular neighborhood \(V_k\) to move the resulting loop onto
\(\partial V_k\), which is disjoint from \(K\).

Set $Z:=\alpha_k(S^1),\qquad Z_0:=\varnothing.$
By Proposition~\ref{induction}(6), $\Gamma_k$ is the polyhedron of a
nonempty, proper, full subcomplex of $\mathcal{T}_k$.  Hence
Lemma~\ref{lem:sublevel} applies and gives
\(Z=\alpha_k(S^1)\subset\Gamma_k=f_k^{-1}(0)
\subset f_k^{-1}([0,t_k))=\operatorname{int}V_k\).
The space $W:=\operatorname{int}V_k$ is a PL $n$-manifold without
boundary and therefore equals its own manifold interior.  Moreover,
$Z\supset Z_0$ and $\Gamma_k$ are compact PL polyhedra in $W$, with
$p_k:=\dim(Z\setminus Z_0)\leq1,
 \qquad
 q_k:=\dim\Gamma_k=1.$

For clarity, we record the special case of the controlled PL
general-position theorem used here.  Let
$\mathsf Z\supset\mathsf Z_0$ and $\mathsf Y$ be compact polyhedra
contained in the interior of a PL $n$-manifold $\mathsf W$, and fix a
metric $d_{\mathsf W}$ inducing the topology of $\mathsf W$.
For every $\varepsilon>0$, there is an $\varepsilon$-push
$(\Psi_t)_{0\leq t\leq1}$ of $\mathsf Z$ in $\mathsf W$, relative
to $\mathsf Z_0$, such that
\[
 \dim\bigl(\Psi_1(\mathsf Z\setminus\mathsf Z_0)
                 \cap\mathsf Y\bigr)
 \leq\dim(\mathsf Z\setminus\mathsf Z_0)+\dim\mathsf Y-n
\]
\cite[Theorem~4.2]{zbMATH01714503}.  Bryant states the theorem for
arbitrary polyhedra; compactness is imposed here only because it holds
in the application.  By the definition of a push
\cite[\S4]{zbMATH01714503}, $(\Psi_t)_{0\leq t\leq1}$ is an ambient
PL isotopy of $\mathsf W$ satisfying
\[
 \Psi_0=\operatorname{id}_{\mathsf W},
 \qquad
 \Psi_t|_{\mathsf Z_0}=\operatorname{id}_{\mathsf Z_0}
 \quad(0\leq t\leq1).
\]
For every $w\in\mathsf W$, the track
$\{\Psi_t(w):0\leq t\leq1\}$ has $d_{\mathsf W}$-diameter less than
$\varepsilon$, and $\Psi_t$ is the identity outside
\(\bigl\{w\in\mathsf W:
d_{\mathsf W}(w,\mathsf Z)<\varepsilon\bigr\}\).
In particular, the theorem does not require $\mathsf Z$ and
$\mathsf Y$ to be disjoint.

Apply this statement with
\[
\begin{aligned}
 \mathsf W&=\operatorname{int}V_k,\\
 d_{\mathsf W}
 &=d_{g_{st}}|_{\operatorname{int}V_k\times\operatorname{int}V_k},\\
 \mathsf Z&=\alpha_k(S^1),\qquad
 \mathsf Z_0=\varnothing,\\
 \mathsf Y&=\Gamma_k,\qquad \varepsilon=1.
\end{aligned}
\]
Let $(\Psi_t)_{0\leq t\leq1}$ denote the resulting ambient PL isotopy
of $\operatorname{int}V_k$.  Then $\Psi_0=\operatorname{id}$ and
\(\dim\bigl(\Psi_1(\alpha_k(S^1))\cap\Gamma_k\bigr)
\leq p_k+q_k-n\leq2-n<0\).
Thus $\Psi_1(\alpha_k(S^1))\cap\Gamma_k=\varnothing$.

Let $\beta_k:=\Psi_1\circ\alpha_k$.  Then
$\beta_k(S^1)\subset\operatorname{int}V_k\setminus\Gamma_k$.  The map
$(z,t)\mapsto \Psi_t(\alpha_k(z))$ is a free homotopy in
$\operatorname{int}V_k$ from $\alpha_k$ to $\beta_k$.  Consequently, the
free homotopy class of $\beta_k$ in $V_k$ is the conjugacy class of
$\mathfrak{m}_k(a)$.

Proposition~\ref{induction}(6) gives
$V_k=f_k^{-1}([0,t_k])$, $\Gamma_k=f_k^{-1}(0)$, and
$\partial V_k=f_k^{-1}(t_k)$ for the $0$--$1$ simplicial level function
$f_k$.  Applying the punctured-sublevel homeomorphism
\eqref{eq:punctured-sublevel-product} from Lemma~\ref{lem:sublevel}, with
$c=t_k$, gives a homeomorphism of pairs
\[
 (V_k\setminus\Gamma_k,\partial V_k)
 \cong
 (\partial V_k\times(0,1],\partial V_k\times\{1\}).
\]
Transporting the factor homotopy
$(z,\vartheta)\mapsto(z,(1-u)\vartheta+u)$ therefore gives a strong deformation
retraction
\[
R_{k,u}:V_k\setminus\Gamma_k
\longrightarrow
V_k\setminus\Gamma_k,
\qquad
0\leq u\leq1,
\]
such that
\[
R_{k,0}=\operatorname{id},
\qquad
R_{k,u}\big|_{\partial V_k}
=
\operatorname{id}_{\partial V_k},
\qquad
R_{k,1}(V_k\setminus\Gamma_k)=\partial V_k.
\]

Define $\delta_k:=R_{k,1}\circ\beta_k$.  Then
$\delta_k(S^1)\subset\partial V_k$.
Concatenating the free homotopy \(\Psi_t\circ\alpha_k\) from
\(\alpha_k\) to \(\beta_k\) with the homotopy
\(R_{k,t}\circ\beta_k\) from \(\beta_k\) to \(\delta_k\), we conclude
that \(\delta_k\) is freely homotopic to \(\alpha_k\) in \(V_k\).
Hence the free homotopy class of \(\delta_k\) in \(V_k\) is the
conjugacy class of $\mathfrak{m}_k(a)$.

We next verify that $\delta_k$ lies in $M\setminus D$.  Since
$\delta_k(S^1)\subset\partial V_k\subset V_k$
and \(D\cap V_k=\varnothing\), we have
$\delta_k(S^1)\cap D=\varnothing$.
On the other hand, the nesting of the tower gives
$K\subset V_{k+1}\subset\operatorname{int}V_k$, whereas
$\delta_k(S^1)\subset\partial V_k$.
Therefore $\delta_k(S^1)\cap K=\varnothing$.
Combining the two conclusions yields
$\delta_k(S^1)\subset S^n\setminus(D\cup K)=M\setminus D$.

Suppose, toward a contradiction, that $\delta_k$ were null-homotopic in
$M\setminus C$.  Since
$M\setminus C=\operatorname{int}V_1\setminus K\subset V_1$, the same null-homotopy
would make the free homotopy class of $\delta_k$ trivial in $V_1$.
Let $\iota_1^k:V_k\hookrightarrow V_1$ denote the inclusion.  By
Proposition~\ref{induction},
$(\iota_1^k)_*\circ\mathfrak{m}_k
=\mathfrak{m}_1\circ\omega^{\,k-1}.$
It follows that the free homotopy class of \(\delta_k\), when viewed
in \(V_1\), is the conjugacy class of
$\mathfrak{m}_1\bigl(\omega^{\,k-1}(a)\bigr)$.  By Lemma~\ref{lem:endo}, the
endomorphism $\omega:F_2\to F_2$ is injective.  Since $a\neq1$,
injectivity gives $\omega^{\,k-1}(a)\neq1$.
As $\mathfrak{m}_1$ is an isomorphism, it follows that
$\mathfrak{m}_1\bigl(\omega^{\,k-1}(a)\bigr)\neq1$.
The identity element is conjugate only to itself, so the conjugacy
class of this element is nontrivial.  Hence the free homotopy class of
$\delta_k$ in $V_1$ is nontrivial.  This contradicts the assumed
null-homotopy of \(\delta_k\) in \(M\setminus C\).

Therefore, for the fixed compact set $C=C_1$, every compact set
$D\subset M$ containing $C$ admits a loop
$\delta_k:S^1\to M\setminus D$ that is not null-homotopic in
$M\setminus C$.  Consequently, $M$ is not simply connected at infinity.

Simple connectivity at infinity, with the convention fixed above, is
invariant under homeomorphism.  Indeed, a homeomorphism carries compact
sets to compact sets and restricts to a homeomorphism of the corresponding
complements; hence it preserves both the inverse system of components of
complements, and therefore one-endedness, and the loop null-homotopy
condition.  The closed balls $\overline B(0,m)$ form a cofinal compact
exhaustion of $\mathbb R^n$, and
$\mathbb R^n\setminus\overline B(0,m)$ is connected for $n\geq2$; thus
$\mathbb R^n$ is one-ended.  If $n\geq3$ and
$C'\subset\mathbb R^n$ is compact, choose a closed ball $D'$ containing
$C'$.  Since $\mathbb R^n\setminus D'$ deformation retracts onto
$S^{n-1}$ and $\pi_1(S^{n-1})=0$, every loop in
$\mathbb R^n\setminus D'$ is null-homotopic there, hence also in
$\mathbb R^n\setminus C'$.  Therefore $\mathbb R^n$ is simply connected
at infinity for $n\geq3$.  Since $n\geq4$ here, it follows that
$M\not\cong\mathbb R^n$.
\end{proof}

\subsection{The metric and its shifts}

Our metric construction follows the scheme of Karakhanyan
\cite[Theorem~1.1]{zbMATH07873599}, who shows that $S^n\setminus K$ carries a
complete scalar-flat metric conformal to the round metric if and only
if the critical Bessel capacity of $K$ vanishes.  In the fixed chart of
this section, this condition is $\operatorname{cap}(K)=0$.  After
stereographic projection, our topological construction yields a
compact set $F\subset\mathbb R^n$ with
$\mathcal{C}_{1+2/n,\,n/2}^{\mathbb R^n}(F)=0$.  To construct the metric in
Theorem~\ref{thm:counterexample-psc}, we need a probability measure on $F$ whose
Newtonian potential gives infinite conformal length to every curve
approaching $F$.  This is governed by the Wolff potential of the
measure: its divergence at every point of $F$ will imply completeness,
whereas capacity zero initially provides only averaged divergence.

The following Evans-type theorem converts this averaged information
into pointwise divergence.  For the parameters of the conformal
problem, the statement appears in \cite[Proposition~2.1]{zbMATH07873599}.
The decisive compactness step there is only sketched; since pointwise
divergence of the Wolff potential is not, in general, preserved by
weak-\(*\) convergence (Remark~\ref{rem:B-compactness}), we give a
complete argument for arbitrary $(d,\alpha,q)$ with $q\ge2$, the range
needed here.

\begin{theorem}\label{thm:evans-wolff}
Let $d\ge1$, $\alpha>0$, and $2\le q<\infty$, with $\alpha q\le d$.
Set $p:=q/(q-1)$, $\theta:=p-1=1/(q-1)\in(0,1]$, and
$s:=d-\alpha q\ge0$.  For a finite positive Radon measure $\nu$ on
$\mathbb{R}^d$ write
\[
 \mathcal{W}^\nu(x)
 :=
 \int_0^1\left(\frac{\nu(B(x,r))}{r^s}\right)^{\theta}\frac{dr}{r}.
\]
If $E\subset\mathbb{R}^d$ is nonempty and compact with
$\mathcal{C}_{\alpha,q}^{\mathbb R^d}(E)=0$, then there is a Radon
probability measure $\mu$ with
\[
 \operatorname{supp}\mu=E,
 \qquad
 \mathcal{W}^\mu(x)=+\infty\quad\text{for every }x\in E.
\]
\end{theorem}

The full-support conclusion is included as part of the sharp
measure-theoretic statement.  The metric construction below uses only
$\mu(\mathbb{R}^d\setminus E)=0$ and
$\mathcal{W}^{\mu}(x)=+\infty$ for every $x\in E$.

The idea of the proof is as follows. By the Hedberg--Wolff inequality \cite[Theorem~1]{zbMATH03801897}, zero capacity forces every probability measure on $E$ to have infinite Wolff self-energy. Since the full Wolff potential is not stable under weak-\(*\) convergence, we replace it by finite jointly continuous smoothed truncations. Compactness then promotes pointwise divergence of the truncated self-energies to uniform divergence of their minima. Sion's minimax theorem \cite[Theorem~3.4]{zbMATH03133049}, using the concavity in the measure afforded by $\theta\leq1$, produces measures whose truncated potentials are uniformly large on $E$. A convex series of these measures then yields a probability measure whose full Wolff potential diverges at every point of $E$.

\begin{proof}
Throughout, $B(x,r)$ denotes the open ball in $\mathbb{R}^d$, and
$\mathcal{P}(E)$ denotes the set of Borel probability measures on the
compact set $E$; each $\nu\in\mathcal{P}(E)$ is also regarded as a
measure on $\mathbb{R}^d$ vanishing off $E$.

\smallskip\noindent\emph{Step 1 (the space $\mathcal{P}(E)$ and the
self-energy).}
We equip $\mathcal{P}(E)$ with the weak topology of
\cite[Ch.~II, \S6]{zbMATH03245885}: $\nu_k\to\nu$ means
$\int_Eg\,d\nu_k\to\int_Eg\,d\nu$ for every $g\in C(E)$.  By
\cite[Ch.~II, Thm.~5.8]{zbMATH03245885}, $\nu\mapsto\bigl(g\mapsto\int_Eg\,d\nu\bigr)$
identifies $\mathcal{P}(E)$ with the set of positive linear functionals
$\ell$ on $C(E)$ with $\ell(1)=1$, and under this identification the
weak topology is the restriction of the weak-\(*\) topology of $C(E)^*$.
Thus $\mathcal{P}(E)$ is a convex subset of the locally convex Hausdorff
space $C(E)^*$ endowed with its weak-\(*\) topology, and since $E$ is a compact metric
space, it is compact and metrizable
\cite[Ch.~II, Thm.~6.4]{zbMATH03245885}.  In particular, every sequence
in $\mathcal{P}(E)$ has a weak-\(*\) convergent subsequence, and
continuity of functions on $\mathcal{P}(E)$ may be tested along
sequences.

For a finite positive Radon measure $\nu$ on $\mathbb{R}^d$, the map
$(x,r)\longmapsto\nu(B(x,r))$ is lower semicontinuous on
$\mathbb{R}^d\times(0,\infty)$: if
$(x_k,r_k)\to(x,r)$ and
$L\subset B(x,r)$ is compact, then $L\subset B(x_k,r_k)$ for all large
$k$, so $\liminf_k\nu(B(x_k,r_k))\ge\nu(L)$, and inner regularity gives
$\liminf_k\nu(B(x_k,r_k))\ge\nu(B(x,r))$.  Consequently the integrand
of $\mathcal{W}^\nu$ is Borel on $\mathbb{R}^d\times(0,1]$,
$\mathcal{W}^\nu:\mathbb{R}^d\to[0,\infty]$ is Borel by Tonelli's
theorem, and the self-energy
$J(\nu):=\int_E\mathcal{W}^\nu\,d\nu\in[0,\infty]$ is well defined.
Three elementary properties of
$\nu\mapsto\mathcal{W}^\nu$ are used repeatedly.  It is
\emph{monotone}: $\nu\le\nu'$ as measures implies
$\nu(B(x,r))\le\nu'(B(x,r))$ for every ball, hence
$\mathcal{W}^\nu\le\mathcal{W}^{\nu'}$ pointwise.  It is
\emph{$\theta$-homogeneous}: $\mathcal{W}^{c\nu}=c^\theta\mathcal{W}^\nu$
for $c>0$.  Finally, it has \emph{finite tails}: since
$\nu(B(x,r))\le\nu(\mathbb{R}^d)$, for $0<a\le1$
\begin{equation}\label{eq:B-tail}
 \int_{a}^{1}\Bigl(\frac{\nu(B(x,r))}{r^s}\Bigr)^{\theta}\frac{dr}{r}
 \ \le\ \nu(\mathbb{R}^d)^\theta\,c(a),
 \qquad
 c(a):=\int_a^1r^{-s\theta-1}\,dr<\infty.
\end{equation}

\smallskip\noindent\emph{Step 2 (capacity zero forces infinite
self-energy).}
Suppose that $J(\nu)<\infty$ for some $\nu\in\mathcal{P}(E)$.  In the
notation of Hedberg--Wolff
\cite[pp.~161, 164]{zbMATH03801897}, their capacity exponent $q$ is
our $q$, and their conjugate exponent $p$, determined by
$\frac1p+\frac1q=1$, is our $p$.  In that notation,
\[
 W^{\nu}_{\alpha,q}(x)
 =\int_0^1\Bigl(\frac{\nu(B(x,r))}{r^{\,d-\alpha q}}\Bigr)^{p-1}\frac{dr}{r};
\]
since $s=d-\alpha q$ and $\theta=p-1$, this is $W^{\nu}_{\alpha,q}=\mathcal{W}^\nu$.
The Hedberg--Wolff inequality \cite[Theorem~1, p.~165]{zbMATH03801897}
applies, because $\nu$ is a
positive Radon measure and $1<q\le d/\alpha$, and gives a constant
$A=A(d,\alpha,q)$ with
\[
 \|G_\alpha*\nu\|_{L^p(\mathbb{R}^d)}^{p}
 \le A\int_{\mathbb{R}^d}W^{\nu}_{\alpha,q}\,d\nu
 =A\int_E\mathcal{W}^\nu\,d\nu
 =A\,J(\nu)<\infty,
\]
the middle equality because $\nu$ vanishes off $E$.  Since $G_\alpha>0$
and $\nu\neq0$, also $\|G_\alpha*\nu\|_{L^p}>0$.

Now let $f$ be admissible in the definition of
$\mathcal{C}_{\alpha,q}^{\mathbb R^d}(E)$, that is, $f\ge0$, $f\in L^q(\mathbb{R}^d)$,
and $G_\alpha*f\ge1$ everywhere on $E$.  By Tonelli's theorem (all
integrands are nonnegative), the symmetry $G_\alpha(x-y)=G_\alpha(y-x)$,
and H\"older's inequality,
\[
\begin{aligned}
1=\nu(E)
&\le \int_E(G_\alpha*f)\,d\nu
 =\int_E\int_{\mathbb R^d}G_\alpha(x-y)f(y)\,dy\,d\nu(x)\\
&=\int_{\mathbb R^d}f\,(G_\alpha*\nu)\,dy
 \le \|f\|_{L^q}\,\|G_\alpha*\nu\|_{L^p}.
\end{aligned}
\]
Consequently, $\|f\|_{L^q}^q\ge\|G_\alpha*\nu\|_{L^p}^{-q}$, and taking
the infimum over all admissible $f$,
\(\mathcal{C}_{\alpha,q}^{\mathbb R^d}(E)\ge\|G_\alpha*\nu\|_{L^p}^{-q}>0\),
contrary to the hypothesis.  Hence
\begin{equation}
 \mathcal{C}_{\alpha,q}^{\mathbb R^d}(E)=0
 \quad\Longrightarrow\quad
 J(\nu)=+\infty\qquad\text{for every }\nu\in\mathcal{P}(E).
 \label{eq:B-all-energies}
\end{equation}
The restriction $\alpha q\le d$ is invoked here to apply the
Hedberg--Wolff inequality; in the remaining steps we use only its
consequence $s=d-\alpha q\ge0$.

\smallskip\noindent\emph{Step 3 (smoothed truncated potentials).}
Fix once and for all a continuous $\chi:[0,\infty)\to[0,1]$ with
$\chi\equiv1$ on $[0,1]$ and $\chi\equiv0$ on $[2,\infty)$; being
continuous and constant outside a compact set, $\chi$ is uniformly
continuous.  For $\nu\in\mathcal{P}(E)$, $x\in\mathbb{R}^d$ and $r>0$
put
\[
 P_r\nu(x):=\int_E\chi\!\left(\frac{|x-y|}{r}\right)d\nu(y).
\]
Then:

\textup{(3a)} $\nu(B(x,r))\le P_r\nu(x)\le\nu(B(x,2r))\le1$, because
$\chi(|x-y|/r)=1$ for $|x-y|<r$, $=0$ for $|x-y|\ge2r$, and
$0\le\chi\le1$.

\textup{(3b)} $\nu\mapsto P_r\nu(x)$ is the restriction to
$\mathcal{P}(E)$ of a linear functional, hence affine.

\textup{(3c)} $(\nu,x,r)\mapsto P_r\nu(x)$ is jointly continuous on
$\mathcal{P}(E)\times\mathbb{R}^d\times(0,\infty)$.  Indeed, let
$(\nu_k,x_k,r_k)\to(\nu,x,r)$ and put $g_k(y):=\chi(|x_k-y|/r_k)$,
$g(y):=\chi(|x-y|/r)$.  Since $E$ is bounded and $r_k\to r>0$,
$\sup_{y\in E}\bigl||x_k-y|/r_k-|x-y|/r\bigr|\to0$, hence
$\|g_k-g\|_{C(E)}\to0$ by uniform continuity of $\chi$, and
\[
 |P_{r_k}\nu_k(x_k)-P_r\nu(x)|
 \le\|g_k-g\|_{C(E)}\,\nu_k(E)
 +\Bigl|\int_Eg\,d\nu_k-\int_Eg\,d\nu\Bigr|\longrightarrow0,
\]
the second term by weak-\(*\) convergence, as $g\in C(E)$.

For $m\in\mathbb{N}$, $\nu,\lambda\in\mathcal{P}(E)$ and
$x\in\mathbb{R}^d$ define
\[
 T^\nu_m(x):=\int_{2^{-m-2}}^{1/2}
 \left(\frac{P_r\nu(x)}{r^s}\right)^{\theta}\frac{dr}{r},
 \qquad
 F_m(\nu,\lambda):=\int_ET^\nu_m\,d\lambda,
 \qquad
 I_m(\nu):=F_m(\nu,\nu).
\]
These have the following properties.

\textup{(3d)} \emph{Boundedness:}
$0\le T^\nu_m(x)\le\int_{2^{-m-2}}^{1/2}r^{-s\theta-1}dr<\infty$, by
(3a).

\textup{(3e)} \emph{Monotonicity in $m$:} $T^\nu_m\le T^\nu_{m+1}$
pointwise, since the interval of integration grows and the integrand
is nonnegative; hence also $F_m\le F_{m+1}$ and $I_m\le I_{m+1}$.

\textup{(3f)} \emph{Joint continuity:} $(\nu,x)\mapsto T^\nu_m(x)$ is
continuous on $\mathcal{P}(E)\times\mathbb{R}^d$.  Indeed, the integrand
$(P_r\nu(x)/r^s)^\theta r^{-1}$ is jointly continuous in $(\nu,x,r)$ by
(3c) and the continuity of $t\mapsto t^\theta$ on $[0,\infty)$, and it
is bounded by $r^{-s\theta-1}\le2^{(m+2)(s\theta+1)}$ on the compact
interval $[2^{-m-2},\tfrac12]$; dominated convergence applies.

\textup{(3g)} \emph{Concavity in $\nu$:} for fixed $x$ and $m$, the map
$\nu\mapsto T^\nu_m(x)$ is concave on $\mathcal{P}(E)$.  Indeed, since $0<\theta\leq1$, the map $t\mapsto t^\theta$ is concave and nondecreasing on $[0,\infty)$; this is the only point at which $q\geq2$ is used. Hence, by (3b), its composition with the affine nonnegative map $\nu\mapsto P_r\nu(x)$ is concave in $\nu$, and integration against the positive measure $dr/r$ preserves concavity.

\textup{(3h)} \emph{Comparison with $\mathcal{W}$:} for all $\nu$, $m$,
$x$,
\begin{equation}\label{eq:B-TleW}
 T^\nu_m(x)\ \le\ 2^{s\theta}\,\mathcal{W}^\nu(x).
\end{equation}
Indeed, by the upper half of (3a) and the substitution $\rho=2r$ (so
that $r^s=2^{-s}\rho^s$ and $dr/r=d\rho/\rho$),
\[
 T^\nu_m(x)
 \le\int_{2^{-m-2}}^{1/2}\Bigl(\frac{\nu(B(x,2r))}{r^s}\Bigr)^{\theta}\frac{dr}{r}
 =2^{s\theta}\int_{2^{-m-1}}^{1}\Bigl(\frac{\nu(B(x,\rho))}{\rho^s}\Bigr)^{\theta}\frac{d\rho}{\rho}
 \le2^{s\theta}\,\mathcal{W}^\nu(x).
\]
This is the reason the truncation stops at $r=\tfrac12$: after
doubling, the range of $\rho$ stays inside $(0,1]$, the range of
integration of $\mathcal{W}^\nu$.

\textup{(3i)} \emph{Lower bound in the limit:} for every $\nu$ and
$x$, by the lower half of (3a), monotone convergence in $m$, and
\eqref{eq:B-tail} with $a=\tfrac12$,
\[
 \mathcal{W}^\nu(x)
 \ \le\ \int_0^{1/2}\Bigl(\frac{\nu(B(x,r))}{r^s}\Bigr)^{\theta}\frac{dr}{r}+c(\tfrac12)
 \ \le\ \lim_{m\to\infty}T^\nu_m(x)+c(\tfrac12).
\]

\smallskip\noindent\emph{Step 4 (joint continuity of $F_m$; the
minima of $I_m$ diverge).}
We first show that $F_m$ is jointly continuous on
$\mathcal{P}(E)\times\mathcal{P}(E)$.  Let
$(\nu_k,\lambda_k)\to(\nu,\lambda)$.  Then $T^{\nu_k}_m\to T^\nu_m$
uniformly on $E$: otherwise there would be $\varepsilon>0$ and
$x_k\in E$ with $|T^{\nu_k}_m(x_k)-T^\nu_m(x_k)|\ge\varepsilon$; after passing
to a subsequence with $x_k\to x\in E$, we obtain from (3f) that
$T^{\nu_k}_m(x_k)\to T^\nu_m(x)$ and $T^\nu_m(x_k)\to T^\nu_m(x)$, a
contradiction.  Hence
\[
 |F_m(\nu_k,\lambda_k)-F_m(\nu,\lambda)|
 \le\|T^{\nu_k}_m-T^\nu_m\|_{C(E)}
 +\Bigl|\int_ET^\nu_m\,d\lambda_k-\int_ET^\nu_m\,d\lambda\Bigr|
 \longrightarrow0,
\]
the last term by weak-\(*\) convergence, since $T^\nu_m\in C(E)$ by
(3f).  By (3g), $\nu\mapsto F_m(\nu,\lambda)$ is concave for fixed
$\lambda$; by definition, $\lambda\mapsto F_m(\nu,\lambda)$ is affine
for fixed $\nu$.  In particular, $I_m$ is continuous on the compact set
$\mathcal{P}(E)$, so
$a_m:=\min_{\nu\in\mathcal{P}(E)}I_m(\nu)$ is attained, and
$a_m\le a_{m+1}$ by (3e).  We claim that
\begin{equation}
 a_m\longrightarrow+\infty\qquad(m\to\infty).
 \label{eq:B-am}
\end{equation}
For each fixed $\nu\in\mathcal{P}(E)$, integration of (3i) against
$\nu$ gives
\(J(\nu)\le\int_E\lim_{m\to\infty}T^\nu_m\,d\nu+c(\tfrac12)\).
Since $J(\nu)=+\infty$ by \eqref{eq:B-all-energies}, it follows that
$\int_E\lim_mT^\nu_m\,d\nu=+\infty$, and the monotone convergence
theorem together with (3e) yields
\(I_m(\nu)\nearrow\int_E\lim_{m\to\infty}T^\nu_m\,d\nu=+\infty\);
thus $I_m\to+\infty$ pointwise on $\mathcal{P}(E)$.  Suppose
\eqref{eq:B-am} fails.  Since $(a_m)$ is nondecreasing, it is then
bounded, say $a_m\le A<\infty$ for all $m$, and there are
$\nu_m\in\mathcal{P}(E)$ with $I_m(\nu_m)\le A$.  By compactness and
metrizability, a subsequence $(\nu_{m_j})$ converges weak-\(*\) to some
$\nu\in\mathcal{P}(E)$.  Fix $l$.  For $m_j\ge l$, (3e) gives
$I_l(\nu_{m_j})\le I_{m_j}(\nu_{m_j})\le A$, and the continuity of
$I_l$ gives $I_l(\nu)=\lim_jI_l(\nu_{m_j})\le A$.  Letting
$l\to\infty$ contradicts $I_l(\nu)\to+\infty$.  This proves
\eqref{eq:B-am}.  Observe that compactness is applied only to
the continuous functionals $I_l$ with $l$ fixed, never to
$\mathcal{W}$.

\smallskip\noindent\emph{Step 5 (minimax: from the diagonal to a
uniform lower bound).}
For fixed $\lambda\in\mathcal{P}(E)$ the map $\nu\mapsto F_m(\nu,\lambda)$
is concave: $P_r\nu(x)$ depends affinely on $\nu$, $z\mapsto z^\theta$
is concave on $[0,\infty)$ because $0<\theta=\frac1{q-1}\le1$, and
concavity is preserved by the integrations against $dr/r$ and $d\lambda$
(this is (3g)).  For fixed $\nu$ the map
$\lambda\mapsto F_m(\nu,\lambda)=\int_ET^\nu_m\,d\lambda$ is affine,
hence convex.  Moreover, $F_m$ is real-valued and jointly continuous
on the compact convex set $\mathcal{P}(E)\times\mathcal{P}(E)$ (Step~4).
In particular, $F_m$ is upper semicontinuous and quasi-concave in $\nu$
and lower semicontinuous and quasi-convex in $\lambda$, so Sion's
minimax theorem \cite[Theorem~3.4]{zbMATH03133049} applies with $\nu$ as
the maximizing and $\lambda$ as the minimizing variable:
\[
 \sup_{\nu\in\mathcal{P}(E)}\,\inf_{\lambda\in\mathcal{P}(E)}F_m(\nu,\lambda)
 =\inf_{\lambda\in\mathcal{P}(E)}\,\sup_{\nu\in\mathcal{P}(E)}F_m(\nu,\lambda).
\]
All extrema are attained: the inner ones because $F_m$ is continuous
on the compact set $\mathcal{P}(E)$.  For completeness, joint continuity
on the compact product implies uniform continuity, and
\[
\left|\min_\lambda F_m(\nu,\lambda)
      -\min_\lambda F_m(\nu',\lambda)\right|
\leq
\sup_\lambda|F_m(\nu,\lambda)-F_m(\nu',\lambda)|.
\]
By uniform continuity, the right-hand side tends to zero as
\(\nu'\to\nu\).  The analogous estimate, with \(\max_\nu\) and the roles of
the variables reversed, proves continuity of the other value function.
Compactness therefore gives the outer extrema as well.  Thus
\[
 \max_{\nu\in\mathcal{P}(E)}\,\min_{\lambda\in\mathcal{P}(E)}F_m(\nu,\lambda)
 =\min_{\lambda\in\mathcal{P}(E)}\,\max_{\nu\in\mathcal{P}(E)}F_m(\nu,\lambda).
\]
For fixed $\nu$, continuity of $T^\nu_m$ and compactness of $E$ give a
point $x_\nu\in E$ with $T^\nu_m(x_\nu)=\min_{x\in E}T^\nu_m(x)$.  For
every $\lambda\in\mathcal{P}(E)$,
\(F_m(\nu,\lambda)=\int_ET^\nu_m\,d\lambda\ge\min_{x\in E}T^\nu_m(x)\),
with equality for $\lambda=\delta_{x_\nu}$.  Consequently
$\min_{\lambda\in\mathcal{P}(E)}F_m(\nu,\lambda)=\min_{x\in E}T^\nu_m(x)$,
and therefore
\begin{equation}\label{eq:B-minimax}
\begin{aligned}
 \max_{\nu\in\mathcal{P}(E)}\,\min_{x\in E}T^\nu_m(x)
 &=\min_{\lambda\in\mathcal{P}(E)}\,\max_{\nu\in\mathcal{P}(E)}F_m(\nu,\lambda)\\
 &\ge\min_{\lambda\in\mathcal{P}(E)}F_m(\lambda,\lambda)
 =\min_{\lambda\in\mathcal{P}(E)}I_m(\lambda)=a_m,
\end{aligned}
\end{equation}
the inequality because $\max_\nu F_m(\nu,\lambda)\ge F_m(\lambda,\lambda)$
for each $\lambda$.  Thus the diagonal quantity $a_m$, which diverges by
\eqref{eq:B-am}, bounds from below a quantity that is uniform in $x\in E$.

\smallskip\noindent\emph{Step 6 (assembly by a convex series).}
For each integer $k\geq1$, choose by \eqref{eq:B-am} an index $m(k)$ with
$a_{m(k)}\ge2^{s\theta}\,k\,2^{k\theta}$, and let
$\nu_k\in\mathcal{P}(E)$ attain the maximum in \eqref{eq:B-minimax}
for $m=m(k)$.  Then $T^{\nu_k}_{m(k)}(x)\ge2^{s\theta}k\,2^{k\theta}$
for every $x\in E$, and \eqref{eq:B-TleW} gives
\begin{equation}\label{eq:B-nuk}
 \mathcal{W}^{\nu_k}(x)\ \ge\ 2^{-s\theta}\,T^{\nu_k}_{m(k)}(x)
 \ \ge\ k\,2^{k\theta}
 \qquad\text{for every }x\in E.
\end{equation}
Set $\mu:=\sum_{k\ge1}2^{-k}\nu_k$.
The series converges in total variation, so $\mu$ is a positive Radon
measure with $\mu(E)=\sum_{k\geq1}2^{-k}=1$ and $\mu(\mathbb{R}^d\setminus E)=0$;
thus $\mu\in\mathcal{P}(E)$ and $\operatorname{supp}\mu\subset E$.
Since $\mu\ge2^{-k}\nu_k$ as measures, the monotonicity and
$\theta$-homogeneity of $\nu\mapsto\mathcal{W}^\nu$ (Step~1) and
\eqref{eq:B-nuk} give, for every $x\in E$ and every integer $k\geq1$,
\[
 \mathcal{W}^\mu(x)\ \ge\ \mathcal{W}^{2^{-k}\nu_k}(x)
 =2^{-k\theta}\,\mathcal{W}^{\nu_k}(x)\ \ge\ 2^{-k\theta}\,k\,2^{k\theta}=k.
\]
Hence $\mathcal{W}^\mu\equiv+\infty$ on $E$.  This termwise lower bound
avoids attempting to pass pointwise divergence through weak-\(*\)
convergence.

Finally, $\operatorname{supp}\mu=E$.  If $x\in E\setminus\operatorname{supp}\mu$,
there is $\rho>0$ with $\mu(B(x,\rho))=0$, hence $\mu(B(x,r))=0$ for
$0<r\le\rho$, and \eqref{eq:B-tail} with $a=\min\{\rho,1\}$ gives
$\mathcal{W}^\mu(x)\le c(\min\{\rho,1\})<\infty$, the integral over
$(0,a]$ being zero; this contradicts $\mathcal{W}^\mu(x)=+\infty$.
The proof is complete.
\end{proof}

\begin{remark}[Why compactness alone cannot work]\label{rem:B-compactness}
If $\nu_k\to\nu$ weak-\(*\) in $\mathcal{P}(E)$, then
$\nu(B(x,r))\le\liminf_k\nu_k(B(x,r))$ for every open ball
\cite[Ch.~II, Thm.~6.1(d)]{zbMATH03245885}, hence by Fatou's lemma
$\mathcal{W}^\nu(x)\le\liminf_k\mathcal{W}^{\nu_k}(x)$.  This one-sided
inequality does not transfer pointwise divergence to the weak-\(*\) limit.
For instance, suppose that $s\ge1$ and let
$E=[0,1]\times\{0\}^{d-1}\subset\mathbb{R}^d$.  Since
$\mathcal{H}^s(E)<\infty$, the finite-critical-measure theorem used in
the proof of Lemma~\ref{lem:thin} gives
$\mathcal{C}_{\alpha,q}^{\mathbb R^d}(E)=0$.  Let $\lambda$ be the normalized length
measure on $E$ and $\nu$ the normalized length measure on
$[\tfrac12,1]\times\{0\}^{d-1}$.  Since $\lambda(B(x,r))\ge r$ for $x\in E$ and
$r\le\tfrac12$, $\mathcal{W}^\lambda\ge\int_0^{1/2}r^{(1-s)\theta-1}dr=+\infty$
on $E$, so $\nu_k:=(1-2^{-k})\nu+2^{-k}\lambda$ satisfies
$\mathcal{W}^{\nu_k}\equiv+\infty$ on $E$ for every $k$; but $\nu_k\to\nu$,
and $\mathcal{W}^\nu(0)\le\int_{1/2}^1r^{-s\theta-1}dr<\infty$ because
$\nu(B(0,r))=0$ for $r\le\tfrac12$.  This is why Steps~3--6 apply
compactness only to the truncated functionals $T^\nu_m$ and assemble the
final measure by a convex series rather than by a limit.
\end{remark}

Theorem~\ref{thm:evans-wolff} is the sole measure-theoretic input: in the fixed
stereographic chart it yields a probability measure \(\mu\) carried by
\(F:=\sigma(K)\) such that
$\mathcal{W}^\mu\equiv+\infty$ on \(F\).
This divergence gives both the blow-up of the Newtonian potential on
\(F\) and the completeness of the conformal metric, while \(\mu(F)=1\)
controls the potential at infinity, yielding the smooth extension across
the pole and the global positive lower bound needed for negative shifts.
For the remainder of this section we use the sign convention
$\Delta_g:=\operatorname{div}_g\nabla$.  Thus the conformal Laplacian is
\[
 L_g:=-\frac{4(n-1)}{n-2}\Delta_g+\mathrm{Sc}_g,
\]
and its conformal covariance is
\begin{equation}\label{eq:conformal-covariance-metric}
\mathrm{Sc}_{w^{4/(n-2)}g}
=
w^{-(n+2)/(n-2)}L_gw
\qquad (w>0\ \text{smooth}).
\end{equation}
Karakhanyan's existence theorem \cite[Theorem~1.1]{zbMATH07873599} does
not provide these controls, and its existence direction relies on
\cite[Proposition~2.1]{zbMATH07873599}, whose decisive compactness step
is only sketched (Remark~\ref{rem:B-compactness}).  The following lemma
therefore adapts \cite[Theorem~1.1, implication
\textup{(ii)}\(\Rightarrow\)\textup{(i)},
equations~(4.1)--(4.7)]{zbMATH07873599}, retaining the arguments for
regularity at the pole, uniform blow-up at the singular set, and
completeness; the lemma after it bounds the conformal factor from below
through its potential representation and uses the linearity of \(L_g\)
to construct the shifted family.
It is stated for an arbitrary compact set \(E\) and pole \(p_\infty\);
it is applied with \((\Omega,E,p_\infty)=(M,K,p_\sigma)\), where
\(\sigma_{p_\sigma}=\sigma\) is the fixed chart.

\begin{lemma}[Controlled one-chart scalar-flat metric]\label{lem:onechart}
Let $n\ge4$, let $E\subset S^n$ be nonempty and compact, and
assume that $\Omega:=S^n\setminus E$
is connected.  Let $p_\infty\in\Omega$, and let
$\sigma_{p_\infty}:S^n\setminus\{p_\infty\}\to \mathbb{R}^n$ be the
stereographic projection with pole $p_\infty$.  Suppose that
$\mathcal{C}_{1+2/n,\,n/2}^{\mathbb R^n}\bigl(\sigma_{p_\infty}(E)\bigr)=0$.
Then there is a smooth function $\phi:\Omega\to(0,\infty)$ such that
$h:=\phi^{4/(n-2)}g_{st}\big|_\Omega$ is complete and scalar-flat.  Moreover,
\[
 \phi(x)\longrightarrow+\infty
 \quad\text{uniformly as }\;
 d_{g_{st}}(x,E)\longrightarrow0,
 \qquad\text{and}\qquad
 \phi(p_\infty)=2^{-(n-2)/2}.
\]
\end{lemma}

\begin{proof}
Put
$F:=\sigma_{p_\infty}(E),\qquad
R_F:=\max_{y\in F}|y|,\qquad
\rho_F:=\frac{1}{1+2R_F}>0.$
Then $F$ is a nonempty compact subset of $\mathbb{R}^n$.  Recall from
\eqref{eq:parameters} and \eqref{eq:driving-identities} that
\[
 s=\frac{n-2}{2},
 \qquad
 \theta=\frac{2}{n-2},
 \qquad
 s\theta=1,
 \qquad
 (n-2)\theta=2,
 \qquad
 n-2-s=s,
\]
and put $\beta_0:=\frac{n-2}{2}$ for the exponent of the round conformal
factor.  For
$(d,\alpha,q)=\left(n,1+\frac2n,\frac n2\right),$
the assumption $n\ge4$ gives $q=\frac n2\ge2$ and
$\alpha q=\frac{n+2}{2}\le n$.  Thus Theorem~\ref{thm:evans-wolff} applies and
yields a Radon probability measure $\mu$ such that
\[
 \mu(\mathbb R^n\setminus F)=0,
 \qquad
 \mathcal{W}^{\mu}(x)=+\infty
 \quad\text{for every }x\in F.
\]
Define
\[
 u(x):=\int_F|x-y|^{2-n}\,d\mu(y)
 \in(0,\infty],
 \qquad x\in\mathbb{R}^n.
\]

\smallskip\noindent
\emph{Step 1: the Newtonian potential.}
Let $Q\subset\mathbb{R}^n\setminus F$ be compact and put
$\delta:=\operatorname{dist}(Q,F)>0$.  Every $x$-derivative of order
$j$ of the kernel $|x-y|^{2-n}$ is bounded on $Q\times F$ by a constant
depending only on $j$ and $\delta$, so differentiation under the
integral sign is legitimate.  Hence $u$ is positive, smooth, and
harmonic on $\mathbb{R}^n\setminus F$.

We first prove that $u(x)=+\infty$ for every $x\in F$.
Fix $x_0\in F$ and suppose, to the contrary, that $u(x_0)<\infty$.
For $0<r\le1$, since $|x_0-y|^{2-n}>r^{2-n}$ on $B(x_0,r)$,
\[
\begin{aligned}
 \mu(B(x_0,r))\,r^{2-n}
 &\le
 \int_{B(x_0,r)}|x_0-y|^{2-n}\,d\mu(y)
 \le
 u(x_0),\\
 \text{and hence}\qquad
 \mu(B(x_0,r))
 &\le
 u(x_0)\,r^{n-2}.
\end{aligned}
\]
Using $(n-2-s)\theta=s\theta=1$, we obtain
\begin{align*}
 \mathcal{W}^{\mu}(x_0)
 &=
 \int_0^1
 \left(
 \frac{\mu(B(x_0,r))}{r^s}
 \right)^\theta\frac{dr}{r}
 \le
 u(x_0)^\theta
 \int_0^1
 r^{(n-2-s)\theta}\frac{dr}{r}
 =
 u(x_0)^\theta\int_0^1dr
 <\infty,
\end{align*}
contradicting the choice of $\mu$.

The function $u$ is lower semicontinuous on $\mathbb{R}^n$ by Fatou's
lemma, so $\{u>A\}$ is open for every $A>0$.  Since $u\equiv+\infty$ on
$F$, this open set contains the compact set $F$, and hence it contains
a uniform neighborhood of $F$: there exists $\varepsilon_A>0$ such that
$\bigl\{x:\operatorname{dist}(x,F)<\varepsilon_A\bigr\}\subset\{u>A\}$.
Thus
\begin{equation}\label{eq:u-divergence}
 u(x)\longrightarrow+\infty
 \quad\text{uniformly as }\;
 \operatorname{dist}(x,F)\longrightarrow0.
\end{equation}

We also record the behavior of $u$ at infinity.  If
$|x|\ge1+2R_F$ and $y\in F$, then
\(|x-y|\ge |x|-R_F\ge\frac{|x|}{2}\),
and hence \(u(x)\le2^{n-2}|x|^{2-n}\).
Moreover, $(|x|/|x-y|)^{n-2}\to1$ uniformly for $y\in F$ as
$|x|\to\infty$, and the preceding estimate bounds this integrand by
$2^{n-2}$ for all sufficiently large $|x|$.  Since $\mu(F)=1$,
dominated convergence gives
\begin{equation}\label{eq:u-infinity}
 \lim_{|x|\to\infty}|x|^{n-2}u(x)
 =\lim_{|x|\to\infty}\int_F\Bigl(\frac{|x|}{|x-y|}\Bigr)^{n-2}d\mu(y)
 =1.
\end{equation}

\smallskip\noindent
\emph{Step 2: construction of the metric and extension across the
pole.}
Set
$U_{\mathrm{rd}}(\xi):=\left(\frac{2}{1+|\xi|^2}\right)^{\beta_0}.$
Then $0<U_{\mathrm{rd}}\le U_{\mathrm{rd}}(0)=2^{\beta_0}$, and
$(\sigma_{p_\infty}^{-1})^*g_{st}
=U_{\mathrm{rd}}^{4/(n-2)}g_{\mathrm E},$
where $g_{\mathrm E}$ is the Euclidean metric.  Define $\phi$ on
$\Omega\setminus\{p_\infty\}$ by
\begin{equation}\label{eq:phi-chart}
 \phi\circ\sigma_{p_\infty}^{-1}:=\frac{u}{U_{\mathrm{rd}}}
 \qquad\text{on }\sigma_{p_\infty}(\Omega\setminus\{p_\infty\})=\mathbb{R}^n\setminus F.
\end{equation}
It follows that
\[
 (\sigma_{p_\infty}^{-1})^*
 \left(\phi^{4/(n-2)}g_{st}\right)
 =
 u^{4/(n-2)}g_{\mathrm E}
 =:\widetilde h.
\]
Since $u$ is positive and harmonic on $\mathbb{R}^n\setminus F$ and
$\mathrm{Sc}_{g_{\mathrm E}}=0$, formula
\eqref{eq:conformal-covariance-metric} gives
$\mathrm{Sc}_{\widetilde h}=0$.  Thus $h$ is smooth and scalar-flat on
$\Omega\setminus\{p_\infty\}$.

It remains to extend the metric across $p_\infty$.  Define
$I(z):=\frac{z}{|z|^2},\qquad
\psi(z):=\sigma_{p_\infty}^{-1}(I(z))\quad (z\ne0),\qquad
\psi(0):=p_\infty.$
Viewing $S^n\subset\mathbb R^{n+1}$ as the unit sphere, let
$\sigma_{-p_\infty}$ denote stereographic projection from the antipode
$-p_\infty$, using the same orthogonal identification of
$p_\infty^\perp$ with $\mathbb R^n$.  Since
$I\circ\sigma_{p_\infty}=\sigma_{-p_\infty}$ on
$S^n\setminus\{\pm p_\infty\}$, we have
$\psi=\sigma_{-p_\infty}^{-1}$, including at $z=0$; hence $\psi$ is a
smooth inverse chart about $p_\infty$.  For $0<|z|<\rho_F$, put
$\widehat\phi:=\phi\circ\psi$; this is well defined because
$|I(z)|>1+2R_F>R_F$, and hence $I(z)\notin F$.  Since
\[
 I^*g_{\mathrm E}=|z|^{-4}g_{\mathrm E}
 \qquad\text{and}\qquad
 U_{\mathrm{rd}}(I(z))
 =
 \left(\frac{2|z|^2}{1+|z|^2}\right)^{\beta_0}
 =
 |z|^{n-2}U_{\mathrm{rd}}(z),
\]
we have $\psi^*g_{st}=U_{\mathrm{rd}}(z)^{4/(n-2)}g_{\mathrm E}$: the round
metric has the same conformal factor in the $z$-coordinate.
Furthermore,
\[
 U_{\mathrm{rd}}(z)\widehat\phi(z)
 =
 \frac{U_{\mathrm{rd}}(z)\,u(I(z))}{U_{\mathrm{rd}}(I(z))}
 =
 |z|^{2-n}u(I(z))
 =:v(z).
\]
Since
$\Delta v(z)=|z|^{-n-2}(\Delta u)(I(z))$, the function $v$ is harmonic
on the punctured ball $0<|z|<\rho_F$.  For $0<|z|\le\rho_F$ we have
$|I(z)|\ge1+2R_F$, so Step~1 gives
\(0<v(z)
\le
|z|^{2-n}\,2^{n-2}|I(z)|^{2-n}
=
2^{n-2}\).
Hence $v$ is bounded near the origin and extends harmonically across
$z=0$, and by \eqref{eq:u-infinity},
\[
 v(0)
 =
 \lim_{z\to0}|z|^{2-n}u(I(z))
 =
 \lim_{|x|\to\infty}|x|^{n-2}u(x)
 =
 1.
\]
Consequently, $\widehat\phi=v/U_{\mathrm{rd}}$ extends smoothly and
positively across $z=0$, with
$\widehat\phi(0)=v(0)/U_{\mathrm{rd}}(0)=2^{-\beta_0}$.  Since $\psi$ is
a smooth inverse chart about $p_\infty$, this defines a smooth positive
extension of $\phi$ across $p_\infty$, with
$\phi(p_\infty)=2^{-\beta_0}=2^{-(n-2)/2}.$
In the $z$-coordinate,
\[
 \psi^*h
 =
 \widehat\phi^{\,4/(n-2)}U_{\mathrm{rd}}^{4/(n-2)}g_{\mathrm E}
 =
 v^{4/(n-2)}g_{\mathrm E}
 \qquad\text{on }|z|<\rho_F,
\]
which is smooth and scalar-flat by
\eqref{eq:conformal-covariance-metric}, because $v$ is positive and
harmonic.  Thus $h$ is smooth and scalar-flat on all of $\Omega$.

We next transfer \eqref{eq:u-divergence} to $\phi$.  Since
$U_{\mathrm{rd}}\le2^{\beta_0}$, equation \eqref{eq:phi-chart} gives
$\phi\ge2^{-\beta_0}\,u\circ\sigma_{p_\infty}
\qquad\text{on }\Omega\setminus\{p_\infty\}.$
Fix $A>0$ and apply \eqref{eq:u-divergence} with the threshold
$2^{\beta_0} A$: there is $\varepsilon_A>0$ such that
$\operatorname{dist}(x,F)<\varepsilon_A$ implies $u(x)>2^{\beta_0} A$.
The set
\[
 Q_A:=
 \{p_\infty\}\cup
 \left\{
 x\in S^n\setminus\{p_\infty\}:
 \operatorname{dist}(\sigma_{p_\infty}(x),F)\ge\varepsilon_A
 \right\}
\]
is closed in $S^n$ because its complement is the open set
\[
 \sigma_{p_\infty}^{-1}\!\left(
 \left\{z\in\mathbb R^n:\operatorname{dist}(z,F)<\varepsilon_A\right\}
 \right).
\]
Thus $Q_A$ is compact, and it is disjoint from $E$.  Therefore
$\eta_A:=d_{g_{st}}(Q_A,E)>0$.  If $d_{g_{st}}(x,E)<\eta_A$,
then $x\notin Q_A$, so $x\ne p_\infty$ and
$\operatorname{dist}(\sigma_{p_\infty}(x),F)<\varepsilon_A$, whence
$\phi(x)\ge2^{-\beta_0}u(\sigma_{p_\infty}(x))>A.$
It follows that
\begin{equation}\label{eq:phi-divergence}
 \phi(x)\longrightarrow+\infty
 \quad\text{uniformly as }\;
 d_{g_{st}}(x,E)\longrightarrow0.
\end{equation}

\smallskip\noindent
\emph{Step 3: completeness.}
For a Riemannian manifold $(Z,g)$, call a locally Lipschitz map
$c:[0,T)\to Z$, $0<T\le\infty$, \emph{divergent} if, for every compact
$Q\subset Z$, there is
$t_Q\in(0,T)$ such that $c(t)\notin Q$ for all $t>t_Q$.  Here local
Lipschitz continuity is measured with respect to the Riemannian
distance $d_g$.  By continuity and the Lindel\"of property of
$[0,T)$, there is a countable cover by parameter intervals, each of
which is mapped by $c$ into a single coordinate neighborhood on which
the coordinate distance and $d_g$ are bi-Lipschitz equivalent.
Rademacher's theorem applied to the corresponding coordinate
representatives therefore gives a vector $c'(t)\in T_{c(t)}Z$ for
almost every $t$.  On chart overlaps, the chain rule for the smooth
transition maps shows that this vector is independent of the chosen
chart.  For every $0<T'<T$, the restriction $c|_{[0,T']}$ is
Lipschitz and hence absolutely continuous.  We define its Riemannian
length by
\[
 \operatorname{Len}_g(c)
 :=\sup_{0<T'<T}\int_0^{T'}|c'(t)|_g\,dt
 =\int_0^T|c'(t)|_g\,dt\in[0,\infty].
\]
We use the same notation, with the integral taken over the given
parameter interval, for locally Lipschitz curves defined on other
intervals and for their restrictions.
We first record the
criterion that if $(Z,g)$ is connected and every divergent locally
Lipschitz curve in $Z$ has infinite $g$-length, then $(Z,g)$ is complete.
Indeed, if $(Z,g)$ were incomplete, Hopf--Rinow would produce a maximal
unit-speed geodesic $c:[0,T)\to Z$ with $T<\infty$.  This geodesic is
divergent.  Otherwise there would be a compact set $Q\subset Z$ and
times $t_\ell\uparrow T$ such that $c(t_\ell)\in Q$.  The vectors
$(c(t_\ell),c'(t_\ell))$ lie in the compact
unit-sphere bundle over $Q$, so, after passage to a subsequence, they
converge to some $v\in TZ$.  Local existence and continuous dependence
for the smooth geodesic vector field give an open neighborhood
$W_v\subset TZ$ of $v$ and a number $\delta>0$ such that its flow is
defined on $[0,\delta]\times W_v$.  For all sufficiently large $\ell$,
$(c(t_\ell),c'(t_\ell))\in W_v$ and $T-t_\ell<\delta$.  By uniqueness, the
projected flow line with this initial datum agrees with $c$ on
$[t_\ell,T)$ and remains defined beyond $T$, contradicting maximality.
Thus an incomplete connected Riemannian manifold admits a divergent
locally Lipschitz curve of finite length.  The contrapositive is the
criterion used below; compare \cite[Theorem~2.3]{zbMATH07873599}.

By this criterion, applied to $(Z,g)=(\Omega,h)$, it suffices to prove that
every divergent locally Lipschitz curve in $\Omega$ has infinite
$h$-length.

Since $E\ne\varnothing$
and $p_\infty\in\Omega$, the manifold $\Omega$ is a nonempty proper open subset
of $S^n$.  It is noncompact, since otherwise it would be both open and
closed in the connected sphere $S^n$.

For the completeness step specifically, the following dyadic-annulus
argument is adapted from
\cite[proof of Theorem~1.1, implication
\textup{(ii)}$\Rightarrow$\textup{(i)}, equations~(4.3)--(4.7)]{zbMATH07873599}.
Let $c:[0,T)\to\Omega$ be a divergent locally Lipschitz curve.  Choose $\rho_0>0$ with
$\overline B_{S^n}(p_\infty,\rho_0)\subset\Omega$.  Since this closed
ball is compact and $c$ is divergent, there is $t_0<T$ such that
$c(t)\notin\overline B_{S^n}(p_\infty,\rho_0)$ for $t\ge t_0$.
Therefore $\gamma:=\sigma_{p_\infty}\circ c|_{[t_0,T)}$
is a bounded locally Lipschitz curve in $\mathbb{R}^n\setminus F$, because
$S^n\setminus B_{S^n}(p_\infty,\rho_0)$ is a compact subset of the stereographic
chart.  The curve $\gamma$ is
also divergent there: if $Q\subset\mathbb{R}^n\setminus F$ is compact,
then $\sigma_{p_\infty}^{-1}(Q)$ is a compact subset of $\Omega$, which the tail
of $c$ eventually avoids.  Since
$(\sigma_{p_\infty}^{-1})^*h=\widetilde h=u^{4/(n-2)}g_{\mathrm E}$, the
$\widetilde h$-length element is
$ds_{\widetilde h}=u^{2/(n-2)}\,ds_{\mathrm E}=u^\theta\,ds_{\mathrm E},$
and it suffices to show that
$\operatorname{Len}_{\widetilde h}(\gamma)=\infty$.

Choose times $\varpi_m\uparrow T$.  Since $\gamma$ is bounded, after passing
to a subsequence we may assume that $\gamma(\varpi_m)\to x_0$ for some
$x_0\in\mathbb{R}^n$.  Necessarily $x_0\in F$: otherwise a closed
Euclidean ball of sufficiently small positive radius centered at $x_0$
would be a compact subset of
$\mathbb{R}^n\setminus F$ visited by $\gamma$ at arbitrarily late times,
contradicting divergence.  Set $r_k:=2^{-k}$ and choose $k_0\ge0$ such
that $r_{k_0}\le|\gamma(t_0)-x_0|$, which is possible because
$\gamma(t_0)\notin F\ni x_0$.  For $k\ge k_0$, define
\[
\begin{aligned}
 b_k
 &:=
 \inf\left\{
 t\in[t_0,T):
 |\gamma(t)-x_0|\le r_{k+1}
 \right\},\\
 a_k
 &:=
 \sup\left\{
 t\in[t_0,b_k]:
 |\gamma(t)-x_0|\ge r_k
 \right\}.
\end{aligned}
\]
The defining set for $b_k$ is nonempty because $\gamma(\varpi_m)\to x_0$.
Since $|\gamma(t_0)-x_0|\ge r_k>r_{k+1}$, continuity at $t_0$ gives
$t_0<b_k$, while nonemptiness of the defining set gives $b_k<T$.
Choose a sequence in that set converging to $b_k$.  Continuity gives
$|\gamma(b_k)-x_0|\le r_{k+1}$.  On the other hand, every
$t\in[t_0,b_k)$
satisfies $|\gamma(t)-x_0|>r_{k+1}$; letting $t\uparrow b_k$ gives the
reverse inequality.  Hence $|\gamma(b_k)-x_0|=r_{k+1}$.

The defining set for $a_k$ is a nonempty closed subset of
$[t_0,b_k]$, so its supremum is attained.  Since
$|\gamma(b_k)-x_0|=r_{k+1}<r_k$, one has $a_k<b_k$.  Moreover,
$|\gamma(a_k)-x_0|\ge r_k$.  If this inequality were strict,
continuity would give a point $t\in(a_k,b_k)$ with
$|\gamma(t)-x_0|\ge r_k$, contrary to the definition of $a_k$.
Thus $|\gamma(a_k)-x_0|=r_k$.  The definitions now give
\[
\begin{gathered}
 |\gamma(a_k)-x_0|=r_k,
 \qquad
 |\gamma(b_k)-x_0|=r_{k+1},
 \\
 \gamma((a_k,b_k))
 \subset
 A_k:=
 \left\{
 x:r_{k+1}<|x-x_0|<r_k
 \right\}.
\end{gathered}
\]
The annuli $A_k$ are pairwise disjoint, so the intervals $(a_k,b_k)$
are pairwise disjoint.  Indeed, if
$t\in(a_k,b_k)\cap(a_\ell,b_\ell)$ for $k\ne\ell$, then
$\gamma(t)\in A_k\cap A_\ell$, a contradiction.  Consequently, the
closed intervals $[a_k,b_k]$ have pairwise disjoint interiors and can
intersect only at endpoints.  Since finitely many endpoints have
Lebesgue measure zero, for every $K\ge k_0$,
\begin{equation}\label{eq:finite-annuli}
 \operatorname{Len}_{\widetilde h}(\gamma)
 \ge
 \sum_{k=k_0}^{K}
 \operatorname{Len}_{\widetilde h}
 \bigl(\gamma|_{[a_k,b_k]}\bigr).
\end{equation}
Each crossing has Euclidean length at least
$|\gamma(a_k)-\gamma(b_k)|\ge r_k-r_{k+1}=r_k/2$.  If $x\in A_k$ and
$y\in B(x_0,r_{k+2})$, then
\[
\begin{aligned}
 |x-y|
 &\le
 |x-x_0|+|x_0-y| \le
 r_k+r_{k+2}
 =
 \frac54\,r_k,\\
 \text{and hence}\qquad
 u(x)
 &\ge
 \left(\frac45\right)^{n-2}
 \mu(B(x_0,r_{k+2}))\,r_k^{2-n}.
\end{aligned}
\]
Using $(n-2)\theta=2$ and $\gamma((a_k,b_k))\subset A_k\setminus F$, we
obtain the following estimate.  The endpoints of the crossing lie on
$\partial A_k$ and do not affect the length integral, so the infimum over
$A_k\setminus F$ is valid:
\[
 \operatorname{Len}_{\widetilde h}
 \bigl(\gamma|_{[a_k,b_k]}\bigr)
 \ge
 \Bigl(\inf_{A_k\setminus F}u\Bigr)^\theta
 \frac{r_k}{2}
 \ge
 \frac{8}{25}\,
 \frac{\mu(B(x_0,r_{k+2}))^\theta}{r_k}.
\]
Letting $K\to\infty$ in \eqref{eq:finite-annuli} gives
\begin{equation}\label{eq:annuli}
 \operatorname{Len}_{\widetilde h}(\gamma)
 \ge
 \frac{8}{25}
 \sum_{k\ge k_0}
 \frac{\mu(B(x_0,r_{k+2}))^\theta}{r_k}
 =
 \frac{2}{25}
 \sum_{m\ge k_0+2}
 \frac{\mu(B(x_0,r_m))^\theta}{r_m}.
\end{equation}
On the other hand, since $r\mapsto\mu(B(x_0,r))$ is nondecreasing and
$s\theta=1$,
\begin{equation}\label{eq:dyadic}
\begin{aligned}
 \mathcal{W}^{\mu}(x_0)
 &=
 \sum_{m\ge0}
 \int_{r_{m+1}}^{r_m}
 \left(
 \frac{\mu(B(x_0,r))}{r^s}
 \right)^\theta
 \frac{dr}{r}
 \\
 &\le
 \sum_{m\ge0}
 \mu(B(x_0,r_m))^\theta
 \int_{r_{m+1}}^{r_m}\frac{dr}{r^2}
 \\
 &=
 \sum_{m\ge0}
 \frac{\mu(B(x_0,r_m))^\theta}{r_m}.
\end{aligned}
\end{equation}
Because $x_0\in F$, Theorem~\ref{thm:evans-wolff} gives
$\mathcal{W}^{\mu}(x_0)=+\infty$.  Every term in the series on the
right-hand side of \eqref{eq:dyadic} is finite, since $\mu(B)\leq1$ and
$r_m>0$.  Hence the series, and therefore each of its tails, diverges,
and \eqref{eq:annuli} yields
$\operatorname{Len}_{\widetilde h}(\gamma)=+\infty$.  Thus every
divergent locally Lipschitz curve in $\Omega$ has infinite $h$-length,
and the
criterion established at the beginning of Step~3 shows that $h$ is
complete.
\end{proof}

The preceding metric yields all three scalar-curvature signs after constant
shifts.  The only additional inputs are the lower bound $m_0>0$, which follows
from the fact that $\mu$ is a probability measure supported on a bounded set,
and the linearity of the conformal Laplacian, together with
$L_{g_{st}}\phi=0$ pointwise on \(\Omega\) and
$L_{g_{st}}1=n(n-1)$.  The completeness comparison
and the scalar-curvature computation are independent parts of the argument.

\begin{lemma}[Complete shifts of every sign]\label{lem:shifts}
In the setting of Lemma~\ref{lem:onechart}, let $\phi$ be the
function constructed there and put $m_0:=\inf_\Omega\phi$.  Then
$m_0>0$, and for every $t\in(-m_0,\infty)$ the metric
$g^{(t)}:=(\phi+t)^{4/(n-2)}\,g_{st}\big|_\Omega$ is smooth, complete,
and locally conformally flat, with
\begin{equation}\label{eq:curvature-identity}
 \mathrm{Sc}_{g^{(t)}}
 =n(n-1)\,t\,(\phi+t)^{-(n+2)/(n-2)},
 \qquad
 \lim_{t\to0}\int_\Omega|\mathrm{Sc}_{g^{(t)}}|^{n/2}\,dV_{g^{(t)}}=0.
\end{equation}
In particular, $\mathrm{Sc}_{g^{(t)}}$ is negative, zero, or positive
according as $t<0$, $t=0$, or $t>0$.  For $t>0$, the rescaled metric
$\bar g^{(t)}:=t^{-4/(n-2)}g^{(t)}
=\Bigl(1+\frac{\phi}{t}\Bigr)^{4/(n-2)}g_{st}\big|_\Omega$ satisfies
\[
 0<\mathrm{Sc}_{\bar g^{(t)}}<n(n-1),
 \qquad
 \lim_{\delta\downarrow0}
 \sup_{\substack{x\in\Omega\\d_{g_{st}}(x,E)<\delta}}
 \mathrm{Sc}_{\bar g^{(t)}}(x)=0.
\]
\end{lemma}

\begin{proof}
Retain the notation $F$, $\mu$, and $R_F$ from the construction in
Lemma~\ref{lem:onechart}, and put $\beta_0:=\frac{n-2}{2}$.

\emph{The lower bound.}  In the chart, by \eqref{eq:phi-chart},
\[
 \phi\circ\sigma_{p_\infty}^{-1}(x)
 =2^{-\beta_0}(1+|x|^2)^{\beta_0}
  \int_F|x-y|^{-2\beta_0}\,d\mu(y).
\]
For $y\in F$, $|x-y|\leq|x|+R_F\leq(1+R_F)(1+|x|)$, and
$1+|x|^2\geq\frac12(1+|x|)^2$.  Since $\mu(F)=1$, for every
$x\in\mathbb R^n\setminus F$ one has
\begin{equation}\label{eq:m0}
\begin{aligned}
 \phi\circ\sigma_{p_\infty}^{-1}(x)
 &\geq 2^{-\beta_0}\cdot2^{-\beta_0}(1+|x|)^{2\beta_0}
 (1+R_F)^{-2\beta_0}(1+|x|)^{-2\beta_0}\\
 &=2^{-(n-2)}(1+R_F)^{-(n-2)}.
\end{aligned}
\end{equation}
Equation~\eqref{eq:m0} applies on
$\sigma_{p_\infty}^{-1}(\mathbb R^n\setminus F)
=\Omega\setminus\{p_\infty\}$.  At the remaining point,
$\phi(p_\infty)=2^{-\beta_0}$ and
$2^{-\beta_0}=(\sqrt2)^{-(n-2)}
>[2(1+R_F)]^{-(n-2)}
=2^{-(n-2)}(1+R_F)^{-(n-2)},$
because $2(1+R_F)\geq2>\sqrt2$.  Thus the same lower bound holds at
$p_\infty$, and hence
$m_0\geq2^{-(n-2)}(1+R_F)^{-(n-2)}>0$.

\emph{Smoothness and completeness.}  For $t>-m_0$, $\phi+t\ge m_0+t>0$,
so $g^{(t)}$ is a smooth metric conformal to $g_{st}$, hence
locally conformally flat.  With
$a_t:=\min\{1,1+t/m_0\}>0$ and $b_t:=\max\{1,1+t/m_0\}$, the bound
$\phi\ge m_0$ gives $a_t\le1+t/\phi\le b_t$; that is,
$a_t^{4/(n-2)}h\le g^{(t)}=(1+t/\phi)^{4/(n-2)}h\le b_t^{4/(n-2)}h$.
Taking lengths and then infima over curves gives
\(a_t^{2/(n-2)}d_h\le d_{g^{(t)}}\le b_t^{2/(n-2)}d_h\).
Thus $d_h$ and $d_{g^{(t)}}$ are bi-Lipschitz equivalent, so the
completeness of $h$ implies that of $g^{(t)}$.

\emph{Curvature.}
Since $\mathrm{Sc}_{g_{st}}=n(n-1)$ and
$\mathrm{Sc}_{w^{4/(n-2)}g_{st}}=w^{-(n+2)/(n-2)}L_{g_{st}}w$
for smooth $w>0$, scalar-flatness of $h$ means
$L_{g_{st}}\phi=0$ pointwise on \(\Omega\), and
$L_{g_{st}}1=n(n-1)$.  By
linearity, $L_{g_{st}}(\phi+t)=n(n-1)\,t$, and therefore
$\mathrm{Sc}_{g^{(t)}}
=n(n-1)\,t\,(\phi+t)^{-(n+2)/(n-2)},$
which is the first part of \eqref{eq:curvature-identity}; since
$\phi+t>0$, the sign of $\mathrm{Sc}_{g^{(t)}}$ is the sign of $t$.
For the second part, note that
$dV_{g^{(t)}}=(\phi+t)^{2n/(n-2)}\,dV_{g_{st}}$ and
$\frac{2n}{n-2}-\frac{n(n+2)}{2(n-2)}=-\frac n2$, so that
\[
\begin{aligned}
 \int_\Omega|\mathrm{Sc}_{g^{(t)}}|^{n/2}\,dV_{g^{(t)}}
 &=[n(n-1)]^{n/2}
 \int_\Omega\Bigl(\frac{|t|}{\phi+t}\Bigr)^{n/2}dV_{g_{st}}\\
 &\le[n(n-1)]^{n/2}\operatorname{Vol}_{g_{st}}(S^n)
 \Bigl(\frac{|t|}{m_0+t}\Bigr)^{n/2},
\end{aligned}
\]
where we use $\phi+t\ge m_0+t>0$.  The right-hand side tends to $0$
as $t\to0$.

\emph{The rescaled family.}
Under constant rescaling, scalar curvature transforms by
$\mathrm{Sc}_{\varrho g}=\varrho^{-1}\mathrm{Sc}_g$ for $\varrho>0$.  For
$t>0$, taking $\varrho=t^{-4/(n-2)}$ and using
$\frac4{n-2}+1=\frac{n+2}{n-2}$, we obtain
\[
 \mathrm{Sc}_{\bar g^{(t)}}
 =t^{\frac4{n-2}}\,\mathrm{Sc}_{g^{(t)}}
 =n(n-1)\,t^{\frac{n+2}{n-2}}(\phi+t)^{-\frac{n+2}{n-2}}
 =n(n-1)\Bigl(1+\frac{\phi}{t}\Bigr)^{-\frac{n+2}{n-2}}.
\]
Since $\phi>0$, this lies in $(0,n(n-1))$, and the uniform divergence
\eqref{eq:phi-divergence} of $\phi$ at $E$ gives
$\mathrm{Sc}_{\bar g^{(t)}}(x)\to0$ uniformly as
$d_{g_{st}}(x,E)\to0$.
\end{proof}

Finally, we assemble the preceding topological and analytic
constructions.

\begin{proof}[Proof of Theorem~\ref{thm:counterexample-psc}]
Fix $n\ge4$, fix the tower of Proposition~\ref{induction}, and put
$K:=\bigcap_{j\ge1}V_j$ and $M:=S^n\setminus K$.

\smallskip\noindent\emph{The set $K$.}
By Proposition~\ref{prop:limitset}, $K$ is a nondegenerate continuum
with $\dim_{\mathcal{H}}K=1$,
$K\subset V_1\subset\operatorname{int}B\Subset S^n\setminus\{p_\sigma\}$, and
$\operatorname{cap}(K)=
\mathcal{C}_{1+2/n,\,n/2}^{\mathbb R^n}\bigl(\sigma(K)\bigr)=0.$
In particular, $p_\sigma\in M$, and $M$, being an open subset of $S^n$,
is a smooth manifold.  It is noncompact: indeed, $M$ is nonempty and
proper because $p_\sigma\in M$ and $K\neq\varnothing$, and if $M$ were compact,
then it would be both open and closed in the connected sphere $S^n$.

\smallskip\noindent\emph{Assertion (1).}
$M$ is contractible by Proposition~\ref{prop:contract} and not simply
connected at infinity by Proposition~\ref{prop:sci}.

\smallskip\noindent\emph{Assertion (2): the scalar-flat metric.}
We apply Lemma~\ref{lem:onechart} with
$(\Omega,E,p_\infty):=(M,K,p_\sigma)$.  Its
hypotheses hold: $n\ge4$; $E=K$ is nonempty and compact; $\Omega=M$ is
connected, being contractible; $p_\infty=p_\sigma\in M$; and the
stereographic chart with pole $p_\sigma$ is the fixed chart $\sigma$, so the
capacity assumption
$\mathcal{C}_{1+2/n,\,n/2}^{\mathbb R^n}(\sigma(K))=0$ is exactly
$\operatorname{cap}(K)=0$.  The lemma provides a smooth function
$\phi:M\to(0,\infty)$ such that $h:=\phi^{4/(n-2)}g_{st}|_M$ is
complete and scalar-flat, with $\phi(x)\to+\infty$ uniformly as
$d_{g_{st}}(x,K)\to0$.

\smallskip\noindent\emph{Assertion (2): the shifted family.}
We apply Lemma~\ref{lem:shifts} to this $\phi$.  It gives
$m_0=\inf_M\phi>0$; for every $t\in(-m_0,\infty)$ the metric
$g^{(t)}=(\phi+t)^{4/(n-2)}g_{st}|_M$ is smooth, complete, and
locally conformally flat, with
$\mathrm{Sc}_{g^{(t)}}=n(n-1)\,t\,(\phi+t)^{-(n+2)/(n-2)},$
which has the sign of $t$ because $\phi+t>0$ on $M$ (and $g^{(0)}=h$).
Moreover,
$\int_M|\mathrm{Sc}_{g^{(t)}}|^{n/2}\,dV_{g^{(t)}}\to0$ as $t\to0$
by \eqref{eq:curvature-identity}.  Finally, for $t>0$,
Lemma~\ref{lem:shifts} gives
$\mathrm{Sc}_{\bar g^{(t)}}=n(n-1)(1+\phi/t)^{-(n+2)/(n-2)}\in(0,n(n-1))$
for $\bar g^{(t)}=t^{-4/(n-2)}g^{(t)}$, and the uniform divergence of
$\phi$ at $K$ yields
\[
 \lim_{\delta\downarrow0}
 \sup_{\substack{x\in M\\ d_{g_{st}}(x,K)<\delta}}
 \mathrm{Sc}_{\bar g^{(t)}}(x)=0.
\]
This proves (2) and completes the proof.
\end{proof}

\begin{corollary}[Failure of Euclidean rigidity without the
additional hypotheses]\label{cor:sharpness}
Fix $n\geq4$, and let $K$, $M=S^n\setminus K$, and
$(g^{(t)})_{t\in(-m_0,\infty)}$ be as in Theorem~\ref{thm:counterexample-psc}.  The
manifold $M$ is connected, smooth, open, and simply connected, and
\(\widetilde H_i(M;\mathbb Z)=0\) for \(i\geq0\).
For every $t\geq0$, the metric $g^{(t)}$ is complete and locally
conformally flat, with
$\mathrm{Sc}_{g^{(t)}}\geq0$.  Moreover, the inclusion
$\iota:M\hookrightarrow S^n$ may be chosen as a developing map of
$(M,g^{(t)})$.  The omitted set, and also the topological boundary of
the developing image, is
$\Lambda:=S^n\setminus\iota(M)=\partial\iota(M)=K$,
which is a nondegenerate continuum.  Nevertheless, $M$ is not
homeomorphic to $\mathbb R^n$.

Consequently, the conclusion of Theorem~\ref{thm:intro-E}\textup{(ii)} can fail
when its two additional topological hypotheses are omitted
simultaneously.  The disjunction of conditions
\textup{(i)}--\textup{(v)} in Corollary~\ref{Contractible for} likewise
cannot be omitted, even if its conclusion is weakened from conformal
equivalence to homeomorphism.
\end{corollary}

\begin{proof}
Theorem~\ref{thm:counterexample-psc}\textup{(1)} states that $M$ is a contractible smooth
open manifold.  Hence $M$ is connected and simply connected and has
vanishing reduced integral homology; in particular,
$\widetilde H_{n-1}(M;\mathbb Z)=0$, as required in
Corollary~\ref{Contractible for}.  The same part of the theorem states
that $M$ is not simply connected at infinity.  Since $\mathbb R^n$ is
simply connected at infinity for $n\geq3$, and simple connectivity at
infinity is invariant under homeomorphism, $M$ is not homeomorphic to
$\mathbb R^n$.

For $t\geq0$, Theorem~\ref{thm:counterexample-psc}\textup{(2)} gives completeness, local
conformal flatness, and $\mathrm{Sc}_{g^{(t)}}\geq0$.  If
$\iota:M\hookrightarrow S^n$ is the inclusion, then
$\iota^*g_{st}=(\phi+t)^{-4/(n-2)}g^{(t)}.$
Thus $\iota$ is a conformal local diffeomorphism and is a developing map
for the locally conformally flat structure of $g^{(t)}$.  Its image is
$M=S^n\setminus K$.  Since $\dim_{\mathcal H}K=1<n$, the set $K$ has
empty interior in $S^n$; hence the omitted set and the topological
boundary of the developing image are both exactly $K$.  Since $M$ is
simply connected, the holonomy is trivial;
accordingly, $K$ is not the classical Kleinian limit set of the holonomy
group (which is empty).

After both additional hypotheses are removed from
Theorem~\ref{thm:intro-E}\textup{(ii)}, the preceding facts verify every remaining
hypothesis, whereas the conclusion $M\cong\mathbb R^n$ fails.  Moreover,
$\lfloor(n-2)/2\rfloor\geq1$, so the failure of simple connectivity at
infinity implies that $M$ is not
$\lfloor(n-2)/2\rfloor$-connected at infinity; as stated at the beginning
of this section, no assertion is made about the small loops condition.
Finally, all standing hypotheses of Corollary~\ref{Contractible for}
hold, including the required homology vanishing, while its conclusion
fails even topologically.  This proves both sharpness assertions.
\end{proof}

\subsection{\texorpdfstring{Uniformly PSC for $n\geq 6$}{Uniformly PSC for n >= 6}}

The construction underlying Theorem~\ref{thm:counterexample-psc} can be modified to
produce locally conformally flat metrics with uniformly positive scalar
curvature on contractible manifolds not homeomorphic to $\mathbb R^n$.
The rose is replaced by a $2$-dimensional Newman spine.  The topological
construction requires only $n\geq 5$, corresponding to codimension at
least three for a $2$-complex, while the analytic argument requires the
spine dimension $\mathfrak d=2$ to satisfy
$\mathfrak d\leq a=\frac{n-2}{2}$, and hence $n\geq 6$.

\begin{theorem}\label{thm:uniform-psc-counterexample}
For every integer $n\geq 6$ there exists a contractible smooth open
$n$-manifold $M$ which is not homeomorphic to $\mathbb R^n$ and which
admits a smooth complete locally conformally flat metric
$\widehat g$ such that
\[
  0<\inf_M\mathrm{Sc}_{\widehat g}
  \le \sup_M\mathrm{Sc}_{\widehat g}<\infty.
\]
\end{theorem}

\medskip
\noindent\textit{Strategy of the proof.}
We refine Newman's construction to obtain a finite pure PL
$2$-complex $P\subset S^n$ whose complement is contractible but not
homeomorphic to $\mathbb R^n$, and whose area measure is Ahlfors
$2$-regular.  Setting $a:=s=(n-2)/2$, we define
$\widehat g=(1+V)^{4/(n-2)}g_{st}$ using the potential of this measure
with chordal kernel $\mathfrak q^{-(2+a)}$ when $a>2$, and its logarithmically
corrected analogue when $a=2$.  Indeed, for a power kernel $\mathfrak q^{-b}$
with $2<b<n-2$, the $2$-regular potential estimate gives
$V\asymp\rho^{2-b}$, while the positive $\mathfrak q^{-b-2}$ term in
$L_{g_{st}}\mathfrak q^{-b}$ gives $L_{g_{st}}V\asymp\rho^{-b}$.  Thus the
matching condition
\(L_{g_{st}}V\asymp V^{(a+2)/a}\)
requires
\(-b=(2-b)\frac{a+2}{a}\),
equivalently \(b=a+2\).
Thus $b=a+2>2$, whereas $b<n-2=2a$ exactly when $a>2$.  At the
endpoint $a=2$, the matching exponent is $b=a+2=n-2$, the
conformal-Laplacian-harmonic exponent, and the logarithmic
correction restores a positive leading term.  In this critical case the
logarithmic exponent is forced by the same matching: if
$V\asymp\rho^{-a}\ell(\rho)^{-\lambda}$ and
$L_{g_{st}}V\asymp\rho^{-a-2}\ell(\rho)^{-\lambda-1}$, then
$L_{g_{st}}V\asymp V^\kappa$, where \(\kappa=(a+2)/a\), requires
$\lambda\kappa=\lambda+1$, and hence
$\lambda=1/(\kappa-1)=a/2$.
Ahlfors-regular potential estimates, combined with the conformal
scalar-curvature calculation,
then prove completeness and the required two-sided scalar-curvature
bound; the condition $2\leq a$ is exactly $n\geq 6$.

The construction in this subsection is independent of the chart
$\sigma$, its pole $p_\sigma$, the ball $B$, and the PL structure fixed
in the initial setup of this section.  For the proof, fix
$p_\infty\in S^n$ and a stereographic chart
$\sigma_\infty:S^n\setminus\{p_\infty\}\to\mathbb R^n$.  Whenever a compact
set or measure is first constructed in $\mathbb R^n$, we use the same
symbol for the image of the set or the push-forward of the measure under
$\sigma_\infty^{-1}$; in particular,
inclusions such as $\mathbb R^5\subset\mathbb R^n\subset S^n$ below
are understood through this identification.  Viewing $S^n$ as the unit
sphere in $\mathbb R^{n+1}$, define the chordal distance by
$\mathfrak q(x,y):=|x-y|_{\mathbb R^{n+1}}
=2\sin\!\left(\frac{d_{g_{st}}(x,y)}{2}\right)\in[0,2].$
Then $\frac{2}{\pi}d_{g_{st}}(x,y)\leq \mathfrak q(x,y)\leq d_{g_{st}}(x,y)$,
and $\mathfrak q(x,y)>0$ precisely when $x\neq y$.
Recall that $a=s=\frac{n-2}{2}$, and set
$\kappa=\frac{n+2}{n-2}=\frac{a+2}{a}$ and
$\ell(u)=\log\frac{8}{u}$,
so that $\ell(\mathfrak q)\geq\log4$ for $0<\mathfrak q\leq2$.  For positive functions on a
specified set, $f\asymp g$ means $cg\leq f\leq Cg$ for positive
constants independent of the variables under consideration; dependence
on fixed parameters is understood.  When a limiting regime is specified,
the comparison is asserted in that regime; otherwise it holds throughout
the indicated set.

A \emph{$\mathfrak d$-regular pair} in $S^n$ is a compact set
$\Lambda\subset S^n$ together with a Borel probability measure $\mu$
with $\mathrm{supp}\,\mu=\Lambda$ such that, for some $r_0,c,C>0$,
\begin{equation}\label{eq:regular}
 cr^{\mathfrak d}\le\mu(B_{\mathfrak q}(p,r))\le Cr^{\mathfrak d}
 \qquad(p\in\Lambda,\ 0<r<r_0).
\end{equation}

The next lemma is Newman's classical construction \cite{zbMATH03054974},
whose general form, starting from a balanced presentation of a perfect
group, is described in \cite[Example~3.2.2]{zbMATH07206284}; the group
used below is Newman's own.  We record only the refinements needed
later: a pure PL spine in an affine chart, its Ahlfors-regular area
measure, and the identification of its complement with the
corresponding open Newman manifold.

\begin{lemma}[A $2$-regular Newman spine]\label{lem:topological}
For every $n\geq 5$ there is a finite connected pure simplicial $2$-complex
$P\subset\mathbb R^5\subset\mathbb R^n\subset S^n$ such that
\(\widetilde H_*(P;\mathbb Z)=0\)
and \(\pi_1(P)\ne1\),
the complement $M=S^n\setminus P$ is contractible but not homeomorphic
to $\mathbb R^n$, and the push-forward of the probability measure
obtained by normalizing Euclidean area on the Euclidean realization of
$P$ makes $(P,\mu)$ a $2$-regular pair.
\end{lemma}

\begin{proof}
Following Newman \cite{zbMATH03054974}, let
\[
 G=\langle x,y\mid r_1,r_2\rangle,
 \qquad
 r_1=y^{-3}(xy)^2,
 \qquad
 r_2=x^{-5}(xy)^2.
\]
Let $X$ be the associated presentation $2$-complex, with unique
$0$-cell $v_0$, oriented generator $1$-cells $x,y$, and one $2$-cell
attached along each of $r_1,r_2$.
Under the substitution $(x,y)=(B,A)$, the relators become Newman's words,
$A^{-3}(BA)^2,\qquad B^{-5}(BA)^2,$
so $X$ is his complex $Z$.  The only group-theoretic fact imported from
Newman at this point is that $G\neq1$.  As Newman observes,
this follows from the fact that the relations hold in the icosahedral
group $A_5$.  Indeed, the assignments
$x\longmapsto(1\,2\,3\,4\,5),\qquad
y\longmapsto(1\,4\,2)$
satisfy $x^5=y^3=(xy)^2=1$, since $xy$ is a product of two disjoint
transpositions, regardless of the convention for composition.  They
therefore define a homomorphism $G\to A_5$ with nontrivial image.  The
exponent-sum matrix
$E=\bigl(\begin{smallmatrix}2&-1\\-3&2\end{smallmatrix}\bigr)$ has
$\det E=1$.  The map $E^T$ is both the map whose cokernel is the
abelianization of $G$ and, by the cellular boundary formula
\cite[Section~2.2, pp.~140--141]{zbMATH02103273}, the cellular boundary
$\partial_2$ of $X$.  Since $X$ has a single $0$-cell and both
$1$-cells are loops at that cell, $\partial_1=0$.  Moreover, the
cellular augmentation $\varepsilon:C_0(X)=\mathbb Z\to\mathbb Z$ is
an isomorphism.  Thus the augmented cellular chain complex of $X$ is
\[
 0\longrightarrow\mathbb Z^2
 \xrightarrow{\ E^T\ }\mathbb Z^2
 \xrightarrow{\ 0\ }\mathbb Z
 \xrightarrow{\ \varepsilon\ }\mathbb Z
 \longrightarrow0.
\]
Because \(\det E=1\), the map $E^T$ is an isomorphism; the augmentation
is an isomorphism as well.  Hence the
cokernel, which is \(G^{\mathrm{ab}}\), is zero, and all reduced cellular
homology groups of \(X\) vanish.  Thus $G$ is perfect, $X$ is acyclic,
and $\pi_1(X)=G\ne1$.  In particular, $X$ is path connected.

The metric part of the lemma requires a pure simplicial realization of
$X$.  Subdivide the attaching circle of each relator $2$-cell into one
edge for each letter of the relator, map that edge affinely onto the
generator edge $x$ or $y$ with the orientation prescribed by the sign
of the letter, and cone the resulting polygonal circle to a new interior
vertex $w$, chosen separately for each relator.  Give every radial edge
the order from $w$ to $v_0$.  For a cone triangle over
a letter edge, label its two boundary
vertices $q_0,q_1$ so that the restriction to the ordered outer face
$[q_0,q_1]$ is the chosen characteristic map of the corresponding
generator edge, rather than its reverse, and order the triangle as
$[w,q_0,q_1]$.  This choice is also available for a letter with exponent
$-1$, since both boundary vertices are identified with $v_0$.  The
restrictions to the radial faces $[w,q_0]$ and $[w,q_1]$
are then the chosen radial characteristic maps.  Moreover, a radial edge
common to two adjacent triangles receives, from both triangles, the order
from $w$ to $v_0$.  Consequently, the restriction
of every triangle characteristic map to each face is exactly the
characteristic map of the corresponding $1$-simplex under the
order-preserving identification of that face with $\Delta^1$.  The
quotient maps embed the interiors of all the vertex, edge, and triangle
simplices just described, and these open simplices partition $X$; moreover,
$X$ carries the quotient topology by construction.  The quotient is
therefore a finite $2$-dimensional $\Delta$-complex.  Repeated
occurrences of a generator edge in an attaching word are permitted in a
$\Delta$-complex.
This $\Delta$-complex is pure, because every radial edge is a face of a
cone triangle and both generator edges occur in $r_1$.

The second barycentric subdivision of a $\Delta$-complex is a simplicial
complex; see \cite[Exercise~23, p.~133]{zbMATH02103273}.  Indeed, after the
first barycentric subdivision the vertices of each simplex correspond
to a strictly increasing flag of faces and are therefore distinct; in
the next barycentric subdivision each simplex is uniquely determined by
the set of vertices corresponding to its flag.  Barycentric subdivision
preserves finiteness, dimension, and purity.  Denote the resulting finite
pure simplicial $2$-complex by $P$; its underlying polyhedron is
canonically homeomorphic to $X$, so $P$ is acyclic and
$\pi_1(P)=G\neq1$.  Since $P$ is a finite polyhedron of dimension
$p=2$ and the target is $\mathbb R^5$,
\[
2p+1=5=\dim\mathbb R^5,
\]
the dimension hypothesis in
\cite[Corollary~4.4, p.~235]{zbMATH01714503} holds at equality.
Apply that corollary to the constant map $P\to\mathbb R^5$, with
$X_0=\varnothing$ and the constant control function $\varepsilon\equiv1$.
The restriction to $X_0$ is vacuously a PL embedding, so the corollary
gives a PL embedding $P\hookrightarrow\mathbb R^5$.
After subdivision, identify $P$ with its linear image in
$\mathbb R^5\subset\mathbb R^n$ and, according to the convention above,
with its inverse-stereographic image in $S^n$.

We make the PL structure used below explicit.  Choose a large affine
$n$-simplex $\mathcal{B}\subset\mathbb R^n$ with
$P\subset\operatorname{int}\mathcal{B}$.  A common subdivision gives a
finite triangulation of $\mathcal{B}$ in which both $P$ and
$\partial\mathcal{B}$ are subcomplexes.  Transport this triangulation to
$\sigma_\infty^{-1}(\mathcal{B})$ and triangulate the complementary
topological $n$-ball by coning the induced boundary triangulation to
$p_\infty$; equivalently, the required homeomorphism is obtained by
compactifying the rays from an interior point of $\mathcal{B}$.  We use
the resulting PL structure on $S^n$.  In particular, the specified
inverse-stereographic image of $P$ is a subpolyhedron of this PL sphere.

After a derived subdivision relative to $P$, assume that $P$ is the
polyhedron of a full subcomplex $\mathcal{L}_P$ of a finite triangulation
$\mathcal{T}_P$ of $S^n$.  The subcomplex $\mathcal{L}_P$ is nonempty
and proper, because $P$ is nonempty and
$P\subset S^n\setminus\{p_\infty\}$.  Let
$\xi_{\mathcal{N}}:|\mathcal{T}_P|\to[0,1]$ be the simplicial map taking
value $0$ on the vertices of $\mathcal{L}_P$ and value $1$ on all other
vertices, and choose the half-neighborhood model
$\mathcal{N}:=\xi_{\mathcal{N}}^{-1}([0,1/2]).$
Lemma~\ref{lem:sublevel}, with $c=1/2$, shows that $\mathcal{N}$ is a
compact regular neighborhood of $P$ and a PL $n$-manifold, with
$\operatorname{int}\mathcal{N}
=\xi_{\mathcal{N}}^{-1}([0,1/2)),\qquad
\partial\mathcal{N}=\xi_{\mathcal{N}}^{-1}(1/2).$
Thus the half-neighborhood is chosen from the outset; no arbitrary
regular neighborhood is being identified silently with Bryant's model.
Set $\mathcal{L}_P^{\mathrm c}:=C(\mathcal{L}_P,\mathcal{T}_P).$
By Lemma~\ref{lem:sublevel}, $\mathcal{L}_P^{\mathrm c}$ is a nonempty
proper full subcomplex of $\mathcal{T}_P$, and its level function is
$1-\xi_{\mathcal{N}}$.  Since the preceding interior formula gives
$S^n\setminus\operatorname{int}\mathcal{N}
=\xi_{\mathcal{N}}^{-1}([1/2,1])$, we may write
\[
C:=S^n\setminus\operatorname{int}\mathcal{N}
=\xi_{\mathcal{N}}^{-1}([1/2,1])
=(1-\xi_{\mathcal{N}})^{-1}([0,1/2]).
\]
A second application of Lemma~\ref{lem:sublevel}, now to
$(\mathcal{T}_P,\mathcal{L}_P^{\mathrm c},1-\xi_{\mathcal{N}})$ with
$c=1/2$, shows that $C$ is a closed regular neighborhood of
$|\mathcal{L}_P^{\mathrm c}|$ and a compact PL $n$-manifold, with
\[
\begin{aligned}
 \partial C
 &=(1-\xi_{\mathcal{N}})^{-1}(1/2)
   =\xi_{\mathcal{N}}^{-1}(1/2)
   =\partial\mathcal{N},\\
 \operatorname{int}C
 &=(1-\xi_{\mathcal{N}})^{-1}([0,1/2))
   =\xi_{\mathcal{N}}^{-1}((1/2,1]).
\end{aligned}
\]
Consequently,
\begin{equation}\label{eq:half-neighborhood-intersection}
 \mathcal{N}\cap C
 =\xi_{\mathcal{N}}^{-1}(1/2)
 =\partial\mathcal{N}
 =\partial C.
\end{equation}
Thus $\partial\mathcal{N}$ and $\partial C$ are the same subpolyhedron
$\xi_{\mathcal{N}}^{-1}(1/2)$ of the fixed PL sphere $S^n$ and hence
carry the same PL structure induced from $S^n$.  In particular, the
one-sided PL collars used below have the same PL base.
We give Newman's argument for this chosen embedding; see
\cite[\S\S~2--3, pp.~193--195]{zbMATH03054974} and
\cite[Example~3.2.2]{zbMATH07206284}.

By the punctured-sublevel conclusion of Lemma~\ref{lem:sublevel},
applied to $(\mathcal{T}_P,\mathcal{L}_P,\xi_{\mathcal{N}})$, there is a
homeomorphism of pairs
\[
 (\mathcal{N}\setminus P,\partial\mathcal{N})
 \cong
 (\partial\mathcal{N}\times(0,1],
  \partial\mathcal{N}\times\{1\}).
\]
The corresponding factor homotopy is a strong deformation retraction
of $\mathcal{N}\setminus P$ onto $\partial\mathcal{N}$; denote its terminal
retraction by
$r:\mathcal{N}\setminus P\to\partial\mathcal{N}$.  Every map in this
deformation fixes $\partial\mathcal{N}$ pointwise.  Since
$P=\xi_{\mathcal{N}}^{-1}(0)$ is contained in
$\operatorname{int}\mathcal{N}$, it is disjoint from $C$.  Hence the
two sets $\mathcal{N}\setminus P$ and $C$ are closed in
$S^n\setminus P$, cover it, and intersect in $\partial\mathcal{N}$ by
\eqref{eq:half-neighborhood-intersection}.
The closed-cover pasting lemma, applied to the products of these two
sets with the parameter interval, therefore pastes this deformation to
the identity deformation on $C$ and gives a strong deformation
retraction
\begin{equation}\label{eq:newman-retraction}
 S^n\setminus P\longrightarrow C.
\end{equation}
Finally, $\mathcal{N}$ collapses onto $P$ by
\cite[Theorem~3.16]{zbMATH01714503}; hence the associated collapse
deformation is a strong deformation retraction fixing $P$ pointwise.

The finite polyhedron $P$ is nonempty, compact, locally contractible, and
proper in $S^n$.  Since $P$ is acyclic, the universal coefficient theorem
gives $\widetilde H^*(P;\mathbb Z)=0$.  The reduced-complement form of
Alexander duality \cite[Corollary~3.45]{zbMATH02103273} and
\eqref{eq:newman-retraction} therefore give
\[
 \widetilde H_i(C;\mathbb Z)
 \cong \widetilde H_i(S^n\setminus P;\mathbb Z)
 \cong \widetilde H^{\,n-i-1}(P;\mathbb Z)=0
 \qquad(i\geq0).
\]
Thus $C$ is acyclic.
The boundary $\partial\mathcal{N}$ is nonempty.  Indeed, $\mathcal{N}$
is connected because it deformation retracts onto the connected
polyhedron $P$; if its boundary were empty, the codimension-zero
submanifold $\mathcal{N}\subset S^n$ would be both open and
closed and hence equal to $S^n$, which is impossible because $\mathcal{N}\simeq P$
and $H_n(P;\mathbb Z)=0$.  In particular, $C$ is nonempty.  The equality
$\widetilde H_0(C;\mathbb Z)=0$ also shows that $C$ is path connected.

We next prove simple connectivity.  Let
$\gamma_0:S^1\to S^n\setminus P$ be a loop and set
$d:=\operatorname{dist}_{g_{st}}(\gamma_0(S^1),P)>0$.  Choose
triangulations of $S^1$ and $S^n$ with $P$ a subcomplex.  After
subdivision, controlled simplicial approximation
\cite[p.~222]{zbMATH01714503}, with control $d/2$, gives a free homotopy
from $\gamma_0$ to a PL loop $\gamma$ for which every point track has
diameter less than $d/2$.  The homotopy therefore lies in
$S^n\setminus P$.  It is enough to contract $\gamma$, because the
homotopy annulus can then be pasted to a filling disk for $\gamma$ to
give a filling disk for $\gamma_0$.  Since $S^n$ is simply connected,
$\gamma$ extends to a continuous map $F_0:D^2\to S^n$.  Extend the
boundary triangulation over $D^2$.  Relative simplicial approximation
\cite{zbMATH03195022} gives a PL map
$F:D^2\to S^n$ with $F|_{\partial D^2}=\gamma$.  Put
$Q:=F(D^2)$ and $Q_0:=\gamma(S^1)$, and choose a common subdivision in
which $Q_0\subset Q$ and $P$ are subpolyhedra.  If $Q=Q_0$, the
map $F$ already misses $P$ and is the required null-homotopy.  Assume
therefore that $Q\ne Q_0$.  The relative PL general-position theorem
\cite[Theorem~4.2]{zbMATH01714503}, applied to
$(Q,Q_0,P)\subset S^n$, gives an ambient PL isotopy $\Psi_t$, fixed on
$Q_0$, such that
\[
 \dim\bigl(\Psi_1(Q\setminus Q_0)\cap P\bigr)
 \leq \dim(Q\setminus Q_0)+\dim P-n
 \leq 2+2-n<0.
\]
Since $Q_0\cap P=\varnothing$, the map $\Psi_1\circ F$ is a null-homotopy
of $\gamma$ in $S^n\setminus P$.  Hence $S^n\setminus P$, and therefore
$C$ by \eqref{eq:newman-retraction}, is simply connected.  Beyond the
embedding $P\hookrightarrow\mathbb R^5\subset\mathbb R^n$, the restriction
$n\geq5$ is used precisely through the inequality $2+2-n<0$, both in
the preceding simple-connectivity argument and in the injectivity argument
below.

If some $\pi_k(C)$ were nonzero, choose the least such $k\geq2$.
Then $C$ would be $(k-1)$-connected, and the Hurewicz theorem
\cite[Theorem~4.32]{zbMATH02103273} would give
$\pi_k(C)\cong H_k(C;\mathbb Z)=0$, a contradiction.  Thus $C$ is
weakly contractible.  As a compact PL manifold, $C$ has the homotopy
type of a finite CW complex.  Choose a finite CW complex \(Z_C\) and a
homotopy equivalence $Z_C\to C$.  The constant map $Z_C\to *$ is a weak
homotopy equivalence, so Whitehead's theorem
\cite[Theorem~4.5]{zbMATH02103273} makes it a homotopy equivalence.
Consequently, $C$ is contractible.

It remains to identify the fundamental group of the boundary.  Put
$\Sigma:=\partial\mathcal{N}=\partial C.$
By the half-neighborhood construction above, the two occurrences of $\Sigma$ carry the
same PL structure.  The PL collaring theorem
\cite[Corollary~2.5]{zbMATH01714503}, applied separately to
$\mathcal{N}$ and $C$, gives, after reparametrizing the collar
coordinates, PL collar embeddings
\[
 c_{\mathcal{N}}:\Sigma\times[-2,0]\longrightarrow\mathcal{N},
 \qquad
 c_C:\Sigma\times[0,2]\longrightarrow C
\]
such that
$c_{\mathcal{N}}(z,0)=c_C(z,0)=z\qquad(z\in\Sigma).$
The normalization at parameter $0$ and injectivity of the collar
embeddings imply
\begin{equation}\label{eq:collar-zero-sections}
 c_{\mathcal{N}}^{-1}(\Sigma)=\Sigma\times\{0\},
 \qquad
 c_C^{-1}(\Sigma)=\Sigma\times\{0\}.
\end{equation}
Indeed, if $c_{\mathcal{N}}(z,t)=w\in\Sigma$, then
$c_{\mathcal{N}}(w,0)=w$ by the normalization, so injectivity of
$c_{\mathcal{N}}$ gives $(z,t)=(w,0)$.  The reverse inclusion follows
from the normalization, and the proof for $c_C$ is identical.
Define
\[
 \overline\chi_\Sigma(z,t):=
 \begin{cases}
  c_{\mathcal{N}}(z,t),&-2\leq t\leq0,\\
  c_C(z,t),&0\leq t\leq2.
 \end{cases}
\]
The two formulas agree on the common subpolyhedron
$\Sigma\times\{0\}$, so the PL pasting lemma shows that
$\overline\chi_\Sigma$ is PL.  We next verify its injectivity.  If the
two parameters both lie in $[-2,0]$, or both lie in $[0,2]$, this
follows from injectivity of $c_{\mathcal{N}}$ or $c_C$, respectively.
In the only remaining case, after interchanging the two points, the
parameters satisfy $t<0<t'$.  Equality of their images would put the
common image in $\mathcal{N}\cap C=\Sigma$ by
\eqref{eq:half-neighborhood-intersection}, contrary to
\eqref{eq:collar-zero-sections}.  Hence $\overline\chi_\Sigma$ is
injective.  Since $\Sigma\times[-2,2]$ is compact and $S^n$ is
Hausdorff, $\overline\chi_\Sigma$ is an embedding.  Therefore
$\chi_\Sigma:=\overline\chi_\Sigma\big|_{\Sigma\times(-2,2)}
:\Sigma\times(-2,2)\to  S^n$ is a PL embedding.  Since
$\Sigma=\partial\mathcal{N}$ and
$\mathcal{N}$ is a compact PL $n$-manifold, $\Sigma$ is a closed PL
$(n-1)$-manifold.  PL invariance of domain therefore shows that the
image of $\chi_\Sigma$ is open.  Thus $\chi_\Sigma$ is a PL bicollar of
$\Sigma$, with
\[
 \chi_\Sigma(\Sigma\times(-2,0])\subset\mathcal{N},
 \qquad
 \chi_\Sigma(\Sigma\times[0,2))\subset C.
\]
More precisely, \eqref{eq:collar-zero-sections} and
$\partial\mathcal{N}=\partial C=\Sigma$ give
\begin{equation}\label{eq:bicollar-strict-sides}
 \chi_\Sigma(\Sigma\times(-2,0))
 \subset\operatorname{int}\mathcal{N},
 \qquad
 \chi_\Sigma(\Sigma\times(0,2))
 \subset\operatorname{int}C.
\end{equation}
The PL bicollar comes from the collaring theorem.  The homeomorphism
$\Phi_c$ of Lemma~\ref{lem:sublevel} is not used to construct
$\chi_\Sigma$.
Define
\[
 \mathcal{O}_{\mathcal{N}}
 :=\operatorname{int}\mathcal{N}\cup \chi_\Sigma(\Sigma\times(-2,1)),
 \qquad
\mathcal{O}_C
 :=\operatorname{int}C\cup \chi_\Sigma(\Sigma\times(-1,2)).
\]
They are open because $\chi_\Sigma$ is a homeomorphism onto an open
subset of $S^n$, and they cover $S^n$ because
$S^n=\operatorname{int}\mathcal{N}\sqcup\Sigma\sqcup\operatorname{int}C$.
For $0\leq u\leq1$, let $H^{\mathcal{N}}_u$ be the identity on
$\mathcal{N}$ and set
\[
 H^{\mathcal{N}}_u\bigl(\chi_\Sigma(z,t)\bigr)
 :=\chi_\Sigma\bigl(z,(1-u)t\bigr)
 \qquad(0\leq t<1).
\]
Likewise, let $H^C_u$ be the identity on $C$ and set
\[
 H^C_u\bigl(\chi_\Sigma(z,t)\bigr)
 :=\chi_\Sigma\bigl(z,(1-u)t\bigr)
 \qquad(-1<t\leq0).
\]
Since
$\mathcal{N}=\operatorname{int}\mathcal{N}\cup\Sigma$ and
$C=\operatorname{int}C\cup\Sigma$, the bicollar-side inclusions give
the covers
\[
 \mathcal{O}_{\mathcal{N}}
 =\mathcal{N}\cup\chi_\Sigma(\Sigma\times[0,1)),
 \qquad
 \mathcal{O}_C
 =C\cup\chi_\Sigma(\Sigma\times(-1,0]).
\]
These are closed covers in their respective ambient open sets.  Indeed,
$\mathcal{N}$ and $C$ are compact and hence closed in $S^n$.  Moreover,
\eqref{eq:half-neighborhood-intersection} and
\eqref{eq:bicollar-strict-sides} give
\[
\begin{aligned}
 \mathcal{N}\cap\chi_\Sigma(\Sigma\times[0,1))&=\Sigma,\\
 C\cap\chi_\Sigma(\Sigma\times(-1,0])&=\Sigma.
\end{aligned}
\]
Consequently,
\[
\begin{aligned}
 \mathcal{O}_{\mathcal{N}}
 \setminus\chi_\Sigma(\Sigma\times[0,1))
 &=\mathcal{N}\setminus\Sigma=\operatorname{int}\mathcal{N},\\
 \mathcal{O}_C
 \setminus\chi_\Sigma(\Sigma\times(-1,0])
 &=C\setminus\Sigma=\operatorname{int}C.
\end{aligned}
\]
The right-hand sides are open, so the two collar strips are closed
relative to $\mathcal{O}_{\mathcal{N}}$ and $\mathcal{O}_C$,
respectively.
The formulas agree on the common set $\Sigma$ in each cover.  Applied
to the products of these closed covers with $[0,1]$, the pasting lemma
therefore shows that
$\{H^{\mathcal{N}}_u\}_{0\leq u\leq1}$ and
$\{H^C_u\}_{0\leq u\leq1}$ are strong deformation retractions of
$\mathcal{O}_{\mathcal{N}}$ onto $\mathcal{N}$ and of
$\mathcal{O}_C$ onto $C$, respectively.  The same side-separation
identities also give
$\mathcal{O}_{\mathcal{N}}\cap\mathcal{O}_C
=\chi_\Sigma(\Sigma\times(-1,1))$
and the homotopy
$\chi_\Sigma(z,t)\mapsto\chi_\Sigma(z,(1-u)t)$ is a strong deformation
retraction of this intersection onto $\Sigma$.  Both $\mathcal{N}$ and
$C$ are path
connected: $\mathcal{N}$ deformation retracts onto the connected complex
$P$, while $C$ is contractible.  The reduced Mayer--Vietoris sequence
for the open cover $S^n=\mathcal{O}_{\mathcal{N}}\cup\mathcal{O}_C$
therefore contains
\[
 0=\widetilde H_1(S^n;\mathbb Z)
 \longrightarrow
 \widetilde H_0(\mathcal{O}_{\mathcal{N}}\cap\mathcal{O}_C;\mathbb Z)
 \longrightarrow
 \widetilde H_0(\mathcal{O}_{\mathcal{N}};\mathbb Z)
 \oplus\widetilde H_0(\mathcal{O}_C;\mathbb Z)=0.
\]
Thus $\mathcal{O}_{\mathcal{N}}\cap\mathcal{O}_C$, and hence $\Sigma$, is
path connected.

Fix $z\in\partial\mathcal{N}$ and let
$\iota:(\partial\mathcal{N},z)\hookrightarrow(\mathcal{N},z)$ be the
inclusion.  We claim that $\iota_*$ is an isomorphism on fundamental
groups.  Choose a PL collar
$c:\partial\mathcal{N}\times[0,2]\to\mathcal{N}$
\cite[Corollary~2.5]{zbMATH01714503}.  Because $P$ is compact and
disjoint from $\partial\mathcal{N}$, restrict and reparametrize this
collar so that its image is disjoint from $P$.  Put
$z'=c(z,1)$, and let $\delta(t)=c(z,t)$, $0\leq t\leq1$.  Pushing a
collar inward shows that
$\operatorname{int}\mathcal{N}\hookrightarrow\mathcal{N}$ is a homotopy equivalence.
Consequently, every element of $\pi_1(\mathcal{N},z)$ is represented by a
loop $\delta*\beta*\overline\delta$, where $\beta$ is a continuous loop
in $\operatorname{int}\mathcal{N}$ based at $z'$.  After relative
simplicial approximation fixing $z'$, we may and do take $\beta$ to be
PL.  Since $z'$ belongs to the
collar, it is disjoint from $P$.  Pass to a common subdivision in which
$\{z'\}\subset\beta(S^1)$ and $P$ are subpolyhedra.  Apply the relative PL
general-position theorem \cite[Theorem~4.2]{zbMATH01714503} in
$\operatorname{int}\mathcal{N}$ to the pair
$(\beta(S^1),\{z'\})$ and the polyhedron $P$.  It gives, in the
nonconstant case, an ambient PL isotopy $\Upsilon_t$ fixed at $z'$ and
satisfying
\[
 \dim\bigl(\Upsilon_1(\beta(S^1)\setminus\{z'\})\cap P\bigr)
 \leq \dim(\beta(S^1)\setminus\{z'\})+\dim P-n
 \leq1+2-n<0.
\]
Thus $\Upsilon_1(\beta(S^1))$ misses $P$.  Since $\Upsilon_t$ fixes
$z'$, replacing $\beta$ by $\Upsilon_1\circ\beta$ does not change its
based homotopy class.  The constant case already misses $P$.  Hence
$\delta*\beta*\overline\delta$ may be chosen in $\mathcal{N}\setminus P$.
The punctured-sublevel deformation retraction
$r:\mathcal{N}\setminus P\to\partial\mathcal{N}$ gives a based homotopy from this loop to
$\iota\circ r\circ(\delta*\beta*\overline\delta)$.  Hence its class lies
in the image of $\iota_*$, proving surjectivity.

For injectivity, let a PL loop $\alpha:S^1\to\partial\mathcal{N}$, based at $z$,
be null-homotopic in $\mathcal{N}$, and put
$\alpha'(w)=c(\alpha(w),1)$ for $w\in S^1$.  The collar homotopy
$(w,t)\mapsto c(\alpha(w),t)$ has basepoint track $\delta$, and hence
\([\alpha']=[\overline\delta*\alpha*\delta]\)
in \(\pi_1(\mathcal{N},z')\).
Since $\alpha$ is null-homotopic in $\mathcal{N}$, this class is trivial.
The isomorphism induced by inclusion
$\pi_1(\operatorname{int}\mathcal{N},z')
\xrightarrow{\ \cong\ }\pi_1(\mathcal{N},z')$
therefore shows that $\alpha'$ is null-homotopic in
$\operatorname{int}\mathcal{N}$.  After extending the boundary
triangulation over $D^2$, relative simplicial approximation
\cite{zbMATH03195022} supplies a PL null-homotopy
$F:(D^2,\partial D^2)\to(\operatorname{int}\mathcal{N},\alpha'(S^1))$ whose
boundary map is $\alpha'$.  Since the collar is disjoint from $P$,
$\alpha'(S^1)\cap P=\varnothing$.  Put
$Q:=F(D^2)$ and $Q_0:=\alpha'(S^1)$.  If $Q=Q_0$, then $F(D^2)$ already
misses $P$; set $\widetilde F:=F$.  If $Q\ne Q_0$, pass to a common
subdivision in which $Q_0\subset Q$ and $P$ are subpolyhedra.  Relative
PL general position in $\operatorname{int}\mathcal{N}$
\cite[Theorem~4.2]{zbMATH01714503}, applied to $(Q,Q_0,P)$, gives an
ambient PL isotopy $\Theta_t$ fixed on $Q_0$ and satisfying
\[
 \dim\bigl(\Theta_1(Q\setminus Q_0)\cap P\bigr)
 \leq\dim(Q\setminus Q_0)+\dim P-n
 \leq2+2-n<0.
\]
Since $Q_0\cap P=\varnothing$, set $\widetilde F:=\Theta_1\circ F$;
this is a null-homotopy of $\alpha'$ in
$\operatorname{int}\mathcal{N}\setminus P$.  In either case, attach to
$\widetilde F$ the
collar annulus
$(w,t)\mapsto c(\alpha(w),t)$, $w\in S^1$ and $0\leq t\leq1$.  Since the collar misses
$P$, the resulting disk is a null-homotopy of $\alpha$ in
$\mathcal{N}\setminus P$.  Composing it with $r$, which fixes $\partial\mathcal{N}$
pointwise, gives a null-homotopy in $\partial\mathcal{N}$.  Hence
$\iota_*$ is injective.
Since $\mathcal{N}$ deformation retracts onto $P$, we have proved
\begin{equation}\label{eq:newman-boundary-group}
 \pi_1(\partial C)=\pi_1(\partial\mathcal{N})
 \cong\pi_1(\mathcal{N})\cong\pi_1(P)=G\ne1.
\end{equation}

Finally, after replacing the factor coordinate $\vartheta\in(0,1]$ by
$1-\vartheta$,
the punctured-sublevel homeomorphism above writes $M=S^n\setminus P$
as $C$ with an open collar $\partial C\times[0,1)$ attached along
$\partial C\times\{0\}$.  Choose an inward PL collar
$c:\partial C\times[0,1]\to C$ with $c(z,0)=z$
\cite[Corollary~2.5]{zbMATH01714503}.  Give the union of the inward and
attached collars the coordinate $u\in[-1,1)$, with $u=-t$ on
$c(\partial C\times[0,1])$ and $u=t$ on the attached collar.  The
homeomorphism $u\mapsto(u-1)/2$ from $[-1,1)$ onto $[-1,0)$ fixes the
endpoint $-1$ and therefore extends by the identity off the inward
collar.  This gives a homeomorphism
$M\cong\operatorname{int}C$.  In particular, $M$ is contractible, as also
follows directly from \eqref{eq:newman-retraction}.  On the punctured
inward collar, introduce a new coordinate $t\in(0,1]$, with
$t\downarrow0$ corresponding to approach to $\partial C$, hence to
infinity in $\operatorname{int}C$.  Choose a homeomorphism
$(0,1]\to[0,\infty)$ with this orientation.  This reparametrization
identifies a cofinal sequence of
connected neighborhoods of infinity in
$\operatorname{int}C$ with
\(\partial C\times(j,\infty)\), \(j=1,2,\ldots\).
With base points chosen along the ray $t\mapsto(z,t)$, projection onto
$\partial C$ identifies the corresponding fundamental-group system
with
\[
 G\xleftarrow{\mathrm{id}}G\xleftarrow{\mathrm{id}}G
 \xleftarrow{\mathrm{id}}\cdots.
\]
By \eqref{eq:newman-boundary-group} this system is not pro-trivial, so
$M$ is not simply connected at infinity
\cite[Example~3.2.16 and Proposition~3.4.36(a)]{zbMATH07206284}.  Simple connectivity at infinity
is a homeomorphism invariant, whereas $\mathbb R^n$ is simply connected
at infinity for $n\geq3$; it follows that
$M\not\cong\mathbb R^n$; this is precisely the conclusion of
\cite[Theorem~3.5.2]{zbMATH07206284}.

It remains to verify \eqref{eq:regular} with $\mathfrak d=2$.  Let
$P_{\mathrm E}\subset\mathbb R^5\subset\mathbb R^n$ be the Euclidean
realization of $P$, and let $\Delta_1,\dots,\Delta_m$ be its
$2$-simplices, each a nondegenerate planar triangle.  In accordance with
the unnormalized Hausdorff convention fixed at the beginning of this
section, put $\mathcal{A}^2:=\frac{\pi}{4}\,\mathcal{H}^2$.
Then $\mathcal{A}^2$ agrees with Euclidean planar area on every affine
$2$-plane.  Distinct $\Delta_i$ meet in common faces, which are
$\mathcal{A}^2$-null, so
$\mathcal{A}^2(P_{\mathrm E})
=\sum_i\operatorname{area}\Delta_i\in(0,\infty)$.
Put
$\mu_{\mathrm E}
:=\frac{\mathcal{A}^2|_{P_{\mathrm E}}}{\mathcal{A}^2(P_{\mathrm E})},\qquad
\mu:=(\sigma_\infty^{-1})_\#\mu_{\mathrm E}$.
Thus the spherical complex already denoted by $P$ is
$\sigma_\infty^{-1}(P_{\mathrm E})$, and $\mu$ is its spherical probability measure.
The map $\sigma_\infty^{-1}$ is smooth and bi-Lipschitz on the compact set
$P_{\mathrm E}$ with respect to Euclidean and chordal distances: there is
$C_{\mathrm{bi}}\ge1$ with
$C_{\mathrm{bi}}^{-1}|x-y|\le \mathfrak q(\sigma_\infty^{-1}(x),\sigma_\infty^{-1}(y))\le C_{\mathrm{bi}}|x-y|$ for
$x,y\in P_{\mathrm E}$.
(For the chart obtained by projecting onto the equatorial hyperplane,
$\mathfrak q(\sigma_\infty^{-1}(x),\sigma_\infty^{-1}(y))
 =2|x-y|\,\bigl((1+|x|^2)(1+|y|^2)\bigr)^{-1/2}$, so
$C_{\mathrm{bi}}=1+R^2$
works if $P_{\mathrm E}\subset B(0,R)$, $R\ge1$.)  Writing $B_{\mathrm E}$
for Euclidean balls, we take $p_{\mathrm E}\in P_{\mathrm E}$, set
$p=\sigma_\infty^{-1}(p_{\mathrm E})$, and let $r>0$; then we obtain
\begin{equation}\label{eq:ball-comparison}
 P_{\mathrm E}\cap B_{\mathrm E}(p_{\mathrm E},r/C_{\mathrm{bi}})
 \subseteq\sigma_\infty\bigl(P\cap B_{\mathfrak q}(p,r)\bigr)
 \subseteq P_{\mathrm E}\cap B_{\mathrm E}(p_{\mathrm E},C_{\mathrm{bi}}r),
\end{equation}
so it suffices to find $c_0,C_0,r_0>0$ with
\[
 c_0r^2\le\mathcal{A}^2\bigl(P_{\mathrm E}\cap
 B_{\mathrm E}(p_{\mathrm E},r)\bigr)\le C_0r^2
 \qquad(p_{\mathrm E}\in P_{\mathrm E},\ 0<r\le r_0).
\]

\emph{Upper estimate.}  For each $i$, the set
$\Delta_i\cap B_{\mathrm E}(p_{\mathrm E},r)$ lies in the intersection
of $B_{\mathrm E}(p_{\mathrm E},r)$ with the affine plane of
$\Delta_i$, which is a planar disc of radius at most $r$.  Hence
$\mathcal{A}^2(\Delta_i\cap B_{\mathrm E}(p_{\mathrm E},r))\le\pi r^2$,
and summing over $i$ gives
$\mathcal{A}^2(P_{\mathrm E}\cap B_{\mathrm E}(p_{\mathrm E},r))
\le m\pi r^2$ for all $r>0$.

\emph{Lower estimate.}  Since $P_{\mathrm E}$ is pure,
$p_{\mathrm E}$ lies in some $\Delta=\Delta_i$.  For $0<t<1$ the
homothety $h_t(y)=p_{\mathrm E}+t(y-p_{\mathrm E})$ maps
the convex set $\Delta$ into itself, satisfies
$|h_t(y)-p_{\mathrm E}|\le t\operatorname{diam}\Delta$, and scales area
by $t^2$.  Thus
$h_t(\Delta)\subseteq\Delta\cap B_{\mathrm E}(p_{\mathrm E},r)$ whenever
$t\operatorname{diam}\Delta<r$, and letting
$t\uparrow r/\operatorname{diam}\Delta$ yields
\[
 \mathcal{A}^2\bigl(P_{\mathrm E}\cap B_{\mathrm E}(p_{\mathrm E},r)\bigr)
 \ge\mathcal{A}^2\bigl(\Delta\cap B_{\mathrm E}(p_{\mathrm E},r)\bigr)
 \ge\frac{\operatorname{area}\Delta}{(\operatorname{diam}\Delta)^2}\,r^2
 \qquad(0<r\le\operatorname{diam}\Delta).
\]
Since there are finitely many simplices, this is the required lower
estimate with
$c_0=\min_i\operatorname{area}\Delta_i/(\operatorname{diam}\Delta_i)^2>0$
and $r_0=\min_i\operatorname{diam}\Delta_i>0$.

Dividing by $\mathcal{A}^2(P_{\mathrm E})$ and using
\eqref{eq:ball-comparison}, we
obtain
\[
 \frac{c_0}{C_{\mathrm{bi}}^2\,\mathcal{A}^2(P_{\mathrm E})}\,r^2
 \le\mu\bigl(B_{\mathfrak q}(p,r)\bigr)
 \le\frac{m\pi C_{\mathrm{bi}}^2}{\mathcal{A}^2(P_{\mathrm E})}\,r^2
 \qquad(p\in P,\ 0<r\le r_0/C_{\mathrm{bi}}),
\]
which is \eqref{eq:regular} with $\mathfrak d=2$.  In particular,
$\mu(B_{\mathfrak q}(p,r))>0$ for every $p\in P$ and
$0<r\leq r_0/C_{\mathrm{bi}}$, and the same
holds for larger $r$ by monotonicity.  Thus
$P\subset\operatorname{supp}\mu$.  Conversely, by construction
$\mu(S^n\setminus P)=0$, and $P$ is closed, so
$\operatorname{supp}\mu\subset P$.  Hence
$\operatorname{supp}\mu=P$, and $(P,\mu)$ is a $2$-regular pair.
\end{proof}

Two lemmas provide the analytic input for Proposition~\ref{prop:metric}.
Lemma~\ref{lem:potential} gives the blow-up estimates needed for
completeness, while Lemma~\ref{lem:kernels} supplies the positivity and
matching growth needed for the scalar-curvature bounds.  In the critical
case $\mathfrak d=a$, the logarithmic correction compensates for the fact that the
pure power $\mathfrak q^{2-n}$ is $L_{g_{st}}$-harmonic off the diagonal.

\begin{lemma}[Ahlfors-regular potential estimate]\label{lem:potential}
Let $(\Lambda,\mu)$ be a $\mathfrak d$-regular pair, put
$\rho=\operatorname{dist}_{\mathfrak q}(\cdot,\Lambda)$, and let
$k:(0,2]\to(0,\infty)$ be continuous with
\[
 k(u)\asymp u^{-\alpha}\ell(u)^{-\lambda}\quad(u\downarrow0),
 \qquad \alpha>\mathfrak d,\quad \lambda\ge0.
\]
Then, for $x\in S^n\setminus\Lambda$,
\begin{equation}\label{eq:potential}
 \int_\Lambda k(\mathfrak q(x,y))\,d\mu(y)
 \asymp \rho(x)^{\mathfrak d-\alpha}\ell(\rho(x))^{-\lambda}
 \qquad(\rho(x)\downarrow0).
\end{equation}
\end{lemma}

\medskip
\noindent\emph{Idea of proof.}
At distance $\delta=\rho(x)$ from $\Lambda$, the portion of $\Lambda$
within distance comparable to $\delta$ has measure comparable to
$\delta^{\mathfrak d}$, while the kernel there has size
$\delta^{-\alpha}\ell(\delta)^{-\lambda}$.  The remaining contribution
is controlled by dyadic annuli, whose estimates form a convergent
geometric series precisely because $\alpha>\mathfrak d$.

\begin{proof}
Recall that $\ell(u)=\log\frac{8}{u}$, so $\ell$ is decreasing,
$\ell\ge\log4$ on $(0,2]$, and
\begin{equation}\label{eq:ell-doubling}
 \ell(2^m u)=\ell(u)-m\log2
 \qquad(m\in\mathbb N_0,\ 2^m u\le2).
\end{equation}
By hypothesis, there are $r_1\in(0,2]$ and $0<c_1\le C_1$ with
\begin{equation}\label{eq:kernel-comparison}
 c_1u^{-\alpha}\ell(u)^{-\lambda}\le k(u)\le
 C_1u^{-\alpha}\ell(u)^{-\lambda}\qquad(0<u\le r_1).
\end{equation}
Set $r_2=\tfrac{1}{2}\min\{r_0,r_1\}$ and assume
$x\in S^n\setminus\Lambda$ and $0<\delta:=\rho(x)<r_2/4$.
Choose $p\in\Lambda$ with $\mathfrak q(x,p)=\delta$ ($\Lambda$ is compact) and
put $\ell_\delta:=\ell(\delta)$.  All constants below may depend on the fixed
regular pair, on $\mathfrak d,\alpha,\lambda$, and on $k$, but not on $x$.

\emph{Lower bound.}  For $y\in B_{\mathfrak q}(p,\delta)\cap\Lambda$ the triangle
inequality gives $\mathfrak q(x,y)\le \mathfrak q(x,p)+\mathfrak q(p,y)<2\delta\le r_1$, while
$\mathfrak q(x,y)\ge\rho(x)=\delta$; hence
$\mathfrak q(x,y)^{-\alpha}\ge(2\delta)^{-\alpha}$ and, since $\ell$ is decreasing
and $\lambda\ge0$, $\ell(\mathfrak q(x,y))^{-\lambda}\ge
\ell_\delta^{-\lambda}$.  Since
$\delta<r_0$, \eqref{eq:regular} gives
$\mu(B_{\mathfrak q}(p,\delta))\ge c\delta^{\mathfrak d}$, so \eqref{eq:kernel-comparison}
yields
\[
 \int_\Lambda k(\mathfrak q(x,y))\,d\mu(y)
 \ge\int_{B_{\mathfrak q}(p,\delta)\cap\Lambda}k(\mathfrak q(x,y))\,d\mu(y)
 \ge c_1c\,2^{-\alpha}\delta^{\mathfrak d-\alpha}
 \ell_\delta^{-\lambda}.
\]

\emph{Upper bound.}  Since $\mathfrak q(x,\cdot)\ge\delta$ on $\Lambda$,
the annuli
\[
A_j=\{y\in\Lambda:2^j\delta\le \mathfrak q(x,y)<2^{j+1}\delta\},
\qquad j\ge0,
\]
are pairwise disjoint and, for every $J\ge0$,
\[
\bigcup_{j=0}^{J}A_j
=\{y\in\Lambda:\mathfrak q(x,y)<2^{J+1}\delta\}.
\]
Let $J$ be maximal with $(2^{J+1}+1)\delta<r_2$; such a finite $J$
exists because $3\delta<r_2$, whereas
$(2^{j+1}+1)\delta\to\infty$ as $j\to\infty$.  For $0\le j\le J$ the
triangle inequality gives
$A_j\subset B_{\mathfrak q}\bigl(p,(2^{j+1}+1)\delta\bigr)$, a ball of radius less
than $r_2<r_0$, so \eqref{eq:regular} and $2^{j+1}+1\le3\cdot2^j$ give
$\mu(A_j)\le3^{\mathfrak d}C(2^j\delta)^{\mathfrak d}$.  On $A_j$ one has
$2^j\delta\le \mathfrak q(x,y)<2^{j+1}\delta\le r_1$, so
\eqref{eq:kernel-comparison} and \eqref{eq:ell-doubling} give
$k(\mathfrak q(x,y))\le
C_1(2^j\delta)^{-\alpha}[\ell_\delta-(j+1)\log2]^{-\lambda}$,
where the bracket equals $\ell(2^{j+1}\delta)>\ell(r_2)>0$.  Hence
\begin{equation}\label{eq:annulus}
 \int_{A_j}k(\mathfrak q(x,y))\,d\mu(y)
 \le C'\delta^{\mathfrak d-\alpha}2^{-j\gamma}\,\ell(2^{j+1}\delta)^{-\lambda}
 \qquad(0\le j\le J),
\end{equation}
where $\gamma=\alpha-\mathfrak d>0$.  Split $\{0,\dots,J\}$ according to whether
$(j+1)\log2\le \ell_\delta/2$.  For the first range, the bracket is at least
$\ell_\delta/2$, and the geometric series
$\sum_{j\ge0}2^{-j\gamma}=(1-2^{-\gamma})^{-1}$ (this is where
$\alpha>\mathfrak d$ enters) bounds that part of the sum of \eqref{eq:annulus}
by $C\delta^{\mathfrak d-\alpha}\ell_\delta^{-\lambda}$.  For the remaining indices,
$(j+1)\log2>\ell_\delta/2$ and
$\ell(2^{j+1}\delta)\ge\ell(r_2)>0$; therefore
\[
 \sum_{\substack{0\le j\le J\\(j+1)\log2>\ell_\delta/2}}
 2^{-j\gamma}\,\ell(2^{j+1}\delta)^{-\lambda}
 \le C\!\!\sum_{j>\ell_\delta/(2\log2)-1}\!\!2^{-j\gamma}
 \le Ce^{-\gamma \ell_\delta/2}
 \le C\ell_\delta^{-\lambda},
\]
the last inequality because
$\sup_{u\ge\log4}u^{\lambda}e^{-\gamma u/2}<\infty$.  Thus this part
also contributes at most
$C\delta^{\mathfrak d-\alpha}\ell_\delta^{-\lambda}$.  Finally,
maximality of $J$ gives $(2^{J+2}+1)\delta\ge r_2$, hence
$2^{J+1}\delta\ge(r_2-\delta)/2\ge3r_2/8$; the remaining set
$\{y\in\Lambda:\mathfrak q(x,y)\ge2^{J+1}\delta\}$ lies at chordal distance at
least $3r_2/8$ from $x$, where the continuous function $k$ is bounded
by $C_{\mathrm{far}}:=\max_{u\in[3r_2/8,2]}k(u)$, so it contributes at most
$C_{\mathrm{far}}\mu(\Lambda)=C_{\mathrm{far}}$, which is absorbed because
$\delta^{\mathfrak d-\alpha}\ell_\delta^{-\lambda}\to\infty$ as
$\delta\downarrow0$.
Together the three parts give the upper bound in
\eqref{eq:potential}.
\end{proof}

\begin{lemma}[Kernel identities]\label{lem:kernels}
For $b>0$ and $x,y\in S^n$ with $x\ne y$,
\begin{equation}\label{eq:power}
 L_{g_{st},x}\mathfrak q(x,y)^{-b}
 =A_{n,b}\mathfrak q(x,y)^{-b}+B_{n,b}\mathfrak q(x,y)^{-b-2},
\end{equation}
where
\[
 A_{n,b}=\frac{(n-1)(n-2-b)(n-b)}{n-2},
 \qquad
 B_{n,b}=\frac{4(n-1)b(n-2-b)}{n-2}.
\]
In particular, both coefficients are positive when $0<b<n-2$.
Moreover, if $\lambda>0$, then
\begin{align}
 L_{g_{st},x}\bigl(\mathfrak q^{2-n}\ell(\mathfrak q)^{-\lambda}\bigr)
 &=\frac{4(n-1)\lambda}{n-2}\mathfrak q^{-n}
   \frac{(n-2)\ell(\mathfrak q)-(\lambda+1)}{\ell(\mathfrak q)^{\lambda+2}}
 \notag\\
 &\quad+\frac{(n-1)\lambda}{n-2}\mathfrak q^{2-n}
   \frac{2\ell(\mathfrak q)+\lambda+1}{\ell(\mathfrak q)^{\lambda+2}}.
 \label{eq:critical}
\end{align}
For $\lambda=a/2=(n-2)/4$, this expression is positive and
\begin{equation}\label{eq:critical-comparison}
 L_{g_{st},x}\bigl(\mathfrak q^{2-n}\ell(\mathfrak q)^{-a/2}\bigr)
 \asymp \mathfrak q^{-n}\ell(\mathfrak q)^{-a/2-1}
 \qquad(0<\mathfrak q\le2).
\end{equation}
\end{lemma}

\begin{proof}
Fix $y\in S^n$ and write
$r=d_{g_{st}}(x,y),\qquad \mathfrak q=2\sin\frac r2$.
For $0<r<\pi$,
\[
 \mathfrak q_r=\cos\frac r2,\qquad
 \mathfrak q_r^2=1-\frac{\mathfrak q^2}{4},\qquad
 \mathfrak q_{rr}=-\frac{\mathfrak q}{4},
 \qquad
 \cot r\,\mathfrak q_r=\frac1{\mathfrak q}-\frac{\mathfrak q}{2},
\]
the last identity because $\sin r=\mathfrak q\,\mathfrak q_r$ and $\cos r=1-\tfrac{\mathfrak q^2}2$.
Hence, if $f=f(\mathfrak q)$ is radial about $y$, then
\begin{align*}
 \Delta_{g_{st}}f
 &=f''\mathfrak q_r^2+f'\bigl(\mathfrak q_{rr}+(n-1)\cot r\,\mathfrak q_r\bigr)\\
 &=\left(1-\frac{\mathfrak q^2}{4}\right)f''
   +\left(\frac{n-1}{\mathfrak q}-\frac{(2n-1)\mathfrak q}{4}\right)f'.
\end{align*}
Although geodesic polar coordinates degenerate at $r=\pi$, both sides
extend continuously there, so the resulting identities hold for every
$0<\mathfrak q\le2$.

For $f(\mathfrak q)=\mathfrak q^{-b}$, one has
$f'=-b\mathfrak q^{-b-1},\qquad
f''=b(b+1)\mathfrak q^{-b-2}$.
Substitution into the preceding formula gives
\[
 \Delta_{g_{st}}\mathfrak q^{-b}
 =b(b-n+2)\mathfrak q^{-b-2}
  +\frac{b(2n-2-b)}4\mathfrak q^{-b}.
\]
Since
$L_{g_{st}}=-\frac{4(n-1)}{n-2}\Delta_{g_{st}}+n(n-1),$
we obtain
\begin{align*}
 L_{g_{st}}\mathfrak q^{-b}
 &=\frac{4(n-1)b(n-2-b)}{n-2}\mathfrak q^{-b-2}\\
 &\quad+
 \frac{n-1}{n-2}
 \bigl[n(n-2)-b(2n-2-b)\bigr]\mathfrak q^{-b}.
\end{align*}
The factorization
$n(n-2)-b(2n-2-b)=(n-2-b)(n-b)$
proves \eqref{eq:power}.  If $0<b<n-2$, then
$n-2-b>0$, $n-b>0$, and $b>0$, so both coefficients are positive.
At $b=n-2$, both coefficients vanish; equivalently,
$\mathfrak q^{2-n}$ is annihilated by $L_{g_{st}}$ off the diagonal.
This is the Green-kernel degeneration underlying the critical case.

For the critical kernel, put
$m=n-2,\qquad t=\ell(\mathfrak q),\qquad
f(\mathfrak q)=\mathfrak q^{-m}t^{-\lambda}$.
Since $t'=-\mathfrak q^{-1}$, direct differentiation gives
\[
 f'
 =\mathfrak q^{-m-1}t^{-\lambda-1}(\lambda-mt)
\]
and
\[
 f''
 =\mathfrak q^{-m-2}t^{-\lambda-2}
 \bigl[m(m+1)t^2-\lambda(2m+1)t+\lambda(\lambda+1)\bigr].
\]
Substituting these derivatives into the radial Laplacian formula and
using $n=m+2$ yields
\[
 \Delta_{g_{st}}f
 =\lambda \mathfrak q^{-n}\frac{(\lambda+1)-mt}{t^{\lambda+2}}
 +\frac14\mathfrak q^{2-n}
   \frac{mn\,t^2-2\lambda t-\lambda(\lambda+1)}{t^{\lambda+2}}.
\]
When $L_{g_{st}}$ is applied, the term containing
$mn\,\mathfrak q^{2-n}t^{-\lambda}$ cancels exactly with
$n(n-1)f$.  The remaining terms are
\begin{align*}
 L_{g_{st}}f
 &=\frac{4(n-1)\lambda}{n-2}\mathfrak q^{-n}
   \frac{(n-2)t-(\lambda+1)}{t^{\lambda+2}}\\
 &\quad+
 \frac{(n-1)\lambda}{n-2}\mathfrak q^{2-n}
   \frac{2t+\lambda+1}{t^{\lambda+2}},
\end{align*}
which is \eqref{eq:critical}.  Finally, take
$\lambda=\frac a2=\frac{n-2}{4}$.  Since $t=\ell(\mathfrak q)\ge\log4$,
\((n-2)t-(\lambda+1)
\ge(n-2)\log4-\frac{n+2}{4}>0\).
Indeed, $(n-2)\log 4-\tfrac{n+2}{4}$ is increasing in $n$ and is already positive at $n=3$, where it equals $\log 4-\tfrac54>0$.  Thus both terms in \eqref{eq:critical} are
positive.  Moreover, for $t\ge\log4$,
\((n-2)t-(\lambda+1)\ge
\Bigl(n-2-\frac{n+2}{4\log4}\Bigr)t\),
and the coefficient in parentheses is positive because it equals
\(\frac{(n-2)\log4-(n+2)/4}{\log4}>0\).
Hence the first term in \eqref{eq:critical} is comparable to
$\mathfrak q^{-n}t^{-\lambda-1}$.  The second term is positive and satisfies
\[
 \mathfrak q^{2-n}\frac{2t+\lambda+1}{t^{\lambda+2}}
 \le C\mathfrak q^{2-n}t^{-\lambda-1}
 =C\mathfrak q^2\mathfrak q^{-n}t^{-\lambda-1}
 \le C\mathfrak q^{-n}t^{-\lambda-1}.
\]
The first term supplies the corresponding lower bound, proving
\eqref{eq:critical-comparison}.
\end{proof}

\begin{proposition}[Metrics from regular pairs]\label{prop:metric}
Let $n\ge3$, let $0<\mathfrak d\le a=(n-2)/2$, and let $(\Lambda,\mu)$ be a
$\mathfrak d$-regular pair in $S^n$ with
$\varnothing\ne\Lambda\subsetneq S^n$.  Define
\[
 \mathcal{K}_{\mathfrak d}(u)=
 \begin{cases}
  u^{-(\mathfrak d+a)},&\mathfrak d<a,\\[2mm]
  u^{-2a}\ell(u)^{-a/2},&\mathfrak d=a,
 \end{cases}
 \qquad
 V(x)=\int_\Lambda \mathcal{K}_{\mathfrak d}(\mathfrak q(x,y))\,d\mu(y),
 \quad x\in S^n\setminus\Lambda.
\]
Then
\begin{equation}\label{eq:metric}
 \widehat g=(1+V)^{4/(n-2)}g_{st}
\end{equation}
is smooth and locally conformally flat on $S^n\setminus\Lambda$; each
connected component is geodesically complete; and its scalar curvature
satisfies
\[
 0<\inf_{S^n\setminus\Lambda}\mathrm{Sc}_{\widehat g}
 \le\sup_{S^n\setminus\Lambda}\mathrm{Sc}_{\widehat g}<\infty.
\]
\end{proposition}

\begin{proof}
\emph{Smoothness.}  Let $Q_{\mathrm{sm}}\subset S^n\setminus\Lambda$ be a
nonempty compact set, so that
$\varepsilon:=\min\{\mathfrak q(x,y):x\in Q_{\mathrm{sm}},\,y\in\Lambda\}>0$.
On $Q_{\mathrm{sm}}\times\Lambda$ one has $\mathfrak q\ge\varepsilon$.  Since
$\mathfrak q^2$ is smooth, $\mathfrak q$ is smooth on a neighborhood of
this compact set.  As $\mathcal{K}_{\mathfrak d}$ is smooth on
$[\varepsilon,2]$, the integrand
$\mathcal{K}_{\mathfrak d}(\mathfrak q(x,y))$ and all its
$x$-derivatives are continuous, hence uniformly bounded on
$Q_{\mathrm{sm}}\times\Lambda$.  As $\mu$ is finite, dominated
convergence justifies differentiation under the integral: $V$ is
smooth and positive on $S^n\setminus\Lambda$ and
\[
 H:=L_{g_{st}}V=\int_\Lambda
 L_{g_{st},x}\mathcal{K}_{\mathfrak d}(\mathfrak q(x,y))\,d\mu(y).
\]
Here $x\notin\Lambda=\operatorname{supp}\mu$, so the kernel identities
are used only off the diagonal; no distributional mass at $x=y$ enters
this calculation.

\emph{Positivity and asymptotics.}  Suppose first that $\mathfrak d<a$ and set
$b=\mathfrak d+a$, so that $0<b<2a=n-2$.  With
\(I_2=\int_\Lambda \mathfrak q(x,y)^{-b-2}\,d\mu(y)\),
identity \eqref{eq:power} has positive coefficients and gives
\(H=A_{n,b}V+B_{n,b}I_2>0\).
The kernels $u^{-b}$ and $u^{-b-2}$ are continuous and positive on
$(0,2]$ and satisfy the hypotheses of Lemma~\ref{lem:potential} with
$(\alpha,\lambda)=(b,0)$ and $(b+2,0)$, respectively; hence, as
$\rho\downarrow0$,
\(V\asymp\rho^{-a}\)
and
\(I_2\asymp\rho^{-a-2}\).
In particular, for $\rho$ sufficiently small,
\(V\le C\rho^{-a}=C\rho^2\rho^{-a-2}\le C'I_2\).
Since $H=A_{n,b}V+B_{n,b}I_2$ and $B_{n,b}>0$, it follows that
\begin{equation}\label{eq:subcritical-asymptotics}
 V\asymp\rho^{-a},
 \qquad
 H\asymp I_2\asymp\rho^{-a-2}.
\end{equation}
If $\mathfrak d=a$, then
$\mathcal{K}_a(\mathfrak q)=\mathfrak q^{2-n}\ell(\mathfrak q)^{-a/2}$,
and by
Lemma~\ref{lem:kernels} its $L_{g_{st}}$-image is again a continuous
positive function of $\mathfrak q$ alone, comparable to
$\mathfrak q^{-n}\ell(\mathfrak q)^{-a/2-1}$ on $(0,2]$ by
\eqref{eq:critical-comparison}.  Hence $H>0$, and
Lemma~\ref{lem:potential}, applied with
$(\alpha,\lambda)=(2a,\tfrac a2)$ to $V$ and with
$(\alpha,\lambda)=(n,\tfrac a2+1)$ to $H$ (both admissible, as
$\alpha>\mathfrak d=a$), gives
\begin{equation}\label{eq:critical-asymptotics}
 V\asymp\rho^{-a}\ell(\rho)^{-a/2},
 \qquad
 H\asymp\rho^{-a-2}\ell(\rho)^{-a/2-1}.
\end{equation}

\emph{The comparison $H\asymp V^\kappa$.}  Since
$\kappa=\frac{a+2}a$, one has
$a\kappa=a+2,\qquad \frac a2\kappa=\frac a2+1,$
so raising the $V$-asymptotics in
\eqref{eq:subcritical-asymptotics}--\eqref{eq:critical-asymptotics} to
the power $\kappa$ reproduces exactly the $H$-asymptotics: in both
cases there is $\rho_1>0$ with $H\asymp V^\kappa$ on
$\{0<\rho\le\rho_1\}$.  The set $\{\rho\ge\rho_1\}$ is a compact
subset of $S^n\setminus\Lambda$, on which $H$ and $V^\kappa$ are
continuous and positive; hence $H/V^\kappa$ is bounded between
positive constants there as well, and $H\asymp V^\kappa$ holds on all
of $S^n\setminus\Lambda$.

\emph{Curvature.}  Set $U=1+V\ge1$.  Since $L_{g_{st}}1=n(n-1)$,
\(L_{g_{st}}U=n(n-1)+H\asymp1+V^\kappa\).
Since $\kappa>1$,
\(1+V^\kappa\le(1+V)^\kappa\le2^{\kappa-1}\bigl(1+V^\kappa\bigr)\),
the second inequality by convexity of $x\mapsto x^\kappa$.
Consequently,
\(L_{g_{st}}U\asymp1+V^\kappa\asymp U^\kappa\).
As $U$ is smooth and $U\ge1$, \eqref{eq:metric} defines a smooth
metric conformal to $g_{st}$, and hence it is locally conformally flat.
Since $a=\frac{n-2}{2}$ and
$\kappa=\frac{n+2}{n-2}=\frac{a+2}{a}$,
we have $\widehat g=U^{4/(n-2)}g_{st}=U^{2/a}g_{st}$.
For
$L_{g_{st}}
=-\frac{4(n-1)}{n-2}\Delta_{g_{st}}+n(n-1),$
the conformal scalar-curvature formula gives
\(\mathrm{Sc}_{\widehat g}
=U^{-(n+2)/(n-2)}L_{g_{st}}U
=U^{-\kappa}L_{g_{st}}U\).
The comparison $L_{g_{st}}U\asymp U^\kappa$ means that, for some
constants $0<c\le C<\infty$,
\(cU^\kappa\le L_{g_{st}}U\le CU^\kappa\)
on \(S^n\setminus\Lambda\).
Since $U>0$, multiplying by $U^{-\kappa}$ yields
$c\le\mathrm{Sc}_{\widehat g}(x)\le C$ for every
$x\in S^n\setminus\Lambda$.  Thus
\(0<c\leq
\inf_{S^n\setminus\Lambda}\mathrm{Sc}_{\widehat g}
\le \sup_{S^n\setminus\Lambda}\mathrm{Sc}_{\widehat g}
\leq C<\infty\).

\emph{Completeness.}  The inequality $U\ge V$ and the lower bounds in
\eqref{eq:subcritical-asymptotics}--\eqref{eq:critical-asymptotics} show that,
after shrinking $\rho_1$, there is $c>0$ such that on
$\{0<\rho\le\rho_1\}$,
\[
 U^{1/a}\ge
 \begin{cases}
  c\rho^{-1},&\mathfrak d<a,\\[1mm]
  \displaystyle\frac{c}{\rho\sqrt{\ell(\rho)}},&\mathfrak d=a.
 \end{cases}
\]
Note that $\widehat g=U^{2/a}g_{st}$, so
$ds_{\widehat g}=U^{1/a}ds_{g_{st}}$, and that $\rho$ is
$1$-Lipschitz for $g_{st}$, because
$|\rho(x)-\rho(x')|\le \mathfrak q(x,x')\le d_{g_{st}}(x,x')$.

Let $\gamma:[0,T)\to S^n\setminus\Lambda$ be a divergent locally
Lipschitz curve; thus, for every compact $Q\subset S^n\setminus\Lambda$,
there is $t_Q\in[0,T)$ such that $\gamma(t)\notin Q$ for all $t>t_Q$.
For every $\varepsilon>0$, the set
$K_\varepsilon=\{x\in S^n:\rho(x)\ge\varepsilon\}$ is a compact
subset of $S^n\setminus\Lambda$.  Since $\gamma$ is
divergent, it eventually avoids every $K_\varepsilon$; hence
$\rho(\gamma(t))\to0$ as $t\to T$, and there is $t_0$ with
$\rho(\gamma(t))\le\rho_1$ for $t\ge t_0$.  Write
$\rho(t):=\rho(\gamma(t))$.  This function is locally Lipschitz, and
$|\rho'(t)|\le|\gamma'(t)|_{g_{st}}$ almost everywhere.  For each
$t<T$, the positive continuous function $\rho$ has a positive minimum
on $[t_0,t]$; hence the absolutely continuous chain rule applies to
$\log\rho$ and to $\sqrt{\ell(\rho)}$ on this interval.

In the subcritical case, for $t_0\le t<T$,
\[
\begin{aligned}
 \operatorname{Len}_{\widehat g}
 \bigl(\gamma|_{[t_0,t]}\bigr)
 &=\int_{t_0}^{t}U(\gamma(\tau))^{1/a}
   |\gamma'(\tau)|_{g_{st}}\,d\tau\\
 &\ge c\int_{t_0}^{t}\frac{|\rho'(\tau)|}{\rho(\tau)}\,d\tau\\
 &\ge c\left|\int_{t_0}^{t}
   \frac{\rho'(\tau)}{\rho(\tau)}\,d\tau\right|\\
 &=c\bigl|\log\rho(t)-\log\rho(t_0)\bigr|
 \xrightarrow[t\uparrow T]{}+\infty.
\end{aligned}
\]
In the critical case, since
$\frac{d}{dr}\bigl(-2\sqrt{\ell(r)}\bigr)
=\frac1{r\sqrt{\ell(r)}},$
the same argument gives
\[
\begin{aligned}
 \operatorname{Len}_{\widehat g}
 \bigl(\gamma|_{[t_0,t]}\bigr)
 &=\int_{t_0}^{t}U(\gamma(\tau))^{1/a}
   |\gamma'(\tau)|_{g_{st}}\,d\tau\\
 &\ge c\int_{t_0}^{t}
   \frac{|\rho'(\tau)|}
        {\rho(\tau)\sqrt{\ell(\rho(\tau))}}\,d\tau\\
 &\ge c\left|\int_{t_0}^{t}
   \frac{\rho'(\tau)}
        {\rho(\tau)\sqrt{\ell(\rho(\tau))}}\,d\tau\right|\\
 &=2c\bigl|\sqrt{\ell(\rho(t))}
          -\sqrt{\ell(\rho(t_0))}\bigr|
 \xrightarrow[t\uparrow T]{}+\infty.
\end{aligned}
\]
In both cases the terminal limit follows from
$\rho(t)\to0$ as $t\uparrow T$: one has
$\log\rho(t)\to-\infty$ and
$\ell(\rho(t))=\log(8/\rho(t))\to+\infty$.
Thus every
divergent locally Lipschitz curve has infinite $\widehat g$-length.

Suppose that some connected component were not geodesically complete.
By the unit-speed formulation of geodesic completeness, one could, after translating
and reversing the parameter if necessary, choose a unit-speed geodesic
$\gamma:[0,T)\to S^n\setminus\Lambda$, with $T<\infty$, that is not
extendible past $T$.  This geodesic must be divergent.  Indeed, otherwise
there would be a compact set $Q\subset S^n\setminus\Lambda$ and a
sequence of times tending to $T$ whose images under $\gamma$ lie in $Q$.
The corresponding unit tangent vectors lie in the compact unit tangent
bundle over $Q$ and therefore have a convergent subsequence.
Local existence for the smooth geodesic vector field then supplies a
common positive existence time for all sufficiently late terms of this
subsequence, and
uniqueness extends $\gamma$ beyond $T$, contradicting its
nonextendibility.
Thus $\gamma$ is a divergent smooth, hence locally Lipschitz, curve, but
its $\widehat g$-length is $T<\infty$, contradicting the preceding
paragraph.  Every connected component is therefore geodesically
complete.
\end{proof}

\begin{proof}[Proof of Theorem~\ref{thm:uniform-psc-counterexample}]
Let $P\subset S^n$, $M=S^n\setminus P$, and $\mu$ be supplied by
Lemma~\ref{lem:topological}.  Then $(P,\mu)$ is $2$-regular, and
the measure has support equal to the nonempty proper subset \(P\subset S^n\).  Moreover,
the analytic hypothesis $\mathfrak d=2\le a=(n-2)/2$ holds because
$n\geq 6$.  The same lemma shows that $M$ is
contractible, and hence connected, and that
$M\not\cong\mathbb R^n$.  Thus $M$ is the unique connected component
of $S^n\setminus P$.  Since the finite complex $P$ is closed, $M$ is an
open subset of the smooth sphere and hence is a smooth open manifold.
Proposition~\ref{prop:metric} therefore gives
the required geodesically complete metric on $M$.
\end{proof}

\begin{question}\label{q:lowdim}
For $n\in\{4,5\}$, is there a contractible open $n$-manifold
$M\not\cong\mathbb R^n$ carrying a complete locally conformally flat
metric with $\mathrm{Sc}\ge1$?
\end{question}

\begin{proof}[Proof of Theorem~\ref{F}]
Fix $n\geq4$ and choose any $t>0$ as in Theorem~\ref{thm:counterexample-psc}.  The metric
$g^{(t)}$ is complete and locally conformally flat, and its scalar
curvature is strictly positive; its underlying manifold is contractible
and is not homeomorphic to $\mathbb R^n$.  This proves the first
assertion.  For every $n\geq4$, the uniform-decay conclusion of
Theorem~\ref{thm:counterexample-psc}, together with
$\mathrm{Sc}_{g^{(t)}}=
t^{-4/(n-2)}\mathrm{Sc}_{\bar g^{(t)}}$,
shows that $\mathrm{Sc}_{g^{(t)}}(x)\to0$ as
$d_{g_{st}}(x,K)\to0$.  By Proposition~\ref{prop:limitset},
$\dim_{\mathcal{H}}K=1<n$, so $K$ has empty interior in $S^n$.
Fix $p\in K$.  Since $K$ has empty interior, for every $i\geq1$ one may
choose $x_i\in B_{g_{st}}(p,1/i)\setminus K\subset M$.
Then $d_{g_{st}}(x_i,K)<1/i$, and hence
$\mathrm{Sc}_{g^{(t)}}(x_i)\to0$.  Since
$\mathrm{Sc}_{g^{(t)}}>0$ on $M$, it follows that
$\inf_M\mathrm{Sc}_{g^{(t)}}=0$.
Thus the shifted metrics of
Theorem~\ref{thm:counterexample-psc} are not uniformly PSC in any dimension.  For
$n\geq6$, however, Theorem~\ref{thm:uniform-psc-counterexample} supplies a manifold of the
same kind carrying a complete locally conformally flat metric whose
scalar curvature has a positive lower bound, proving the second
assertion. 
\end{proof}

\addcontentsline{toc}{section}{\refname}
\bibliographystyle{alpha}
\bibliography{reference}

\end{document}